\documentclass[11pt,a4paper,leqno]{article}
\usepackage[latin1]{inputenc}
\usepackage{amsmath,amsthm}
\usepackage{amsfonts,amssymb,latexsym}
\usepackage{graphicx}
\usepackage[french]{babel}
\usepackage[all]{xy}

\DeclareMathOperator{\Spec}{Spec}

\DeclareMathOperator{\et}{et}

\DeclareMathOperator{\verifiant}{v\acute{e}rifiant}

\DeclareMathOperator{\rg}{rg}

\DeclareMathOperator{\nor}{nor}
\DeclareMathOperator{\Gal}{Gal}

\DeclareMathOperator{\Lip}{Lip}

\DeclareMathOperator{\eff}{eff}

\DeclareMathOperator{\Val}{Val}

\DeclareMathOperator{\Card}{Card}
\DeclareMathOperator{\Hom}{Hom}

\DeclareMathOperator{\PGCD}{pgcd}

\DeclareMathOperator{\Pic}{Pic}

\DeclareMathOperator{\di}{div}

\DeclareMathOperator{\card}{card}

\DeclareMathOperator{\dd}{d}

\DeclareMathOperator{\Eff}{Eff}

\DeclareMathOperator{\Vol}{Vol}
\newcommand{\CC}{\mathbf{C}}
\newcommand{\RR}{\mathbf{R}}
\newcommand{\ZZ}{\mathbf{Z}}
\newcommand{\NN}{\mathbf{N}}
\newcommand{\QQ}{\mathbf{Q}}
\newcommand{\xx}{\boldsymbol{x}}

\newcommand{\yy}{\boldsymbol{y}}
\newcommand{\kk}{\boldsymbol{k}}
\newcommand{\zz}{\boldsymbol{z}}
\newcommand{\qq}{\boldsymbol{q}}

\newcommand{\rr}{\boldsymbol{r}}
\newcommand{\ww}{\boldsymbol{w}}
\renewcommand{\tt}{\boldsymbol{t}}
\newcommand{\jj}{\boldsymbol{j}}

\renewcommand{\aa}{\boldsymbol{a}}
\newcommand{\aalpha}{\boldsymbol{\alpha}}

\newcommand{\0}{\boldsymbol{0}}
\newcommand{\ii}{\boldsymbol{i}}
\newcommand{\uu}{\boldsymbol{u}}
\newcommand{\vv}{\boldsymbol{v}}
\newcommand{\ee}{\boldsymbol{\varepsilon}}
\newcommand{\bb}{\boldsymbol{b}}
\newcommand{\ra}{\rightarrow}
\newcommand{\mt}{\mapsto}

\newcommand{\PP}{\mathbf{P}}

\renewcommand{\AA}{\mathbf{A}}
\newcommand{\OO}{\mathcal{O}}
\newcommand{\MM}{\mathcal{M}}
\newcommand{\BB}{\mathcal{B}}
\newcommand{\II}{\mathcal{I}}
\newcommand{\JJ}{\mathcal{J}}

 \newtheorem{thm}{Th\'eor\`eme}[section]
 \newtheorem{prop}[thm]{Proposition}
 \newtheorem{lemma}[thm]{Lemme}
 \newtheorem{cor}[thm]{Corollaire}
 \newtheorem{Def}[thm]{D\'efinition}
  \newtheorem{rem}[thm]{Remarque}

  \usepackage[refpage]{nomencl}

\makenomenclature

\usepackage{makeidx}
\makeindex

\begin{document}

\title{POINTS DE HAUTEUR BORN\'EE \\ SUR LES HYPERSURFACES LISSES \\
DES VARI\'ET\'ES TORIQUES}

\author{ Teddy Mignot}

\maketitle

\begin{abstract}
Nous d\'emontrons ici la conjecture de Batyrev et Manin pour le nombre de points de hauteur born\'ee des hypersurfaces de certaines vari\'et\'es toriques dont le rang du groupe de Picard est $ 2 $. La m\'ethode utilis\'ee est inspir\'ee de celle d\'evelopp\'ee par Schindler pour le cas des hypersurfaces de l'espace biprojectif, qui elle-m\^eme s'inspire de la m\'ethode du cercle de Hardy-Littlewood. La constante obtenue dans la formule asymptotique finale est exactement celle conjectur\'ee par Peyre. 
\end{abstract}

\tableofcontents

\section{Introduction}

On consid\`ere une vari\'et\'e torique compl\`ete lisse $ X=X(\Delta) $ de dimension $ n $ d\'efinie par le r\'eseau $ N=\ZZ^{n} $ et un \'eventail $ \Delta $ ayant $ n+2 $ ar\^etes engendr\'ees par des vecteurs not\'es $ v_{0},v_{1},...,v_{n},v_{n+1}\in \RR^{n} $. De telles vari\'et\'es ont \'et\'e classifi\'ees par Kleinschmidt dans \cite{K}. Nous supposerons par ailleurs que le groupe de Picard et le c\^one effectif de $ X $ sont engendr\'es par les classes de diviseurs associ\'es aux ar\^etes de $ 2 $ vecteurs g\'en\'erateurs de l'\'eventail, disons $ v_{0} $ et $ v_{n+1} $. On note note $ D_{0} $ et $ D_{n+1} $ les diviseurs associ\'es, et $ [D_{0}] $, $ [D_{n+1}] $ leurs classes dans $ \Pic(X) $.  
On peut alors \'ecrire 

\[ \Pic(X)=\ZZ[D_{0}]\oplus\ZZ[D_{n+1}], \]
\[ C^{1}_{\Eff}=\RR^{+}[D_{0}]+\RR^{+}[D_{n+1}], \]
et la classe du diviseur anticanonique de $ X $ est de la forme \[ [-K_{X}]=n_{1}[D_{0}]+n_{2}[D_{n+1}] \] avec $ n_{1},n_{2}\in \ZZ $. D'autre part, pour $ d_{1},d_{2}\in \NN $ fix\'es consid\'erons un diviseur de classe $ d_{1}[D_{0}]+d_{2}[D_{n+1}] $ et une hypersurface $ Y $ de dimension suppos\'ee sup\'erieure ou \'egale \`a $ 3 $, d\'efinie par une section de ce diviseur. On supposera que l'hypersurface choisie est lisse. La classe du diviseur anticanonique de $ Y $ est alors donn\'ee par \[  [-K_{Y}]=(n_{1}-d_{1})[\tilde{D}_{0}]+(n_{2}-d_{2})[\tilde{D}_{n+1}], \] o\`u $ \tilde{D}_{0} $ et $ \tilde{D}_{n+1} $ d\'esignent les diviseurs induits par $ D_{0} $ et $ D_{n+1} $ sur $ Y $. En utilisant par exemple la construction d\'ecrite par Salberger dans \cite[\S 10]{Sa}, on peut construire explicitement la hauteur $ H $ sur $ X $ associ\'ee \`a $ (n_{1}-d_{1})[D_{0}]+(n_{2}-d_{2})[D_{n+1}] $. Elle induit une hauteur sur $ Y $ qui est la hauteur associ\'ee \`a $ [-K_{Y}] $, et que l'on notera encore $ H $. L'objectif est alors de donner une formule asymptotique pour le nombre \[ \mathcal{N}_{U}(B)=\Card\{ P\in Y(\QQ)\cap U \; | \; H(P)\leqslant B \}, \] pour un ouvert $ U $ bien choisi. Plus pr\'ecis\'ement nous allons montrer (cf. proposition \ref{conclusion}) que $ \mathcal{N}_{U}(B) $ v\'erifie la conjecture de Manin, i.e que pour un nombre de variables $ n+2 $ assez grand (condition analogue \`a celle donn\'ee par Birch dans \cite{Bi} pour les hypersurfaces de l'espace projectif), ce cardinal  est de la forme \[\mathcal{N}_{U}(B)=  C_{H}(Y)B\log(B)+O(B), \] o\`u $ C_{H}(Y) $ est la constante conjectur\'ee par Peyre. \\
 
Dans la section $ 2 $ nous fixons pr\'ecis\'ement le cadre de notre \'etude. Nous y d\'ecrivons entre autres les vari\'et\'es toriques auxquelles nous nous int\'eresserons, l'expression de la hauteur, et la forme des \'equations d\'efinissant les hypersurfaces. Nous montrons par ailleurs que le calcul de $ \mathcal{N}_{U}(B) $ peut se ramener \`a celui de  \begin{multline*} N_{d,U}(B)=\card\left\{ (\xx,\yy,\zz)\in (\ZZ^{r+1}\times \ZZ^{m-r}\times\ZZ^{n-m+1}) \cap U\; | \; \xx\neq\0, \; \right. \\ \left. \; (\yy,\zz)\neq (\0,\0), \; F(d\xx,\yy,\zz)=0, \;  |\xx|^{m+1-d_{1}}\max\left( \frac{|\yy|}{d|\xx|},|\zz|\right)^{n-r+1-d_{2}} \leqslant B \right\} \end{multline*}
o\`u $ m,r,d_{1},d_{2} $ sont des entiers fix\'es, et $ F $ un polyn\^ome homog\`ene de degr\'e $ d_{1} $ (resp. $ d_{2} $) en $ (\xx,\yy) $ (resp. $ (\yy,\zz) $). \\

La m\'ethode utilis\'ee pour \'evaluer les $ N_{d,U}(B )$ est fortement inspir\'ee de celle d\'evelopp\'ee par Schindler dans \cite{S2} pour traiter le cas des hypersurfaces des espaces biprojectifs. Cette m\'ethode consiste dans un premier temps \`a donner une formule asymptotique pour le nombre $ N_{d,U}(P_{1},P_{2}) $ de points $ (\xx,\yy,\zz)  $ de $ U\cap \ZZ^{n+2}$ tels que $ |\xx|\leqslant P_{1} $ et  $ \max\left( \frac{|\yy|}{d|\xx|},|\zz|\right)\leqslant P_{2} $ pour des bornes $ P_{1},P_{2} $ fix\'ees. Dans la section $ 3 $, en utilisant des arguments issus de la m\'ethode du cercle, on \'etablit une formule asymptotique pour $ N_{d,U}(P_{1},P_{2}) $ lorsque $ P_{1} $ et $ P_{2} $ sont \og relativement proches \fg \ en un sens que nous pr\'eciserons. Dans la section $ 4 $ (resp. $ 5 $), pour un $ \xx\in \ZZ^{r+1} $ (resp. $ \zz\in \ZZ^{n-m+1} $) fix\'e, on donne une formule asymptotique pour le nombre de points $ (\yy,\zz) $ (resp. $ (\xx,\yy) $) v\'erifiant $ F(\xx,\yy,\zz)=0 $ tels que $ \max\left( \frac{|\yy|}{d|\xx|},|\zz|\right)\leqslant P_{2} $ (resp. $ |\xx|\leqslant P_{1} $) en utilisant \`a nouveau la m\'ethode du cercle. Les r\'esultats obtenus combin\'es avec ceux de la section $ 2 $ nous permettrons dans la section $ 6 $ d'\'etablir une formule asymptotique pour $ N_{d,U}(P_{1},P_{2}) $ avec $ P_{1},P_{2} $ quelconques. Dans la section $ 7 $, on utilise les r\'esultats \'etablis par Blomer et Br\"{u}dern dans \cite{BB} pour conclure quant \`a la valeur de $ N_{d,U}(B) $ \`a partir des estimations obtenues dans les sections pr\'ec\'edentes. Enfin, dans la section $ 8 $, on conclut en d\'emontrant le th\'eor\`eme\;\ref{thmconcl} donnant une formule asymptotique pour $ \mathcal{N}_{U}(B) $. On v\'erifie en particulier que la constante obtenue est bien celle avanc\'ee par Peyre dans \cite{Pe}.

\section{Pr\'eliminaires}

\subsection{Notations et premi\`eres propri\'et\'es}\label{premprop}

Rappelons les d\'efinitions suivantes : 
\begin{Def}
\'Etant donn\'e un r\'eseau $ N $, un \emph{\'eventail} est un ensemble $ \Delta $ de c\^ones polyh\'edraux de $ N_{\RR}=N\otimes \RR $ v\'erifiant : 
\begin{enumerate}
\item Pour tout c\^one $ \sigma\in \Delta $, on a $ 0\in \sigma $;
\item Toute face d'un c\^one de $ \Delta $ est un c\^one de $ \Delta $;
\item L'intersection de deux c\^ones de $ \Delta  $ est une face de chacun de ces deux c\^ones. 
\end{enumerate}
On dit de plus que l'\'eventail est \begin{itemize}
\item \emph{complet} si $ \bigcup_{\sigma\in \Delta}\sigma=N_{\RR} $,
\item \emph{r\'egulier} si chaque c\^one de $ \Delta $ est engendr\'e par une famille de vecteurs pouvant \^etre compl\'et\'ee en une base de $ N_{\RR} $. 
\end{itemize}
\end{Def}
Pour tout \'eventail $ \Delta $ nous noterons $ \Delta_{\max} $ l'ensemble des c\^ones de dimension maximale, et pour tout c\^one $ \sigma\in \Delta $, on notera $ \sigma(1) $ l'ensemble des vecteurs g\'en\'erateurs des ar\^etes de $ \sigma $). Pour un c\^one polyh\'edral $ \sigma $ de $ N_{\RR} $ donn\'e on d\'efinit semi-groupe \[ S_{\sigma}=\sigma^{\vee}\cap N^{\vee}, \]
o\`u $ \sigma^{\vee} $ (resp. $ N^{\vee}=M $) d\'esigne le c\^one (resp. r\'eseau) dual de $ \sigma $ (resp. $ N $). La \emph{ vari\'et\'e torique affine } sur un corps $ k $ associ\'ee \`a $ \sigma $ est la vari\'et\'e affine : \begin{equation}
U_{\sigma}=\Spec(k[S_{\sigma}])
\end{equation}
On remarque que si $ \sigma, \tau $ sont deux c\^ones de $ N_{\RR} $, alors \[ \tau \subset \sigma \Rightarrow U_{\tau}\subset U_{\sigma}. \]
\'Etant donn\'e un r\'eseau $ N $ et un \'eventail $ \Delta $, on d\'efinit une vari\'et\'e alg\'ebrique $ X=X(\Delta) $ sur $ k $ par recollement des ouverts $ U_{\sigma} $ pour $ \sigma\in \Delta $. Nous renvoyons le lecteur \`a \cite[\S 1,2,3]{F} pour plus de d\'etails sur les vari\'et\'es toriques. Remarquons que la vari\'et\'e $ X(\Delta) $ est lisse (resp. compl\`ete) si $ \Delta $ est r\'egulier (resp. complet).  \\

Dans ce qui va suivre nous allons consid\'erer $ X $ une vari\'et\'e torique de dimension $ n $ d\'efinie par un \'eventail $ \Delta $ \`a $ d=n+r $ ar\^etes dont les g\'en\'erateurs seront not\'es dans cette section, $ v_{1},v_{2},...,v_{n}, v_{n+1},...,v_{n+r}\in \ZZ^{n} $, et un r\'eseau $ N=\ZZ^{n} $. On note $ D_{1},...,D_{n},...,D_{n+r} $ les diviseurs associ\'es aux vecteurs g\'en\'erateurs (voir \cite[\S 3.3]{F}). Rappelons que dans le cas o\`u la vari\'et\'e torique $ X $ est lisse, le groupe de Picard de $ X $ est de rang $ r $. Pour simplifier nous allons imposer une premi\`ere condition aux vari\'et\'es torique que nous consid\`ererons : nous nous int\'eresserons exclusivement aux vari\'et\'es toriques compl\`etes lisses dont le c\^one effectif est simplicial et que tout diviseur effectif soit combinaison lin\'eaire de $ r $ diviseurs $ D_{i} $, disons $ [D_{n+1}],...,[D_{n+r}] $. Une premi\`ere question naturelle est de se demander si ceci peut se traduire en termes de propri\'et\'es sur les c\^ones de l'\'eventail. Nous allons r\'epondre \`a cette question dans ce qui va suivre. \\

On souhaite donc avoir, pour tout $ i\in \{1,...,n\} $  \[ [D_{i}]=\sum_{j=1}^{r}a_{i,j}[D_{n+j}] \] avec $ a_{i,j}\in \NN $ pour tous $ i,j $. Ceci \'equivaut \`a dire qu'il existe des entiers naturels $ a_{i,j}$ tels que les diviseurs $ D_{i}-\sum_{j=1}^{r}a_{i,j}D_{n+j} $ soient principaux pour tous $ i\in \{1,...,n\} $. Rappelons que les diviseurs principaux de $ X $ sont exactement les diviseurs $ \di(\chi^{u}) $ associ\'es aux caract\`eres $ \chi^{u} $ du tore de $ X $ (voir \cite{F}) 
 pour $ u\in M=N^{\vee}=\ZZ^{n} $ d\'efinis par : \[  \di(\chi^{u})=\sum_{k=1}^{n+r}\langle u,v_{k}\rangle D_{k}. \] On cherche donc des vecteurs $ u_{1},...,u_{n}\in \ZZ^{n} $ tels que pour tous $ i,j \in \{1,...,n\} $, \begin{equation}\label{Cond1}   \langle u_{i},v_{j}\rangle=\delta_{i,j}  \end{equation} (i.e $ (u_{1},...,u_{n}) $ est la base duale de $ (v_{1},...,v_{n}) $ au sens des espaces vectoriels) et \begin{equation}\label{Cond2}
 \langle u_{i},v_{k}\rangle\leqslant 0 
 \end{equation} pour tout $ k\in \{n+1,...,n+r\} $. Quitte \`a permuter les $ v_{i} $, on peut supposer que $ (v_{1},...,v_{n}) $ est une famille g\'en\'eratrice d'un c\^one maximal (i.e. de dimension $ n $) de $ \Delta $. Puisque l'on a suppos\'e que $ X $ est lisse, $ (v_{1},...,v_{n}) $ est alors une base du r\'eseau $ \ZZ^{n} $ dont $ (u_{1},...,u_{n}) $ est la base duale (au sens des r\'eseaux). La condition \eqref{Cond2} impose d'autre part que pour cette base duale $ (u_{1},...,u_{n}) $ : \[ \forall k \in \{n+1,...,n+r\}, \; \;  \langle u_{i},v_{k}\rangle\leqslant 0. \] Une condition n\'ecessaire et suffisante pour que ceci soit v\'erifi\'e est que : \[ v_{n+1},...,v_{n+r}\in C\langle -v_{1},-v_{2},...,-v_{n}\rangle \] o\`u $ C\langle -v_{1},-v_{2},...,-v_{n}\rangle $ d\'esigne le c\^one de $ \RR^{n} $ engendr\'e par $ -v_{1},...,-v_{n} $. 
 \begin{rem}
 Si l'on note \[ \forall k\in \{1,...,r\}, \; \;  v_{n+k}=-\sum_{i=1}^{n}a_{i,k}v_{i}, \] avec $ a_{i,k}\in \NN $, on v\'erifie qu'alors : \[ \forall i\in \{1,...,n\}, \; \;  [D_{i}]=\sum_{k=1}^{r}a_{i,k}[D_{n+k}]. \]
 \end{rem}

\subsection{Hauteurs sur les hypersurfaces de vari\'et\'es toriques}\label{hauteur}

\'Etant donn\'ee une vari\'et\'e torique compl\`ete lisse $ X $ d\'efinie par un \'eventail $ \Delta $ \`a $ n+r $ ar\^etes et un r\'eseau $ N=\ZZ^{n} $, dont le groupe de Picard et le c\^one effectif sont engendr\'es par $ [D_{n+1}],...,[D_{n+r}] $ (cf. section pr\'ec\'edente), on consid\`ere la classe du diviseur anticanonique de $ X $ qui sera de la forme : \[ [-K_{X}]=\sum_{i=1}^{n+r}[D_{i}]=\sum_{k=1}^{r}n_{k}[D_{n+k}], \] avec $ n_{1},...,n_{r}\in \NN $. On consid\`ere alors un diviseur de classe $ \sum_{k=1}^{r}d_{k}[D_{n+k}] $, avec $ d_{1},...,d_{r}\in \NN $. Une section globale $ s $ du fibr\'e en droites associ\'e \`a ce diviseur sur $ X $ permet de d\'efinir une hypersurface de $ X $ que l'on notera $ Y $. La classe du diviseur anticanonique sur $ Y $ sera induite par la classe du diviseur \begin{equation}\label{DiviseurD0} D_{0}=\sum_{k=1}^{r}(n_{k}-d_{k})D_{n+k}. \end{equation} Nous allons donner une construction de la hauteur associ\'ee \`a $ \OO(D_{0}) $ sur $ X $. Pour cela, nous utiliserons la construction des hauteurs sur les vari\'et\'es toriques d\'ecrite par Salberger dans \cite{Sa}. \\

Soit $ \nu $ une place sur $ \QQ $, et $ |.|_{\nu} : \QQ^{\ast}\ra \RR^{+} $  la valeur absolue associ\'ee. On pose, comme dans la section pr\'ec\'edente, $ N=\ZZ^{n} $, $ M=N^{\vee}=\ZZ^{n} $ et $ U(\QQ_{\nu}) $ le tore $ \Hom(M,\QQ^{\ast}_{\nu}) $ qui peut \^etre identifi\'e avec un ouvert dense de Zariski de $ X(\QQ_{\nu}) $ \`a condition de fixer un point de cet ouvert. L'application $ \log|.|_{\nu} : \QQ_{\nu}^{\ast}\ra \RR $ induit un  morphisme \[ L : U(\QQ_{\nu}) \ra N_{\RR}= \RR^{n}. \] Pour tout $ \sigma\in \Delta $, $ L^{-1}(-\sigma) $ est un sous-ensemble ferm\'e de $ U(\QQ_{\nu}) $. On note alors $ C_{\sigma,\nu} $ l'adh\'erence de $ L^{-1}(-\sigma)  $ dans $ X(\QQ_{\nu}) $. On utilise ces ensembles $ C_{\sigma,\nu} $ pour construire une norme $ ||.||_{D,\nu} $ sur $ \OO(D) $ pour tout diviseur de Weil $ D $ sur $ X $, via la proposition suivante : 

\begin{prop}
Soit $ D=\sum_{i=1}^{n+r}a_{i}D_{i} $ un diviseur de Weil sur $ X $ et $ s $ une section locale analytique de $ \OO(D) $ d\'efinie en $ P\in X(\QQ_{\nu}) $. Le point $ P\in X(\QQ_{\nu}) $ appartient \`a $ C_{\sigma,\nu} $ pour un certain $ \sigma \in \Delta $. Soit $ \chi^{u(\sigma)} $ un caract\`ere sur $ U $ repr\'esentant le diviseur de Cartier correspondant \`a $ D $ sur $ U_{\sigma} $ (i.e. $ \langle u(\sigma),v_{i}\rangle=-a_{i} $ pour tout $ v_{i}\in \sigma(1) $). On pose alors : \[ ||s(P)||_{D,\nu}=|s(P)\chi^{u(\sigma)}(P)|_{\nu}, \] et cette expression est ind\'ependante du choix de $ \sigma\in \Delta $ tel que $ P\in C_{\sigma,\nu} $. \end{prop}
\begin{proof} Voir \cite[Proposition 9.2]{Sa}. \end{proof}

On a alors la proposition suivante qui nous sera utile par la suite. 

\begin{prop}
Soit $ D=\sum_{i=1}^{n+r}a_{i}D_{i} $ un diviseur de Weil sur $ X $ tel que $ \OO(D) $ est engendr\'e par ses sections globales. Alors, pour $ \sigma\in \Delta_{\max} $, si $ \chi^{-u(\sigma)} $ d\'esigne l'unique caract\`ere sur $ U $ qui engendre $ \OO(D) $ sur $ U_{\sigma} $ (i.e. $ \langle u(\sigma),v_{i}\rangle=-a_{i} $ pour tout $ v_{i}\in \sigma(1) $), alors $ \chi^{-u(\sigma)} $ est une section globale de $ \OO(D) $ et $ \chi^{-u(\sigma)}(P)\neq 0 $ pour tout $ P\in U_{\sigma}(\QQ_{\nu}) $. Si $ s $ est une section locale de $ \OO(D) $ d\'efinie en $ P\in X(\QQ_{\nu}) $, alors \[ ||s(P)||_{D,\nu}=\inf_{\sigma\in \Delta_{\max}}|s(P)\chi^{u(\sigma)}(P)|_{\nu}, \] o\`u $ \Delta_{\max} $ d\'esigne l'ensemble des c\^ones de $ \Delta $ de dimension $ n $. De plus, si $ D $ est ample et $ \sigma\in \Delta_{\max} $, alors $ C_{\sigma,\nu} $ est l'ensemble des $ P\in X(\QQ_{\nu})  $ tels que $ |\chi^{u(\sigma)-u(\tau)}(P)|_{\nu}\leqslant 1 $ pour tout $ \tau\in \Delta_{\max} $. \end{prop}
\begin{proof} Voir \cite[Proposition 9.8]{Sa}. \end{proof}

On peut alors d\'efinir la hauteur associ\'ee \`a un diviseur $ D $. Si $ D=\sum_{i=1}^{n+r}a_{i}D_{i} $ est un diviseur de Weil sur $ X $ et $ P\in X(\QQ) $, la hauteur associ\'ee \`a $ D $ est l'application $ H_{D} : X(\QQ)\ra [0,\infty[ $ d\'efinie par \[ H_{D}(P)=\prod_{\nu\in \Val(\QQ)}||s(P)||_{D,\nu}^{-1}, \] o\`u $ \Val(\QQ) $ d\'esigne l'ensemble des places de $ \QQ $, et $ s $ une section locale de $ \OO(D) $ d\'efinie en $ P $ telle que $ s(P)\neq 0 $. 

\begin{rem}
Comme on peut le voir dans \cite[Proposition 10.12]{Sa}, pour tout $ P\in U(\QQ) $, $ H_{D}(P) $ ne d\'epend que de la classe de $ D $ dans $ \Pic(X) $. \end{rem}

Par la suite, on notera $ H $ la hauteur sur $ X $ associ\'ee au diviseur $ D_{0} $ d\'efini par \eqref{DiviseurD0}. Notre objectif sera alors d'\'evaluer \[ \mathcal{N}_{V}(B)=\Card\{P\in V(\QQ)\cap Y(\QQ) \; | \; H_{D_{0}}(P)\leqslant B\}, \] pour un certain ouvert dense $ V\subset U $ de $ X $. Pour \'evaluer cette quantit\'e il est plus pratique de se ramener \`a compter le nombre de points de hauteur born\'ee sur un torseur universel (voir \cite[\S 3]{Sa} pour la d\'efinition de torseurs universels) associ\'e \`a $ X $. Pour les vari\'et\'es toriques, la construction du torseur universel est relativement simple et est donn\'ee dans \cite[\S 8]{Sa}. Nous allons rappeler cette construction. \\

On consid\`ere le r\'eseau $ N_{0}=\ZZ^{n+r} $ et $ M_{0}=N_{0}^{\vee}=\ZZ^{n+r} $. \`A tout g\'en\'erateur $ v_{i} $ d'une ar\^ete du c\^one $ \Delta $ on associe l'\'el\'ement $ e_{0,i} $ de la base canonique de $ N_{0}=\ZZ^{n+r} $. On pose alors $ N_{1}=N_{0} $ et $ \Delta_{1} $ l'\'eventail constitu\'e de tous les c\^ones engendr\'es par les $ e_{0,i} $. La vari\'et\'e torique $ X_{1} $ d\'etermin\'ee par $ (N_{1},\Delta_{1}) $ est alors l'espace affine $ \AA^{n+r} $. Pour tout $ \sigma \in \Delta $, on note d'autre part $ \sigma_{0} $ le c\^one de $ N_{0,\RR} $ engendr\'e par les $ e_{0,i} $ pour $ i $ tel que $ v_{i}\in \sigma $. Les c\^ones $ \sigma_{0} $ ainsi associ\'es forment alors un \'eventail r\'egulier $ \Delta_{0} $ de $ N_{0,\RR} $ (cf. \cite[Proposition 8.4]{Sa}), et $ (\Delta_{0},N_{0}) $ d\'efinit une vari\'et\'e torique $ X_{0}\subset X_{1} $. Soit $ U_{0,\sigma}=\Spec(\QQ[S_{\sigma_{0}}]) $ o\`u $ S_{\sigma_{0}}=\sigma_{0}^{\vee}\cap M_{0} $. Les morphismes toriques $ \pi_{\sigma} : U_{0,\sigma}\ra U_{\sigma} $ d\'efinies par les applications naturelles de $ \sigma_{0} $ sur $ \sigma $ se recollent en un morphisme $ \pi : X_{0}\ra X $ qui est alors un torseur universel sur $ X $ (cf. \cite[Proposition 8.5]{Sa}).\\

\'Etant donn\'e que $ X_{0}\subset X_{1}=\AA^{n+r}_{\QQ} $ les points de $ X_{0} $ s'\'ecrivent sous forme de $ (n+r) $-uplets de coordonn\'ees $ \xx=(x_{1},...,x_{n},x_{n+1},...,x_{n+r}) $. On notera alors pour tout diviseur $ D=\sum_{i=1}^{n+r}a_{i}D_{i}  $ : \[ \xx^{D}=\prod_{i=1}^{n+r}x_{i}^{a_{i}}. \] 

\begin{rem}
Si $ \sigma\in \Delta $ , on note \[ \underline{\sigma}=\sum_{i\; |\; v_{i}\notin \sigma(1)}D_{i}, \] alors $ U_{0,\sigma} $ est l'ouvert de $ X_{1} $ d\'etermin\'e par $ \xx^{\underline{\sigma}}\neq 0 $, et donc $ X_{0} $ est l'ouvert de $ X_{1} $ d\'efini par  : \[ \xx\in X_{0} \Leftrightarrow \exists \sigma\in \Delta_{\max} \; | \; \xx^{\underline{\sigma}}\neq 0. \] 
\end{rem}

En rappelant que $ D_{0}=\sum_{k=1}^{r}(n_{k}-d_{k})D_{n+k} $, on d\'efinit alors les diviseurs $ D(\sigma) $ associ\'es : 
\begin{Def}
Soit $ \sigma\in \Delta_{\max} $, et soit $ \chi^{u(\sigma)} $ le caract\`ere de $ U $ tel que $ \chi^{-u(\sigma)} $ engendre $ \OO(D_{0}) $ sur $ U_{\sigma} $. On pose alors \[ D(\sigma)=D_{0}+\sum_{v_{i}\in \sigma(1)}\langle -u(\sigma),v_{i}\rangle D_{i}. \] 
\end{Def}
\begin{rem}
Les diviseurs $ D(\sigma) $ ne d\'ependent que de la classe de $ D_{0} $ dans $ \Pic(X) $. 
\end{rem}
\begin{lemma}\label{effectif}
Soit $ \sigma\in \Delta_{\max} $. Si $ \OO(D_{0}) $ est engendr\'e par ses section globales, alors $ \chi^{-u(\sigma)} $ est une section globale de $ \OO(D_{0}) $, et $ D(\sigma) $ est un diviseur effectif \`a support contenu dans $ \bigcup_{v_{i}\notin \sigma(1)}D_{i} $. \end{lemma}\begin{proof}
Si $ \OO(D_{0}) $ est engendr\'e par ses sections globales alors, pour tout $ \sigma\in \Delta_{\max} $, il existe une section globale de $ \OO(D_{0}) $ qui engendre $ \OO(D_{0}) $ sur $ U_{\sigma} $. Or, $ U_{\sigma} $ est un espace affine, donc \`a multiplication par un scalaire pr\`es, il existe une unique section locale qui engendre $ \OO(D_{0}) $ sur $ U_{\sigma} $. Donc la section locale $ \chi^{-u(\sigma)} $ est en fait une section globale. \\

Puisque $ \chi^{-u(\sigma)} $ est une section globale, on a d'apr\`es la description de $\Gamma(X,D_{0}) $ donn\'ee dans \cite[p.68]{F} : \[ \langle -u(\sigma),v_{i}\rangle\geqslant -a_{i} \] o\`u $ a_{i}=0 $ pour tout $ i\in \{1,...,n\} $ et $ a_{n+k}=(n_{k}-d_{k}) $ pour tout $ k\in \{1,...,r\} $. De plus, on a \[ \langle -u(\sigma),v_{i}\rangle= -a_{i} \] pour tout $ i $ tel que $ v_{i}\in \sigma(1) $. Donc $ D(\sigma) $ est bien effectif et \`a support contenu dans $ \bigcup_{v_{i}\notin \sigma(1)}D_{i} $. \end{proof}
Nous pouvons \`a pr\'esent d\'efinir une fonction hauteur $ H_{0} $ sur $ X_{0}(\QQ) $ en posant simplement $ H_{0}=H\circ \pi $. 

\begin{prop}
On suppose que $ \OO(D_{0}) $ est engendr\'e par ses sections globales. Avec les notation ci-dessus, on a : \[ \forall P_{0}=\xx\in X_{0}(\QQ), \; \; H_{0}(P_{0})=\prod_{\nu\in \Val(\QQ)}\sup_{\sigma\in \Delta_{\max}}|\xx^{D(\sigma)}|_{\nu}. \]
\end{prop}
\begin{proof}
La d\'emonstration de cette proposition est directement inspir\'ee de la preuve de \cite[Proposition 10.14]{Sa}. On consid\`ere un point $ P_{0}\in X_{0}(\QQ) $, $ P=\pi(P_{0}) $, et $ \tau \in \Delta_{\max} $ tel que $ P\in U_{\tau} $. On a alors que $ \chi^{-u(\tau)} $ est une section locale d\'efinie en $ P\in U_{\tau} $, et  \[||\chi^{-u(\tau)}(P)||_{D_{0},\nu}  =\inf_{\sigma\in \Delta_{\max}}|\chi^{u(\sigma)-u(\tau)}|_{\nu}.\]

Remarquons que puisque $ P\in U_{\tau} $, d'apr\`es le lemme \ref{effectif}, $ \xx^{D(\tau)}\neq 0 $ (\'etant donn\'e que $ D(\tau) $ est effectif \`a support contenu dans $ \bigcup_{v_{i}\notin \sigma(1)}D_{i} $), et que \[ \frac{\xx^{D(\sigma)}}{\xx^{D(\tau)}}=\chi^{u(\tau)-u(\sigma)}(P). \]Par cons\'equent, si $ s $ d\'esigne la section locale $ \chi^{-u(\tau)} $, on a alors :  \[||s(P)||_{D_{0},\nu}^{-1}  =\sup_{\sigma\in \Delta_{\max}}\left|\frac{\xx^{D(\sigma)}}{\xx^{D(\tau)}}\right|_{\nu}.\] De plus, par la formule du produit, on a \[ \prod_{\nu\in \Val(\QQ)}|\xx^{D(\tau)}|_{\nu}=1. \] D'o\`u le r\'esultat. 

\end{proof} 

De la m\^eme mani\`ere que nous avons construit $ X_{0} $, on peut construire un $ \ZZ $-torseur universel sur la vari\'et\'e $ \tilde{X} $ sur $ \ZZ $ obtenue \`a partir des ouverts affines $ \tilde{U}_{\sigma}=\Spec(\ZZ[S_{\sigma}]) $ (voir \cite[p. 207]{Sa}). On notera ce torseur $ \tilde{\pi} : \tilde{X}_{0}\ra \tilde{X}  $. On consid\`ere alors la proposition suivante (issue de \cite[Proposition 11.3]{Sa}) : 

\begin{prop}
Soit $ P_{0}=\xx\in X_{0}(\QQ)  $ qui se rel\`eve en un $ \ZZ $-point $ \tilde{P}_{0}=\tilde{\xx} $ de $ \tilde{X}_{0} $. On a alors \[ H_{0}(P_{0})=\sup_{\sigma\in \Delta_{\max}}|\tilde{\xx}^{D(\sigma)}|, \] o\`u $ |.| $ d\'esigne la valeur absolue usuelle sur $ \RR $. \end{prop}

\begin{proof}
Remarquons que comme pr\'ec\'edemment on a une immersion ouverte $ \tilde{X}_{0}\hookrightarrow \ZZ^{n+r} $. Soit $ p $ un nombre premier, et soit \[ Y_{0}\subset \Spec(\ZZ/p\ZZ[(x_{i})_{1\leqslant i\leqslant n+r}]) \] la r\'eduction modulo $ p $ de $ \tilde{X}_{0} $. On a alors, comme pr\'ec\'edemment, que $ \tilde{\xx}^{D(\sigma)}\neq 0 $ dans $ \ZZ/p\ZZ $ pour un certain $ \sigma\in \Delta $. On a donc que $ \sup_{\tau\in \Delta_{\max}}|\tilde{\xx}^{D(\tau)}|_{p}=1 $, et ainsi \[ H_{0}(P_{0})=\prod_{\nu\in \Val(\QQ)}\sup_{\tau\in \Delta_{\max}}|\tilde{\xx}^{D(\tau)}|_{\nu}=\sup_{\sigma\in \Delta_{\max}}|\tilde{\xx}^{D(\sigma)}|. \] \end{proof}
 
Plut\^ot que de compter les $ \QQ $-points de hauteur born\'ee de $ X $, nous allons compter les $ \ZZ $-points de $ \tilde{X}_{0} $ en utilisant le lemme ci-dessous : 

\begin{lemma}\label{zpoints} Pour $ m\in \NN $, soient \[ c(m)=\Card\{ P\in U(\QQ)\; |\; H(P)=m\}, \] \[ c_{0}(m)=\Card\{ P\in \tilde{X}_{0}\cap U_{0}(\QQ)\; |\; H_{0}(P_{0})=m\}.  \] Alors $ c(m)=c_{0}(m)/2^{r} $. \end{lemma}
\begin{proof} Voir d\'emonstration de \cite[Lemme 11.4.a)]{Sa}. \end{proof}

Ainsi, \'etant donn\'e un ouvert de Zariski $ V $ de $ X $, si l'on note \[ \mathcal{N}_{0,V}(B)=\Card\{P_{0}\in \tilde{Y}_{0}(\ZZ)\cap U_{0}(\QQ)\cap \pi^{-1}(V)\; |\; H_{0}(P_{0})\leqslant B \}\] (o\`u $ \tilde{Y}_{0} $ est l'hypersurface de $ \tilde{X}_{0} $ correspondant \`a l'hypersurface $ Y $ de $ X $), on a alors \[ \mathcal{N}_{V}(B)=\mathcal{N}_{0,V}(B)/2^{r}. \] Nous chercherons donc dor\'enavant \`a \'evaluer $ \mathcal{N}_{0,V}(B) $. Nous allons le faire pour le cas des vari\'et\'es toriques compl\`etes lisses  \`a $ n+2 $ g\'en\'erateurs (i.e. cas o\`u $ r=2 $). Nous allons d'abord, dans la section suivante, d\'ecrire ces vari\'et\'es, puis construire la hauteur sur les torseurs universels correspondants.

\subsection{Cas des vari\'et\'es toriques \`a $ n+2 $ g\'en\'erateurs} 
On consid\`ere $ n+2 $ vecteurs $ v_{0},v_{1},...,v_{n},v_{n+1}\in \ZZ^{n} $ tels que $ (v_{1},...,v_{n}) $ forme une base de $ \ZZ^{n} $ et  \[ \left\{\begin{array}{l} v_{0}=-\sum_{i=1}^{r}v_{i}-\sum_{i=r+1}^{m}a_{i}v_{i} \\ v_{n+1}=-\sum_{i=r+1}^{n}v_{i}, \end{array}\right. \] o\`u $ 1\leqslant r\leqslant m\leqslant n $, et $ a_{i}\in \ZZ $. On pose alors $ I=\{0,...,r\} $ et $ J=\{r+1,...,n+1\} $. On consid\`ere alors l'\'eventail $ \Delta $ d\'efini par les c\^ones maximaux : \[ \sigma_{i,j}=C\langle (v_{k})_{\substack{k\in I \\ k\neq i }}, (v_{l})_{\substack{l\in J \\ l\neq j }}\rangle \] pour tous $ i\in I $ et $ j\in J $. D'apr\`es \cite[Th\'eor\`eme 1]{K}, nous savons que toute vari\'et\'e torique compl\`ete lisse dont l'\'eventail $ \Delta $ admet $ n+2 $ ar\^etes est isomorphe \`a une vari\'et\'e torique de ce type pour un certain $ (r,m,(a_{i})_{i\in \{r+1,...,m\}}) $ fix\'e. \\

Dans ce qui va suivre, pour simplifier, nous nous int\'eresserons exclusivement \`a la sous-famille de ces vari\'et\'es d\'efinies par $ a_{r+1}=...=a_{m}=1 $ de sorte que $ v_{0}=\sum_{i=1}^{m}v_{i} $.\\

 Remarquons \`a pr\'esent que dans ce cas pr\'ecis, d'apr\`es les r\'esultats obtenus dans la section \ref{premprop}, si, pour tout $ i\in \{0,1,...,n\} $, $ D_{i} $ d\'esigne le diviseur associ\'e \`a $ v_{i} $, on a : \begin{align*}
 [D_{1}]  =[D_{0}] & & [D_{r+1}] =[D_{0}]+[D_{n+1}] && [D_{m+1}] =[D_{n+1}]
 \\ [D_{2}] =[D_{0}] & & [D_{r+2}]=[D_{0}]+[D_{n+1}] & & [D_{m+2}] =[D_{n+1}]
 \\ \ldots \; \; \; \; \; \; \; \; \; \;  &  & \ldots \; \; \; \; \; \; \; \; \; \; \; \; \; \;   & & \ldots  \; \; \; \; \; \; \; \;     
 \\ [D_{r}] =[D_{0}] & & [D_{m}]=[D_{0}]+[D_{n+1}] & &  [D_{n}] =[D_{n+1}]
 \end{align*}

La classe du diviseur anticanonique de $ X $ est alors donn\'e par (cf. \cite{F}) \[ [-K_{X}]=\sum_{i=0}^{n+1}[D_{i}]=(m+1)[D_{0}]+(n-r+1)[D_{n+1}]. \]

 Consid\'erons \`a pr\'esent une hypersurface $ Y $ de $ X $ donn\'ee par une section globale $ s $ de $ \OO(D) $ o\`u $ D $ d\'esigne le diviseur $ d_{1}D_{0}+d_{2}D_{n+1} $. Le diviseur anticanonique de $ Y $ est alors le diviseur induit par \[ (m+1-d_{1})[D_{0}]+(n-r+1-d_{2})[D_{n+1}]. \]
 \begin{rem}
 Dans tout ce qui va suivre, nous supposerons que la diviseur anticanonique de $ Y $ appartient \`a l'int\'erieur du c\^one effectif. Ce qui revient \`a dire, d'apr\`es ce qui pr\'ec\`ede que $ m+1>d_{1} $ et $ n-r+1>d_{2} $. 
 \end{rem}
 D'autre part, les sections globales de $ \OO(D) $ sont donn\'ees par (cf. \cite[\S 3.4]{F}) : \[ \Gamma(X,\OO(D))=\bigoplus_{u\in P_{D}\cap \ZZ^{n}}\CC.\chi^{u}, \] o\`u $ \chi^{u} $ est le caract\`ere associ\'e \`a $ u $, et $ P_{D} $ le polytope : \begin{multline*} P_{D}=\left\{ u\in \ZZ^{n}\; | \; \forall k\in \{1,...,n\},\; \langle u,v_{k}\rangle\geqslant 0, \right. \\ \left. \langle u,v_{0}\rangle\geqslant -d_{1}\;  \et \;  \langle u,v_{n+1}\rangle\geqslant -d_{2} \right\} \end{multline*}
Chaque section (\`a coefficients rationnels) $ s=\sum_{u\in P_{D}\cap \ZZ^{n}}\alpha_{u}\chi^{u} $ o\`u $ \alpha_{u}\in \QQ $ d\'efinit une hypersurface $ Y $ (que l'on suppose lisse) de $ X $, et se rel\`eve en une fonction $ f : \tilde{X}_{0}\ra \RR $ d\'efinie par pour tous $ (\xx,\yy,\zz)\in \QQ^{n+2} $ tels que $ x_{0}\neq 0 $ et $ z_{n+1}\neq 0 $ : \[ f(\xx,\yy,\zz)=\sum_{u\in P_{D}\cap \ZZ^{n}}\alpha_{u}\prod_{i=0}^{r}x_{i}^{\langle u,v_{i}\rangle}\prod_{j=r+1}^{m}y_{j}^{\langle u,v_{j}\rangle}\prod_{k=m+1}^{n+1}z_{k}^{\langle u,v_{k}\rangle}. \] L'hypersurface de $ \tilde{X}_{0} $ d\'efinie par l'annulation de cette fonction correspond alors au torseur universel au-dessus de $ Y $. Par cons\'equent, en utilisant le lemme \ref{zpoints}, on a que les $ \QQ $-points de $ Y $ correspondent (modulo l'action des points de torsion de $ T_{NS} $) aux $ \ZZ $-points $ (\xx,\yy,\zz) $ de $ \tilde{X}_{0} $ tels que $ F(\xx,\yy,\zz)=0 $ o\`u $ F $ est le polyn\^ome : \begin{multline*}
F(\xx,\yy,\zz)  =x_{0}^{d_{1}}z_{n+1}^{d_{2}}f(\xx,\yy,\zz) \\ =\sum_{u\in P_{D}\cap \ZZ^{n}}\alpha_{u}\left(x_{0}^{d_{1}+\langle u,v_{0}\rangle}z_{n+1}^{d_{2}+\langle u,v_{n+1}\rangle}\prod_{i=1}^{r}x_{i}^{\langle u,v_{i}\rangle}\prod_{j=r+1}^{m}y_{j}^{\langle u,v_{j}\rangle}\prod_{k=m+1}^{n}z_{k}^{\langle u,v_{k}\rangle} \right)
\end{multline*} 
\begin{rem}\begin{itemize}
\item On remarque que le polyn\^ome ainsi d\'efini est de degr\'e homog\`ene \'egal \`a $ d_{1} $ en $ (\xx,\yy) $ et de degr\'e homog\`ene $ d_{2} $ en $ (\yy,\zz) $, c'est-\`a-dire, pour tous $ \lambda,\mu\in \CC $ : \[F(\lambda \xx, \lambda\mu \yy,\mu\zz)=\lambda^{d_{1}}\mu^{d_{2}}F(\xx,\yy,\zz). \] En effet le degr\'e de chaque mon\^ome en $ (\xx,\yy) $ est \[ d_{1}+\langle u,v_{0}\rangle+\sum_{i=1}^{m}\langle u,v_{i}\rangle=d_{1}, \] car $ v_{0}=-\sum_{i=1}^{m}v_{i} $, et de m\^eme pour $ (\yy,\zz) $. 
\item R\'eciproquement on peut voir que tout polyn\^ome en $ (\xx,\yy,\zz) $ de degr\'e homog\`ene $ d_{1} $ en $ (\xx,\yy) $ et de degr\'e homog\`ene $ d_{2} $ en $ (\yy,\zz) $ est un polyn\^ome correspondant \`a une unique section globale $ s $ de $ \OO(D) $. 
\end{itemize} 
\end{rem}
\begin{rem}
Dans tout ce qui va suivre on supposera que l'hypersurface $ Y $ d\'efinie par $ F(\xx,\yy,\zz)=0 $ est lisse. En fait cette propri\'et\'e est vraie pour un ouvert dense de Zariski de coefficients $ (\alpha_{u})_{u\in P_{D}\cap \ZZ^{n} } $. En effet on r\'ealise un plongement de $ X $ dans un espace projectif $ \PP^{N} $ en consid\'erant l'application $ f $ qui \`a $ \pi(\xx,\yy,\zz) $ associe la classe de \[\left( (x_{0}^{d_{1}+\langle u,v_{0}\rangle}z_{n+1}^{d_{2}+\langle u,v_{n+1}\rangle}\prod_{i=1}^{r}x_{i}^{\langle u,v_{i}\rangle}\prod_{j=r+1}^{m}y_{j}^{\langle u,v_{j}\rangle}\prod_{k=m+1}^{n}z_{k}^{\langle u,v_{k}\rangle}\right)_{u\in P_{D}\cap \ZZ^{n}}.  \] Par ailleurs, d'apr\`es le th\'eor\`eme de Bertini (cf. \cite{Ha}), pour une famille ouverte dense d'hyperplans projectifs $ H_{\aalpha}=\{(X_{u})_{u\in P_{D}\cap \ZZ^{n}} \;  | \; \sum_{u\in P_{D}\cap \ZZ^{n}}\alpha_{u}X_{u}=0 \}\subset \PP^{N} $, on a que $ X\cap H_{\aalpha} $ est lisse. Or on remarque que $ X\cap H_{\aalpha}=Y $ et par cons\'equent, $ Y $ est lisse pour un ouvert dense de coefficients $ (\alpha_{u})_{u\in P_{D}\cap \ZZ^{n}} $.
\end{rem}

Nous allons \`a pr\'esent construire la hauteur sur $ X $ associ\'ee au diviseur $ D_{Y}=(m+1-d_{1})D_{0}+(n-r+1-d_{2})D_{n+1} $ (correspondant au diviseur anticanonique sur $ Y $). Comme pr\'ec\'edemment, d'apr\`es \cite[\S 3.4]{F}, les sections globales de $ \OO(D_{Y}) $ sont donn\'ees par : \[ \Gamma(X,\OO(D_{Y}))=\bigoplus_{u\in P_{D_{Y}}\cap \ZZ^{n}}\CC.\chi^{u}, \] o\`u $ P_{D_{Y}} $ est le polytope : \begin{multline*} P_{D_{Y}}=\left\{ u\in \ZZ^{n}\; | \; \forall k\in \{1,...,n\},\; \langle u,v_{k}\rangle\geqslant 0, \right. \\ \left. \langle u,v_{0}\rangle\geqslant m+1-d_{1}\;  \et \;  \langle u,v_{n+1}\rangle\geqslant n-r+1-d_{2} \right\} \end{multline*}
Une base des sections globales est donc donn\'ee par les $ (\chi^{u})_{u\in P_{D_{Y}}} $, qui se rel\`event en des fonctions $ (f_{u})_{u\in P_{D_{Y}}} $ de $ \tilde{X}_{0} $ dans $ \RR $ qui sont exactement les mon\^omes en $ (\xx,\yy,\zz) $ de degr\'es $ (m+1-d_{1}) $ en $ (\xx,\yy) $ et de degr\'e $ (n-r+1-d_{2}) $ en $ (\yy,\zz) $. La hauteur $ H $ associ\'ee \`a $ D_{Y} $ est donc d\'efinie sur $ \tilde{X}_{0}(\ZZ)\subset \ZZ^{n+2} $ par pour tout $  \qq=(\xx,\yy,\zz)\in \tilde{X}_{0}(\ZZ); $ : \begin{align*}  H(\qq) & =\max_{\substack{\forall i,j,k, \; \alpha_{i},\beta_{j},\gamma_{k}\in \NN \\ \sum_{i=0}^{r}\alpha_{i}+\sum_{j=r+1}^{m}\beta_{j}=m+1-d_{1} \\ \sum_{j=r+1}^{m}\beta_{j}+\sum_{k=m+1}^{n+1}\gamma_{k}=n-r+1-d_{2}}}\prod_{i=0}^{r}|x_{i}|^{\alpha_{i}}\prod_{j=r+1}^{m}|y_{j}|^{\beta_{j}}\prod_{k=m+1}^{n+1}|z_{k}|^{\gamma_{k}} \\ & =\max_{\substack{\alpha,\beta,\gamma\in \NN \\ \alpha+\beta=m+1-d_{1} \\ \beta+\gamma=n-r+1-d_{2}}}|\xx|^{\alpha}|\yy|^{\beta}|\zz|^{\gamma} \\ & =\max\{|\xx|^{m+1-d_{1}}|\zz|^{n-r+1-d_{2}},|\xx|^{(m+1-d_{1})-(n-r+1-d_{2})}|\yy|^{n-r+1-d_{2}}\} \\ & = |\xx|^{m+1-d_{1}}\max\left( \frac{|\yy|}{|\xx|},|\zz|\right)^{n-r+1-d_{2}}. \end{align*}

Remarquons efin que dans le cas pr\'esent, $ \tilde{X}_{0}(\ZZ)\subset \ZZ^{n+2} $ peut \^etre d\'ecrit comme l'ensemble des $ (n+2) $-uplets d'entiers not\'es $ \qq=(\xx,\yy,\zz) $, avec $ \xx=(x_{0},x_{1},...,x_{r}) $, $ \yy=(y_{r+1},...,y_{m}) $, $ \zz=(z_{m+1},...,z_{n+1}) $ tels que (cf. \cite[11.5]{Sa}) : 
\begin{equation}\label{condexistence}
\exists \sigma \in \Delta_{\max} \; | \; \ \qq^{ \underline{\sigma}}\neq 0,
\end{equation}  

 \begin{equation}\label{condprim}
\PGCD_{\sigma \in \Delta_{\max}}(\qq^{\underline{\sigma}})=1,
\end{equation} o\`u \begin{equation}
\qq^{ \underline{\sigma}}=\prod_{i\notin \sigma(1)}q_{i}.
\end{equation}
Par la d\'efinition de $ \Delta $, et des c\^ones maximaux $ (\sigma_{i,j})_{(i,j)\in I\times J} $, on observe que : \[ \qq^{\underline{\sigma_{i,j}}}=\prod_{l\notin\sigma_{i,j}(1) }q_{l}=q_{i}q_{j}. \] Par cons\'equent, la condition\;\eqref{condexistence} \'equivaut \`a : \[  \exists (i,j)\in I\times J \; |\; q_{i}q_{j}\neq 0 \] soit encore \[  \exists i\in I \; |\; q_{i}\neq 0, \; \et \; \exists j\in J \; |\; q_{j}\neq 0, \] et donc\;\eqref{condexistence} \'equivaut \`a : 
\begin{equation}\label{condexistence2}
\xx\neq \0 \; \; \et \; \; (\yy,\zz)\neq \0 .
\end{equation}  
De m\^eme, on remarque que : \begin{align*} \PGCD_{\sigma \in \Delta_{\max}}(\qq^{\underline{\sigma}}) & =\PGCD_{(i,j)\in I\times J}(q_{i}q_{j}) \\ & =(\PGCD_{i\in I}q_{i})(\PGCD_{j\in J}q_{j}) \\ & =\PGCD(\xx)\PGCD(\yy,\zz), \end{align*} et la condition\;\eqref{condprim} \'equivaut donc \`a  
 \begin{equation}\label{condprim2}
\PGCD(\xx)=1 \; \; \et \; \; \PGCD(\yy,\zz)=1.
\end{equation}

Ainsi, calculer   \[ \mathcal{N}(B)=\card\{P\in Y(\QQ) \; | \; H(P)\leqslant B\} \] revient \`a calculer le nombre de points de \[ \{ \qq=(\xx,\yy,\zz)\in \tilde{X}_{0}(\ZZ) \; | \; H(\xx,\yy,\zz)\leqslant B\}. \] Par ailleurs, quitte \`a appliquer une inversion de M\"{o}bius (en un sens que nous pr\'eciserons ult\'erieurement), on peut se ramener au calcul de \begin{multline*} N_{d,U}(B)=\card\left\{ (\xx,\yy,\zz)\in \ZZ^{n+2}\cap U\; | \; \xx\neq\0, \; (\yy,\zz)\neq (\0,\0), \; \right. \\ \left. \; F(d\xx,\yy,\zz)=0, \;   H_{d}(\xx,\yy,\zz)\leqslant B \right\} \end{multline*} pour un certain ouvert $ U $ que nous pr\'eciserons ult\'erieurement, et pour tout $ d\in \NN^{\ast} $, avec \begin{equation}
H_{d}(\xx,\yy,\zz)=|\xx|^{m+1-d_{1}}\max\left( \frac{|\yy|}{d|\xx|},|\zz|\right)^{n-r+1-d_{2}}.
\end{equation} Dans ce qui va suivre nous allons donc chercher \`a obtenir une formule asymptotique pour $ N_{d,U}(B) $.

\section{Premi\`ere \'etape}

Nous allons \'etablir une formule asymptotique pour $ N_{U,d}(B) $, pour un $ d\in \NN^{\ast} $ fix\'e, en nous inspirant de la m\'ethode d\'ecrite par Schindler dans \cite{S1} et \cite{S2}. L'id\'ee g\'en\'erale est de consid\'erer la fonction $ h_{d} : \NN^{2}\ra [0,\infty[ $ d\'efinie par \begin{multline}\label{fonctionh}
h_{d}(k,l)=\card\left\{ (\xx,\yy,\zz)\in \ZZ^{n+2}\cap U \; | \; |\xx|=k,  \right. \\ \left. \max\left( \left\lfloor\frac{|\yy|}{d|\xx|}\right\rfloor,|\zz|\right)=l\;  \et \; F(d\xx,\yy,\zz)=0 \right\}
\end{multline} 
(o\`u $ U $ est un ouvert de Zariski de $ \AA^{n+2} $ que nous pr\'eciserons ult\'erieurement), de donner des formules asymptotiques pour \[\sum_{k\leqslant P_{1}}\sum_{l\leqslant P_{2}}h_{d}(k,l), \; \; \;  \sum_{k\leqslant P_{1}}   h_{d}(k,l) \; \; \et \; \;  \sum_{l\leqslant P_{2}}h_{d}(k,l), \] afin de pouvoir appliquer un r\'esultat de Blomer et Br\"{u}dern (voir \cite{BB}) pour en d\'eduire une formule asymptotique pour \[ \sum_{k^{m+1-d_{1}}l^{n-r+1-d_{2}}\leqslant B}h_{d}(k,l)\sim_{B\ra \infty}N_{U,d}(B). \]

 Dans cette premi\`ere partie, pour des r\'eels $ P_{1},P_{2}\geqslant 1 $ fix\'es, nous allons chercher \`a calculer \begin{multline}
  N_{d}(P_{1},P_{2})=\card\{ (\xx,\yy,\zz)\in  (P_{1}\BB_{1}\times P_{1}P_{2}\BB_{2}\times P_{2}\BB_{3})\cap\ZZ^{n+2} \; | \; \\  |\yy|\leqslant d|\xx|P_{2} \;  \et \;  F(d\xx,\yy,\zz)=0 \},
\end{multline}   o\`u $ \BB_{1}=[-1,1]^{r+1} $, $ \BB_{2}=[-1,1]^{m-r} $, $ \BB_{3}=[-1,1]^{n-m+1} $. Plus pr\'ecis\'ement, nous allons montrer que pour $ n $ assez grand on a : \begin{multline*} N_{d}(P_{1},P_{2})=\sigma_{d} P_{1}^{m+1-d_{1}}P_{2}^{n-r+1-d_{2}} \\ +O(d^{\upsilon}P_{1}^{m+1-d_{1}}P_{2}^{n-r+1-d_{2}}\min\{P_{1},P_{2}\}^{-\delta}), \end{multline*} o\`u $ \sigma_{d} $ est une constante (ne d\'ependant que de $ d $), $ \delta>0 $ un r\'eel arbitrairement petit et $ \upsilon $ un r\'eel que nous pr\'eciserons. Ceci qui nous permettra plus tard d'obtenir une  formule pour $ \sum_{k\leqslant P_{1}}\sum_{l\leqslant P_{2}}h_{d}(k,l) $. 

\subsection{Une in\'egalit\'e de Weyl}\label{Weyl}

Dans toute cette partie nous allons supposer $ 1\leqslant P_{2}\leqslant P_{1} $. On notera donc $ P_{1}=P_{2}^{b} $ avec $ b\geqslant 1 $. Nous allons \'evaluer $ N_{d}(P_{1},P_{2}) $ en nous inspirant de la m\'ethode du cercle de Hardy-Littlewood. Pour cela, on introduit la fonction g\'en\'eratrice d\'efinie par \begin{equation}
S_{d}(\alpha)=\sum_{\substack{\xx\in \ZZ^{r+1} \\ |\xx|\leqslant P_{1}}}\sum_{\substack{\yy\in \ZZ^{m-r}\\|\yy|\leqslant d|\xx|P_{2}}}\sum_{\substack{\zz\in \ZZ^{n-m+1}\\ |\zz|\leqslant P_{2}}}e(\alpha F(d\xx,\yy,\zz)). 
\end{equation} pour $ \alpha\in [0,1] $, et o\`u $ e $ d\'esigne la fonction $ x\mt \exp(2i\pi x) $. On remarque alors que \[ N_{d}(P_{1},P_{2})=\int_{0}^{1}S_{d}(\alpha)\dd\alpha. \] \'Etant donn\'es $ \xx \in \ZZ^{r+1} $ et $ \yy\in \ZZ^{m-r} $, on constate que \[ |\yy|\leqslant d|\xx|P_{2} \Leftrightarrow |\xx|\geqslant \frac{|\yy|}{dP_{2}} \Leftrightarrow |\xx|\geqslant \left\lceil\frac{|\yy|}{dP_{2}}\right\rceil. \] En posant $ N= \left\lceil\frac{|\yy|}{dP_{2}}\right\rceil $ (ce qui \'equivaut \`a dire que $ |\yy|\in ]d(N-1)P_{2}, dNP_{2}] $), on remarque que $ S(\alpha) $ peut \^etre r\'eexprim\'e sous la forme : \[ S_{d}(\alpha)=\sum_{N=1}^{P_{1}}S_{d,N}(\alpha), \] o\`u \begin{equation}
S_{d,N}(\alpha)=\sum_{\substack{ N\leqslant |\xx|\leqslant P_{1}}}\sum_{\substack{d(N-1)P_{2}<|\yy|\leqslant dNP_{2}}}\sum_{\substack{|\zz|\leqslant P_{2}}}e(\alpha F(d\xx,\yy,\zz)).
\end{equation} Si \[ \mathcal{E}_{N}=\{\yy\in \ZZ^{m-r}\;|\; d(N-1)P_{2}<|\yy|\leqslant dNP_{2}\},  \] on remarque que \[ \mathcal{E}_{N}=\bigcup_{\II\subset \{r+1,...,m\}}\BB_{N,\II}, \] o\`u \[\BB_{N,\II}=\{\yy\in \mathcal{E}_{N}\; |\; \forall i\in \II, \; y_{i}\geqslant 0 \; \et \; \forall i\notin \II, \; y_{i}<0 \}. \] On observe par ailleurs que l'on peut \'ecrire pour tout $ \II $ : \begin{equation} \BB_{N,\II}=\bigcup_{\substack{\JJ\subset \{r+1,...,m\} \\ \JJ\neq \emptyset}}\mathcal{C}_{N,\II,\JJ}, \end{equation}  avec \begin{equation} \mathcal{C}_{N,\II,\JJ}=\{\yy\in \BB_{N,\II}\; |\;  \forall j\in \JJ, \; |y_{j}|>d(N-1)P_{2} \; \et \; \forall j\notin \JJ, \; |y_{j}|\leqslant dNP_{2}\}. \end{equation} On a alors \begin{equation} S_{d,N}(\alpha)\ll \sum_{\substack{\II,\JJ\subset \{r+1,...,m\} \\ \JJ\neq \emptyset }}|S_{d,N,\II,\JJ}(\alpha)| \end{equation} o\`u \begin{equation} S_{d,N,\II,\JJ}(\alpha)=\sum_{\substack{|\xx|\leqslant P_{1}}}\sum_{\yy\in \mathcal{C}_{N,\II,\JJ}}\sum_{\substack{|\zz|\leqslant P_{2}}}e(\alpha F(d\xx,\yy,\zz)). \end{equation}
Par une in\'egalit\'e de H\"{o}lder, on a, pour $ N $ fix\'e \begin{equation}
|S_{d,N,\II,\JJ}(\alpha)|^{2^{d_{2}-1}}\ll P_{1}^{(r+1)(2^{d_{2}-1}-1)}\sum_{|\xx|\leqslant P_{1}}|S_{d,N,\II,\JJ,\xx}(\alpha)|^{2^{d_{2}-1}}
\end{equation} o\`u l'on a not\'e \begin{equation}
S_{d,N,\II,\JJ,\xx}(\alpha)=\sum_{\yy\in \mathcal{C}_{N,\II,\JJ}}\sum_{\substack{|\zz|\leqslant P_{2}}}e(\alpha F(d\xx,\yy,\zz)).
\end{equation}

Dans ce qui va suivre, nous allons chercher \`a \og lin\'eariser\fg \ le polyn\^ome $ F $ en appliquant un opr\'erateur $ \Delta $ d\'efini de la fa\c{c}on suivante : pour tout polyn\^ome $ f $ \`a $ N $ variables on pose pour tous $ \tt_{1},\tt_{2}\in \RR^{N} $ : \[ \Delta_{\tt_{1}}(\tt_{2})=f(\tt_{1}+\tt_{2})-f(\tt_{1}). \] Dans ce qui suit, nous appliquons $ d_{2}-1 $ fois l'op\'erateur $ \Delta $ \`a $ F $ en les variables $ (\yy,\zz) $, et nous obtenons un polyn\^ome en $ d_{2}(n-r+1)+r+1 $ variables $ (\xx,\yy^{(j)},\zz^{(j)})_{j\in \{1,...,d_{2}\}} $. Puis, en appliquant l'op\'erateur $ \Delta $ $ d_{1}-1 $ fois \`a ce polyn\^ome en les variables $ (\xx,\yy^{(j)})_{j\in \{1,..,d_{2}\}} $, nous obtenus finalement un polyn\^ome en $ (r+1)d_{1}+(m-r)d_{1}d_{2}+(n-m+1)d_{2} $ variables du type $ (\xx^{(i)},\yy^{(j,i)},\zz^{(j)})_{\substack{i \in \{1,...,d_{1}\} \\ j\in \{1,..,d_{2}\}}} $ de la forme : \begin{multline*} \Gamma_{d}^{(1)}\left((\xx^{(i)},\yy^{(j,i)},\zz^{(j)})_{\substack{i \in \{1,...,d_{1}\} \\ j\in \{1,..,d_{2}\}}}\right) + G_{1}\left((\xx^{(i)},\yy^{(j,i)},\zz^{(j)})_{\substack{i \in \{1,...,d_{1}-1\} \\ j\in \{1,..,d_{2}\}}}\right)\\ +G_{2}\left((\xx^{(i)},\yy^{(j,i)},\zz^{(j)})_{\substack{i \in \{1,...,d_{1}\} \\ j\in \{1,..,d_{2}-1\}}}\right) \end{multline*} o\`u $ G_{1} $ (resp. $ G_{2} $) est ind\'ependant de $ (\xx^{(d_{1})},\yy^{(j,d_{1})})_{j\in \{1,...,d_{2}\}} $ (resp. \linebreak $ (\yy^{(d_{2},i)},\zz^{(d_{2})})_{i\in \{1,...,d_{1}\}} $), et $ \Gamma_{d}^{(1)} $ est lin\'eaire en $ (\xx^{(i)},\yy^{(j,i)})_{j\in\{1,...,d_{2}\}} $ pour tout $ i\in\{1,...,d_{1}\} $ et lin\'eaire en $ (\yy^{(j,i)},\zz^{(j)})_{i\in\{1,...,d_{1}\}} $ pour tout $ j\in \{1,...,d_{2}\} $. \\

Pour $ N,\II,\JJ $ fix\'es, posons \[ \mathcal{U}_{N,\II,\JJ}=\mathcal{C}_{N,\II,\JJ}\times P_{2}\BB_{3} \subset dP_{1}P_{2}\BB_{2}\times P_{2}\BB_{3}, \] et on d\'efinit \[ \mathcal{U}_{N,\II,\JJ}^{D}=\mathcal{U}_{N,\II,\JJ}-\mathcal{U}_{N,\II,\JJ}, \] \begin{multline*} \mathcal{U}_{N,\II,\JJ}((\yy^{(1)},\zz^{(1)}),...,(\yy^{(t)},\zz^{(t)})) \\ =\bigcap_{(\varepsilon_{1},...,\varepsilon_{t})\in \{0,1\}^{t}}(\mathcal{U}_{N,\II,\JJ}-\varepsilon_{1}(\yy^{(1)},\zz^{(1)})-...-\varepsilon_{t}(\yy^{(t)},\zz^{(t)})). \end{multline*}

Si l'on note $ \mathcal{F}(\yy,\zz)=\alpha F(d\xx,\yy,\zz) $ (pour $ \xx $ fix\'e), et \begin{multline}\label{formefd}
\mathcal{F}_{t}((\yy^{(1)},\zz^{(1)}),...,(\yy^{(t)},\zz^{(t)})) \\ =\sum_{(\varepsilon_{1},...,\varepsilon_{t})\in \{0,1\}^{t}}(-1)^{\varepsilon_{1}+...+\varepsilon_{t}}\mathcal{F}(\varepsilon_{1}(\yy^{(1)},\zz^{(1)})+...+\varepsilon_{t}(\yy^{(t)},\zz^{(t)})),
\end{multline}
en utilisant l'\'equation $ (11.2) $ de \cite{Sm}, on obtient la majoration \begin{multline*} 
|S_{d,N,\II,\JJ,\xx}|^{2^{d_{2}-1}}\ll | \mathcal{U}_{N,\II,\JJ}^{D}|^{2^{d_{2}-1}-d_{2}}\sum_{(\yy^{(1)},\zz^{(1)})\in \mathcal{U}_{N,\II,\JJ}^{D}}...\sum_{(\yy^{(d_{2}-2)},\zz^{(d_{2}-2)})\in \mathcal{U}_{N,\II,\JJ}^{D}}\\ \left| \sum_{\substack{(\yy^{(d_{2}-1)},\zz^{(d_{2}-1)}) \\ \in\mathcal{U}_{N,\II,\JJ}((\yy^{(1)},\zz^{(1)}),...,(\yy^{(d_{2}-2)},\zz^{(d_{2}-2)}))}}e(\mathcal{F}_{d_{2}-1}((\yy^{(1)},\zz^{(1)}),...,(\yy^{(d_{2}-1)},\zz^{(d_{2}-1)})))\right|^{2} \end{multline*} que l'on peut encore majorer par \begin{multline*} ((dP_{1}P_{2})^{m-r}P_{2}^{n-m+1})^{2^{d_{2}-1}-d_{2}}\sum_{\substack{|\yy^{(1)}|\leqslant 2dP_{1}P_{2} \\ |\zz^{(1)}|\leqslant 2P_{2} }}...\sum_{\substack{|\yy^{(d_{2}-2)}|\leqslant 2dP_{1}P_{2} \\ |\zz^{(d_{2}-2)}|\leqslant 2P_{2} }}\\ \left| \sum_{\substack{(\yy^{(d_{2}-1)},\zz^{(d_{2}-1)}) \\ \in\mathcal{U}_{N,\II,\JJ}((\yy^{(1)},\zz^{(1)}),...,(\yy^{(d_{2}-2)},\zz^{(d_{2}-2)}))}}e(\mathcal{F}_{d_{2}-1}((\yy^{(1)},\zz^{(1)}),...,(\yy^{(d_{2}-1)},\zz^{(d_{2}-1)})))\right|^{2} 
\end{multline*}
On remarque que pour tous $ (\yy,\zz),(\yy',\zz')\in \mathcal{U}_{N,I,J}((\yy^{(1)},\zz^{(1)}),...,(\yy^{(d_{2}-2)},\zz^{(d_{2}-2)})) $ on a : \begin{multline*}\mathcal{F}_{d_{2}-1}((\yy^{(1)},\zz^{(1)}),...,(\yy,\zz))-\mathcal{F}_{d_{2}-1}((\yy^{(1)},\zz^{(1)}),...,(\yy',\zz')) \\ =\mathcal{F}_{d_{2}}((\yy^{(1)},\zz^{(1)}),...,(\yy^{(d_{2}},\zz^{(d_{2})})-\mathcal{F}_{d_{2}-1}((\yy^{(1)},\zz^{(1)}),...,(\yy^{(d_{2}-1)},\zz^{(d_{2}-1)})), \end{multline*} pour  \[ (\yy^{(d_{2}-1)},\zz^{(d_{2}-1)})\in \mathcal{U}_{N,\II,\JJ}((\yy^{(1)},\zz^{(1)}),...,(\yy^{(d_{2}-2)},\zz^{(d_{2}-2)}))^{D} \] et \[ (\yy^{(d_{2})},\zz^{(d_{2})})\in \mathcal{U}_{N,\II,\JJ}((\yy^{(1)},\zz^{(1)}),...,(\yy^{(d_{2}-1)},\zz^{(d_{2}-1)})), \] donn\'es par : \[ (\yy,\zz)=(\yy^{(d_{2})},\zz^{(d_{2}}), \] \[ (\yy',\zz')=(\yy^{(d_{2}-1)}+\yy^{(d_{2})},\zz^{(d_{2}-1)}+\zz^{(d_{2})}). \]  On obtient donc la majoration : \begin{multline*} |S_{d,N,\II,\JJ,\xx}(\alpha)|^{2^{d_{2}-1}}\ll (d^{m-r}P_{1}^{m-r}P_{2}^{n-r+1})^{2^{d_{2}-1}-d_{2}}\sum_{\yy^{(1)},\zz^{(1)}}...\sum_{\yy^{(d_{2}-2)},\zz^{(d_{2}-2)}} \\  \sum_{\yy^{(d_{2}-1)},\zz^{(d_{2}-1)}}\sum_{\yy^{(d_{2})},\zz^{(d_{2})}} e(\mathcal{F}_{d_{2}}((\yy^{(1)},\zz^{(1)}),...,(\yy^{(d_{2}},\zz^{(d_{2})}) \\ -\mathcal{F}_{d_{2}-1}((\yy^{(1)},\zz^{(1)}),...,(\yy^{(d_{2}-1)},\zz^{(d_{2}-1)}))). \end{multline*}
o\`u chaque $ \yy^{(i)} $ (resp. $ \zz^{(i)} $) appartient \`a une union de bo\^ites de taille au plus $ dP_{1}P_{2} $ (resp. $ P_{2} $). \\

D'apr\`es \cite{Sm}[Lemme 11.4], on a que \begin{multline*} \mathcal{F}_{d_{2}}((\yy^{(1)},\zz^{(1)}),...,(\yy^{(d_{2}},\zz^{(d_{2})})-\mathcal{F}_{d_{2}-1}((\yy^{(1)},\zz^{(1)}),...,(\yy^{(d_{2}-1)},\zz^{(d_{2}-1)})) \\ = \alpha F_{1}(d\xx,\tilde{\yy},\tilde{\zz})+ \alpha F_{2}(d\xx,\hat{\yy},\hat{\zz}). \end{multline*} o\`u l'on a not\'e \begin{align*}
\tilde{\yy}=(\yy^{(1)},...,\yy^{(d_{2})}) & \; \; \; \tilde{\zz}=(\zz^{(1)},...,\zz^{(d_{2})}) 
 \\ \hat{\yy}=(\yy^{(1)},...,\yy^{(d_{2}-1)}) & \; \; \; \hat{\zz}=(\zz^{(1)},...,\zz^{(d_{2}-1)}).
\end{align*} avec $ F_{1} $ est une forme multilin\'eaire en $ (\tilde{\yy},\tilde{\zz}) $ de la forme :  \[ \sum_{\ii=(i_{1},...,i_{d_{2}})\in\{r+1,...,n+1\}^{d_{2}}}E_{\ii}(d\xx)t_{i_{1}}^{(1)}...t_{i_{d_{2}}}^{(d_{2})} \] o\`u \[ t_{i}^{(j)}=\left\{\begin{array}{lll} y_{i}^{(j)} & \mbox{si} & i\in \{ r+1,...,m\} \\ z_{i}^{(j)} & \mbox{si} & i\in \{ m+1,...,n+1\} \end{array}\right. \] pour tout $ j\in \{1,...,d_{2}\} $, et $ E_{\ii}(d\xx) $ sym\'etrique en $ \ii $. Remarquons par ailleurs que $ F_{1} $ et $ F_{2} $ sont homog\`enes de degr\'e $ d_{1} $ en $ (\xx,\tilde{\yy}) $. \\

Pour $ \tilde{\zz}\in [-P_{2},P_{2}]^{d_{2}(n-m+1)} $ fix\'e, on note \[ S_{d,\tilde{\zz}}(\alpha)=\sum_{|\xx|\leqslant P_{1}}\sum_{\yy^{(1)},...,\yy^{(d_{2})}}e(\alpha F_{1}(d\xx,\tilde{\yy},\tilde{\zz})+ \alpha F_{2}(d\xx,\hat{\yy},\hat{\zz})), \]
et d'apr\`es ce qui pr\'ec\`ede, on a \begin{multline*} |S_{d,N,\II,\JJ}(\alpha)|^{2^{d_{2}-1}}\ll \left(P_{1}^{r+1}\right)^{2^{d_{2}-1}-1}\left((dP_{1}P_{2})^{m-r}\right)^{2^{d_{2}-1}-d_{2}} \\ \left(P_{2}^{n-m+1}\right)^{2^{d_{2}-1}-d_{2}}\sum_{|\tilde{\zz}|\leqslant P_{2}}|S_{d,\tilde{\zz}}(\alpha)|. \end{multline*} En posant $ \tilde{d}=d_{1}+d_{2}-2 $, on en d\'eduit : \begin{multline}\label{dtilde} |S_{d,N,\II,\JJ}(\alpha)|^{2^{\tilde{d}}}\ll \left(P_{1}^{r+1}\right)^{2^{\tilde{d}}-2^{d_{1}-1}}\left((dP_{1}P_{2})^{m-r}\right)^{2^{\tilde{d}}-d_{2}2^{d_{1}-1}} \\ \left(P_{2}^{n-m+1}\right)^{2^{\tilde{d}}-d_{2}2^{d_{1}-1}}\left(\sum_{|\tilde{\zz}|\leqslant P_{2}}|S_{d,\tilde{\zz}}(\alpha)|\right)^{2^{d_{1}-1}}. \end{multline}
Par une in\'egalit\'e de H\"{o}lder, on a \[ \left(\sum_{|\tilde{\zz}|\leqslant P_{2}}|S_{d,\tilde{\zz}}(\alpha)|\right)^{2^{d_{1}-1}} \ll \left(P_{2}^{d_{2}(n-m+1)}\right)^{2^{d_{1}-1}-1}\sum_{|\tilde{\zz}|\leqslant P_{2}}|S_{d,\tilde{\zz}}(\alpha)|^{2^{d_{1}-1}},  \] et ainsi \eqref{dtilde} devient \begin{multline}\label{dtilde} |S_{d,N,\II,\JJ}(\alpha)|^{2^{\tilde{d}}}\ll \left(P_{1}^{r+1}\right)^{2^{\tilde{d}}-2^{d_{1}-1}}\left((dP_{1}P_{2})^{m-r}\right)^{2^{\tilde{d}}-d_{2}2^{d_{1}-1}} \\ \left(P_{2}^{n-m+1}\right)^{2^{\tilde{d}}-d_{2}}\sum_{|\tilde{\zz}|\leqslant P_{2}}|S_{d,\tilde{\zz}}(\alpha)|^{2^{d_{1}-1}}. \end{multline}
Par ailleurs, en appliquant le proc\'ed\'e de diff\'erenciation pr\'ec\'edent \`a $ S_{d,\tilde{\zz}}(\alpha) $, on obtient : \begin{multline*} |S_{d,\tilde{\zz}}(\alpha)|^{2^{d_{1}-1}} \ll \left(P_{1}^{r+1}\right)^{2^{d_{1}-1}-d_{1}}\left((dP_{1}P_{2})^{d_{2}(m-r)}\right)^{2^{d_{1}-1}-d_{1}} \\ \sum_{|\xx^{(1)}|\leqslant P_{1}}\sum_{|\yy^{(1,1)}|\leqslant dP_{1}P_{2}}...\sum_{|\yy^{(d_{2},1)}|\leqslant dP_{1}P_{2}}  \sum_{|\xx^{(2)}|\leqslant P_{1}}\\ \sum_{|\yy^{(1,2)}|\leqslant dP_{1}P_{2}}...\sum_{|\xx^{(d_{1})}|\leqslant P_{1}}\sum_{|\yy^{(1,d_{1})}|\leqslant dP_{1}P_{2}}...\sum_{|\yy^{(d_{2},d_{1})}|\leqslant dP_{1}P_{2}} \\ e\left( \sum_{i=1,2} \mathcal{F}_{d_{1}}^{(i)}((\xx^{(1)},(\yy^{(j,1)})_{j\in \{1,...,d_{2}\}},...,(\xx^{(d_{2})},(\yy^{(j,d_{1})})_{j\in \{1,...,d_{2}\}}) \right.\\ \left. -\mathcal{F}^{(i)}_{d_{1}-1}((\xx^{(1)},(\yy^{(j,1)})_{j\in \{1,...,d_{2}\}},...,(\xx^{(d_{1}-1)},(\yy^{(j,d_{1}-1)})_{j\in \{1,...,d_{2}\}}  \right), \end{multline*}
o\`u pour $ i \in \{1,2\} $, $ \mathcal{F}_{k}^{(i)} $ d\'esigne la forme de \eqref{formefd} associ\'ee \`a $ \mathcal{F}(\xx,\tilde{\yy})=\alpha F_{i}(d\xx,\tilde{\yy},\tilde{\zz}) $ pour un $ \tilde{\zz} $ fix\'e. On remarque que \begin{multline*} \mathcal{F}_{d_{1}}^{(1)}\left((\xx^{(i)},\yy^{(j,i)})_{\substack{i\in \{1,...,d_{1}\} \\ j\in \{1,...,d_{2}\}}}\right)-\mathcal{F}_{d_{1}-1}^{(1)}\left((\xx^{(i)},\yy^{(j,i)})_{\substack{i\in \{1,...,d_{1}-1\} \\ j\in \{1,...,d_{2}\}}}\right) \\ =\Gamma_{d}^{(1)}\left((\xx^{(i)},\yy^{(j,i)},\zz^{(j)})_{\substack{i\in \{1,...,d_{1}\} \\ j\in \{1,...,d_{2}\}}}\right)+ g_{d}\left((\xx^{(i)},\yy^{(j,i)},\zz^{(j)})_{\substack{i\in \{1,...,d_{1}-1\} \\ j\in \{1,...,d_{2}\}}}\right) \end{multline*}
 o\`u $ \Gamma_{d}^{(1)} $ est une forme lin\'eaire en $ (\xx^{(i)},\yy^{(j,i)})_{\substack{ j\in \{1,...,d_{2}\}}} $ pour chaque $ i\in \{1,...,d_{1}\}$, de la forme \[ \alpha\sum_{\ii=(i_{1},...,i_{d_{1}})\in I^{d_{1}}}G_{d,\ii}(\tilde{\zz})u_{i_{1}}^{(1)}...u_{i_{d_{1}}}^{(d_{1})}\] avec \[ I=\{ 0,1,...r \} \cup \{(r+1,1),...(m,1)...(r+1,d_{2}),...,(m,d_{2})\}, \]\begin{equation}\label{uij} u_{i}^{(j)}=\left\{\begin{array}{lll} x_{i}^{(j)} & \mbox{si} & i\in \{ 0,1,...,r\} \\ y_{k}^{(l,j)} & \mbox{si} & i=(k,l)\in \{r+1,...m\}\times\{1,...,d_{2}\} \end{array}\right. \end{equation} avec $ G_{d,\ii}(\tilde{\zz})\in \ZZ[d,\tilde{\zz}] $ sym\'etrique en $ \ii $ et dont le degr\'e en $ d $ est \[ f_{\ii}=\card\{k\in \{1,...,d_{1}\} \; | \; i_{k}\in \{0,...,r\}\} \] et on peut donc \'ecrire \[ G_{d,\ii}(\tilde{\zz})=d^{f_{\ii}} G_{\ii}(\tilde{\zz}) \] avec $ G_{\ii}(\tilde{\zz}) $ sym\'etrique en $ \ii $. \\

  D'autre part, on remarque que puisque $ F_{2} $ ne d\'ependait que de $ \xx,\hat{\yy},\hat{\zz} $, la partie \[ \mathcal{F}_{d_{1}}^{(2)}\left((\xx^{(i)},\yy^{(j,i)})_{\substack{i\in \{1,...,d_{1}\} \\ j\in \{1,...,d_{2}\}}}\right)-\mathcal{F}_{d_{1}-1}^{(2)}\left((\xx^{(i)},\yy^{(j,i)})_{\substack{i\in \{1,...,d_{1}-1\} \\ j\in \{1,...,d_{2}\}}}\right) \] est en fait un polyn\^ome en $ \tilde{\xx},\hat{\zz}, (\yy^{(j,i)})_{\substack{i\in \{1,...,d_{1}\}\\ j \in \{1,...,d_{2}-1\}}} $ de la forme \begin{multline*}\Gamma^{(2)}_{d}\left((\xx^{(i)},\yy^{(j,i)},\zz^{(j)})_{\substack{i\in \{1,...,d_{1}\} \\ j\in \{1,...,d_{2}-1\}}}\right)+ h_{d}\left((\xx^{(i)},\yy^{(j,i)},\zz^{(j)})_{\substack{i\in \{1,...,d_{1}-1\} \\ j\in \{1,...,d_{2}-1\}}}\right) \end{multline*} o\`u $ \Gamma^{(2)}_{d} $ est une forme lin\'eaire en $ (\xx^{(i)},\yy^{(j,i)})_{\substack{ j\in \{1,...,d_{2}-1\}}} $ pour tout $ i\in \{1,...,d_{1}\} $, de la forme \[ \alpha\sum_{\ii=(i_{1},...,i_{d_{1}})\in (I')^{d_{1}}}H_{d,\ii}(\tilde{\zz})u_{i_{1}}^{(1)}...u_{i_{d_{1}}}^{(d_{1})}\] avec $ I'=\{ 0,1,...r \} \cup \{(r+1,1),...(m,1)...(r+1,d_{2}-1),...,(m,d_{2}-1)\} $. On observe en particulier que $ \Gamma^{(2)}_{d} $ est ind\'ependant de $ (\yy^{(d_{2},i)},\zz^{(d_{2})})_{\substack{i\in \{1,...,d_{1}\}}} $. En regroupant les r\'esultats obtenus on trouve \begin{multline}\label{Salpha} |S_{d,N,\II,\JJ}(\alpha)|^{2^{\tilde{d}}}\ll \left(P_{1}^{r+1}\right)^{2^{\tilde{d}}-d_{1}}\left((dP_{1}P_{2})^{m-r}\right)^{2^{\tilde{d}}-d_{1}d_{2}} \\ \left(P_{2}^{n-m+1}\right)^{2^{\tilde{d}}-d_{2}}\sum_{\tilde{\zz}}\sum_{(\xx^{(i)},\yy^{(j,i)})_{\substack{i\in \{1,...,d_{1}-1\} \\ j\in \{1,...,d_{2}\}}}}\left|\sum_{\xx^{(d_{1})},\yy^{(1,d_{1})},...,\yy^{(d_{2},d_{1})}} \right. \\ \left. e\left( \Gamma^{(1)}_{d}\left((\xx^{(i)},\yy^{(j,i)},\zz^{(j)})_{\substack{i\in \{1,...,d_{1}\} \\ j\in \{1,...,d_{2}\}}}\right)+\Gamma^{(2)}_{d}\left((\xx^{(i)},\yy^{(j,i)},\zz^{(j)})_{\substack{i\in \{1,...,d_{1}\} \\ j\in \{1,...,d_{2}-1\}}}\right)\right)\right| \end{multline}

Avant d'aller plus loin, il convient de faire la remarque suivante :
\begin{lemma}\label{remarquesym}

Remarquons que si l'on avait diff\'erenci\'e la forme $ F $ en $ (\xx,\yy) $ puis en $ (\tilde{\yy},\zz) $ plut\^ot qu'en $ (\yy,\zz) $ puis en $ (\xx,\tilde{\yy}) $, on aurait obtenu : \begin{multline}\label{Salpha} |S_{d,N,\II,\JJ}(\alpha)|^{2^{\tilde{d}}}\ll \left(P_{1}^{r+1}\right)^{2^{\tilde{d}}-d_{1}}\left((dP_{1}P_{2})^{m-r}\right)^{2^{\tilde{d}}-d_{1}d_{2}} \\ \left(P_{2}^{n-m+1}\right)^{2^{\tilde{d}}-d_{2}}\sum_{\tilde{\xx}}\sum_{(\yy^{(j,i)},\zz^{(j)})_{\substack{i\in \{1,...,d_{1}\} \\ j\in \{1,...,d_{2}-1\}}}}\left|\sum_{\xx^{(d_{1})},\yy^{(1,d_{1})},...,\yy^{(d_{2},d_{1})}} \right. \\ \left. e\left( \Gamma^{(1)'}_{d}\left((\xx^{(i)},\yy^{(j,i)},\zz^{(j)})_{\substack{i\in \{1,...,d_{1}\} \\ j\in \{1,...,d_{2}\}}}\right)+\Gamma^{(2)'}_{d}\left((\xx^{(i)},\yy^{(j,i)},\zz^{(j)})_{\substack{i\in \{1,...,d_{1}-1\} \\ j\in \{1,...,d_{2}\}}}\right)\right)\right|, \end{multline} avec la propri\'et\'e $ \Gamma^{(1)'}_{d}=\Gamma^{(1)}_{d} $.
 \end{lemma}
 \begin{proof}
 On pose \[ F(\xx,\yy,\zz)= \sum_{\substack{m_{1},m_{2},m_{3}\in \NN \\ m_{1}+m_{2}=d_{1}\\ m_{2}+m_{3}=d_{2}}}\sum_{\substack{\ii=\{0,...,r\}^{m_{1}} \\ \jj=\{r+1,...,m\}^{m_{2}} \\ \kk=\{m+1,...,n+1\}^{m_{3}}}}\alpha_{\ii,\jj,\kk}x_{i_{1}}...x_{i_{m_{1}}}y_{j_{1}}...y_{j_{m_{2}}}z_{k_{1}}...z_{k_{m_{3}}}, \] (avec $ \alpha_{\ii,\jj,\kk} $ sym\'etrique en $ \ii,\jj,\kk $).  La forme multilin\'eaire $ F_{1}(d\xx,\tilde{\yy},\tilde{\zz}) $ pr\'ec\'edente est alors  \begin{multline*} d_{2}!\sum_{\substack{m_{1},m_{2},m_{3}\in \NN \\ m_{1}+m_{2}=d_{1}\\ m_{2}+m_{3}=d_{2}}}d^{m_{1}}\sum_{\substack{\ii=\{0,...,r\}^{m_{1}} \\ \jj=\{r+1,...,m\}^{m_{2}} \\ \kk=\{m+1,...,n+1\}^{m_{3}}}}\alpha_{\ii,\jj,\kk}x_{i_{1}}...x_{i_{m_{1}}} \\ \sum_{\sigma\in \MM(d_{2},m_{2})}y_{j_{1}}^{\sigma(1)}...y_{j_{m_{2}}}^{\sigma(m_{2})}z_{k_{1}}^{\sigma(m_{2}+1)}...z_{k_{m_{3}}}^{\sigma(m_{2}+m_{3})} \end{multline*}
 o\`u $ \MM(1,...,d_{2}) $ d\'esigne l'ensemble des permutations $ \sigma $ de $ \{1,...,d_{2}\} $ telles que $ \sigma(1)<\sigma(2)<...<\sigma(m_{2}) $ et $ \sigma(m_{2}+1)<\sigma(m_{2}+2)<...<\sigma(m_{2}+m_{3})=\sigma(d_{2}) $. 
La forme multilin\'eaire $ \Gamma^{(1)}_{d} $ obtenue en diff\'erenciant en $ (\xx,\tilde{\yy}) $ est alors \begin{multline*} d_{1}!d_{2}!\sum_{\substack{m_{1},m_{2},m_{3}\in \NN \\ m_{1}+m_{2}=d_{1}\\ m_{2}+m_{3}=d_{2}}}d^{m_{1}}\sum_{\substack{\ii=\{0,...,r\}^{m_{1}} \\ \jj=\{r+1,...,m\}^{m_{2}} \\ \kk=\{m+1,...,n+1\}^{m_{3}}}}\alpha_{\ii,\jj,\kk}\sum_{\tau\in \MM(d_{1},m_{1})}\sum_{\sigma\in \MM(d_{2},m_{2})} \\ x_{i_{1}}^{\tau(1)}...x_{i_{m_{1}}}^{\tau(m_{1})}y_{j_{1}}^{(\sigma(1),\tau(m_{1}+1))}...y_{j_{m_{2}}}^{(\sigma(m_{2}),\tau(m_{1}+m_{2})}z_{k_{1}}^{\sigma(m_{2}+1)}...z_{k_{m_{3}}}^{\sigma(m_{2}+m_{3})} \end{multline*}
 Il est alors clair que l'on obtient le m\^eme r\'esultat en diff\'erenciant en $ (\xx,\yy) $ puis en $ (\tilde{\yy},\zz) $. 
 \end{proof}

On remarque que, pour $ \tilde{\zz} $ et $ (\xx^{(i)},\yy^{(j,i)})_{\substack{i\in \{1,...,d_{1}-1\} \\ j\in \{1,...,d_{2}\}}} $ fix\'es si l'on note $ \Gamma_{d}=\Gamma^{(1)}_{d}+\Gamma^{(2)}_{d} $: \begin{multline}\label{produit}  \left|\sum_{\xx^{(d_{1})},\yy^{(1,d_{1})},...,\yy^{(d_{2},d_{1})}}  e\left( \Gamma_{d}\left((\xx^{(i)},\yy^{(j,i)},\zz^{(j)})_{\substack{i\in \{1,...,d_{1}\} \\ j\in \{1,...,d_{2}\}}}\right)\right)\right| \\ = \prod_{i\in I}\left|\sum_{u_{i}^{(d_{1})}}
e\left(\alpha u_{i}^{(d_{1})}\left(\sum_{\substack{\ii\in I\; | \; i_{d_{1}}=i }}G_{d,\ii}(\tilde{\zz})u_{i_{1}}^{(1)}...u_{i_{d_{1}-1}}^{(d_{1}-1)}+\sum_{\ii'\in I' \; | \; i_{d_{1}}'=i}
H_{d,\ii}(\tilde{\zz})u_{i_{1}'}^{(1)}...u_{i_{d_{1}-1}'}^{(d_{1}-1)}\right)\right)\right| \end{multline} 
o\`u la somme sur $ u_{i}^{(d_{1})} $ porte sur $  u_{i}^{(d_{1})} $ appartenant \`a un intervalle de taille $ O(P_{1}) $ si $ i\in \{0,...,r\} $ et de taille $ O(dP_{1}P_{2}) $ pour \[ i\in \{(r+1,1),...(m,1)...(r+1,d_{2}),...,(m,d_{2})\}. \] Pour simplifier les notations on pose \begin{multline}\label{gammai1}
\gamma_{d,i}^{(1)}\left((\xx^{(k)},\yy^{(j,k)},\zz^{(j)})_{\substack{k\in \{1,...,d_{1}-1\} \\ j\in
 \{1,...,d_{2}\}}}\right)=\sum_{\substack{\ii\in I^{d_{1}}\; | \; i_{d_{1}}=i }}G_{d,\ii}(\tilde{\zz})u_{i_{1}}^{(1)}...u_{i_{d_{1}-1}}^{(d_{1}-1)} \end{multline} \begin{multline}\label{gammai2}
\gamma_{d,i}^{(2)}\left((\xx^{(k)},\yy^{(j,k)},\zz^{(j)})_{\substack{k\in \{1,...,d_{1}-1\} \\ j\in
 \{1,...,d_{2}\}}}\right)=\sum_{\ii'\in (I')^{d_{1}} \; | \; i_{d_{1}}'=i}
H_{d,\ii}(\tilde{\zz})u_{i_{1}'}^{(1)}...u_{i_{d_{1}-1}'}^{(d_{1}-1)} \end{multline} o\`u les $ u_{i}^{(j)} $ sont les variables d\'efinies par\;\ref{uij}, et \begin{equation}\label{gamma} \gamma_{d,i}=\gamma_{d,i}^{(1)}+\gamma_{d,i}^{(2)} \end{equation} En notant pour tout r\'eel $ x $ \[ ||x||=\inf_{m\in \ZZ }|x-m|, \] et en consid\'erant la majoration \[ \sum_{m\in I(P)\cap\ZZ}e(\beta m) \ll \min\{P,||\beta||^{-1}\} \] pour tout intervalle $ I(P) $ de taille $ O(P) $ avec $ P\geqslant 1 $, on peut alors majorer \eqref{produit} par : 
\begin{multline*}
\prod_{i\in I}\min\left(H_{i}, \left|\left|\alpha \gamma_{d,i}\left(\xx^{(k)},\yy^{(j,k)},\zz^{(j)})_{\substack{k\in \{1,...,d_{1}-1\} \\ j\in
 \{1,...,d_{2}\}}}\right) \right|\right|^{-1}\right) \end{multline*}
 o\`u \[ H_{i}=\left\{\begin{array}{lll} P_{1} & \mbox{si} & i\in \{ 0,1,...,r\} \\ dP_{1}P_{2} & \mbox{si} & i=(k,l)\in \{r+1,...m\}\times\{1,...,d_{2}\} \end{array}\right.. \]
Pour tout $ \rr=(r_{i})_{i\in I}\in \prod_{i\in I}(\NN\cap [0,H_{i}[) $, et pour $ (\xx^{(i)},\yy^{(j,i)},\zz^{(j)})_{\substack{i\in \{1,...,d_{1}-1\} \\ j\in \{1,...,d_{2}-1\}}}  $ fix\'es, on note $ \mathcal{A}\left((\xx^{(i)},\yy^{(j,i)},\zz^{(j)})_{\substack{i\in \{1,...,d_{1}-1\} \\ j\in \{1,...,d_{2}-1\}}},\rr\right) $, l'ensemble des \'el\'ements $ \zz^{(d_{2})},\yy^{(d_{2},1)},...,\yy^{(d_{2},d_{1}-1)} $ tels que $ |\zz^{(d_{2})}|\leqslant P_{1} $, $ |\yy^{(d_{2},k)}|\leqslant dP_{1}P_{2} $ pour tout $ k\in \{1,...,d_{1}-1\} $ et \[ \forall i\in I,\; \; r_{i}H_{i}^{-1}\leqslant \left\{\alpha\gamma_{d,i}\left((\xx^{(k)},\yy^{(j,k)},\zz^{(j)})_{\substack{k\in \{1,...,d_{1}-1\} \\ j\in
 \{1,...,d_{2}\}}}\right)\right\}< (r_{i}+1)H_{i}^{-1} \] et $ A\left((\xx^{(i)},\yy^{(j,i)},\zz^{(j)})_{\substack{i\in \{1,...,d_{1}-1\} \\ j\in \{1,...,d_{2}-1\}}},\rr\right) $ le cardinal de cet ensemble. On a alors l'estimation \begin{multline}\label{inegA}
 \sum_{\zz^{(d_{2})},\yy^{(d_{2},1)},...,\yy^{(d_{2},d_{1}-1)}}\left|\sum_{\xx^{(d_{1})},\yy^{(1,d_{1})},...,\yy^{(d_{2},d_{1})}}  e\left( \Gamma_{d}\left((\xx^{(i)},\yy^{(j,i)},\zz^{(j)})_{\substack{i\in \{1,...,d_{1}\} \\ j\in \{1,...,d_{2}\}}}\right)\right)\right| \\ \ll \sum_{\rr}A\left((\xx^{(i)},\yy^{(j,i)},\zz^{(j)})_{\substack{i\in \{1,...,d_{1}-1\} \\ j\in \{1,...,d_{2}-1\}}},\rr\right) \prod_{i\in I }\min\left( H_{i},\max\left( \frac{H_{i}}{r_{i}},\frac{H_{i}}{H_{i}-r_{i}-1}\right)\right)
\end{multline}
Par ailleurs, si \begin{multline*} (\zz^{(d_{2})},(\yy^{(d_{2},i)})_{i\in \{1,...d_{1}-1\}}),  (\zz^{'(d_{2})},(\yy^{'(d_{2},i)})_{i\in \{1,...d_{1}-1\}}) \\ \in \mathcal{A}\left((\xx^{(i)},\yy^{(j,i)},\zz^{(j)})_{\substack{i\in \{1,...,d_{1}-1\} \\ j\in \{1,...,d_{2}-1\}}},\rr\right), \end{multline*} on a alors, pour tout $ i\in I $ : \begin{multline*}
\gamma_{d,i}\left((\xx^{(k)},\yy^{(j,k)},\zz^{(j)})_{\substack{k\in \{1,...,d_{1}-1\} \\ j\in
 \{1,...,d_{2}-1\}}},\zz^{(d_{2})},(\yy^{(d_{2},k)})_{k\in \{1,...d_{1}-1\}}\right) \\ -\gamma_{d,i}\left((\xx^{(k)},\yy^{(j,k)},\zz^{(j)})_{\substack{k\in \{1,...,d_{1}-1\} \\ j\in
 \{1,...,d_{2}-1\}}},\zz^{'(d_{2})},(\yy^{'(d_{2},k)})_{k\in \{1,...d_{1}-1\}}\right) \\ = \gamma_{d,i}^{(1)}\left((\xx^{(k)},\yy^{(j,k)},\zz^{(j)})_{\substack{k\in \{1,...,d_{1}-1\} \\ j\in
 \{1,...,d_{2}-1\}}}, \right. \\ \left. \zz^{(d_{2})}-\zz^{'(d_{2})},(\yy^{(d_{2},k)}-\yy^{'(d_{2},k)})_{k\in \{1,...d_{1}-1\}}\right)
\end{multline*}
(car $ \gamma_{d,i}^{(2)} $ ne d\'epend pas de $ (\zz^{(d_{2})},(\yy^{(d_{2},k)})_{k\in \{1,...d_{1}-1\}}) $ et $ \gamma_{d,i}^{(1)} $ est lin\'eaire en $ (\zz^{(d_{2})},(\yy^{(d_{2},i)})_{i\in \{1,...d_{1}-1\}}) $). En notant $ N\left((\xx^{(i)},\yy^{(j,i)},\zz^{(j)})_{\substack{i\in \{1,...,d_{1}-1\} \\ j\in \{1,...,d_{2}-1\}}}\right)  $ le cardinal de l'ensemble des $ \zz^{(d_{2})},\yy^{(d_{2},1)},...,\yy^{(d_{2},d_{1}-1)} $ tels que $ |\zz^{(d_{2})}|\leqslant P_{1} $, $ |\yy^{(d_{2},j)}|\leqslant dP_{1}P_{2} $ et \[ \forall i\in I,\; \;  \left|\left|\alpha\gamma_{d,i}^{(1)}\left((\xx^{(k)},\yy^{(j,k)},\zz^{(j)})_{\substack{k\in \{1,...,d_{1}-1\} \\ j\in
 \{1,...,d_{2}\}}}\right)\right|\right|<H_{i}^{-1}, \] on a alors \[ A\left((\xx^{(i)},\yy^{(j,i)},\zz^{(j)})_{\substack{i\in \{1,...,d_{1}-1\} \\ j\in \{1,...,d_{2}-1\}}},\rr\right) \ll N\left((\xx^{(i)},\yy^{(j,i)},\zz^{(j)})_{\substack{i\in \{1,...,d_{1}-1\} \\ j\in \{1,...,d_{2}-1\}}}\right),  \] et donc \eqref{inegA} donne \begin{multline*} \sum_{\zz^{(d_{2})},\yy^{(d_{2},1)},...,\yy^{(d_{2},d_{1}-1)}}\left|\sum_{\xx^{(d_{1})},\yy^{(1,d_{1})},...,\yy^{(d_{2},d_{1})}}  e\left( \Gamma\left((\xx^{(i)},\yy^{(j,i)},\zz^{(j)})_{\substack{i\in \{1,...,d_{1}\} \\ j\in \{1,...,d_{2}\}}}\right)\right)\right| \\ \ll N\left((\xx^{(i)},\yy^{(j,i)},\zz^{(j)})_{\substack{i\in \{1,...,d_{1}-1\} \\ j\in \{1,...,d_{2}-1\}}}\right)\sum_{\rr} \prod_{i\in I }\min\left( H_{i},\max\left( \frac{H_{i}}{r_{i}},\frac{H_{i}}{H_{i}-r_{i}-1}\right)\right) \\ \ll N\left((\xx^{(i)},\yy^{(j,i)},\zz^{(j)})_{\substack{i\in \{1,...,d_{1}-1\} \\ j\in \{1,...,d_{2}-1\}}}\right) (P_{1}\log P_{1})^{r+1}(dP_{1}P_{2}\log(dP_{1}P_{2}))^{d_{2}(m-r)}
\end{multline*}
En r\'esum\'e, si, pour tous $ H_{1}^{(i)},H_{2}^{(i,j)},H_{3}^{(j)}\geqslant 1$ et $ B_{1},B_{2}\geqslant 1 $, \[ M\left( \alpha, (H_{1}^{(i)},H_{2}^{(i,j)},H_{3}^{(j)})_{\substack{i\in \{1,...,d_{1}-1 \}\\ j\in \{1,...,d_{2}\}}}, B_{1}^{-1},B_{2}^{-1} \right)   \] d\'esigne le cardinal de l'ensemble des $(\xx^{(i)},\yy^{(j,i)},\zz^{(j)})_{\substack{i\in \{1,...,d_{1}-1\} \\ j\in \{1,...,d_{2}\}}}$ tels que $ |\xx^{(i)}|\leqslant H_{1}^{(i)} $, $ |\yy^{(i,j)}|\leqslant H_{2}^{(i,j)} $, $ |\zz^{(j)}|\leqslant H_{2}^{(j)} $ pour tous $ (i,j)\in \{1,...,d_{1}-1\}\times \{1,...,d_{2}\} $ et \[ \forall k \in \{0,...,r\}  , \; \; \left|\left|\alpha\gamma_{d,k}^{(1)}\left((\xx^{(i)},\yy^{(j,i)},\zz^{(j)})_{\substack{i\in \{1,...,d_{1}-1\} \\ j\in
 \{1,...,d_{2}\}}}\right)\right|\right|<B_{1}^{-1}, \] \[  \forall k \in \{r+1,...m\}\times\{1,...,d_{2}\} , \; \; \left|\left|\alpha\gamma_{d,k}^{(1)}\left((\xx^{(i)},\yy^{(j,i)},\zz^{(j)})_{\substack{i\in \{1,...,d_{1}-1\} \\ j\in
 \{1,...,d_{2}\}}}\right)\right|\right|<B_{2}^{-1}, \] en reprenant la formule \eqref{Salpha}, on obtient la majoration (pour $ \varepsilon>0 $ arbitrairement petit) \begin{multline*}
 |S_{d,N,\II,\JJ}(\alpha)|^{2^{\tilde{d}}}\ll \left(P_{1}^{r+1}\right)^{2^{\tilde{d}}-(d_{1}-1)+\varepsilon}\left((dP_{1}P_{2})^{m-r}\right)^{2^{\tilde{d}}-(d_{1}-1)d_{2}+\varepsilon}\left(P_{2}^{n-m+1}\right)^{2^{\tilde{d}}-d_{2}} \\ M\left(\alpha,(P_{1},dP_{1}P_{2},P_{2})_{\substack{i\in \{1,...,d_{1}-1\} \\ j\in \{1,...,d_{2}\}}}, P_{1}^{-1},(dP_{1}P_{2})^{-1} \right).
\end{multline*}

 On en d\'eduit (en sommant sur $ N\in \{0,...,P_{1}\} $ et sur les $ \II,\JJ\subset\{r+1,...,m\} $) le lemme ci-dessous 
 \begin{lemma}\label{dilemme1}
 Pour $ \varepsilon>0 $ arbitrairement petit, et pour $ \kappa>0 $, $ P>0 $ des r\'eels fix\'es, l'une au moins des assertions suivantes est vraie  \begin{enumerate}
\item  $ |S_{d}(\alpha)|\ll d^{m-r+\varepsilon+\frac{d_{1}(r+1)}{2^{\tilde{d}}}}P_{1}^{m+2+\varepsilon}P_{2}^{n-r+1+\varepsilon}P^{-\kappa}, $
\item \begin{multline*}  M\left(\alpha,(P_{1},dP_{1}P_{2},P_{2})_{\substack{i\in \{1,...,d_{1}-1 \}\\ j\in \{1,...,d_{2}\}}}, P_{1}^{-1},(dP_{1}P_{2})^{-1} \right) \\  \gg  d^{d_{1}(r+1)}\left(P_{1}^{r+1}\right)^{(d_{1}-1)}\left((dP_{1}P_{2})^{m-r}\right)^{(d_{1}-1)d_{2}}\left(P_{2}^{n-m+1}\right)^{d_{2}}P^{-2^{\tilde{d}}\kappa}\end{multline*}
\end{enumerate}

 \end{lemma}
\begin{rem}
Si $ \kappa $ est petit, la condition $ 1 $ donne une majoration de $ |S_{d}(\alpha)| $ plus grande que la majoration triviale, \[ |S_{d}(\alpha)| \ll d^{m-r}P_{1}^{m+1}P_{2}^{n-r+1}, \] c'est pourquoi nous utiliserons uniquement cette majoration pour $ P^{\kappa}>P_{1}^{d_{1}}P_{2}^{d_{2}} $. \end{rem}

\subsection{G\'eom\'etrie des nombres}

Nous allons \`a pr\'esent \'etablir des r\'esultats de g\'eom\'etrie des nombres qui nous serons utiles pour la suite de cette section. Il s'agit en fait de g\'en\'eralisations de \cite[Lemme 12.6]{D} et de \cite[Lemme 3.1]{S1}.

\begin{lemma}\label{geomnomb1}
On consid\`ere deux entiers $ n_{1},n_{2}>0 $, des r\'eels $ (\lambda_{i,j})_{\substack{1\leqslant i\leqslant n_{1} \\1\leqslant i\leqslant n_{2}} } $ et des formes lin\'eaires \[ \forall i \in  \{1,...,n_{1}\}, \; \forall \uu=(u_{1},...,u_{n_{2}}) \; L_{i}(\uu)=\sum_{j=1}^{n_{2}}\lambda_{i,j}u_{j}, \] et \[ \forall j \in  \{1,...,n_{2}\}, \; \forall \uu=(u_{1},...,u_{n_{1}}) \; L^{t}_{j}(\uu)=\sum_{i=1}^{n_{1}}\lambda_{i,j}u_{i}. \] Soient $ a_{1},...,a_{n_{2}},b_{1},...,b_{n_{1}}>1 $ des r\'eels fix\'es. Pour tout $ 0\leqslant Z\leqslant 1 $, on note \begin{multline*} U(Z)=\Card\left\{ (u_{1},...,u_{n_{2}},u_{n_{2}+1},...,u_{n_{2}+n_{1}})\in \ZZ^{n_{1}+n_{2}} \; | \; \forall j \in \{1,...,n_{2}\} \; \right.\\ \left.  |u_{j}|\leqslant a_{j}Z \;\et \; \forall i\in \{1,...,n_{1}\}\; |L_{i}(u_{1},...,u_{n_{2}})-u_{n_{2}+i}|\leqslant b_{i}^{-1}Z \right\}, \end{multline*} \begin{multline*} U^{t}(Z)=\Card\left\{ (u_{1},...,u_{n_{1}},u_{n_{1}+1},...,u_{n_{1}+n_{2}})\in \ZZ^{n_{1}+n_{2}} \; | \; \forall i \in \{1,...,n_{1}\} \; \right.\\ \left. |u_{i}|\leqslant b_{i}Z \;  \et \; \forall j\in \{1,...,n_{2}\}\; |L^{t}_{j}(u_{1},...,u_{n_{1}})-u_{n_{1}+i}|\leqslant a_{j}^{-1}Z \right\}. \end{multline*} Si $ 0<Z_{1}\leqslant Z_{2}\leqslant 1 $, on a alors : \[ U(Z_{2})\ll_{n_{1},n_{2}} \max\left( \left(\frac{Z_{2}}{Z_{1}}\right)^{n_{2}}U(Z_{1}), \frac{Z_{2}^{n_{2}}}{Z_{1}^{n_{1}}}\frac{\prod_{j=1}^{n_{2}}a_{j}}{\prod_{i=1}^{n_{1}}b_{i}}U^{t}(Z_{1})\right). \]
\end{lemma}
\begin{rem}
Le lemme\;3.1 de \cite{S1} pr\'esente uniquement le cas o\`u $ a_{1}=...=a_{n_{2}}=a $ et $ b_{1}=...=b_{n_{1}}=b $. Cette g\'en\'eralisation aux $ a_{i} $ et $ b_{i} $ distincts permet de donner des estimations du nombre de points dans un r\'eseau dont les coordonn\'ees sont born\'ees par des bornes distinctes.
\end{rem}

\begin{proof}[D\'emonstration du lemme\;\ref{geomnomb1}]
On consid\`ere le r\'eseau $ \Lambda $ de $ \RR^{n_{2}+n_{1}} $ d\'efini comme l'ensemble des points \[ (x_{1},...,x_{n_{2}},x_{n_{2}+1},...,x_{n_{2}+n_{1}})\in \RR^{n_{1}+n_{2}} \] tels qu'il existe \[ (u_{1},...,u_{n_{2}},u_{n_{2}+1},...,u_{n_{2}+n_{1}})\in \ZZ^{n_{1}+n_{2}} \] tels que   \[ \begin{array}{rcl} a_{1}x_{1} & = & u_{1}, \\ & \vdots & \\ a_{n_{2}}x_{n_{2}} & = & u_{n_{2}}, \\ b_{1}^{-1}x_{n_{2}+1} & = & L_{1}(u_{1},...,u_{n_{2}})+u_{n_{2}+1}, \\ & \vdots & \\ b_{n_{1}}^{-1}x_{n_{2}+n_{1}} & = & L_{n_{1}}(u_{1},...,u_{n_{2}})+u_{n_{2}+n_{1}}.
\end{array}\]

Ce r\'eseau est d\'efini par la matrice (i.e une base de ce r\'eseau est donn\'ee par les colonnes de la matrice) \[ A=\begin{pmatrix}
a_{1}^{-1} &  & (0) & 0 & \cdots & 0 \\ 
 & \ddots &  & \vdots &  & \vdots \\ 
(0) &  & a_{n_{2}}^{-1} & 0 & \cdots & 0 \\ 
b_{1}\lambda_{1,1} & \cdots & b_{1}\lambda_{1,n_{2}} & b_{1} &  & (0) \\ 
\vdots &  & \vdots &  & \ddots &  \\ 
b_{n_{1}}\lambda_{n_{1},1} & \cdots & b_{n_{1}}\lambda_{n_{1},n_{2}} & (0) &  & b_{n_{1}}
\end{pmatrix}. \] On remarque que $ U(Z) $ est alors le nombre de points $ (x_{1},...,x_{n_{1}+n_{2}}) $ de $ \Lambda $ tels que $ |x_{i}|\leqslant Z $ pour tout $ i\in\{1,...,n_{1}+n_{2}\} $. Par ailleurs,  \[ B=(A^{t})^{-1}=\begin{pmatrix}
a_{1} &  & (0) & -a_{1}\lambda_{1,1} & \cdots & -a_{1}\lambda_{n_{1},1} \\ 
 & \ddots &  & \vdots &  & \vdots \\ 
(0) &  & a_{n_{2}} & -a_{n_{2}}\lambda_{1,n_{2}} & \cdots & -a_{n_{2}}\lambda_{n_{1},n_{2}} \\ 
0& \cdots & 0 & b_{1}^{-1} &  & (0) \\ 
\vdots &  & \vdots &  & \ddots &  \\ 
0 & \cdots & 0 & (0) &  & b_{n_{1}}^{-1}
\end{pmatrix}. \]
d\'efinit un r\'eseau $ \Omega $ ayant les m\^emes minima successifs que le r\'eseau $ \tilde{\Omega} $ d\'efini par la matrice  \[ \tilde{B}=\begin{pmatrix}
b_{1}^{-1} &  & (0) & 0 & \cdots &  \\ 
 & \ddots &  & \vdots &  & \vdots \\ 
(0) &  & b_{n_{1}}^{-1} & 0 & \cdots & 0 \\ 
a_{1}\lambda_{1,1}& \cdots & a_{1}\lambda_{n_{1},1} & a_{1} &  & (0) \\ 
\vdots &  & \vdots &  & \ddots &  \\ 
a_{n_{2}}\lambda_{1,n_{2}} & \cdots & a_{n_{2}}\lambda_{n_{1},n_{2}} & (0) &  & a_{n_{2}}
\end{pmatrix}. \]

On pose $ c=\left(\frac{\prod_{j=1}^{n_{2}}a_{j}}{\prod_{i=1}^{n_{1}}b_{i}}\right)^{\frac{1}{n_{1}+n_{2}}} $ et $ \Lambda^{\nor}=c\Lambda $, $ \Omega^{\nor}=c^{-1}\tilde{\Omega} $ les r\'eseaux normalis\'es (i.e de d\'eterminant $ 1 $) associ\'es \`a $ \Lambda $ et $ \Omega $. Par la d\'emonstration de \cite[Lemme 3.1]{S1}, on a alors \[ U(Z_{2})\ll_{n_{1},n_{2}} \max\left( \left(\frac{Z_{2}}{Z_{1}}\right)^{n_{2}}U(Z_{1}), \frac{Z_{2}^{n_{2}}}{Z_{1}^{n_{1}}}c^{n_{1}+n_{2}}U^{t}(Z_{1})\right). \] d'o\`u le r\'esultat.
\end{proof}

En particulier, lorsque $ n_{1}=n_{2}=n $, $ a_{i}=b_{i} $ pour tout $ i $ et $ \lambda_{i,j}=\lambda_{j,i} $ on obtient le r\'esultat suivant : 

\begin{lemma}\label{geomnomb2}
Soit $ n>0 $ un entier et $ (\lambda_{i,j})_{\substack{1\leqslant i,j\leqslant n} } $ des r\'eels tels que $ \lambda_{i,j}=\lambda_{j,i} $ pour tous $ i,j $, et des formes lin\'eaires \[ \forall i \in  \{1,...,n_{1}\}, \; \forall \uu=(u_{1},...,u_{n})\;  L_{i}(\uu)=\sum_{j=1}^{n}\lambda_{i,j}u_{j}. \] Soient $ a_{1},...,a_{n}>1 $ des r\'eels fix\'es. Pour tout $ 0\leqslant Z\leqslant 1 $, on note \begin{multline*} U(Z)=\Card\left\{ (u_{1},...,u_{n},u_{n+1},...,u_{2n}) \; | \; \forall j \in \{1,...,n\} \; |u_{j}|\leqslant a_{j}Z \; \right.\\ \left. \et \; \forall i\in \{1,...,n\}\; |L_{i}(u_{1},...,u_{n})-u_{n+i}|\leqslant a_{i}^{-1}Z \right\}. \end{multline*} On a alors \[ U(Z_{2})\ll_{n} \left(\frac{Z_{2}}{Z_{1}}\right)^{n}U(Z_{1}). \]
\end{lemma}

Revenons \`a pr\'esent \`a la situation de la section pr\'ec\'edente, et consid\'erons, pour $ (\xx^{(i)},\yy^{(i,j)},\zz^{(j)})_{\substack{i\in \{1,...,d_{1}-2\} \\ j\in \{1,...,d_{2} \}}} $ fix\'es les $ N=(r+1)+d_{2}(m-r) $ formes lin\'eaires en $ (\xx^{(d_{1}-1)},\yy^{(j,d_{1}-1)})_{j\in \{1,...,d_{2}\}} $ donn\'ees par les $ \alpha \gamma_{d,k}^{(1)} $ pour $ k\in I $. Remarquons que d'apr\`es \eqref{gammai1} on a pour tout $ k\in I $\begin{align*}
\gamma_{d,k}^{(1)}\left((\xx^{(i)},\yy^{(j,i)},\zz^{(j)})_{\substack{i\in \{1,...,d_{1}-1\} \\ j\in
 \{1,...,d_{2}\}}}\right) & =\sum_{\substack{\ii\in I^{d_{1}-1} }}G_{d,\ii,k}(\tilde{\zz})u_{i_{1}}^{(1)}...u_{i_{d_{1}-1}}^{(d_{1}-1)} \\ &=\sum_{\substack{\ii\in I^{d_{1}-1} }}d^{f_{\ii,k}}G_{\ii,k}(\tilde{\zz})u_{i_{1}}^{(1)}...u_{i_{d_{1}-1}}^{(d_{1}-1)}\end{align*} (o\`u $ \tilde{\zz}=(\zz^{(1)},...,\zz^{(d_{2})}) $ et les $ u_{i}^{(j)} $ sont donn\'es par\;\eqref{uij}) et donc pour tous $ k,l\in I $ le coefficient $ \lambda_{k,l} $ en $ u_{l}^{(d_{1}-1)} $ s'\'ecrit : \begin{equation*}
\lambda_{k,l}=\sum_{\substack{\ii\in I^{d_{1}-2} }}d^{f_{\ii,l,k}}G_{\ii,l,k}(\tilde{\zz})u_{i_{1}}^{(1)}...u_{i_{d_{1}-2}}^{(d_{1}-2)} \end{equation*} et on observe que, puisque les $ G_{\ii}(\tilde{\zz}) $ sont sym\'etriques en $ \ii \in I^{d_{1}} $.  \begin{equation*}
\lambda_{k,l}=\lambda_{l,k}. \end{equation*} Pour $ P>0 $ fix\'e, et $  \theta\in [0,1] $ suppos\'es tels que $ P^{\theta}\leqslant P_{1} $, on pose $  Z_{2}=1 $, $ Z_{1}=(dP_{1})^{-1}P^{\theta} $, $ a_{k}=P_{1} $ pour tout $ k\in I_{1}=\{0,...,r\} $, et $ a_{k}=dP_{1}P_{2} $ pour $ k\in I_{2}=\{r+1,...m\}\times\{1,...,d_{2}\} $ de sorte que (en remarquant que $ I=I_{1}\cup I_{2} $) : \[ \begin{array}{lrclcrcl}\forall k \in I_{1}, & a_{k}Z_{2} & = & P_{1}, & \;  & a_{k}Z_{1}& = &P^{\theta}/d \\ \forall k\in I_{2}, &  a_{k}Z_{2} & = & dP_{1}P_{2}, & \;  & a_{k}Z_{1}& = & P_{2}P^{\theta} \\ \forall k \in I_{1}, &  a_{k}^{-1}Z_{2} & = & P_{1}^{-1}, & \;  & a_{k}^{-1}Z_{1} & = & d^{-1}P_{1}^{-2}P^{\theta} \\ \forall k\in I_{2}, & a_{k}^{-1}Z_{2} & = & (dP_{1}P_{2})^{-1}, & \;  & a_{k}^{-1}Z_{1} & = & (dP_{1})^{-2}P_{2}^{-1}P^{\theta} 
\end{array}
\]
En appliquant le lemme \ref{geomnomb2}, on obtient \[ U(Z_{2})\ll \left(\frac{dP_{1}}{P^{\theta}}\right)^{r+1+d_{2}(m-r)}U(Z_{1}), \] avec \begin{multline*} U(Z_{2})=\card\left\{ (\xx^{(d_{1}-1)},(\yy^{(j,d_{1}-1)})_{j\in \{1,...,d_{2}\}})\; | \; |\xx^{(d_{1}-1)}| \leqslant P_{1}, \; |\yy^{(j,d_{1}-1)}|\leqslant dP_{1}P_{2}, \right.\\ \et  \forall k \in I_{1}  , \; \; \left|\left|\alpha\gamma_{d,k}^{(1)}\left((\xx^{(i)},\yy^{(j,i)},\zz^{(j)})_{\substack{i\in \{1,...,d_{1}-1\} \\ j\in
 \{1,...,d_{2}\}}}\right)\right|\right|<P_{1}^{-1}, \\ \left. \forall k \in I_{2} , \; \; \left|\left|\alpha\gamma_{d,k}^{(1)}\left((\xx^{(i)},\yy^{(j,i)},\zz^{(j)})_{\substack{i\in \{1,...,d_{1}-1\} \\ j\in
 \{1,...,d_{2}\}}}\right)\right|\right|<(dP_{1}P_{2})^{-1}\right\}, \end{multline*} et \begin{multline*} U(Z_{1})=\card\left\{ (\xx^{(d_{1}-1)},(\yy^{(j,d_{1}-1)})_{j\in \{1,...,d_{2}\}})\; | \; |\xx^{(d_{1}-1)}| \leqslant P^{\theta}/d, \; |\yy^{(j,d_{1}-1)}|\leqslant P^{\theta}P_{2}, \right.\\ \et  \forall k \in I_{1}  , \; \; \left|\left|\alpha\gamma_{d,k}^{(1)}\left((\xx^{(i)},\yy^{(j,i)},\zz^{(j)})_{\substack{i\in \{1,...,d_{1}-1\} \\ j\in
 \{1,...,d_{2}\}}}\right)\right|\right|<d^{-1}P_{1}^{-2}P^{\theta}, \\ \left. \forall k \in I_{2} , \; \; \left|\left|\alpha\gamma_{d,k}^{(1)}\left((\xx^{(i)},\yy^{(j,i)},\zz^{(j)})_{\substack{i\in \{1,...,d_{1}-1\} \\ j\in
 \{1,...,d_{2}\}}}\right)\right|\right|<d^{-2}P_{1}^{-2}P_{2}^{-1}P^{\theta}\right\}, \end{multline*} 
En sommant sur les $ (\xx^{(i)},\yy^{(i,j)},\zz^{(j)})_{\substack{i\in \{1,...,d_{1}-2\} \\ j\in \{1,...,d_{2} \}}} $, on obtient alors \begin{multline*}
M\left(\alpha,(B_{1}^{(i)},B_{2}^{(j,i)},B_{3}^{(j)})_{\substack{i\in \{1,...,d_{1}-1 \}\\ j\in \{1,...,d_{2}\}}}, P_{1}^{-1},(dP_{1}P_{2})^{-1} \right)\\ \ll \left(\frac{dP_{1}}{P^{\theta}}\right)^{r+1+d_{2}(m-r)}M\left(\alpha,(H_{1}^{(i)},H_{2}^{(j,i)},H_{3}^{(j)})_{\substack{i\in \{1,...,d_{1}-1 \}\\ j\in \{1,...,d_{2}\}}},  d^{-1}P_{1}^{-2}P^{\theta},d^{-2}P_{1}^{-2}P_{2}^{-1}P^{\theta} \right)
\end{multline*}
o\`u \[\forall  \;  i\in \{1,...,d_{1}-1\}, \; \; B_{1}^{(i)}=P_{1} \]\[ \forall  \;  i\in \{1,...,d_{1}-1\}, \; \;B_{2}^{(j,i)}= dP_{1}P_{2}, \]\[ B_{3}^{(j)}=P_{2}, \] et \[ H_{1}^{(i)}=\left\{\begin{array}{lll} P_{1} & \mbox{si} & i\in \{1,...,d_{1}-2\}
\\ P^{\theta}/d& \mbox{si} &  i=d_{1}-1 
\end{array}\right., \]\[ H_{2}^{(j,i)}=\left\{\begin{array}{lll} dP_{1}P_{2} & \mbox{si} & i\in \{1,...,d_{1}-2\}
\\ P^{\theta}P_{2} & \mbox{si} &  i=d_{1}-1 
\end{array}\right., \]\[ H_{3}^{(j)}=P_{2}. \]
Par la suite, on applique le lemme de la m\^eme mani\`ere avec $ (\xx^{(i)},\yy^{(j,i)},\zz^{(j)})_{\substack{i\notin\{d_{1},d_{1}-l\}}} $ fix\'es (pour $ l $ variant de $ 1 $ \`a $ d_{1}-1 $), et en consid\'erant les $ \alpha \gamma_{d,k}^{(1)} $ comme des formes lin\'eaires en $ (\xx^{(d_{1}-l)},\yy^{(j,d_{1}-l)})_{j\in \{1,...,d_{2}\}} $, et en choisissant $  Z_{2}=d^{-\frac{(l-1)}{2}}P_{1}^{-\frac{(l-1)}{2}}P^{\frac{(l-1)\theta}{2}} $, $ Z_{1}=d^{-\frac{(l+1)}{2}}P_{1}^{-\frac{(l+1)}{2}}P^{\frac{(l+1)\theta}{2}} $, $ a_{k}=d^{\frac{(l-1)}{2}}P_{1}^{\frac{(l+1)}{2}}P^{-\frac{(l-1)\theta}{2}} $ pour tout $ k\in I_{1}$, et $ a_{k}=d^{\frac{(l+1)}{2}}P_{1}^{\frac{(l+1)}{2}}P_{2}P^{-\frac{(l-1)\theta}{2}} $ pour $ k\in I_{2} $ de sorte que \[ \begin{array}{lrclcrcl}\forall k \in I_{1}, & a_{k}Z_{2} & = & P_{1}, & \;  & a_{k}Z_{1}& = &P^{\theta}/d \\ \forall k\in I_{2}, &  a_{k}Z_{2} & = & dP_{1}P_{2}, & \;  & a_{k}Z_{1}& = & P_{2}P^{\theta} \\ \forall k \in I_{1}, &  a_{k}^{-1}Z_{2} & = & d^{-(l-1)}P_{1}^{-l}P^{(l-1)\theta}, & \;  & a_{k}^{-1}Z_{1} & = & d^{-(l+1)}P_{1}^{-(l+1)}P^{l\theta} \\ \forall k\in I_{2}, & a_{k}^{-1}Z_{2} & = & d^{-l}P_{1}^{-l}P_{2}^{-1}P^{(l-1)\theta}, & \;  & a_{k}^{-1}Z_{1} & = & d^{-(l+1)}P_{1}^{-(l+1)}P_{2}^{-1}P^{l\theta} 
\end{array}
\]
On obtient alors (\`a l'\'etape $ l $) la majoration : 
\begin{multline*}
M\left(\alpha,(B_{1}^{(i)},B_{2}^{(j,i)},B_{3}^{(j)})_{\substack{i\in \{1,...,d_{1}-1 \}\\ j\in \{1,...,d_{2}\}}},d^{-(l-1)}P_{1}^{-l}P^{(l-1)\theta},d^{-l}P_{1}^{-l}P_{2}^{-1}P^{(l-1)\theta} \right)\\ \ll \left(\frac{dP_{1}}{P^{\theta}}\right)^{r+1+d_{2}(m-r)}M\left(\alpha,(H_{1}^{(i)},H_{2}^{(j,i)},H_{3}^{(j)})_{\substack{i\in \{1,...,d_{1}-1 \}\\ j\in \{1,...,d_{2}\}}}, \right. \\ \left. d^{-(l+1)}P_{1}^{-(l+1)}P^{l\theta} ,d^{-(l+1)}P_{1}^{-(l+1)}P_{2}^{-1}P^{l\theta}  \right)
\end{multline*}
o\`u \[ B_{1}^{(i)}=\left\{\begin{array}{lll} P_{1} & \mbox{si} & i\in \{1,...,d_{1}-l\}
\\ P^{\theta}/d& \mbox{si} &  i\in\{d_{1}-l+1,...,d_{1}-1\} 
\end{array}\right., \]\[ B_{2}^{(j,i)}=\left\{\begin{array}{lll} dP_{1}P_{2} & \mbox{si} & i\in \{1,...,d_{1}-l\}
\\ P^{\theta}P_{2} & \mbox{si} &  i\in\{d_{1}-l+1,...,d_{1}-1\} 
\end{array}\right., \]\[ B_{3}^{(j)}=P_{2}. \] et \[ H_{1}^{(i)}=\left\{\begin{array}{lll} P_{1} & \mbox{si} & i\in \{1,...,d_{1}-l-1\}
\\ P^{\theta}/d & \mbox{si} &  i\in\{d_{1}-l,...,d_{1}-1\} 
\end{array}\right., \]\[ H_{2}^{(j,i)}=\left\{\begin{array}{lll} dP_{1}P_{2} & \mbox{si} & i\in \{1,...,d_{1}-l-1\}
\\ P^{\theta}P_{2} & \mbox{si} & i\in\{d_{1}-l,...,d_{1}-1\} 
\end{array}\right., \]\[ H_{3}^{(j)}=P_{2}. \]

On obtient donc finalement, au rang $ l=d_{1}-1 $ :  
\begin{multline}\label{majo1}
M\left(\alpha,(P_{1},dP_{1}P_{2},P_{2})_{\substack{i\in \{1,...,d_{1}-1 \}\\ j\in \{1,...,d_{2}\}}}, P_{1}^{-1},(dP_{1}P_{2})^{-1} \right)\\ \ll \left(\frac{dP_{1}}{P^{\theta}}\right)^{(r+1+d_{2}(m-r))(d_{1}-1)}M\left(\alpha,(P^{\theta}/d,P^{\theta}P_{2},P_{2})_{\substack{i\in \{1,...,d_{1}-1 \}\\ j\in \{1,...,d_{2}\}}}, \right. \\ \left. d^{-(d_{1}-1)}P_{1}^{-d_{1}}P^{(d_{1}-1)\theta}, d^{-d_{1}}P_{1}^{-d_{1}}P_{2}^{-1}P^{(d_{1}-1)\theta} \right).
\end{multline}
Nous allons \`a pr\'esent chercher \`a \'etablir des majorations analogues avec les $ n_{2}=(m-r)(d_{1}-1)+(n-m+1) $-uplets de variables donn\'es par les $ (\yy^{(j,i)},\zz^{(j)})_{i\in \{1,...,d_{1}-1\}} $ pour $ j\in \{1,...,d_{2}\} $, en consid\'erant toujours les formes lin\'eaires $ \alpha  \gamma_{d,k}^{(1)} $. Fixons donc $ (\xx^{(i)},\yy^{(j,i)},\zz^{(j)})_{\substack{i\in \{1,...,d_{1}-1\} \\ j\in \{1,...,d_{2}-1 \}}} $ v\'erifiant les $ (m-r) $ in\'egalit\'es donn\'ees par \begin{equation}\label{eqisole} \left|\left|\alpha\gamma_{d,(l,d_{2})}^{(1)}\left((\xx^{(i)},\yy^{(j,i)},\zz^{(j)})_{\substack{i\in \{1,...,d_{1}-1\} \\ j\in
 \{1,...,d_{2}-1\}}}\right)\right|\right|<d^{-d_{1}}P_{1}^{-d_{1}}P_{2}^{-1}P^{(d_{1}-1)\theta} \end{equation} pour $ l\in\{r+1,...,m\} $ (les formes $ \gamma^{(1)}_{d,(l,d_{2})} $ ne d\'ependant pas des $ \yy^{(d_{2},i)},\zz^{(d_{2})} $). On consid\`ere les variables $ (\yy^{(d_{2},i)},\zz^{(d_{2})})_{i\in \{1,...,d_{1}-1\}} $ et les $ n_{1}=(r+1)+(d_{2}-1)(m-r) $ formes lin\'eaires $ \alpha  \gamma_{d,k}^{(1)} $, $ k\neq (l,d_{2}) $ correspondantes. On applique le lemme \ref{geomnomb1} en choisissant $  Z_{2}=d^{-\frac{d_{1}}{2}}P_{1}^{-\frac{d_{1}}{2}}P_{2}^{\frac{1}{2}}P^{\frac{(d_{1}-1)\theta}{2}} $, $ Z_{1}=d^{-\frac{d_{1}}{2}}P_{1}^{-\frac{d_{1}}{2}}P_{2}^{-\frac{1}{2}}P^{\frac{(d_{1}+1)\theta}{2}} $, $ a_{k}=d^{\frac{d_{1}}{2}}P_{1}^{\frac{d_{1}}{2}}P_{2}^{\frac{1}{2}}P^{\frac{-(d_{1}-1)\theta}{2}} $ pour tout $ k\in J_{1}=\{m+1,...,n+1\} $, $ a_{k}=d^{\frac{d_{1}}{2}}P_{1}^{-\frac{d_{1}}{2}}P_{2}^{\frac{1}{2}}P^{\frac{-(d_{1}-3)\theta}{2}} $ pour $ k\in J_{2}=\{r+1,...m\}\times\{1,...,d_{1}-1\} $, $ b_{k}=d^{\frac{d_{1}}{2}-1}P_{1}^{\frac{d_{1}}{2}}P_{2}^{\frac{1}{2}}P^{-\frac{(d_{1}-1)\theta}{2}} $ pour $ k\in I_{1}=\{0,...,r\} $ et $ b_{k}=d^{\frac{d_{1}}{2}}P_{1}^{\frac{d_{1}}{2}}P_{2}^{\frac{3}{2}}P^{-\frac{(d_{1}-1)\theta}{2}} $ pour $ k\in I_{2}'=\{r+1,...m\}\times\{1,...,d_{2}-1\} $ de sorte que \[ \begin{array}{lrclcrcl}\forall k \in J_{1}, & a_{k}Z_{2} & = & P_{2}, & \;  & a_{k}Z_{1}& = &P^{\theta} \\ \forall k\in J_{2}, &  a_{k}Z_{2} & = & P^{\theta}P_{2}, & \;  & a_{k}Z_{1}& = & P^{2\theta} \\  \forall k \in I_{1}, &  b_{k}^{-1}Z_{2} & = & d^{-(d_{1}-1)}P_{1}^{-d_{1}}P^{(d_{1}-1)\theta}, & \;  & b_{k}^{-1}Z_{1} & = & d^{-(d_{1}-1)}P_{1}^{-d_{1}}P_{2}^{-1}P^{d_{1}\theta} \\ \forall k\in I_{2}', & b_{k}^{-1}Z_{2} & = & d^{-d_{1}}P_{1}^{-d_{1}}P_{2}^{-1}P^{(d_{1}-1)\theta}, & \;  & b_{k}^{-1}Z_{1} & = & d^{-d_{1}}P_{1}^{-d_{1}}P_{2}^{-2}P^{d_{1}\theta}
\end{array}
\]
et de plus \[ \begin{array}{l} \forall k \in {J_{1}}, \; \;  a_{k}^{-1}Z_{1}=d^{-d_{1}}P_{1}^{-d_{1}}P_{2}^{-1}P^{d_{1}\theta} \\ \forall k \in {J_{2}}, \; \;  a_{k}^{-1}Z_{1}=d^{-d_{1}}P_{1}^{-d_{1}}P_{2}^{-1}P^{(d_{1}-1)\theta} \\ \forall k\in I_{1}, \; \;  b_{k}Z_{1}=P^{\theta}/d \\ \forall k\in I_{2}', \; \;  b_{k}Z_{1}=P_{2}P^{\theta}. \end{array} \]

 On trouve alors \[ U(Z_{2})\ll \max\left(\left(\frac{P_{2}}{P^{\theta}}\right)^{n_{2}}U(Z_{1}),\frac{Z_{2}^{n_{2}}}{Z_{1}^{n_{1}}}\frac{\prod_{k\in J}a_{k}}{\prod_{k\in I_{1}\cup I_{2}'}b_{k}}U^{t}(Z_{1})\right), \] avec 
 \begin{align*}
 \frac{Z_{2}^{n_{2}}}{Z_{1}^{n_{1}}}\frac{\prod_{k\in J}a_{k}}{\prod_{k\in I_{1}\cup I_{2}'}b_{k}} &=\frac{\prod_{k\in J}a_{k}Z_{2}}{\prod_{k\in I_{1}\cup I_{2}'}b_{k}Z_{1}} \\ & =d^{r+1}\frac{P_{2}^{n_{2}}}{P^{n_{1}\theta}}\left(\frac{P^{(d_{1}-1)(m-r)\theta}}{P_{2}^{(d_{2}-1)(m-r)}}\right), \end{align*}
 \begin{multline*} U(Z_{2})=\card\left\{ ((\yy^{(d_{2},i)},\zz^{(d_{2})})_{i\in \{1,...,d_{1}-1\}})\; | \; |\zz^{(d_{2})}| \leqslant P_{2}, \; |\yy^{(d_{2},i)}|\leqslant P^{\theta}P_{2}, \right.\\ \et  \forall k \in I_{1}  , \; \; \left|\left|\alpha\gamma_{d,k}^{(1)}\left((\xx^{(i)},\yy^{(j,i)},\zz^{(j)})_{\substack{i\in \{1,...,d_{1}-1\} \\ j\in
 \{1,...,d_{2}\}}}\right)\right|\right|<d^{-(d_{1}-1)}P_{1}^{-d_{1}}P^{(d_{1}-1)\theta}, \\ \left. \forall k \in I_{2}' , \; \; \left|\left|\alpha\gamma_{d,k}^{(1)}\left((\xx^{(i)},\yy^{(j,i)},\zz^{(j)})_{\substack{i\in \{1,...,d_{1}-1\} \\ j\in
 \{1,...,d_{2}\}}}\right)\right|\right|<d^{-d_{1}}P_{1}^{-d_{1}}P_{2}^{-1}P^{(d_{1}-1)\theta}\right\}, \end{multline*} et \begin{multline*} U(Z_{1})=\card\left\{ ((\yy^{(d_{2},i)})_{i\in \{1,...,d_{1}-1\}},\zz^{(d_{2})})\; | \; |\zz^{(d_{2})}| \leqslant P^{\theta}, \; |\yy^{(d_{2},i)}|\leqslant P^{2\theta}, \right.\\ \et  \forall k \in I_{1}  , \; \; \left|\left|\alpha\gamma_{d,k}^{(1)}\left((\xx^{(i)},\yy^{(j,i)},\zz^{(j)})_{\substack{i\in \{1,...,d_{1}-1\} \\ j\in
 \{1,...,d_{2}\}}}\right)\right|\right|<d^{-(d_{1}-1)}P_{1}^{-d_{1}}P_{2}^{-1}P^{d_{1}\theta}, \\ \left. \forall k \in I_{2}' , \; \; \left|\left|\alpha\gamma_{d,k}^{(1)}\left((\xx^{(i)},\yy^{(j,i)},\zz^{(j)})_{\substack{i\in \{1,...,d_{1}-1\} \\ j\in
 \{1,...,d_{2}\}}}\right)\right|\right|<d^{-d_{1}}P_{1}^{-d_{1}}P_{2}^{-2}P^{d_{1}\theta}\right\}, \end{multline*}
\begin{multline*} U^{t}(Z_{1})=\card\left\{ (\xx^{(d_{1})},(\yy^{(j,d_{1})})_{j\in \{1,...,d_{2}-1\}})\; | \; |\xx^{(d_{1})}| \leqslant P^{\theta}/d, \; |\yy^{(j,d_{1})}|\leqslant P^{\theta}P_{2}, \right.\\ \et  \forall k \in J_{1}  , \; \; \left|\left|\alpha(\gamma_{d,k}^{(1)})^{t}\left((\xx^{(i)},\yy^{(j,i)},\zz^{(j)})_{\substack{i\in \{1,...,d_{1}\} \\ j\in
 \{1,...,d_{2}-1\}}}\right)\right|\right|<d^{-d_{1}}P_{1}^{-d_{1}}P_{2}^{-1}P^{d_{1}\theta}, \\ \left. \forall k \in J_{2} , \; \; \left|\left|\alpha(\gamma_{d,k}^{(1)})^{t}\left((\xx^{(i)},\yy^{(j,i)},\zz^{(j)})_{\substack{i\in \{1,...,d_{1}\} \\ j\in
 \{1,...,d_{2}-1\}}}\right)\right|\right|<d^{-d_{1}}P_{1}^{-d_{1}}P_{2}^{-1}P^{(d_{1}-1)\theta}\right\}. \end{multline*}

Rappelons que l'on avait : \[ \Gamma^{(1)}_{d}\left((\xx^{(i)},\yy^{(j,i)},\zz^{(j)})_{\substack{i\in \{1,...,d_{1}\} \\ j\in \{1,...,d_{2}\}}}\right)=\sum_{k\in I}\gamma_{d,k}^{(1)}\left((\xx^{(i)},\yy^{(j,i)},\zz^{(j)})_{\substack{i\in \{1,...,d_{1}-1\} \\ j\in
 \{1,...,d_{2}\}}}\right)u_{k}^{(d_{1})}.  \] Or d'apr\`es la remarque \ref{remarquesym}, on a $ \Gamma^{(1)}_{d}=\Gamma^{(1)'}_{d} $. On pose alors \[ \Gamma^{(1)}_{d}=\sum_{\substack{k\in I_{1}\cup I_{2}' \\ l \in J_{1}\cup J_{2}}}\lambda_{k,l}t_{l}^{(d_{2})}u_{k}^{(d_{1})}+ \sum_{j=r+1}^{m}\alpha_{j}y_{j}^{(d_{1},d_{2})}, \]
On a alors, \[ \gamma_{d,k}^{(1)}\left((\xx^{(i)},\yy^{(j,i)},\zz^{(j)})_{\substack{i\in \{1,...,d_{1}-1\} \\ j\in
 \{1,...,d_{2}\}}}\right)=\sum_{\substack{ l \in J_{1}\cup J_{2}}}\lambda_{k,l}t_{l}^{(d_{2})}\] et  \[ (\gamma_{d,l}^{(1)})^{t}\left((\xx^{(i)},\yy^{(j,i)},\zz^{(j)})_{\substack{k\in \{1,...,d_{1}\} \\ j\in
 \{1,...,d_{2}-1\}}}\right)=\sum_{k\in I_{1}\cup I_{2}'}\alpha_{k,l}u_{k}^{(d_{1})} \] 
Par cons\'equent les formes lin\'eaires $ (\gamma_{d,k}^{(1)})^{t} $ sont exactement celles que l'on aurait obtenu en diff\'erenciant en $ (\xx,\yy) $ puis en $ (\tilde{\yy},\zz) $ et en sommant ensuite sur chaque $ \zz^{d_{2}}, \yy^{d_{2},d_{1}}) $ . En particulier si l'on consid\`ere les formes $ (\gamma_{d,k}^{(1)})^{t}\left((\xx^{(i)},\yy^{(j,i)},\zz^{(j)})_{\substack{i\in \{1,...,d_{1}\} \\ j\in \{1,...,d_{2}-1\}}}\right) $ comme des formes en $ (\yy^{j,i},\zz^{(j)})_{i\in \{1,...,d_{1}-1\}} $ pour un certain $ j\in \{1,...,d_{2}\} $ alors ces formes lin\'eaires v\'erifient la condition de sym\'etrie du lemme \ref{geomnomb2}, et on peut alors appliquer ce lemme comme nous l'avions fait pour les formes en $ (\yy^{j,i},\xx^{(i)})_{j\in \{1,...,d_{2}-1\}} $, pour finalement obtenir, en posant
\begin{multline*}
M^{t}\left(\alpha,(H_{1}^{(i)},H_{2}^{(j,i)},H_{3}^{(j)})_{\substack{i\in \{1,...,d_{1} \}\\ j\in \{1,...,d_{2}-1\}}},  B_{1}^{-1},B_{2}^{-1} \right) \\ =\Card\left\{(\xx^{(i)},\yy^{(j,i)},\zz^{(j)})_{\substack{i\in \{1,...,d_{1}\} \\ j\in \{1,...,d_{2}-1\}}} \; | \; \forall (i,j)\in \{1,...,d_{1}\}\times \{1,...,d_{2}-1\}, \; \right. \\  |\xx^{(i)}|\leqslant H_{1}^{(i)}, \;  |\yy^{(i,j)}|\leqslant H_{2}^{(i,j)}, \;  |\zz^{(j)}|\leqslant H_{2}^{(j)}\; \\   \et \; \forall k \in \{r+1,...,m\}  , \; \; \left|\left|\alpha\gamma_{d,k}^{(1)}\left((\xx^{(i)},\yy^{(j,i)},\zz^{(j)})_{\substack{i\in \{1,...,d_{1}\} \\ j\in
 \{1,...,d_{2}-1\}}}\right)\right|\right|<B_{1}^{-1}, \; \\  \left.\forall k \in \{m+1,...,n+1\}\times\{1,...,d_{1}\} , \; \; \left|\left|\alpha(\gamma_{d,k}^{(1)})^{t}\left((\xx^{(i)},\yy^{(j,i)},\zz^{(j)})_{\substack{i\in \{1,...,d_{1}\} \\ j\in
 \{1,...,d_{2}-1\}}}\right)\right|\right|<B_{2}^{-1}\right\},
\end{multline*}
et en choisissant
 \[ H_{2}^{(j,i)}= P^{\theta}P_{2}   \]\[ H_{1}^{(i)}= P^{\theta}/d, \]\[ H_{3}^{(j)}=P_{2} : \]  \begin{multline*}
 \sum_{\substack{(\xx^{(i)},\yy^{(j,i)},\zz^{(j)})_{\substack{i\in \{1,...,d_{1}-1\} \\ j\in
 \{1,...,d_{2}-1\}}} \; \\ \verifiant \; \eqref{eqisole}}}\frac{Z_{2}^{n_{2}}}{Z_{1}^{n_{1}}}\frac{\prod_{k\in J}a_{k}}{\prod_{k\in I_{1}\cup I_{2}'}b_{k}}U^{t}(Z_{1})  \\  \ll d^{r+1}\frac{P_{2}^{n_{2}}}{P^{n_{1}\theta}}\left(\frac{P^{(d_{1}-1)(m-r)\theta}}{P_{2}^{(d_{2}-1)(m-r)}}\right) \\ M^{t}\left(\alpha,(H_{1}^{(i)},H_{2}^{(j,i)},H_{3}^{(j)})_{\substack{i\in \{1,...,d_{1} \}\\ j\in \{1,...,d_{2}-1\}}},  d^{-d_{1}}P_{1}^{-d_{1}}P_{2}^{-1}P^{d_{1}\theta},d^{-d_{1}}P_{1}^{-d_{1}}P_{2}^{-1}P^{(d_{1}-1)\theta} \right) \\ \ll d^{r+1}\frac{P_{2}^{n_{2}}}{P^{n_{1}\theta}}\left(\frac{P^{(d_{1}-1)(m-r)\theta}}{P_{2}^{(d_{2}-1)(m-r)}}\right)\left(\frac{P_{2}}{P^{\theta}}\right)^{(n-m+1)(d_{2}-1)+(m-r)(d_{2}-1)d_{1}} \\ M^{t}\left(\alpha,(P^{\theta}/d,P^{2\theta},P^{\theta})_{\substack{i\in \{1,...,d_{1} \}\\ j\in \{1,...,d_{2}-1\}}},  d^{-d_{1}}P_{1}^{-d_{1}}P_{2}^{-d_{2}}P^{(\tilde{d}+1)\theta},d^{-d_{1}}P_{1}^{-d_{1}}P_{2}^{-d_{2}}P^{\tilde{d}\theta} \right)\\ = d^{r+1}\frac{P_{2}^{(n-m+1)d_{2}+(m-r)(d_{1}-1)d_{2}}}{P^{(n_{2}(d_{2}-1)+n_{1}+(d_{2}-d_{1})(m-r))\theta}}\\ M^{t}\left(\alpha,(P^{\theta}/d,P^{2\theta},P^{\theta})_{\substack{i\in \{1,...,d_{1} \}\\ j\in \{1,...,d_{2}-1\}}},  d^{-d_{1}}P_{1}^{-d_{1}}P_{2}^{-d_{2}}P^{(\tilde{d}+1)\theta},d^{-d_{1}}P_{1}^{-d_{1}}P_{2}^{-d_{2}}P^{\tilde{d}\theta} \right)
\end{multline*}

En proc\'edant de la m\^eme mani\`ere pour tous les $ n_{2} $-uplets de variables $ (\yy^{(j,i)},\zz^{(j)})_{i\in \{1,...,d_{1}-1\}} $ pour $ j\in \{1,...,d_{2}\} $, on obtient finalement \begin{multline}\label{majo2} M\left(\alpha,(P^{\theta}/d,P^{\theta}P_{2},P_{2})_{\substack{i\in \{1,...,d_{1}-1 \}\\ j\in \{1,...,d_{2}\}}},d^{-(d_{1}-1)}P_{1}^{-d_{1}}P^{(d_{1}-1)\theta},d^{-d_{1}}P_{1}^{-d_{1}}P_{2}^{-1}P^{(d_{1}-1)\theta} \right)\\ \ll \frac{P_{2}^{d_{2}n_{2}}}{P^{\theta(d_{2}-1)n_{2}}}\max\left\{P^{-n_{2}\theta}M\left(\alpha,(P^{\theta}/d,P^{2\theta},P^{\theta})_{\substack{i\in \{1,...,d_{1}-1 \}\\ j\in \{1,...,d_{2}\}}},\right. H_{2}, \; H_{1} \right), \\ \left. d^{r+1}P^{-(n_{1}+(m-r)(d_{2}-d_{1}))\theta}M^{t}\left(\alpha,(P^{\theta}/d,P^{2\theta},P^{\theta})_{\substack{i\in \{1,...,d_{1} \}\\ j\in \{1,...,d_{2}-1\}}},H_{1}, \; H_{1}  \right)\right\} . \end{multline} o\`u l'on a not\'e : \[ H_{1}=d^{-d_{1}}P_{1}^{-d_{1}}P_{2}^{-d_{2}}P^{(\tilde{d}+1)\theta}, \] \[ H_{2}=d^{-(d_{1}-1)}P_{1}^{-d_{1}}P_{2}^{-d_{2}}P^{(\tilde{d}+1)\theta}. \]
 
En regroupant le lemme \ref{dilemme1}, et les majorations \eqref{majo1} et \eqref{majo2}, on obtient le lemme ci-dessous : 
\begin{lemma}\label{dilemme2}
 Pour $ \varepsilon>0 $ arbitrairement petit, et pour $ \kappa>0 $, $ P>0 $ des r\'eels fix\'es, pour tout $ \alpha \in [0,1] $, l'une au moins des assertions suivantes est vraie  \begin{enumerate}
\item  $ |S_{d}(\alpha)|\ll_{n,r,m,\varepsilon} d^{m-r+\varepsilon+\frac{d_{1}(r+1)}{2^{\tilde{d}}}}P_{1}^{m+2+\varepsilon}P_{2}^{n-r+1+\varepsilon}P^{-\kappa}, $
\item \begin{multline*}  M\left(\alpha,(P^{\theta}/d,P^{2\theta},P^{\theta})_{\substack{i\in \{1,...,d_{1}-1 \}\\ j\in \{1,...,d_{2}\}}},H_{2}, \; H_{1}\right) \\ \gg  (P^{\theta})^{(d_{1}-1)(r+1)+2(d_{1}-1)d_{2}(m-r)+d_{2}(n-m+1))}P^{-2^{\tilde{d}}\kappa}\end{multline*}
\item  \begin{multline*}M^{t}\left(\alpha,(P^{\theta}/d,P^{2\theta},P^{\theta})_{\substack{i\in \{1,...,d_{1} \}\\ j\in \{1,...,d_{2}-1\}}}, H_{1}, \; H_{1}\right) \\  \gg  (P^{\theta})^{(d_{1}(r+1)+2d_{1}(d_{2}-1)(m-r)+(d_{2}-1)(n-m+1))}P^{-2^{\tilde{d}}\kappa}.\end{multline*}
\end{enumerate} 
 \end{lemma}
 Consid\'erons \`a pr\'esent un \'el\'ement $ (\xx^{(i)},\yy^{(j,i)},\zz^{(j)})_{\substack{i\in \{1,...,d_{1}-1\} \\ j\in
 \{1,...,d_{2}\}}} $ tel que $ |\xx^{(i)}|\leqslant P^{\theta}/d $, $ |\yy^{(j,i)}|\leqslant P^{2\theta} $, $ |\zz^{(j)}|\leqslant P^{\theta} $, \[\left|\left|\alpha\gamma_{d,k}^{(1)}\left((\xx^{(i)},\yy^{(j,i)},\zz^{(j)})_{\substack{i\in \{1,...,d_{1}-1\} \\ j\in
 \{1,...,d_{2}\}}}\right)\right|\right|<d^{-(d_{1}-1)}P_{1}^{-d_{1}}P_{2}^{-d_{2}}P^{(\tilde{d}+1)\theta} \] pour tout $ k\in I_{1} $ et  \[\left|\left|\alpha\gamma_{d,k}^{(1)}\left((\xx^{(i)},\yy^{(j,i)},\zz^{(j)})_{\substack{i\in \{1,...,d_{1}-1\} \\ j\in
 \{1,...,d_{2}\}}}\right)\right|\right|<d^{-d_{1}}P_{1}^{-d_{1}}P_{2}^{-d_{2}}P^{(\tilde{d}+1)\theta} \] pour tout $ k\in I_{2} $ et supposons qu'il existe $ k_{0}\in I $ tel que \[ \alpha\gamma_{d,k_{0}}^{(1)}\left(\xx^{(i)},\yy^{(j,i)},\zz^{(j)})_{\substack{i\in \{1,...,d_{1}-1\} \\ j\in
 \{1,...,d_{2}\}}}\right)\neq 0. \] On pose alors $ q=\gamma_{k_{0}}^{(1)}\left((\xx^{(i)},\yy^{(j,i)},\zz^{(j)})_{\substack{i\in \{1,...,d_{1}-1\} \\ j\in
 \{1,...,d_{2}\}}}\right) $. Rappelons que d'apr\`es \eqref{gammai1}, on a la relation : \begin{multline*}\gamma_{d,k_{0}}^{(1)}\left((\xx^{(k)},\yy^{(j,k)},\zz^{(j)})_{\substack{k\in \{1,...,d_{1}-1\} \\ j\in
 \{1,...,d_{2}\}}}\right)=\sum_{\substack{\ii\in I^{d_{1}-1}}}G_{d,\ii,k_{0}}(\tilde{\zz})u_{i_{1}}^{(1)}...u_{i_{d_{1}-1}}^{(d_{1}-1)} \\ =\sum_{\substack{\ii\in I^{d_{1}-1} }}d^{f_{\ii,k_{0}}}G_{\ii,k_{0}}(\tilde{\zz})u_{i_{1}}^{(1)}...u_{i_{d_{1}-1}}^{(d_{1}-1)} \end{multline*} Par cons\'equent, si $ k_{0}\in I_{1} $ alors $ d $ divise $ q $ et on a $ q\ll dP^{(\tilde{q}+1)\theta} $ (car $ |\xx^{(i)}|\leqslant P^{\theta}/d $, $ |\yy^{(j,i)}|\leqslant P^{2\theta} $, $ |\zz^{(j)}|\leqslant P^{\theta} $) et si $ a $ est l'entier le plus proche de $ \alpha q $, on a donc \[ |\alpha q-a|\leqslant d^{-(d_{1}-1)}P_{1}^{-d_{1}}P_{2}^{-d_{2}}P^{(\tilde{d}+1)\theta}.  \] Dans le cas o\`u $ k_{0}\in I_{2} $ on a $ q\ll P^{(\tilde{q}+1)\theta} $ et si $ a $ est l'entier le plus proche de $ \alpha q $, \[ |\alpha q-a|\leqslant d^{-d_{1}}P_{1}^{-d_{1}}P_{2}^{-d_{2}}P^{(\tilde{d}+1)\theta}.  \] En proc\'edant de m\^eme avec les \'el\'ements $ (\xx^{(i)},\yy^{(j,i)},\zz^{(j)})_{\substack{i\in \{1,...,d_{1}\} \\ j\in
 \{1,...,d_{2}-1\}}} $ compt\'es par $ M^{t}\left(\alpha,(P^{\theta},P^{2\theta},P^{\theta})_{\substack{i\in \{1,...,d_{1} \}\\ j\in \{1,...,d_{2}-1\}}},P_{1}^{-d_{1}}P_{2}^{-d_{2}}P^{(\tilde{d}+1)\theta} \right) $, on voit que le lemme \ref{dilemme2} implique 
 \begin{lemma}\label{dilemme3}
 Pour $ \varepsilon>0 $ arbitrairement petit, et pour $ \kappa>0 $, $ P>0 $ des r\'eels fix\'es, pour tout $ \alpha \in [0,1] $, l'une au moins des assertions suivantes est vraie  \begin{enumerate}
\item  $ |S_{d}(\alpha)|\ll_{n,r,m,\varepsilon} d^{m-r+\varepsilon+\frac{d_{1}(r+1)}{2^{\tilde{d}}}}P_{1}^{m+2+\varepsilon}P_{2}^{n-r+1+\varepsilon}P^{-\kappa},  $
\item Il existe $ q $ tel que $ d|q $, $ 0<q\leqslant dP^{(\tilde{d}+1)\theta} $ et $ a $ tels que $ 0\leqslant a <q $ et \[ |\alpha q-a|\leqslant d^{-(d_{1}-1)}P_{1}^{-d_{1}}P_{2}^{-d_{2}}P^{(\tilde{d}+1)\theta},  \]
\item Il existe $ q $ tel que $ 0<q\leqslant P^{(\tilde{d}+1)\theta} $ et $ a $ tels que $ 0\leqslant a <q $, $ \PGCD(a,q)=1 $ et \[ |\alpha q-a|\leqslant d^{-d_{1}}P_{1}^{-d_{1}}P_{2}^{-d_{2}}P^{(\tilde{d}+1)\theta},  \]
\item \begin{multline*} \Card\left\{ (\xx^{(i)},\yy^{(j,i)},\zz^{(j)})_{\substack{i\in \{1,...,d_{1}-1\} \\ j\in
 \{1,...,d_{2}\}}} \; | \;  |\xx^{(i)}|\leqslant P^{\theta}/d , \;  |\yy^{(j,i)}|\leqslant P^{2\theta}, \; \right.\\ \left. |\zz^{(j)}|\leqslant P^{\theta}  \; \et  \;  \forall k\in I, \; \;  \gamma_{d,k}^{(1)}\left((\xx^{(i)},\yy^{(j,i)},\zz^{(j)})_{\substack{i\in \{1,...,d_{1}-1\} \\ j\in
 \{1,...,d_{2}\}}}\right)= 0 \right\}  \\  \gg  (P^{\theta})^{(d_{1}-1)(r+1)+2(d_{1}-1)d_{2}(m-r)+d_{2}(n-m+1)}P^{-2^{\tilde{d}}\kappa},\end{multline*}
\item  \begin{multline*} \Card\left\{ (\xx^{(i)},\yy^{(j,i)},\zz^{(j)})_{\substack{i\in \{1,...,d_{1}\} \\ j\in
 \{1,...,d_{2}-1\}}} \; | \;  |\xx^{(i)}|\leqslant P^{\theta}/d , \;  |\yy^{(j,i)}|\leqslant P^{2\theta}, \;\right.\\ \left.   |\zz^{(j)}|\leqslant P^{\theta}  \; \et \;  \forall k\in J, \; \;  (\gamma_{d,k}^{(1)})^{t}\left((\xx^{(i)},\yy^{(j,i)},\zz^{(j)})_{\substack{i\in \{1,...,d_{1}\} \\ j\in
 \{1,...,d_{2}-1\}}}\right)= 0 \right\}  \\  \gg  (P^{\theta})^{d_{1}(r+1)+2d_{1}(d_{2}-1)(m-r)+(d_{2}-1)(n-m+1)}P^{-2^{\tilde{d}}\kappa}.\end{multline*}
\end{enumerate} 
 \end{lemma}
 Avant d'aller plus loin, nous introduisons le lemme ci-dessous qui sera utile \`a plusieurs reprise par la suite :

 \begin{lemma}\label{lemmedebile}
 On consid\`ere $ p,q,r\in \NN $ et $ (L_{i})_{i\in\{1,...,r\} } $ des formes lin\'eaires \`a $ p+q $ variables. pour des constantes $ A $, $ B $ et $ (C_{i})_{i\in I} $fix\'ees on note \begin{multline*} M\left(A,B,(C_{i})_{i\in \{1,...,r\}}\right)=\card\left\{ (\xx,\yy)\in \ZZ^{p}\times \ZZ^{q} \; | \; |\xx|\leqslant A ,\;  |\yy|\leqslant B, \; \right. \\  \left. \forall i  \in \{1,...,r\}, \;  ||L_{i}(\xx,\yy)||< C_{i} \right\}. \end{multline*}
 
 On a alors pour tout $ \xi\geqslant 1 $ : \[ M\left(A,B,(C_{i})_{i\in \{1,...,r\}}\right)\leqslant (2\xi)^{q} M\left(2A,\frac{2B}{\xi},(2C_{i})_{i\in \{1,...,r\}}\right).\]

 \end{lemma}
 \begin{proof}
 On subdivise le cube $ [-B,B]^{q} $ en $ (2\xi)^{q}  $ cubes de taille $ B/\xi $. Prenons un tel cube $ \mathcal{C} $ et consid\'erons \begin{multline*} E(\mathcal{C})=\card\left\{ (\xx,\yy)\in \ZZ^{p}\times \ZZ^{q} \; | \; |\xx|\leqslant A ,\; \yy \in \mathcal{C} , \; \right. \\  \left. \forall i  \in \{1,...,r\}, \; ||L_{i}(\xx,\yy)||\leqslant C_{i}  \right\}. \end{multline*} Si $ (\xx,\yy), (\xx',\yy') $ sont deux points de $ E(\mathcal{C}) $, on a alors que \[ |\xx-\xx'|\leqslant 2A, \; \; |\yy-\yy'|\leqslant 2B/\xi \] et pour tout $ i \in \{1,...,r\} $ : \[ |L_{i}(\xx-\xx',\yy-\yy') |\leqslant 2C_{i}. \] On a donc : \[ E(\mathcal{C})\leqslant  M\left(2A,\frac{2B}{\xi},(2C_{i})_{i\in \{1,...,r\}}\right) \] pour tout cube $ \mathcal{C} $. D'o\`u le r\'esultat. 
 \end{proof}
Consid\'erons \`a pr\'esent le cas $ 4. $ du lemme \ref{dilemme3}. Remarquons avant tout qu'il est facile de voir, en appliquant $ d_{1}-1 $ fois le lemme \ref{lemmedebile} (avec $ L_{i}=\gamma_{d,i}^{(1)} $, et $ C_{i}=1/2^{d_{1}} $) que le cardinal consid\'er\'e peut \^etre major\'e, \`a une constante multiplicative pr\`es, par \begin{multline}\label{majo3} (P^{\theta})^{(d_{1}-1)d_{2}(m-r)}\Card\left\{ (\xx^{(i)},\yy^{(j,i)},\zz^{(j)})_{\substack{i\in \{1,...,d_{1}-1\} \\ j\in
 \{1,...,d_{2}\}}} \; | \;  |\xx^{(i)}|\leqslant 2P^{\theta}/d , \;  \right.\\ \left. |\yy^{(j,i)}|\leqslant 2P^{\theta}, \; |\zz^{(j)}|\leqslant P^{\theta}  \; \et  \;  \forall k\in I, \; \;  \gamma_{d,k}^{(1)}\left((\xx^{(i)},\yy^{(j,i)},\zz^{(j)})_{\substack{i\in \{1,...,d_{1}-1\} \\ j\in
 \{1,...,d_{2}\}}}\right)= 0 \right\} \end{multline} Quitte \`a agrandir $ \theta $, nous pouvons remplacer la borne $ 2P^{\theta} $ sur $ \xx^{(i)} $ et $ \yy^{(j,i)} $ par $ P^{\theta} $. D'autre part, si l'on pose pour tout $ k\in I $ : \begin{multline*}\gamma_{k}^{(1)}\left((\xx^{(i)},\yy^{(j,i)},\zz^{(j)})_{\substack{i\in \{1,...,d_{1}-1\} \\ j\in
 \{1,...,d_{2}\}}}\right) =\sum_{\substack{\ii\in I^{d_{1}-1} }}G_{\ii,k}(\tilde{\zz})u_{i_{1}}^{(1)}...u_{i_{d_{1}-1}}^{(d_{1}-1)}, \end{multline*} on a alors \begin{multline*}\gamma_{d,k}^{(1)}\left((\xx^{(i)},\yy^{(j,i)},\zz^{(j)})_{\substack{i\in \{1,...,d_{1}-1\} \\ j\in
 \{1,...,d_{2}\}}}\right) =d\gamma_{k}^{(1)}\left((d\xx^{(i)},\yy^{(j,i)},\zz^{(j)})_{\substack{i\in \{1,...,d_{1}-1\} \\ j\in
 \{1,...,d_{2}\}}}\right) , \end{multline*} pour tout $ k\in I_{1} $, et \begin{multline*}\gamma_{d,k}^{(1)}\left((\xx^{(i)},\yy^{(j,i)},\zz^{(j)})_{\substack{i\in \{1,...,d_{1}-1\} \\ j\in
 \{1,...,d_{2}\}}}\right) =\gamma_{k}^{(1)}\left((d\xx^{(i)},\yy^{(j,i)},\zz^{(j)})_{\substack{i\in \{1,...,d_{1}-1\} \\ j\in
 \{1,...,d_{2}\}}}\right) , \end{multline*} pour $ k\in I_{2} $. Par cons\'equent, on a la majoration : \begin{multline} \Card\left\{ (\xx^{(i)},\yy^{(j,i)},\zz^{(j)})_{\substack{i\in \{1,...,d_{1}-1\} \\ j\in
 \{1,...,d_{2}\}}} \; | \;  |\xx^{(i)}|\leqslant P^{\theta}/d , \;  |\yy^{(j,i)}|\leqslant P^{\theta}, \;\right.\\ \left.  |\zz^{(j)}|\leqslant P^{\theta}  \; \et  \;  \forall k\in I, \; \;  \gamma_{d,k}^{(1)}\left((\xx^{(i)},\yy^{(j,i)},\zz^{(j)})_{\substack{i\in \{1,...,d_{1}-1\} \\ j\in
 \{1,...,d_{2}\}}}\right)= 0 \right\} \\ \ll \Card\left\{ (\xx^{(i)},\yy^{(j,i)},\zz^{(j)})_{\substack{i\in \{1,...,d_{1}-1\} \\ j\in
 \{1,...,d_{2}\}}} \; | \;  |\xx^{(i)}|\leqslant P^{\theta} , \;  |\yy^{(j,i)}|\leqslant P^{\theta}, \; \right.\\ \left. |\zz^{(j)}|\leqslant P^{\theta}  \; \et  \;  \forall k\in I, \; \;  \gamma_{k}^{(1)}\left((\xx^{(i)},\yy^{(j,i)},\zz^{(j)})_{\substack{i\in \{1,...,d_{1}-1\} \\ j\in
 \{1,...,d_{2}\}}}\right)= 0 \right\}.  \end{multline}

 On consid\`ere la vari\'et\'e affine $ \mathcal{L}_{1} $ d\'efinie par l'ensemble des \'el\'ements $ (\xx^{(i)},\yy^{(j,i)},\zz^{(j)})_{\substack{i\in \{1,...,d_{1}-1\} \\ j\in
 \{1,...,d_{2}\}}} $ de l'espace affine de dimension $ (d_{1}-1)(r+1)+(d_{1}-1)d_{2}(m-r)+d_{2}(n-m+1) $ v\'erifiant les \'equations $  \gamma_{k}^{(1)}=0 $ pour tout $ k\in I $. En posant $ \kappa=K\theta $, d'apr\`es \eqref{majo3}, la condition $ 4. $ du lemme \ref{dilemme3} implique (par la d\'emonstration de \cite[Th\'eor\`eme 3.1]{Br}) : \[ \dim(  \mathcal{L}_{1} )\geqslant (d_{1}-1)(r+1)+(d_{1}-1)d_{2}(m-r)+d_{2}(n-m+1) -2^{\tilde{d}}K. \] On consid\`ere par ailleurs la sous-vari\'et\'e affine $ V_{1}^{\ast} $ de $ \AA_{\CC}^{n+2} $ d\'efinie par les $ (\xx,\yy,\zz)\in \AA_{\CC}^{n+2} $ tels que \[ \forall i \in \{ 0,...,r\}, \; \; \frac{\partial F}{\partial x_{i}}=0, \] \[ \forall j \in \{ r+1,...,m\}, \; \; \frac{\partial F}{\partial y_{j}}=0. \] Notons par ailleurs $ \mathcal{D} $ le sous-espace de l'espace affine de dimension $ (d_{1}-1)(r+1)+(d_{1}-1)d_{2}(m-r)+d_{2}(n-m+1) $ d\'efini par les $ (r+1)(d_{1}-2)+(m-r)((d_{1}-1)d_{2}-1)+(d_{2}-1)(n-m+1) $ \'equations : \[ \xx^{(1)}=\xx^{(2)}=...=\xx^{(d_{1}-1)} \] \[ \forall (i,j)\in \{1,...,d_{1}-1\} \times \{1,...,d_{2}\}, \; \; \yy^{(i,j)}=\yy^{(1,1)}, \]\[ \zz^{(1)}=\zz^{(2)}=...=\zz^{(d_{2})}. \] On a alors  \begin{multline*} \dim(  \mathcal{L}_{1}\cap \mathcal{D} ) \geqslant \dim(\mathcal{L}_{1})-((r+1)(d_{1}-2)+(m-r)((d_{1}-1)d_{2}-1) \\ +(d_{2}-1)(n-m+1))  \geqslant  n+2 -2^{\tilde{d}}K. \end{multline*} D'autre part, $ \mathcal{L}_{1}\cap \mathcal{D} $ est isomorphe \`a $ V_{1}^{\ast} $. Donc, en r\'esum\'e la condition $ 4. $ implique \[  \dim(V_{1}^{\ast})\geqslant n+2 -2^{\tilde{d}}K. \] De la m\^eme mani\`ere, en notant $ V_{2}^{\ast} $ la sous-vari\'et\'e de $ \AA_{\CC}^{n+2} $ d\'efinie par \[ \forall i \in \{ m+1,...,n+1\}, \; \; \frac{\partial F}{\partial z_{i}}=0, \] \[ \forall j \in \{ r+1,...,m\}, \; \; \frac{\partial F}{\partial y_{j}}=0, \] on v\'erifie que la condition $ 5. $ implique \[  \dim(V_{2}^{\ast})\geqslant n+2 -2^{\tilde{d}}K. \] Par cons\'equent, on choisira \begin{equation}
 K=(n+2-\max\{\dim(V_{1}^{\ast}), \dim(V_{2}^{\ast})\}-\varepsilon)/2^{\tilde{d}}
\end{equation}    (pour un $ \varepsilon>0 $ arbitrairement petit) de sorte que les assertions $ 4. $ et $ 5. $ ne soient plus possibles. On posera par ailleurs \begin{equation} P=P_{1}^{d_{1}}P_{2}^{d_{2}}. \end{equation} Rappelons que l'on consid\`ere des r\'eels $ \theta $ tels que $ P^{\theta}\leqslant P_{2}\leqslant P_{1} $, et donc, si $ P_{1}=P_{2}^{b} $, alors $ \theta\leqslant \frac{b}{bd_{1}+d_{2}} $. D'autre part, pour un tel $ \theta $, pour $ a,q $ tels que $ 0<q\leqslant dP^{(\tilde{d}+1)\theta} $, $ d|q $ et $ 0\leqslant a <q $, on d\'efinit les arcs majeurs \begin{equation}
\mathfrak{M}^{(1)}_{a,q}(\theta)=\left\{\alpha\in [0,1] \; | \; |\alpha q-a|\leqslant d^{-(d_{1}-1)}P^{-1+(\tilde{d}+1)\theta} \right\}, 
\end{equation} et \begin{equation}
\mathfrak{M}^{(1)}(\theta)=\bigcup_{\substack{1\leqslant q\leqslant dP^{(\tilde{d}+1)\theta}\\ d|q}}\bigcup_{\substack{0\leqslant a <q }}\mathfrak{M}^{(1)}_{a,q}(\theta).
\end{equation}
De m\^eme pour $ a,q $ tels que $ 0<q\leqslant P^{(\tilde{d}+1)\theta} $, et $ 0\leqslant a <q $, on d\'efinit \begin{equation}
\mathfrak{M}^{(2)}_{a,q}(\theta)=\left\{\alpha\in [0,1] \; | \; |\alpha q-a|\leqslant d^{-(d_{1}-1)}P^{-1+(\tilde{d}+1)\theta} \right\}, 
\end{equation} et \begin{equation}
\mathfrak{M}^{(2)}(\theta)=\bigcup_{\substack{1\leqslant q\leqslant P^{(\tilde{d}+1)\theta}}}\bigcup_{\substack{0\leqslant a <q \\\PGCD(a,q)=1}}\mathfrak{M}_{a,q}^{(2)}(\theta).
\end{equation}

 On notera par ailleurs $ \mathfrak{m}(\theta)=[0,1[\setminus\left(\mathfrak{M}^{(1)}(\theta)\cup \mathfrak{M}^{(2)}(\theta)\right) $ l'ensemble des arcs mineurs. Avec ces notations, le lemme \ref{dilemme3} devient alors \begin{lemma}\label{dilemme4}
 Pour $ \varepsilon>0 $ arbitrairement petit, pour tout $ \alpha \in [0,1] $, l'une au moins des assertions suivantes est vraie  \begin{enumerate}
\item  $ |S_{d}(\alpha)|\ll_{n,m,r,\varepsilon} d^{m-r+\varepsilon+d_{1}(r+1)/2^{\tilde{d}}}P_{1}^{m+2}P_{2}^{n-r+1}P^{-K\theta+\varepsilon},  $
\item Le r\'eel $ \alpha $ appartient \`a $ \mathfrak{M}(\theta)=\mathfrak{M}^{(1)}(\theta)\cup \mathfrak{M}^{(2)}(\theta) $. 
\end{enumerate}
 \end{lemma} 
 \begin{rem}
 Dans le cas particulier o\`u $ d=1 $ on peut consid\'erer les arcs majeurs \[ \mathfrak{M}_{a,q}(\theta)=\left\{\alpha\in [0,1] \; | \; |\alpha q-a|\leqslant P^{-1+(\tilde{d}+1)\theta} \right\}  \] et le lemme \ref{dilemme4} peut \^etre exprim\'e sous la forme suivante : 
 \begin{lemma}\label{dilemmed=1}
 Pour $ \varepsilon>0 $ arbitrairement petit, l'une au moins des assertions suivantes est vraie  \begin{enumerate}
\item  $ |S_{1}(\alpha)|\ll P_{1}^{m+2}P_{2}^{n-r+1}P^{-K\theta+\varepsilon},  $
\item Il existe $ a,q $ tels que $ 1\leqslant q \leqslant P^{-1+(\tilde{d}+1)\theta} $, $ \PGCD(a,q)=1 $, $ 0\leqslant a <q $ et le r\'eel $ \alpha $ appartient \`a $ \mathfrak{M}_{a,q}(\theta) $. 
\end{enumerate}
 \end{lemma} 
 \end{rem}

\subsection{Les arcs mineurs}

On consid\`ere \`a pr\'esent $ \delta>0 $ arbitrairement petit, $ \theta_{0}\leqslant \frac{b}{bd_{1}+d_{2}} $ tels que \begin{equation}\label{cond1}
K-2(\tilde{d}+1)>\left(2\delta +\frac{b}{bd_{1}+d_{2}}\right)\theta_{0}^{-1},
\end{equation}
\begin{equation}\label{cond3}
1>(bd_{1}+d_{2})(5(\tilde{d}+1)\theta_{0}+\delta).
\end{equation}
\begin{rem}\label{remarqueb} Pour que les conditions \eqref{cond1} et \eqref{cond3} puissent \^etre v\'erifi\'ees, il est n\'ecessaire d'avoir \[ 
K-2(\tilde{d}+1)>\frac{b}{bd_{1}+d_{2}}(bd_{1}+d_{2})5(\tilde{d}+1)=5b(\tilde{d}+1), \] Soit encore \[ K>(5b+2)(\tilde{d}+1), \] ce que nous supposerons dor\'enavant. \end{rem}
Avec ces conditions, on a alors le lemme suivant : \begin{lemma}\label{arcsmineurs}
On a la majoration \[ \int_{\alpha\in \mathfrak{m}(\theta)}|S_{d}(\alpha)|d\alpha \ll d^{m-r+\varepsilon+\frac{d_{1}(r+1)}{2^{\tilde{d}}}}P_{1}^{m+1}P_{2}^{n-r+1}P^{-1-\delta}. \] \end{lemma}
\begin{proof}
On consid\`ere une suite $ (\theta_{i})_{i} $ telle que \[ \theta_{T}>\theta_{T-1}>...>\theta_{1}>\theta_{0} ,\] \[ \theta_{T}\leqslant \frac{1}{bd_{1}+d_{2}} \]\[ \theta_{T}K
>2\delta+1+\frac{b}{bd_{1}+d_{2}}, \] \[ \forall i \in \{0,...,T-1\}, \; \; 2(\tilde{d}+1)(\theta_{i+1}-\theta_{i})<\frac{\delta}{2} \] Un tel choix de $ \theta_{T} $ est possible, \'etant donn\'e que \[ \frac{K}{bd_{1}+d_{2}}>2\delta+1+\frac{b}{bd_{1}+d_{2}} \Leftrightarrow K>(2\delta+1)( bd_{1}+d_{2})+b, \] ce qui est assur\'e par la condition $ K>(5b+2)(\tilde{d}+1)$ de la remarque \ref{remarqueb}. Quitte \`a supposer $ P $ assez grand, on suppose de plus que $ T $ est tel que $ T\ll P^{\frac{\delta}{2}} $. On a alors, d'apr\`es le lemme \ref{dilemme4}, \begin{align*} \int_{\alpha\notin \mathfrak{M}(\theta_{T})}|S_{d}(\alpha)|d\alpha & \ll d^{m-r+\varepsilon+\frac{d_{1}(r+1)}{2^{\tilde{d}}}}P_{1}^{m+2}P_{2}^{n-r+1}P^{-K\theta_{T}+\varepsilon} \\ & \ll d^{m-r+\varepsilon+\frac{d_{1}(r+1)}{2^{\tilde{d}}}}P_{1}^{m+1}P_{2}^{n-r+1}P^{-1-\delta}.  \end{align*} Par ailleurs, \begin{align*}
\Vol(\mathfrak{M}^{(1)}(\theta_{i})) & \ll \sum_{\substack{q\leqslant dP^{(\tilde{d}+1)\theta_{i}} \\ d|q }}\sum_{\substack{0\leqslant a<q}}\frac{1}{q}d^{-(d_{1}-1)}P^{-1+(\tilde{d}+1)\theta_{i}} \\ & \ll d^{-(d_{1}-1)}P^{-1+2(\tilde{d}+1)\theta_{i}}, 
\end{align*}\begin{align*}
\Vol(\mathfrak{M}^{(2)}(\theta_{i})) & \ll \sum_{\substack{q\leqslant P^{(\tilde{d}+1)\theta_{i}} }}\sum_{\substack{0\leqslant a<q \\ \PGCD(a,q)=1}}\frac{1}{q}d^{-d_{1}}P^{-1+(\tilde{d}+1)\theta_{i}} \\ & \ll d^{-d_{1}}P^{-1+2(\tilde{d}+1)\theta_{i}},
\end{align*} et donc \[ \Vol(\mathfrak{M}(\theta_{i})) \ll d^{-(d_{1}-1)}P^{-1+2(\tilde{d}+1)\theta_{i}}. \]
On a donc que \begin{multline*} \int_{\alpha\in \mathfrak{M}(\theta_{i+1})\setminus\mathfrak{M}(\theta_{i})}|S_{d}(\alpha)|d\alpha \\ \ll d^{m-r+\varepsilon+d_{1}(r+1)/2^{\tilde{d}}-(d_{1}-1)}P^{-1+2(\tilde{d}+1)\theta_{i+1}}P_{1}^{m+2}P_{2}^{n-r+1}P^{-K\theta_{i}+\varepsilon} \\  \ll  d^{m-r+\varepsilon+d_{1}(r+1)/2^{\tilde{d}}-(d_{1}-1)}P_{1}^{m+2}P_{2}^{n-r+1}P^{(2(\tilde{d}+1)-K)\theta_{i}}P^{-1+2(\tilde{d}+1)(\theta_{i+1}-\theta_{i})+\varepsilon}\\  \ll  d^{m-r+\varepsilon+d_{1}(r+1)/2^{\tilde{d}}-(d_{1}-1)}P_{1}^{m+1}P_{2}^{n-r+1}P^{-1-2\delta+2(\tilde{d}+1)(\theta_{i}-\theta_{i+1})+\varepsilon} \\ \ll d^{m-r+\varepsilon+d_{1}(r+1)/2^{\tilde{d}}-(d_{1}-1)}P_{1}^{m+1}P_{2}^{n-r+1}P^{-1-\frac{3}{2}\delta}. \end{multline*}
On obtient le r\'esultat en sommant sur tous les $ i $ avec $ i\in \{0,...,T-1\} $ et $ T\ll P^{\frac{\delta}{2}} $.

\end{proof}
Ainsi, l'int\'egrale de $ S(\alpha) $ sur les arcs majeurs donne une contribution n\'egligeable par rapport \`a $ d^{m-r+\varepsilon+\frac{d_{1}(r+1)}{2^{\tilde{d}}}}P_{1}^{m+1}P_{2}^{n-r+1}P^{-1} $. Nous allons \`a pr\'esent nous int\'eresser \`a la contribution des arcs majeurs.

\subsection{Les arcs majeurs}
Pour des raisons pratiques, nous allons introduire de nouveaux arcs majeurs. Pour tout $ \theta\in [0,1] $, $ a,q\in \ZZ $, on pose \begin{equation}
\mathfrak{M}_{a,q}'(\theta)=\left\{ \alpha \in [0,1]\; | \; |\alpha q -a| \leqslant qd^{-d_{1}}P^{-1+(\tilde{d}+1)\theta} \right\},
\end{equation}\begin{equation}
\mathfrak{M}'(\theta)=\bigcup_{q\leqslant dP^{(\tilde{d}+1)\theta}}\bigcup_{\substack{0\leqslant a <q \\ \PGCD(a,q)=1}}\mathfrak{M}_{a,q}'(\theta).
\end{equation}
Remarquons que ce nouvel ensemble $ \mathfrak{M}'(\theta) $ contient $ \mathfrak{M}(\theta) $. En effet, si $ \alpha\in\mathfrak{M}_{a,q}^{(1)}(\theta)  $ pour un $ d|q $ et $ q\leqslant dP^{(\tilde{d}+1)\theta} $ on a alors $ q\geqslant d $ et \[ \left|\alpha-\frac{a}{q}\right|\leqslant q^{-1}d^{-(d_{1}-1)}P^{-1+(\tilde{d}+1)\theta} \leqslant d^{-d_{1}}P^{-1+(\tilde{d}+1)\theta}, \] donc si $ \frac{a}{q}=\frac{a'}{q'} $ avec $ \PGCD(a',q')=1 $, on a alors $ q'\leqslant dP^{(\tilde{d}+1)\theta}  $ et \[ \left|\alpha q'-a'\right|\leqslant q'd^{-d_{1}}P^{-1+(\tilde{d}+1)\theta}, \] et donc $ \alpha\in\mathfrak{M}_{a',q'}'(\theta)  $. D'autre part il est imm\'ediat que $  \mathfrak{M}^{(2)}(\theta) \subset \mathfrak{M}'(\theta) $. \\

 Par ailleurs, si $ \theta_{0}\in [0,1] $ v\'erifie les conditions\;\eqref{cond1} et\;\eqref{cond3}, on a le lemme suivant \begin{lemma}
Pour $ d_{1}\geqslant 2 $, les ensembles $ \mathfrak{M}_{a,q}'(\theta_{0}) $ sont disjoints deux \`a deux. 
\end{lemma}
\begin{proof}
Supposons qu'il existe $ \alpha\in \mathfrak{M}_{a,q}'(\theta_{0})\cap \mathfrak{M}_{a',q'}'(\theta_{0}) $, avec $ (a,q)\neq (a',q') $. On a alors (puisque $ \PGCD(a,q)=\PGCD(a',q')=1 $) : \[ \frac{1}{qq'}\leqslant \left| \frac{a}{q}-\frac{a'}{q'}\right|\leqslant \left| \frac{a}{q}-\alpha\right|+\left|\alpha-\frac{a'}{q'}\right|\leqslant 2d^{-d_{1}}P^{-1+(\tilde{d}+1)\theta}. \] On aurait donc  \[ 1 \leqslant 2qq'd^{-d_{1}}P^{-1+(\tilde{d}+1)\theta} \leqslant 2d^{2-d_{1}}P^{-1+3(\tilde{d}+1)\theta}\leqslant 2P^{-1+3(\tilde{d}+1)\theta}, \] ce qui est absurde car d'apr\`es \eqref{cond3}, \[ \theta<\frac{1}{5(\tilde{d}+1)(bd_{1}+d_{2})}<\frac{1}{3(\tilde{d}+1)}.  \]
\end{proof}
Puisque $ \mathfrak{M}(\theta_{0})\subset\mathfrak{M}'(\theta_{0}) $, le lemme \ref{arcsmineurs} implique le r\'esultat suivant : 
\begin{lemma}\label{sommearcsmaj}
On a l'estimation : \begin{multline*} N_{d}(P_{1},P_{2})=\sum_{1\leqslant q\leqslant dP^{(\tilde{d}+1)\theta_{0}}}\sum_{\substack{0\leqslant a<q \\ \PGCD(a,q)=1}}\int_{ \mathfrak{M}_{a,q}'(\theta_{0})}S_{d}(\alpha)d\alpha  \\ + O\left( d^{m-r+\varepsilon+\frac{d_{1}(r+1)}{2^{\tilde{d}}}}P_{1}^{m+1}P_{2}^{n-r+1}P^{-1-\delta}\right). \end{multline*}
\end{lemma}
Par la suite, \'etant donn\'e $ \alpha\in \mathfrak{M}_{a,q}'(\theta_{0}) $, on pose $ \alpha=\frac{a}{q}+\beta $, avec $ |\beta|\leqslant d^{-d_{1}}P^{-1+(\tilde{d}+1)\theta_{0}} $, et on note : \begin{equation}
S_{a,q,d}=\sum_{\bb_{1}\in (\ZZ/q\ZZ)^{r+1}}\sum_{\bb_{2}\in (\ZZ/q\ZZ)^{m-r}}\sum_{\bb_{3}\in (\ZZ/q\ZZ)^{n-m+1}}e\left(\frac{a}{q}F(d\bb_{1},\bb_{2},\bb_{3})\right)
\end{equation}
 et \begin{equation}
 I(\beta)=\int_{\substack{(\uu,\vv,\ww)\in \BB_{1}\times \BB_{2}\times \BB_{3} \\ |\vv|\leqslant |\uu| }}e(\beta F(\uu,\vv,\ww))d\uu d\vv d\ww. 
 \end{equation}
 On \'etablit alors le lemme suivant ; 
 \begin{lemma}\label{separation}
 Soit $ \alpha\in \mathfrak{M}_{a,q}'(\theta_{0}) $. On a alors \begin{multline*} S_{d}(\alpha)=d^{m-r}P_{1}^{m+1}P_{2}^{n-r+1}q^{-(n+2)}S_{a,q,d}I(d^{d_{1}}P\beta) \\ +O\left(d^{m-r+1}P_{1}^{m+1}P_{2}^{n-r+1}P^{2(\tilde{d}+1)\theta_{0}}P_{2}^{-1}\right). \end{multline*}
 \end{lemma}
 \begin{proof}
 On remarque dans un premier temps que \begin{equation}\label{Saq}
 S_{d}(\alpha)=\sum_{\bb_{1}\in (\ZZ/q\ZZ)^{r+1}}\sum_{\bb_{2}\in (\ZZ/q\ZZ)^{m-r}}\sum_{\bb_{3}\in (\ZZ/q\ZZ)^{n-m+1}}e\left(\frac{a}{q}F(d\bb_{1},\bb_{2},\bb_{3})\right)S_{3}(\bb_{1},\bb_{2},\bb_{3})
\end{equation} o\`u \[ S_{3}(\bb_{1},\bb_{2},\bb_{3}) =\sum_{\substack{\xx\equiv \bb_{1}(q)\\ |\xx|\leqslant P_{1}}}\sum_{\substack{\yy\equiv \bb_{2}(q)\\ |\yy|\leqslant d|\xx|P_{2}}}\sum_{\substack{\zz\equiv \bb_{2}(q)\\ |\zz|\leqslant P_{2}}}e(\beta F(d\xx,\yy,\zz)).
 \]
 Soient $ (\xx'\yy',\zz') $ et $ (\xx'',\yy'',\zz'') $ tels que \begin{multline*} (q\xx'+\bb_{1},q\yy'+\bb_{2},q\zz'+\bb_{3})\in P_{1}\BB_{1}\times dP_{1}P_{2}\BB_{2}\times P_{2}\BB_{3}, \; \\ \et  \; |q\yy'+\bb_{2}|\leqslant d|q\xx'+\bb_{1}|P_{2}, \end{multline*}\begin{multline*} (q\xx''+\bb_{1},q\yy''+\bb_{2},q\zz''+\bb_{3})\in P_{1}\BB_{1}\times dP_{1}P_{2}\BB_{2}\times P_{2}\BB_{3}, \; \\ \et  \; |q\yy''+\bb_{2}|\leqslant d|q\xx''+\bb_{1}|P_{2}, \end{multline*}
 \[ |\xx'-\xx''|\leqslant 2, \; \; |\yy'-\yy''|\leqslant 2, \; \; |\zz'-\zz''|\leqslant 2, \]
 On a dans ce cas : \begin{multline*}
 |F(q\xx'+\bb_{1},q\yy'+\bb_{2},q\zz'+\bb_{3})-F(q\xx''+\bb_{1},q\yy''+\bb_{2},q\zz''+\bb_{3})| \\ \ll  qd^{d_{1}}P_{1}^{d_{1}-1}P_{2}^{d_{2}}+qd^{d_{1}}P_{1}^{d_{1}-1}P_{2}^{d_{2}-1}+qd^{d_{1}}P_{1}^{d_{1}}P_{2}^{d_{2}-1} \ll qd^{d_{1}}P_{1}^{d_{1}}P_{2}^{d_{2}-1},
 \end{multline*}

 Remarquons que lorsque $ q>P_{2} $, l'\'egalit\'e du lemme est triviale. En effet, on a dans ce cas la majoration imm\'ediate : \begin{align*}
  |S_{d}(\alpha)| & \ll d^{m-r}P_{1}^{m+1}P_{2}^{n-r+1} \\ & \ll d^{m-r+1}P_{1}^{m+1}P_{2}^{n-r+1}P^{2(\tilde{d}+1)\theta_{0}}P_{2}^{-1}
 \end{align*}
 car $ P_{2}<q\leqslant dP^{(\tilde{d}+1)\theta_{0}} $, et d'autre part : \begin{align*}
d^{m-r}P_{1}^{m+1}P_{2}^{n-r+1}q^{-(n+2)}|S_{a,q,d}||I(d^{d_{1}}P\beta)| & \ll d^{m-r}P_{1}^{m+1}P_{2}^{n-r+1} \\ & \ll d^{m-r+1}P_{1}^{m+1}P_{2}^{n-r+1}P^{2(\tilde{d}+1)\theta_{0}}P_{2}^{-1}.
 \end{align*}
On suppose donc dor\'enavant que $ P_{2}\geqslant q $. En rempla\c{c}ant alors $ S_{3} $ par une int\'egrale on obtient : 
 \begin{multline*}
 S_{3}(\bb_{1},\bb_{2},\bb_{3})=\int_{|q\tilde{\uu}|\leqslant P_{1}}\int_{|q\tilde{\vv}|\leqslant d|q\tilde{\uu}|P_{2}}\int_{|q\tilde{\ww}|\leqslant P_{2}}e(\beta F(dq\tilde{\uu},q\tilde{\vv},q\tilde{\ww}))d\tilde{\uu}d\tilde{\vv}d\tilde{\ww} \\ + O\left(q|\beta|d^{d_{1}}P_{1}^{d_{1}}P_{2}^{(d_{2}-1)}\left(\frac{P_{1}}{q}\right)^{r+1}\left(\frac{dP_{1}P_{2}}{q}\right)^{m-r}\left(\frac{P_{2}}{q}\right)^{n-m+1}\right) \\ +O\left( \left(\frac{P_{1}}{q}\right)^{r+1}\left(\frac{dP_{1}P_{2}}{q}\right)^{m-r}\left(\frac{P_{2}}{q}\right)^{n-m}\right). 
 \end{multline*}
 En rappelant que $ |\beta|\leqslant d^{-d_{1}}P^{-1+(\tilde{d}+1)\theta_{0}} $, et en effectuant le changement de variables \[ \uu=qP_{1}^{-1}\tilde{\uu}, \; \;  \vv=q(dP_{1}P_{2})^{-1}\tilde{\vv}, \; \; \ww=qP_{1}^{-1}\tilde{\ww}, \]
 on trouve (puisque $ P=P_{1}^{d_{1}}P_{2}^{d_{2}} $. \begin{multline*}
 S_{3}(\bb_{1},\bb_{2},\bb_{3})=d^{m-r}P_{1}^{m+1}P_{2}^{n-r+1}q^{-(n+2)} \\ \int_{|\uu|\leqslant 1}\int_{|\vv|\leqslant |\uu|}\int_{|\ww|\leqslant 1}e(\beta F(dP_{1}\uu,dP_{1}P_{2}\vv,P_{2}\ww))d\uu d\vv d\ww \\ + O\left(\underbrace{q}_{\leqslant dP^{(\tilde{d}+1)\theta_{0}}}|\beta|d^{d_{1}}P_{1}^{d_{1}}P_{2}^{(d_{2}-1)}\left(\frac{P_{1}}{q}\right)^{r+1}\left(\frac{dP_{1}P_{2}}{q}\right)^{m-r}\left(\frac{P_{2}}{q}\right)^{n-m+1}\right) \\ +O\left( \left(\frac{P_{1}}{q}\right)^{r+1}\left(\frac{dP_{1}P_{2}}{q}\right)^{m-r}\left(\frac{P_{2}}{q}\right)^{n-m}\right) \\ =d^{m-r}P_{1}^{m+1}P_{2}^{n-r+1}q^{-(n+2)}I(d^{d_{1}}P\beta) \\ + O\left(d^{m-r+1}P_{1}^{m+1}P_{2}^{n-r+1}P_{2}^{-1}q^{-(n+2)}P^{2(\tilde{d}+1)\theta_{0}}\right). 
 \end{multline*}
 Puis, en rempla\c{c}ant $ S_{3} $ par cette expression dans \eqref{Saq}, on obtient le r\'esultat.
 \end{proof}
 En regroupant les lemmes \ref{sommearcsmaj} et \ref{separation}, on trouve : \begin{multline*}
  N_{d}(P_{1},P_{2})=d^{m-r}P_{1}^{m+1}P_{2}^{n-r+1}\sum_{1\leqslant q\leqslant dP^{(\tilde{d}+1)\theta_{0}}}q^{-(n+2)} \\ \sum_{\substack{0\leqslant a<q \\ \PGCD(a,q)=1}}S_{a,q,d}\int_{|\beta|\leqslant d^{-d_{1}}P^{-1+(\tilde{d}+1)\theta_{0}}}I(d^{d_{1}}P\beta)d\beta \\ +O\left(d^{m-r+1}P_{1}^{m+1}P_{2}^{n-r+1}P^{2(\tilde{d}+1)\theta_{0}}P_{2}^{-1}\Vol(\mathfrak{M}'(\theta_{0}))\right) \\ +O\left(d^{m-r+\varepsilon+\frac{d_{1}(r+1)}{2^{\tilde{d}}}}P_{1}^{m+1}P_{2}^{n-r+1}P^{-1-\delta}\right). 
 \end{multline*}
 En remarquant que \[ \Vol(\mathfrak{M}'(\theta_{0}))\ll \sum_{1\leqslant q\leqslant dP^{(\tilde{d}+1)\theta_{0}}}\sum_{\substack{0\leqslant a<q \\ \PGCD(a,q)=1}}d^{-d_{1}}P^{-1+(\tilde{d}+1)\theta_{0}}\ll d^{2-d_{1}}P^{-1+3(\tilde{d}+1)\theta_{0}}, \] et que \[ \int_{|\beta|\leqslant d^{-d_{1}}P^{-1+(\tilde{d}+1)\theta_{0}}}I(d^{d_{1}}P\beta)d\beta=d^{-d_{1}}P^{-1}\int_{|\beta|\leqslant P^{(\tilde{d}+1)\theta_{0}}}I(\beta)d\beta, \]  et en notant \begin{equation}
  \mathfrak{S}_{d}(Q)=\sum_{1\leqslant q\leqslant Q}q^{-(n+2)}\sum_{\substack{0\leqslant a<q \\ \PGCD(a,q)=1}}S_{a,q,d}
\end{equation} et \begin{equation}
J(\phi)=\int_{|\beta|\leqslant \phi}I(\beta)d\beta,
\end{equation}
on a \begin{multline}\label{formule} N_{d}(P_{1},P_{2})=d^{m-r-d_{1}}P_{1}^{m+1-d_{1}}P_{2}^{n-r+1-d_{2}} \mathfrak{S}_{d}(dP^{(\tilde{d}+1)\theta_{0}})J(P^{(\tilde{d}+1)\theta_{0}}) \\ +O\left(d^{m-r+3-d_{1}}P_{1}^{m+1}P_{2}^{n-r+1}P^{-1+5(\tilde{d}+1)\theta_{0}}P_{2}^{-1}\right)\\ +O\left(d^{m-r+\varepsilon+\frac{d_{1}(r+1)}{2^{\tilde{d}}}}P_{1}^{m+1}P_{2}^{n-r+1}P^{-1-\delta}\right).  \end{multline} Or, d'apr\`es \eqref{cond3} on a suppos\'e $ 5(\tilde{d}+1)\theta_{0}+\delta<\frac{1}{bd_{1}+d_{2}} $, donc on obtient : \[ d^{m-r+3-d_{1}}P_{1}^{m+1}P_{2}^{n-r+1}P^{-1+5(\tilde{d}+1)\theta_{0}}P_{2}^{-1} \ll d^{m-r+3-d_{1}}P_{1}^{m+1}P_{2}^{n-r+1}P^{-1-\delta}.   \] On d\'efinit \`a pr\'esent \begin{equation}
\mathfrak{S}_{d}=\sum_{q=1}^{\infty}\sum_{\substack{0\leqslant a<q \\ \PGCD(a,q)=1}}q^{-(n+2)}S_{a,q,d},
\end{equation}
\begin{equation}
J=\int_{\beta\in \RR}I(\beta)d\beta.
\end{equation}
Afin de pouvoir rempla\c{c}er $ J(P^{(\tilde{d}+1)\theta_{0}}) $ par $ J $ dans \eqref{formule}, nous allons \'etablir le lemme ci-dessous :
\begin{lemma}\label{integsing}
L'int\'egrale $ J $ est absolument convergente, et on a, pour tout $ \phi $ assez grand : \[ |J-J(\phi)|\ll \phi^{-1}. \]
\end{lemma}
\begin{proof}
 On choisit un \'el\'ement $ \theta\in [0,1] $ v\'erifiant les m\^emes conditions \eqref{cond1} et \eqref{cond3} que $ \theta_{0} $. Soit $ \beta $ tel que $ |\beta|>\phi $, on consid\`ere $ P_{1} $, $ P_{2} $, $ P $ tels que $ 2|\beta|=P^{(\tilde{d}+1)\theta} $ et on prend $ d=1 $. On a alors que $ P^{-1}\beta\in \mathfrak{M}_{0,1}(\theta) $, et d'apr\`es le lemme \ref{separation}  \begin{equation}\label{J1}
 S_{1}( P^{-1}\beta)=P_{1}^{m+1}P_{2}^{n-r+1}I(\beta)+O\left(P_{1}^{m+1}P_{2}^{n-r+1}P^{2(\tilde{d}+1)\theta}P_{2}^{-1}\right).
\end{equation} D'autre part, $ P^{-1}\beta $ appartient au bord de $ \mathfrak{M}_{0,1}(\theta) $, donc, puisque les $ \mathfrak{M}_{a,q}(\theta) $ sont disjoints, pour tout $ \varepsilon>0 $ arbitrairement petit, on a, par le lemme \ref{dilemmed=1}, \begin{equation}\label{J2}
S_{1}( P^{-1}\beta)\ll P_{1}^{m+2}P_{2}^{n-r+1}P^{-K\theta+\varepsilon}. 
\end{equation}
Par cons\'equent, en regroupant \eqref{J1} et \eqref{J2}, on trouve : \begin{align*}
|I(\beta)| & \ll P_{1}P^{-K\theta+\varepsilon}+O\left(P_{2}^{-1}P^{2(\tilde{d}+1)\theta}\right)
\\ & =P^{\frac{b}{bd_{1}+d_{2}}-K\theta+\varepsilon}+O\left(P^{-\frac{1}{bd_{1}+d_{2}}+2(\tilde{d}+1)\theta}\right).
\end{align*} 
Or on a d'apr\`es \eqref{cond3} que \[ \frac{1}{bd_{1}+d_{2}}-2(\tilde{d}+1)\theta>3(\tilde{d}+1)\theta+\delta,  \] donc \[ P^{-\frac{1}{bd_{1}+d_{2}}+2(\tilde{d}+1)\theta} \ll P^{-3(\tilde{d}+1)\theta-\delta} \ll |\beta|^{-3}. \] Par ailleurs, d'apr\`es \eqref{cond1}, on a \[ K\theta-2(\tilde{d}+1)\theta>2\delta+\frac{b}{bd_{1}+d_{2}}, \] et donc \[ P^{\frac{b}{bd_{1}+d_{2}}-K\theta+\varepsilon}\ll P^{-2(\tilde{d}+1)\theta} \ll |\beta|^{-2}. \] On en d\'eduit donc \[ \int_{|\beta|>\phi}|I(\beta)|d\beta \ll \phi^{-1}. \] D'o\`u le r\'esultat du lemme. 
\end{proof}
De m\^eme, pour pouvoir rempla\c{c}er $ \mathfrak{S}_{d}(dP^{(\tilde{d}+1)\theta_{0}}) $ par $ \mathfrak{S}_{d} $ dans \eqref{formule}, on \'etablit :
\begin{lemma}\label{seriesing}
Pour $ d_{1}\geqslant 2 $, la s\'erie $ \mathfrak{S}_{d}  $ est absolument convergente, et on a, pour tout $ Q\geqslant d $ assez grand : \[ |\mathfrak{S}_{d} -\mathfrak{S}_{d} (Q)|\ll \max\{d^{d_{1}(r+1)/2^{\tilde{d}}+\varepsilon},d\}Q^{-\delta}, \] pour $ \delta>0 $ arbitrairement petit.
\end{lemma}
\begin{proof}
 On choisit un \'el\'ement $ \theta\in [0,1] $ v\'erifiant les conditions \eqref{cond1} et \eqref{cond3}. Soit $ q>Q\geqslant d $ quelconque et $ a $ tel que $ 0\leqslant a <q $ et $ \PGCD(a,q)=1 $. On choisit $ P_{1},P_{2}\geqslant 1 $ tels que $ q=dP^{(\tilde{d}+1)\theta} $ avec $ P=P_{1}^{d_{1}}P_{2}^{d_{2}} $. D'apr\`es le lemme \ref{separation}, si $ \alpha=\frac{a}{q} $, on a \begin{multline*} |S_{d}(\alpha)|= d^{m-r}P_{1}^{m+1}P_{2}^{n-r+1}q^{-(n+2)}S_{a,q,d}I(0) \\ +O\left(d^{m-r+1}P_{1}^{m+1}P_{2}^{n-r+1}P^{2(\tilde{d}+1)\theta}P_{2}^{-1}\right).\end{multline*} Par ailleurs, si l'on pose $ \theta'=\theta-\nu $ avec $ \nu>0 $ arbitrairement petit, on a alors que $ \alpha=\frac{a}{q}\notin \mathfrak{M}(\theta') $. En effet, supposons qu'il existe $ a',q' $ tels que $ d|q' $ $ q'\leqslant dP^{(\tilde{d}+1)\theta'}<dP^{(\tilde{d}+1)\theta}=q $, $ 0\leqslant a'<q' $,  et $ \alpha\in  \mathfrak{M}_{a',q'}^{(1)}(\theta') $. Si $ aq'=qa' $, on a donc, puisque $ \PGCD(a,q)=1 $,  $ a|a' $ et donc si $ a'\neq 0 $ $ q'=\frac{a'}{a}q $, ce qui est absurde car $ q>q' $, et si $ a=a'=0 $, alors $ q=1 $ ce qui contredit encore $ q>q' $. On a alors \[ 1\leqslant |aq'-a'q|\leqslant qd^{-(d_{1}-1)}P^{-1+(\tilde{d}+1)\theta'}<d^{2-d_{1}}P^{-1+2(\tilde{d}+1)\theta} \] ce qui est absurde car $ \theta<\frac{1}{2(\tilde{d}+1)} $ d'apr\`es \eqref{cond3}. De la m\^eme mani\`ere, s'il existe $ a',q' $ tels que $ q'\leqslant P^{(\tilde{d}+1)\theta'}<dP^{(\tilde{d}+1)\theta}=q $, $ 0\leqslant a'<q' $, $ \PGCD(a',q')=1 $ et $ \alpha\in  \mathfrak{M}_{a',q'}^{(2)}(\theta') $, on a 
 \[ 1\leqslant |aq'-a'q|\leqslant qd^{-d_{1}}P^{-1+(\tilde{d}+1)\theta'}<d^{1-d_{1}}P^{-1+2(\tilde{d}+1)\theta}. \]

 Par cons\'equent, d'apr\`es le lemme \ref{dilemme4}, on a \[  |S(\alpha)|\ll d^{m-r+\varepsilon+d_{1}(r+1)/2^{\tilde{d}}}P_{1}^{m+2}P_{2}^{n-r+1}P^{-K\theta'+\varepsilon} \] et on obtient donc (\'etant donn\'e que $ I(0)\asymp 1 $), pour $ \nu $ assez petit :  \begin{multline*} |S_{a,q,d}|\ll q^{n+2}d^{d_{1}(r+1)/2^{\tilde{d}}+\varepsilon}P_{1}P^{-K\theta'+\varepsilon} + \left (dq^{n+2}P^{2(\tilde{d}+1)\theta}P_{2}^{-1}\right) \\ \ll d^{d_{1}(r+1)/2^{\tilde{d}}+\varepsilon}q^{n+2}P^{\frac{b}{bd_{1}+d_{2}}-K\theta+2\varepsilon} + \left (dq^{n+2}P^{2(\tilde{d}+1)\theta-\frac{1}{bd_{1}+d_{2}}}\right). \end{multline*} Or, par les conditions \eqref{cond1} et \eqref{cond3} on a (pour $ \delta=\delta'(\tilde{d}+1) $). \[ P^{\frac{b}{bd_{1}+d_{2}}-K\theta+2\varepsilon} \ll  P^{-2(\tilde{d}+1)\theta-\delta'(\tilde{d}+1)\theta}=q^{-2-\delta'}\] \[ P^{2(\tilde{d}+1)\theta-\frac{1}{bd_{1}+d_{2}}} \ll      P^{-3(\tilde{d}+1)\theta-\delta} \ll q^{-3}.\] On a donc \[q^{-(n+2)}|S_{a,q,d}|\ll d^{d_{1}(r+1)/2^{\tilde{d}}+\varepsilon}q^{-2-\delta'}+dq^{-3} \] et ainsi \begin{align*}
 |\mathfrak{S}_{d} -\mathfrak{S}_{d} (Q)| & \ll \sum_{q>Q}\sum_{\substack{0\leqslant a <q \\ \PGCD(a,q)=1}}q^{-(n+2)}|S_{a,q,d}| \\ & \ll \max\{d^{d_{1}(r+1)/2^{\tilde{d}}+\varepsilon},d\}\sum_{q>Q}\sum_{\substack{0\leqslant a <q \\ \PGCD(a,q)=1}}q^{-2-\delta'} \\ & \ll \max\{d^{d_{1}(r+1)/2^{\tilde{d}}+\varepsilon},d\}Q^{-\delta'}. 
\end{align*}  
\end{proof}
\begin{rem}\label{remseriesing}
Remarquons que \[ \mathfrak{S}_{d} (d)\ll d^{2}, \] et donc le lemme pr\'ec\'edent nous donne une majoration de $ \mathfrak{S}_{d} $ ; \[ |\mathfrak{S}_{d}| \ll d^{2}+ \max\{d^{d_{1}(r+1)/2^{\tilde{d}}+\varepsilon},d\}d^{-\delta} \ll \max\{d^{d_{1}(r+1)/2^{\tilde{d}}},d^{2}\}.  \]
\end{rem}

En utilisant les lemmes \ref{seriesing} et \ref{integsing}, et en notant \begin{equation}
\sigma_{d}=d^{m-r-d_{1}}\mathfrak{S}_{d}J,
\end{equation}  on obtient finalement le r\'esultat suivant 

\begin{prop}\label{propgeneral}
Pour, $ P_{1}=P_{2}^{b} $ avec $ b\geqslant 1 $, si $ d_{1}\geqslant 2 $ et si l'on suppose que , $ K=(n+2-\max\{\dim V_{1}^{\ast},\dim V_{2}^{\ast}\}-\varepsilon)/2^{\tilde{d}} $ est tel que \[ K>\max\{ bd_{1}+d_{2}, (5b+2)(d_{1}+d_{2}-1)\}  , \] on a alors \begin{multline*} N_{d}(P_{1},P_{2})=\sigma_{d} P_{1}^{m+1-d_{1}}P_{2}^{n-r+1-d_{2}} \\ +O\left(d^{m-r}\max\{d^{d_{1}(r+1)/2^{\tilde{d}}+\varepsilon},d^{3-d_{1}}\}P_{1}^{m+1-d_{1}-\delta}P_{2}^{n-r+1-d_{2}-\delta}\right), \end{multline*}
pour un r\'eel $ \delta>0 $ arbitrairement petit. 
\end{prop}
\begin{rem}
Remarquons que dans le cas o\`u $ P_{1}\leqslant P_{2} $, et $ P_{2}=P_{1}^{u} $, on obtient exactement la m\^eme estimation de $ N_{d}(P_{1},P_{2}) $ lorsque \[ K>\max\{ d_{1}+ud_{2}, 7(d_{1}+d_{2}-1)\}. \]
\end{rem}

\section{Deuxi\`eme \'etape}

Dans cette section nous allons \'etablir, pour un $ \xx\in \ZZ^{r+1} $ fix\'e, en notant $ k=|\xx| $, une formule asymptotique pour  \begin{multline*}
N_{d,\xx}(P_{2})=\Card\left\{ (\yy,\zz)\in  (dkP_{2}\BB_{2}\times P_{2}\BB_{3})\cap\ZZ^{n-r+1} \; | \;   F(d\xx,\yy,\zz)=0 \right\}. 
\end{multline*}
\`A cette fin on pose \begin{equation}
S_{d,\xx}(\alpha)=\sum_{\substack{\yy\in \ZZ^{m-r}\\|\yy|\leqslant dkP_{2}}}\sum_{\substack{\zz\in \ZZ^{n-m+1}\\ |\zz|\leqslant P_{2}}}e(\alpha F(d\xx,\yy,\zz)),
\end{equation}
et on remarque que \[ N_{d,\xx}(P_{2})=\int_{0}^{1}S_{d,\xx}(\alpha)d\alpha. \]

\subsection{Somme d'exponentielles}
En appliquant le m\^eme proc\'ed\'e que pour la section \ref{Weyl}, on a, pour $ \xx $ fix\'e : \begin{multline*}
|S_{d,\xx}(\alpha)|^{2^{d_{2}-1}}\ll \left((dkP_{2})^{m-r}\right)^{2^{d_{2}-1}-d_{2}}\left(P_{2}^{n-m+1}\right)^{2^{d_{2}-1}-d_{2}} \\ \sum_{\substack{\yy^{(1)},\zz^{(1)} \\ |\yy^{(1)}|\leqslant dkP_{2} \\ |\zz^{(1)}|\leqslant P_{2}}}...\sum_{\substack{\yy^{(d_{2}-1)},\zz^{(d_{2}-1)} \\ |\yy^{(d_{2}-1)}|\leqslant dkP_{2} \\ |\zz^{(d_{2}-1)}|\leqslant P_{2}}} \prod_{j=r+1}^{n+1}\min\left\{ H_{j}, \left|\left| \alpha\gamma_{d,\xx,j}\left((\yy^{(i)},\zz^{(i)})_{i\in \{1,...,d_{2}-1\}}\right)\right|\right|^{-1}\right\}
\end{multline*}
avec \[ H_{j}=\left\{ \begin{array}{rcl} dkP_{2} & \mbox{si} & j\in \{r+1,...,m\} \\ P_{2} & \mbox{si} & j\in \{m+1,...,n+1\}\end{array}\right., \] \[ \gamma_{d,\xx,j}\left((\yy^{(i)},\zz^{(i)})_{i\in \{1,...,d_{2}-1\}}\right)=\sum_{\ii=(i_{1},...,i_{d_{2}-1})\in \{r+1,...,n+1\}^{d_{2}-1}}F_{d\xx,\ii,j}u_{i_{1}}^{(1)}...u_{i_{d_{2}-1}}^{(d_{2}-1)},\] o\`u \[ u_{i}=\left\{ \begin{array}{rcl} y_{i} & \mbox{si} & i\in \{r+1,...,m\} \\ z_{i} & \mbox{si} & i\in \{m+1,...,n+1\}\end{array}\right., \] et les coefficients $ F_{d\xx,\ii,j} $ sont sym\'etriques en $ (\ii,j)\in \{r+1,...,n+1\}^{d_{2}}  $. \`A partir de l\`a, on montre, comme dans la section \ref{Weyl} que 
\begin{multline*} |S_{d,\xx}(\alpha)|^{2^{d_{2}-1}}\ll \left((dkP_{2})^{m-r+\varepsilon}\right)^{2^{d_{2}-1}-d_{2}+1}\left(P_{2}^{n-m+1+\varepsilon}\right)^{2^{d_{2}-1}-d_{2}+1} \\  M_{d,\xx}\left(\alpha,dkP_{2},P_{2},(dkP_{2})^{-1},P_{2}^{-1}\right), \end{multline*}
o\`u l'on a not\'e pour tous r\'eels strictement positifs $ H_{1},H_{2},B_{1},B_{2} $ : \begin{multline*} M_{d,\xx}\left(\alpha,H_{1},H_{2},B_{1}^{-1},B_{2}^{-1}\right)= \card\left\{ (\yy^{(1)},\zz^{(1)},...,\yy^{(d_{2}-1},\zz^{(d_{2}-1)}) \; | \; |\yy^{(i)}|\leqslant H_{1}, \; \right.\\  |\zz^{(i)}|\leqslant H_{2}, \;\et, \forall j \in\{r+1,...,m\}\;\left|\left| \alpha\gamma_{d,\xx,j}\left((\yy^{(i)},\zz^{(i)})_{i\in \{1,...,d_{2}-1\}}\right)\right|\right|\leqslant B_{1}^{-1} \\ \left. \forall j\in \{m+1,...,n+1\}\;\left|\left| \alpha\gamma_{d,\xx,j}\left((\yy^{(i)},\zz^{(i)})_{i\in \{1,...,d_{2}-1\}}\right)\right|\right|\leqslant B_{2}^{-1}\right\}.  \end{multline*} On en d\'eduit : \begin{lemma}
Si $ P>1 $, $ \kappa>0 $ et $ \varepsilon>0 $ arbitrairement petit, l'une au moins des assertions suivantes est v\'erifi\'ee : \begin{enumerate}
\item $ |S_{d,\xx}(\alpha)|\ll (dk)^{m-r+\varepsilon}P_{2}^{n+1-r+\varepsilon}P^{-\kappa}, $ 
\item  $ M_{d,\xx}\left(\alpha,dkP_{2},P_{2},(dkP_{2})^{-1},P_{2}^{-1}\right)\gg \left((dkP_{2})^{m-r}\right)^{d_{2}-1}\left(P_{2}^{n-m+1}\right)^{d_{2}-1}P^{-2^{d_{2}-1}\kappa}. $
\end{enumerate}
\end{lemma}
Or pour $ (\yy^{(i)},\zz^{(i)})_{i\in \{1,...,d_{2}-2\}} $ fix\'es, le r\'eseau d\'efini par les $ (\yy^{(d_{2}-1)},\zz^{(d_{2}-1)}) $ et les formes lin\'eaires $  \alpha\gamma_{d,\xx,j} $ est sym\'etrique (i.e. si $ \gamma_{d,\xx,j}(\uu)=\sum_{l\in \{r+1,...,n+1\}}\lambda_{j,l}u_{l} $, alors $ \lambda_{j,l}=\lambda_{l,j} $). On peut donc appliquer le lemme \ref{geomnomb2}, avec des param\`etres  $ a_{j},Z,Z' $ bien choisis. 

On pose \[ \begin{array}{l} Z=1 \\ Z'=P_{2}^{-1}P^{\theta} \\ \forall j \in \{r+1,...,m\}, \; \; a_{j}=dkP_{2} \\ \forall j \in \{m+1,...,n+1\}, \; \; a_{j}=P_{2} 
\end{array}, \] avec $ \theta $ tel que $ P^{\theta}\leqslant P_{2} $ de sorte que 
\[ \begin{array}{lrclrcl}
\forall j \in \{r+1,...,m\}, & a_{j}Z & = & dkP_{2} & a_{k}^{-1}Z & = & (kP_{2})^{-1} \\ \forall j \in \{m+1,...,n+1\}, & a_{j}Z & = & P_{2} & a_{j}^{-1}Z & = & P_{2}^{-1} \\ \forall j \in \{r+1,...,m\}, & a_{j}Z' & = & dkP^{\theta}& a_{j}^{-1}Z' & = & (dk)^{-1}P_{2}^{-2}P^{\theta} \\ \forall j \in \{m+1,...,n+1\}, & a_{j}Z' & = & P^{\theta} & a_{j}^{-1}Z' & = & P_{2}^{-2}P^{\theta}
\end{array} \]

Puis, on r\'eit\`ere ce proc\'ed\'e avec $ (\yy^{(d_{2}-i)},\zz^{(d_{2}-i)}) $, pour $ i\in\{2,...,d_{2}-1\} $, en choisissant : \[ \begin{array}{l} Z=P_{2}^{-\frac{(i-1)}{2}}P^{\frac{(i-1)\theta}{2}} \\ Z'=P_{2}^{-\frac{(i+1)}{2}}P^{\frac{(i+1)\theta}{2}} \\ \forall j \in \{r+1,...,m\}, \; \; a_{j}=dkP_{2}^{\frac{(i+1)}{2}}P^{-\frac{(i-1)\theta}{2}}  \\ \forall j \in \{m+1,...,n+1\}, \; \; a_{j}=P_{2}^{\frac{(i+1)}{2}}P^{-\frac{(i-1)\theta}{2}} 
\end{array}, \] de sorte que 
\[ \begin{array}{lrclrcl}
\forall j \in \{r+1,...,m\}, & a_{j}Z & = & dkP_{2} & a_{j}^{-1}Z & = & (dk)^{-1}P_{2}^{-i}P^{(i-1)\theta} \\ \forall j \in \{m+1,...,n+1\}, & a_{j}Z & = & P_{2} & a_{j}^{-1}Z & = &  P_{2}^{-i}P^{(i-1)\theta} \\ \forall j \in \{r+1,...,m\}, & a_{j}Z' & = & dkP^{\theta}& a_{j}^{-1}Z' & = & (dk)^{-1}P_{2}^{-(i+1)}P^{i\theta} \\ \forall j \in \{m+1,...,n+1\}, & a_{j}Z' & = & P^{\theta} & a_{j}^{-1}Z' & = & P_{2}^{-(i+1)}P^{i\theta}
\end{array}. \]

On obtient alors finalement : \begin{multline*} 
M_{d,\xx}\left(\alpha,dkP_{2},P_{2},(dkP_{2})^{-1},P_{2}^{-1}\right) \\ \ll \left(\frac{P_{2}}{P^{\theta}}\right)^{(d_{2}-1)(n-r+1)}
  M_{d,\xx}\left(\alpha,dkP^{\theta},P^{\theta},(dk)^{-1}P_{2}^{-d_{2}}P^{(d_{2}-1)\theta},P_{2}^{-d_{2}}P^{(d_{2}-1)\theta}\right)
\end{multline*}
On remarque par ailleurs (en utilisant le lemme \ref{lemmedebile}) que 
 \begin{multline*} 
  M_{d,\xx}\left(\alpha,dkP^{\theta},P^{\theta},(dk)^{-1}P_{2}^{-d_{2}}P^{(d_{2}-1)},P_{2}^{-d_{2}}P^{(d_{2}-1)}\right) \\ \ll (dk)^{(d_{2}-1)(m-r)}M_{d,\xx}\left(\alpha,P^{\theta},P^{\theta},(dk)^{-1}P_{2}^{-d_{2}}P^{(d_{2}-1)\theta},P_{2}^{-d_{2}}P^{(d_{2}-1)\theta}\right)
\end{multline*}
On a donc le lemme suivant : 

\begin{lemma}
Si $ P>1 $, $ \kappa>0 $ et $ \varepsilon>0 $ arbitrairement petit, l'une au moins des assertions suivantes est v\'erifi\'ee : \begin{enumerate}
\item $ |S_{d,\xx}(\alpha)|\ll (dk)^{m-r+\varepsilon}P_{2}^{n+1-r+\varepsilon}P^{-\kappa}, $ 
\item  \begin{multline*} M_{d,\xx}\left(\alpha,P^{\theta},P^{\theta},(dk)^{-1}P_{2}^{-d_{2}}P^{(d_{2}-1)},P_{2}^{-d_{2}}P^{(d_{2}-1)}\right) \\ \gg \left(P^{\theta}\right)^{(n-r+1)(d_{2}-1)}P^{-2^{d_{2}-1}\kappa}. \end{multline*}
\end{enumerate}
\end{lemma}

Remarquons \`a pr\'esent que s'il existe $ j_{0}\in \{r+1,...,n+1\} $ tel que $ \gamma_{d,\xx,j_{0}}\left((\yy^{(i)},\zz^{(i)})_{i\in \{1,...,d_{2}-1\}}\right)\neq 0 $ pour un certain $ (\yy^{(i)},\zz^{(i)})_{i\in \{1,...,d_{2}-1\}} $ tel que $ |\yy^{(i)}|\leqslant P^{\theta} $, $ |\zz^{(i)}|\leqslant P^{\theta} $ pour tout $ i\in \{1,...,d_{2}-1\} $, et \[ \left|\left|\alpha\gamma_{d,\xx,j_{0}}\left((\yy^{(i)},\zz^{(i)})_{i\in \{1,...,d_{2}-1\}}\right)\right|\right|\leqslant P_{2}^{-d_{2}}P^{(d_{2}-1)\theta},\]  alors en posant $ q=\gamma_{d,\xx,j_{0}}\left((\yy^{(i)},\zz^{(i)})_{i\in \{1,...,d_{2}-1\}}\right) $, on a $ q\ll d^{d_{1}}k^{d_{1}}P^{(d_{2}-1)\theta} $ et il existe $ a $ tel que \[ |\alpha q-a|\leqslant P_{2}^{-d_{2}}P^{(d_{2}-1)\theta}. \] Quitte \`a changer $ \theta $, on peut supposer $ q\leqslant d^{d_{1}}k^{d_{1}}P^{(d_{2}-1)\theta} $, $ 0\leqslant a <q $ et $ \PGCD(a,q)=1 $.
Dans ce qui suit, on posera \begin{equation}
\mathfrak{M}_{a,q}^{d,\xx}(\theta)=\left\{ \alpha\in [0,1[ \; | \;  2|\alpha q-a|\leqslant P_{2}^{-d_{2}}P^{(d_{2}-1)\theta}\right\},
\end{equation}
\begin{equation}
\mathfrak{M}^{d,\xx}(\theta)=\bigcup_{q\leqslant d^{d_{1}}k^{d_{1}}P^{(d_{2}-1)\theta}}\bigcup_{\substack{0\leqslant a <q \\ \PGCD(a,q)=1 }}\mathfrak{M}_{a,q}^{d,\xx}(\theta).
\end{equation}

 On en d\'eduit donc : \begin{lemma}\label{dilemme23}
Si $ P>1 $, $ \kappa>0 $ et $ \varepsilon>0 $ arbitrairement petit, l'une au moins des assertions suivantes est v\'erifi\'ee : \begin{enumerate}
\item $ |S_{d,\xx}(\alpha)|\ll (dk)^{m-r+\varepsilon}P_{2}^{n+1-r+\varepsilon}P^{-\kappa}, $ 
\item le r\'eel $ \alpha $ appartient \`a $ \mathfrak{M}^{d,\xx}(\theta) $,
\item  \begin{multline*} 
\Card\left\{(\yy^{(i)},\zz^{(i)})_{i\in \{1,...,d_{2}-1\}}, \; |\yy^{(i)}|\leqslant P^{\theta}, \; |\zz^{(i)}|\leqslant P^{\theta}, \; \right. \\ \left. \et \; \forall j \in \{r+1,...,n+1\},\;  \gamma_{d,\xx,j}\left((\yy^{(i)},\zz^{(i)})_{i\in \{1,...,d_{2}-1\}}\right)=0\right\} \\  \gg \left(P^{\theta}\right)^{(n-r+1)(d_{2}-1)}P^{-2^{d_{2}-1}\kappa}. \end{multline*}
\end{enumerate}
\end{lemma}

On d\'efinit \`a pr\'esent, pour $ \xx $ fix\'e : \begin{multline}
V_{2,\xx}^{\ast}=\left\{ (\yy,\zz) \in \CC^{n-r+1} \; | \; \forall j \in \{r+1,...,m\}, \frac{\partial F}{\partial y_{j}}(\xx,\yy,\zz) =0 \; \right.  \\ \left. \et \; \forall k \in \{m+1,...,n+1\}, \;  \frac{\partial F}{\partial z_{k}}(\xx,\yy,\zz) =0 \right\}.
\end{multline}
Remarquons que, puisque $ F $ est homog\`eme de degr\'e $ d_{1} $ en $ (\xx,\yy) $, on a pour tous $ j,k $: \[ \frac{\partial F}{\partial y_{j}}(d\xx,d\yy,\zz)=d^{d_{1}-1}\frac{\partial F}{\partial y_{j}}(\xx,\yy,\zz) \] \[ \frac{\partial F}{\partial z_{k}}(d\xx,d\yy,\zz)=d^{d_{1}}\frac{\partial F}{\partial z_{k}}(\xx,\yy,\zz)\] et donc l'application $ (\yy,\zz)\mt (d\yy,\zz) $ r\'ealise un isomorphisme de $ V_{2,d\xx}^{\ast} $ sur $ V_{2,\xx}^{\ast} $, donc en particulier : \[ \dim V_{2,d\xx}^{\ast}=\dim V_{2,\xx}^{\ast}. \]
On note par ailleurs : \begin{equation}
\mathcal{A}_{2}^{\lambda}=\left\{ \xx \in \CC^{r+1} \; | \; \dim V_{2,\xx}^{\ast} < \dim V_{2}^{\ast} -(r+1)+\lambda \right\},
\end{equation}
o\`u $ \lambda\in \NN $ est un param\`etre que nous pr\'eciserons ult\'erieurement. Par abus de langage on notera  \[\mathcal{A}_{2}^{\lambda}(\ZZ)=\mathcal{A}_{2}^{\lambda}\cap \ZZ^{r+1}. \] 

Supposons \`a pr\'esent que $ \xx\in \mathcal{A}_{2}^{\lambda}(\ZZ)  $ et que l'assertion $ 3. $ du lemme \ref{dilemme23} est v\'erifi\'ee. Posons par ailleurs $ K_{2}=\kappa/\theta $. Si $ \mathcal{L}_{2,d,\xx} $ est la sous-vari\'et\'e affine de $ \AA^{(n-r+1)(d_{2}-1)} $ d\'efinie par les \'equations \[ \gamma_{d,\xx,j}\left((\yy^{(i)},\zz^{(i)})_{i\in \{1,...,d_{2}-1\}}\right)=0, \] 
on a alors, d'apr\`es la d\'emonstration de \cite[Th\'eor\`eme 3.1]{Br} : \[ \dim  \mathcal{L}_{2,d,\xx} \geqslant  (n-r+1)(d_{2}-1)-2^{d_{2}-1}K_{2}. \] On consid\`ere d'autre part l'intersection avec la diagonale \[ \mathcal{D}_{2} : \left\{ \begin{array}{l} \yy^{(1)}=...=\yy^{(d_{2}-1)} \\ \zz^{(1)}=...=\zz^{(d_{2}-1)} 
\end{array}\right. . \] On a \begin{align*} \dim  (\mathcal{L}_{2,d,\xx}\cap \mathcal{D}_{2}) & \geqslant \dim  \mathcal{L}_{2,d,\xx}- (n-r+1)(d_{2}-2) \\ & \geqslant n-r+1-2^{d_{2}-1}K_{2} \end{align*}
On remarque par ailleurs que $ \mathcal{L}_{2,d,\xx}\cap \mathcal{D}_{2} $ est isomorphe \`a $ V_{2,d\xx}^{\ast} $, et donc, puisque $ \xx\in  \mathcal{A}_{2}^{\lambda}(\ZZ) $, et $ \dim V_{2,d\xx}^{\ast}=\dim V_{2,\xx}^{\ast} $, on obtient \[ 2^{d_{2}-1}K_{2}\geqslant n-r+1-\dim V_{2,\xx}^{\ast}> n+2-\dim V_{2}^{\ast}-\lambda. \] On posera donc dor\'enavant \begin{equation}
K_{2}=(n+2-\dim V_{2}^{\ast}-\lambda)/2^{d_{2}-1}
\end{equation} et le lemme \ref{dilemme23} devient alors : 
\begin{lemma}\label{dilemme24}
Si $ \varepsilon>0 $ est un r\'eel arbitrairement petit, l'une au moins des assertions suivantes est v\'erifi\'ee : \begin{enumerate}
\item $ |S_{d,\xx}(\alpha)|\ll (dk)^{m-r+\varepsilon}P_{2}^{n+1-r+\varepsilon}P^{-K_{2}\theta}, $ 
\item le r\'eel $ \alpha $ appartient \`a $ \mathfrak{M}^{d,\xx}(\theta) $.
\end{enumerate}
\end{lemma}
Pour tout le reste de cette section on fixera $ P=P_{2} $. 
Avant d'aller plus loin, nous \'etablissons une propri\'et\'e de l'ensemble $ \mathcal{A}_{2}^{\lambda} $ 
\begin{prop}\label{propA2}
L'ensemble $ \mathcal{A}_{2}^{\lambda} $ est un ouvert de Zariski de $ \AA_{\CC}^{r+1} $, et de plus, on a \[ \Card\left\{ \xx \in [-P_{1},P_{1}]^{r+1}\cap (\mathcal{A}_{2}^{\lambda})^{c}\cap \ZZ^{r+1} \right\}\ll P_{1}^{r+1-\lambda}. \] 
\end{prop}

\begin{proof}
On commence par montrer que $ \{\xx\in \AA^{r+1}_{\CC} \; | \; \dim V_{2,\xx}^{\ast}\geqslant \lambda \} $ est un ferm\'e de Zariski de $ \AA^{r+1}_{\CC} $. \\

Notons $ Y $ le ferm\'e de $ \AA_{\CC}^{r+1}\times\PP^{n-r}_{\CC} $ d\'efinit par : \begin{multline*} Y=\left\{(\xx,\yy,\zz)\in\AA_{\CC}^{r+1}\times\PP^{n-r}_{\CC}\; | \;\forall j \in \{m+1,...,n+1\}, \frac{\partial F}{\partial y_{j}}(\xx,\yy,\zz) =0 \; \right.  \\ \left. \et \; \forall k \in \{m+1,...,n+1\}, \;  \frac{\partial F}{\partial z_{k}}(\xx,\yy,\zz) =0 \right\}.  \end{multline*} La projection canonique \[ \pi : Y\subset \AA_{\CC}^{r+1}\times\PP^{n-r}_{\CC} \ra \AA^{r+1}_{\CC}, \] est un morphisme projectif, donc ferm\'e. Par cons\'equent, d'apr\`es \cite[Corollaire 13.1.5]{G-D}, \[ \{\xx\in \AA^{r+1}_{\CC} \; | \; \dim Y_{\xx}\geqslant \lambda-1 \} \] est un ferm\'e, et puisque $ \dim Y_{\xx}=  \dim V_{2,\xx}^{\ast}-1 $, l'ensemble \[ \{\xx\in \AA^{r+1}_{\CC} \; | \; \dim V_{2,\xx}^{\ast}\geqslant \lambda\} \] est un ferm\'e de Zariski de $ \AA^{r+1}_{\CC}  $.\\

Nous allons \`a pr\'esent montrer que $ \dim(\mathcal{A}_{2}^{\lambda})^{c}\leqslant r+1-\lambda $. On remarque que \[ Y\cap ((\mathcal{A}_{2}^{\lambda})^{c}\times\PP^{n-r}_{\CC})=\bigsqcup_{\xx\in (\mathcal{A}_{2}^{\lambda})^{c} }\pi^{-1}(\xx). \]
On a alors \[ \dim(\mathcal{A}_{2}^{\lambda})^{c}+\dim V_{2}^{\ast}-(r+1)+\lambda-1\leqslant \dim Y=\dim V_{2}^{\ast}-1, \] ce qui implique \[  \dim(\mathcal{A}_{2}^{\lambda})^{c}\leqslant r+1-\lambda, \]
et donc \[ \card\{\xx\in [-P_{1},P_{1}]^{r+1}\cap(\mathcal{A}_{2}^{\lambda})^{c}(\ZZ) \}\ll P_{1}^{r+1-\lambda}\] (cf. d\'emonstration de \cite[Th\'eor\`eme 3.1]{Br})
\end{proof}

\subsection{M\'ethode du cercle}

On fixe \`a pr\'esent un r\'eel $ \theta\in [0,1] $. On suppose de plus que \begin{equation}
K_{2}>2(d_{2}-1).
\end{equation}On notera \begin{equation}
\phi(d,k,\theta)=(dk)^{d_{1}}P_{2}^{(d_{2}-1)\theta},
\end{equation} \begin{equation}
\Delta_{2}(\theta,K_{2})=\theta(K_{2}-2(d_{2}-1))
\end{equation}

Nous allons \`a pr\'esent, comme dans la section pr\'ec\'edente, s\'eparer l'int\'egrale sur $ [0,1] $ de $ S(\alpha) $ en int\'egrales sur les arcs majeurs et les arcs mineurs. Commen\c{c}ons par traiter le cas des arcs mineurs.

\begin{lemma}\label{arcsmin2}
Pour tout $ \xx\in \mathcal{A}_{2}^{\lambda}(\ZZ) $, on a la  majoration : \[ \int_{\alpha\notin \mathfrak{M}^{d,\xx}(\theta)}|S_{d,\xx}(\alpha)|d\alpha \ll (dk)^{d_{1}+m-r+\varepsilon}P_{2}^{n+1-r-d_{2}-\Delta_{2}(\theta,K_{2})+\varepsilon}. \] \end{lemma}

\begin{proof}
On consid\`ere une subdivision de l'intervalle $ [0,1] $  \[ 0<\theta=\theta_{0}<\theta_{1}<...<\theta_{T-1}<\theta_{T}=1 \] telle que \begin{equation}
2(\theta_{i+1}-\theta_{i})(d_{2}-1)<\varepsilon 
\end{equation}
et $ T\ll P_{2}^{\varepsilon} $ pour $ \varepsilon>0 $ arbitrairement petit (et $ P_{2} $ assez grand). Puisque $ \xx\in \mathcal{A}_{2}^{\lambda}(\ZZ) $, le lemme \ref{dilemme24} donne \begin{align*}
\int_{\alpha\notin \mathfrak{M}^{d,\xx}(\theta_{T})}|S_{d,\xx}(\alpha)|d\alpha &  \ll (dk)^{m-r+\varepsilon}P_{2}^{n+1-r-K_{2}\theta_{T}+\varepsilon} \\ & \ll (dk)^{m-r+\varepsilon}P_{2}^{n+1-r-d_{2}-\Delta_{2}(\theta,K_{2})+\varepsilon}.
\end{align*}
Par ailleurs, on remarque que \begin{align*}
\Vol\left(\mathfrak{M}^{d,\xx}(\theta)\right) & \ll \sum_{q\leqslant \phi(d,k,\theta)}\sum_{\substack{0\leqslant a <q \\ \PGCD(a,q)=1}}q^{-1}P_{2}^{-d_{2}+(d_{2}-1)\theta} \\ & \ll  (dk)^{d_{1}}P_{2}^{-d_{2}+2(d_{2}-1)\theta}.
\end{align*} 

On a alors pour tout $ i\in \{0,...,T-1\} $ : \begin{align*}
\int_{\alpha\in \mathfrak{M}^{d,\xx}(\theta_{i+1})\setminus  \mathfrak{M}^{d,\xx}(\theta_{i})}|S_{d,\xx}(\alpha)|d\alpha & \ll (dk)^{m-r+\varepsilon}P_{2}^{n+1-r-K_{2}\theta_{i}+\varepsilon}\Vol\left(\mathfrak{M}^{d,\xx}(\theta_{i+1})\right) \\  & \ll  (dk)^{m-r+d_{1}+\varepsilon}P_{2}^{n+1-r-K_{2}\theta_{i}+\varepsilon-d_{2}+2(d_{2}-1)\theta} 
\end{align*}

Or, \[ 2\theta_{i+1}(d_{2}-1)-K_{2}\theta_{i} =2(\theta_{i+1}-\theta_{i})(d_{2}-1)-\Delta_{2}(\theta_{i},K_{2})<\varepsilon-\Delta_{2}(\theta,K_{2}) \] et donc \[ \int_{\alpha\in \mathfrak{M}^{d,\xx}(\theta_{i+1})\setminus  \mathfrak{M}^{\xx}(\theta_{i}) }|S_{d,\xx}(\alpha)|d\alpha \ll  (dk)^{m-r+d_{1}+\varepsilon}P_{2}^{n+1-r-d_{2}-\Delta_{2}(\theta,K_{2})+\varepsilon} \] et  on obtient le r\'esultat souhait\'e en sommant sur les $ i\in \{0,...,T-1\} $. 
\end{proof}

On d\'efinit \`a pr\'esent la nouvelle famille d'arcs majeurs : 
\begin{equation}
\mathfrak{M}_{a,q}^{'d,\xx}(\theta)=\left\{ \alpha\in [0,1[ \; | \;  2|\alpha q-a|\leqslant qP_{2}^{-d_{2}}P^{(d_{2}-1)\theta}\right\},
\end{equation}
\begin{equation}
\mathfrak{M}^{'d,\xx}(\theta)=\bigcup_{q\leqslant (dk)^{d_{1}}P^{(d_{2}-1)\theta}}\bigcup_{\substack{0\leqslant a <q \\ \PGCD(a,q)=1 }}\mathfrak{M}_{a,q}^{'d,\xx}(\theta).
\end{equation}
\begin{lemma}\label{disjoint2} Si l'on suppose $ (dk)^{2d_{1}}P_{2}^{-d_{2}+3\theta(d_{2}-1)}<1 $, alors les arcs majeurs $ \mathfrak{M}_{a,q}^{\xx'}(\theta) $ sont disjoints deux \`a deux. \end{lemma}
\begin{proof}
Supposons qu'il existe $ \alpha\in\mathfrak{M}_{a,q}^{'d,\xx}(\theta)\cap\mathfrak{M}_{a',q'}^{'d,\xx}(\theta)  $ pour $ (a,q)\neq (a',q') $, $ q,q'\leqslant \phi(d,k,\theta) $, $ 0\leqslant a<q $, $ 0\leqslant a'<q' $ et $ \PGCD(a,q)=\PGCD(a',q')=1 $. On a alors \[ \frac{1}{qq'}\leqslant \left| \frac{a}{q}-\frac{a'}{q'}\right| \leqslant P_{2}^{-d_{2}+\theta(d_{2}-1)} \] et donc \[ 1 \leqslant qq'P_{2}^{-d_{2}+\theta(d_{2}-1)} \leqslant  (dk)^{2d_{1}}P_{2}^{-d_{2}+3\theta(d_{2}-1)}, \]
d'o\`u le r\'esultat. 
\end{proof}

Remarquons que puisque $ \mathfrak{M}^{d,\xx}(\theta)\subset  \mathfrak{M}^{'d,\xx}(\theta) $, d'apr\`es le lemme \ref{arcsmin2}, on a le r\'esultat suivant :
\begin{lemma}\label{arcmaj2}
Soit $ \xx\in \mathcal{A}_{2}^{\lambda}(\ZZ)  $, on a alors que : \begin{multline*} N_{d,\xx}(P_{2})=\sum_{q\leqslant \phi(d,k,\theta) }\sum_{\substack{0\leqslant a <q \\ \PGCD(a,q)=1}}\int_{\alpha\in \mathfrak{M}_{a,q}^{'d,\xx}(\theta)}S(d,\alpha)d\alpha \\ + O\left( (dk)^{d_{1}+m-r+\varepsilon}P_{2}^{n+1-r-d_{2}-\Delta_{2}(\theta,K_{2})+\varepsilon} \right). 
\end{multline*}
\end{lemma}

On consid\`ere \`a pr\'esent $ \xx\in \ZZ^{r+1} $ quelconque, et on suppose $ \alpha\in  \mathfrak{M}_{a,q}^{'d,\xx}(\theta) $. On pose $ \beta=\alpha-\frac{a}{q} $ et donc $ |\beta|\leqslant \frac{1}{2}P_{2}^{-d_{2}+(d_{2}-1)\theta} $. On a alors le lemme ci-dessous : 

\begin{lemma}\label{separation2}
On a l'estimation \begin{multline*}
S_{d,\xx}(\alpha)=(dk)^{m-r}P_{2}^{n-r+1}q^{-(n-r+1)}S_{a,q,d}(\xx)I_{\xx}(d^{d_{1}}P_{2}^{d_{2}}\beta) \\ +O\left( (dk)^{2d_{1}+m-r}P_{2}^{n-r+2\theta(d_{2}-1)} \right), \end{multline*} avec \begin{equation}
S_{a,q,d}(\xx)=\sum_{(\bb_{2},\bb_{3})\in (\ZZ/q\ZZ)^{m-r}\times (\ZZ/q\ZZ)^{n-m+1}}e\left(\frac{a}{q}F(d\xx,\bb_{2},\bb_{3})\right),
\end{equation} \begin{equation}
I_{\xx}(\beta)=\int_{(\vv,\ww)\in [-1,1]^{m-r}\times [-1,1]^{n-m+1}}e\left(\beta F(\xx,k\vv,\ww)\right)d\vv d\ww. 
\end{equation}
\end{lemma}

\begin{proof}
Lorsque $ P_{2}>q $, l'\'egalit\'e est trivialement v\'erifi\'ee. En effet, dans ce cas, on observe que : \begin{multline*}
|S_{d,\xx}(\alpha)|\ll (dk)^{m-r}P_{2}^{n-r+1} \ll (dk)^{m-r}P_{2}^{n-r}q  \\ \ll (dk)^{2d_{1}+m-r}P_{2}^{n-r+2\theta(d_{2}-1)},
\end{multline*} et 
\begin{multline*}
(dk)^{m-r}P_{2}^{n-r+1}q^{-(n-r+1)}|S_{a,q,d}(\xx)||I_{\xx}(d^{d_{1}}P_{2}^{d_{2}}\beta)| \\ \ll (dk)^{m-r}P_{2}^{n-r+1} \ll (dk)^{2d_{1}+m-r}P_{2}^{n-r+2\theta(d_{2}-1)},
\end{multline*}
d'o\`u le r\'esultat. Nous supposerons donc dor\'enavant que $ q<P_{2} $. On peut \'ecrire \begin{equation}\label{S32}
S_{d,\xx}(\alpha)=\sum_{(\bb_{2},\bb_{3})\in (\ZZ/q\ZZ)^{m-r}\times (\ZZ/q\ZZ)^{n-m+1}}e\left(\frac{a}{q}F(d\xx,\bb_{2},\bb_{3})\right)S_{3}(\bb_{2},\bb_{3}),
\end{equation}  o\`u \[ S_{3}(\bb_{2},\bb_{3})=\sum_{\substack{\yy\equiv \bb_{2} (q) \\ |\yy|\leqslant dkP_{2} }}\sum_{\substack{\zz\equiv \bb_{3} (q) \\ |\zz|\leqslant P_{2}}}e\left(\beta F(d\xx,\yy,\yy)\right).  \]

Si l'on consid\`ere  $ q\yy'+\bb_{2}, q\yy''+\bb_{2} \in [-dkP_{2},dkP_{2}]$ et $ q\zz'+\bb_{3}, q\zz''+\bb_{3} \in [-P_{2},P_{2}]$ avec \[ |\yy'-\yy''|\ll 1, \; \; |\zz'-\zz''|\ll 1, \]
on a \[ \left| F(d\xx,q\yy'+\bb_{2},q\zz'+\bb_{3})-F(d\xx,q\yy''+\bb_{2},q\zz''+\bb_{3})\right| \ll q(dk)^{d_{1}}P_{2}^{d_{2}-1}. \] 
Ainsi : \begin{multline*} S_{3}(\bb_{2},\bb_{3})=\int_{\substack{q\tilde{\vv}\in [-dkP_{2},dkP_{2}]^{m-r} \\ q\tilde{\ww} \in [-P_{2},P_{2}]^{n-m+1}} }e\left( \beta F(d\xx,q\tilde{\vv}, q\tilde{\ww})\right) d\tilde{\vv}d\tilde{\ww} \\ + O\left( q|\beta| (dk)^{d_{1}}P_{2}^{d_{2}-1}\left(\frac{dkP_{2}}{q}\right)^{m-r}\left(\frac{P_{2}}{q}\right)^{n-m+1}\right) \\ + O\left(\left(\frac{dkP_{2}}{q}\right)^{m-r}\left(\frac{P_{2}}{q}\right)^{n-m}\right). \end{multline*}
En rappelant que $ |\beta|\leqslant \frac{1}{2}P_{2}^{-d_{2}+(d_{2}-1)\theta} $, $ q\leqslant \phi(d,k,\theta)=(dk)^{d_{1}}P^{(d_{2}-1)\theta} $ et en consid\'erant le changement de variables $ q\tilde{\vv}=dkP_{2}\vv $, $ q\tilde{\ww}=P_{2}\ww $ on trouve \begin{multline*}
(dk)^{m-r}P_{2}^{n-r+1}q^{-(n-r+1)}I_{\xx}(d^{d_{1}}P_{2}^{d_{2}}\beta) \\ + O\left(q^{-(n-r)}(dk)^{m-r+d_{1}}P_{2}^{n-r+(d_{2}-1)\theta}\right). \end{multline*}
En rempla\c{c}ant $ S_{3} $ par cette nouvelle expression dans \eqref{S32}, on obtient le r\'esultat. 
\end{proof}

On pose dor\'enavant \begin{equation}
\tilde{\phi}(P_{2},\theta)=\frac{1}{2}P_{2}^{\theta(d_{2}-1)},
\end{equation}\begin{equation}
\eta(\theta)=1-5\theta(d_{2}-1).
\end{equation}
 \begin{lemma}\label{intermediairex} Pour $ \xx\in \mathcal{A}_{2}^{\lambda}(\ZZ) $, et $ \varepsilon>0 $ arbitrairement petit, on a l'estimation suivante : \begin{multline*}
N_{d,\xx}(P_{2})=(dk)^{m-r}P_{2}^{n-r+1-d_{2}}\mathfrak{S}_{d,\xx}(\phi(d,k,\theta))J_{d,\xx}(\tilde{\phi}(P_{2},\theta)) \\ + O\left((dk)^{d_{1}+m-r+\varepsilon}P_{2}^{n-r+1-d_{2}-\Delta_{2}(\theta,K_{2})+\varepsilon}\right)+ O\left((dk)^{4d_{1}+m-r}P_{2}^{n-r+1-d_{2}-\eta(\theta)}\right), \end{multline*} o\`u \begin{equation}
\mathfrak{S}_{d,\xx}(\phi(d,k,\theta))=\sum_{q\leqslant \phi(d,k,\theta)}q^{-(n-r+1)}\sum_{\substack{0\leqslant a < q \\ \PGCD(a,q)=1}}S_{a,q,d}(\xx),
\end{equation}\begin{equation}
J_{d,\xx}(\tilde{\phi}(P_{2},\theta))=\int_{|\beta|\leqslant \tilde{\phi}(\theta)}I_{\xx}(d^{d_{1}}\beta)d\beta. 
\end{equation}
\end{lemma}
\begin{proof}
On notera \[ E_{1}=(dk)^{d_{1}+m-r+\varepsilon}P_{2}^{n-r+1-d_{2}-\Delta_{2}(\theta,K_{2})+\varepsilon}, \] \[ E_{2}=(dk)^{2d_{1}+m-r}P_{2}^{n-r+2\theta(d_{2}-1)}\Vol\left(\mathfrak{M}^{'d,\xx}(\theta)\right). \]  D'apr\`es les lemmes \ref{separation2} et \ref{arcmaj2}, on a : \begin{multline*}
N_{d,\xx}(P_{2})=(dk)^{m-r}P_{2}^{n-r+1}\sum_{q\leqslant \phi(d,k,\theta)}q^{-(n-r+1)}\sum_{\substack{0\leqslant a < q \\ \PGCD(a,q)=1}}S_{a,q,d}(\xx) \\ \int_{|\beta|\leqslant P_{2}^{-d_{2}}\tilde{\phi}(P_{2},\theta)}I_{\xx}(d^{d_{1}}P_{2}^{d_{2}}\beta)d\beta
+O(E_{1})+O(E_{2}).
\end{multline*}
Par un changement de variable, on a \begin{align*} \int_{|\beta|\leqslant P_{2}^{-d_{2}}\tilde{\phi}(P_{2},\theta)}I_{\xx}(d^{d_{1}}P_{2}^{d_{2}}\beta)d\beta & =P_{2}^{-d_{2}}\int_{|\beta|\leqslant \tilde{\phi}(P_{2},\theta)}I_{\xx}(d^{d_{1}}\beta)d\beta \\ & =P_{2}^{-d_{2}}J_{d,\xx}(\tilde{\phi}(P_{2},\theta)). \end{align*}
On remarque par ailleurs que \begin{align*}
\Vol\left(\mathfrak{M}^{'\xx}(\theta)\right) & \ll \sum_{q\leqslant \phi(d,k,\theta)}\sum_{\substack{0\leqslant a < q \\ \PGCD(a,q)=1}}P_{2}^{-d_{2}+(d_{2}-1)\theta}  \\ & \ll (dk)^{2d_{1}}P_{2}^{-d_{2}+3(d_{2}-1)\theta},
\end{align*}
et donc \[ E_{2}\ll (dk)^{4d_{1}+m-r}P_{2}^{n-r-d_{2}+5\theta(d_{2}-1)}=(dk)^{4d_{1}+m-r}P_{2}^{n-r+1-d_{2}-\eta(\theta)}, \] ce qui cl\^ot la d\'emonstration du lemme.
\end{proof}
Par la suite, on pose
\begin{equation}
\mathfrak{S}_{d,\xx}=\sum_{q=1}^{\infty}q^{-(n-r+1)}\sum_{\substack{0\leqslant a < q \\ \PGCD(a,q)=1}}S_{a,q}(\xx),
\end{equation}
\begin{equation}
J_{d,\xx}=\int_{\RR}I_{\xx}(d^{d_{1}}\beta)d\beta.
\end{equation}
\begin{lemma}\label{Jx}
Soit $ \xx\in \mathcal{A}_{2}^{\lambda}(\ZZ) $, et $ \varepsilon>0 $ arbitrairement petit. On suppose par ailleurs que $ d_{2}\geqslant 2 $. L'int\'egrale $ J_{d,\xx} $ est absolument convergente, et on a : \[ |J_{d,\xx}(\tilde{\phi}(P_{2},\theta))-J_{d,\xx}|\ll P_{2}^{\theta((d_{2}-1)-K_{2})}\max\left\{P_{2}^{\varepsilon},(dk)^{\varepsilon}\right\}. \] On a de plus $ |J_{d,\xx}|\ll (dk)^{\varepsilon}. $
\end{lemma}
\begin{proof}
On consid\`ere $ \beta $ tel que $ |\beta|\geqslant \tilde{\phi}(\theta) $. On choisit alors des param\`etres $ P $ et $ \theta' $ tels que \begin{equation}
|\beta|=\frac{1}{2}P^{\theta'(d_{2}-1)}, 
\end{equation} \begin{equation}\label{deuxegal}
 P^{-K_{2}\theta'}=P^{-1+2\theta'(d_{2}-1)}(dk)^{2d_{1}}.
\end{equation} Remarquons que ces deux \'egalit\'es impliquent \begin{equation}
\theta'=\frac{\log(2|\beta|)}{(d_{2}-1)\left(\left(2+\frac{K_{2}}{(d_{2}-1)}\right)\log(2|\beta|)+2d_{1}\log(dk)\right)}
\end{equation}
donc en particulier \begin{equation}
\theta'\gg \min\left\{1,\frac{\log(2|\beta|)}{\log(dk)}\right\}.
\end{equation} Par ailleurs, l'\'egalit\'e \eqref{deuxegal} implique \[ P^{-2+4\theta'(d_{2}-1)}(dk)^{4d_{1}}<1, \]donc, pour $ d_{2}\geqslant 2 $, \[  P^{-d_{2}+3\theta'(d_{2}-1)}(dk)^{2d_{1}}<1, \] et ainsi, d'apr\`es le lemme \ref{disjoint2}, les arcs majeurs $ \mathfrak{M}_{a,q}^{d,\xx}(\theta') $ correspondant \`a $ P $ et $ \theta' $ sont disjoints deux \`a deux. Le r\'eel $ P^{-d_{2}}\beta $ appartient au bord de $ \mathfrak{M}_{0,1}(\theta') $, et donc par le lemme \ref{dilemme24}, on a l'estimation : 

\[ |S_{d,\xx}(P^{-d_{2}}\beta )|\ll (dk)^{m-r}P^{n-r+1-K_{2}\theta'+\varepsilon}. \] 
D'autre part, le lemme \ref{separation2} donne : \[ S_{d,\xx}(P^{-d_{2}}\beta )=(dk)^{m-r}P^{n-r+1}I(d^{d_{1}}\beta)+O\left((dk)^{m-r+2d_{1}}P^{n-r+2\theta'(d_{2}-1)} \right). \] On a ainsi : \begin{align*} |I(d^{d_{1}}\beta)| & \ll P^{-K_{2}\theta'+\varepsilon}+ (dk)^{2d_{1}}P^{-1+2\theta'(d_{2}-1)} 
\\ & \ll P^{-K_{2}\theta'+\varepsilon} \\ & \ll |\beta|^{-\frac{K_{2}}{(d_{2}-1)}+\frac{\varepsilon}{\theta'(d_{2}-1)}}. \end{align*}
Remarquons que, puisque $ \theta'\gg \min\left\{1,\frac{\log(2|\beta|)}{\log(dk)}\right\} $, \[ |\beta|^{\frac{\varepsilon}{\theta'(d_{2}-1)}}\ll \max\{|\beta|^{\varepsilon'}, (dk)^{\varepsilon'}\}, \] pour $ \varepsilon'>0 $ arbitrairement petit.
On a donc 
\begin{align*} |J_{d,\xx}(\tilde{\phi}(\theta))-J_{d,\xx}| & \ll \int_{|\beta|>\tilde{\phi(\theta)}}|\beta|^{-\frac{K_{2}}{(d_{2}-1)}} \max\{|\beta|^{\varepsilon'}, (dk)^{\varepsilon'}\}d\beta \\ & \ll \tilde{\phi}(\theta)^{1-\frac{K_{2}}{(d_{2}-1)}} \max\{ \tilde{\phi}(\theta)^{\varepsilon'}, (dk)^{\varepsilon'}\} \\ & \ll P_{2}^{\theta(d_{2}-1-K_{2})}\max\{ P_{2}^{\varepsilon''}, (dk)^{\varepsilon''}\} , \end{align*} avec $ \varepsilon'' $ arbitrairement petit.

D'autre part, en choisissant $ P_{2} \ll 1   $, cette in\'egalit\'e donne \[ |J_{d,\xx}(\tilde{\phi}(\theta))-J_{d,\xx}| \ll (dk)^{\varepsilon''}, \] et puisque $ |J_{d,\xx}(\tilde{\phi}(\theta))|\ll 1 $ lorsque $ P_{2} \ll 1   $, on a imm\'ediatement \[ |J_{d,\xx}| \ll (dk)^{\varepsilon''} \]

\end{proof}

\begin{lemma}\label{Sx}

Soit $ \xx\in \mathcal{A}_{2}^{\lambda}(\ZZ) $, et $ \varepsilon>0 $ arbitrairement petit. On suppose par ailleurs que $ d_{2}\geqslant 2 $. La s\'erie $ \mathfrak{S}_{\xx} $ est absolument convergente, et on a : \[ |\mathfrak{S}_{d,\xx}(\phi(d,k,\theta))-\mathfrak{S}_{d,\xx}|\ll (dk)^{2d_{1}+\varepsilon}P_{2}^{\theta(2(d_{2}-1)-K_{2})}. \] On a de plus $ |\mathfrak{S}_{d,\xx}|\ll (dk)^{2d_{1}+\varepsilon}. $
\end{lemma}

Pour d\'emontrer ce lemme on introduit pour $ \xx $ fix\'e et $ P\geqslant 1 $ la nouvelle s\'erie g\'en\'eratrice : 
\[ S_{d,\xx}'(\alpha)=\sum_{|\yy|\leqslant P}\sum_{|\zz|\leqslant P}e\left(\alpha F(d\xx,\yy,\zz)\right). \] De la m\^eme mani\`ere que pour le lemme \ref{dilemme24}, on \'etablit : 

\begin{lemma}\label{dilemme24bis}
Si $ \varepsilon>0 $ est un r\'eel arbitrairement petit, l'une au moins des assertions suivantes est v\'erifi\'ee : \begin{enumerate}
\item $ |S_{d,\xx}'(\alpha)|\ll P^{n+1-r+\varepsilon-K_{2}\theta}, $ 
\item le r\'eel $ \alpha $ appartient \`a $ \mathfrak{M}^{d,\xx}(\theta) $.
\end{enumerate}
\end{lemma}
\begin{proof}[D\'emonstration du lemme \ref{Sx}]
Soit $ q>\phi(d,k,\theta) $ et $ \alpha=\frac{a}{q} $ avec $ 0\leqslant a<q $ et $ \PGCD(a,q)=1 $. On a alors que $ S_{a,q,d}(\xx)=S_{d,\xx}'(\alpha) $ avec $ P=q $. On consid\`ere $ \theta' $ tel que $ q=(dk)^{d_{1}} q^{(d_{2}-1)\theta'} $. Si $ \theta''=\theta'-\nu $ pour $ \nu>0 $ arbitrairement petit, on a alors que $ \alpha\notin  \mathfrak{M}^{d,\xx}(\theta'') $. En effet s'il existait $ a',q'\in \ZZ $ tels que $ 0\leqslant a'<q' $, $ \PGCD(a',q')=1 $, $ q'\leqslant (dk)^{d_{1}}q^{\theta''(d_{2}-1)} <q $ et $ \alpha \in \mathfrak{M}_{a',q'}^{d,\xx}(\theta'') $, on aurait alors \[ 1 \leqslant |aq'-a'q|< q^{1-d_{2}+\theta'(d_{2}-1)}, \] ce qui est absurde pour $ d_{2}\geqslant 2 $. On a donc, d'apr\`es le lemme pr\'ec\'edent : \[ |S_{a,q,d}(\xx)|\ll q^{n+1-r+\varepsilon-K_{2}\theta'}. \] Par cons\'equent \begin{align*}
\left| \mathfrak{S}_{d,\xx}(\phi(d,k,\theta))-\mathfrak{S}_{d,\xx}\right| & \ll \sum_{q>\phi(d,k,\theta)}q^{-(n-r+1)}\sum_{0\leqslant a <q}|S_{a,q,d}(\xx)| \\ & \ll \sum_{q>\phi(d,k,\theta)}q^{-(n-r+1)}\sum_{0\leqslant a <q}q^{n+1-r+\varepsilon-K_{2}\theta'} \\ & \ll 
\sum_{q>\phi(d,k,\theta)}q^{-\frac{K_{2}}{(d_{2}-1)}+1+\varepsilon}(dk)^{\frac{d_{1}K_{2}}{(d_{2}-1)}}\\ & \ll (dk)^{\frac{d_{1}K_{2}}{(d_{2}-1)}}\phi(d,k,\theta)^{-\frac{K_{2}}{(d_{2}-1)}+2+\varepsilon} \\ & \ll (dk)^{2d_{1}+\varepsilon}P_{2}^{\theta(2(d_{2}-1)-K_{2})+\varepsilon}
\end{align*}
Par ailleurs, en prenant $ P_{2}\ll 1 $ cette majoration donne : \[ \left| \mathfrak{S}_{d,\xx}(\phi(d,k,\theta))-\mathfrak{S}_{d,\xx}\right|  \ll (dk)^{2d_{1}+\varepsilon}, \] et en consid\'erant la majoration triviale $ | \mathfrak{S}_{d,\xx}(\phi(d,k,\theta))|\ll (dk)^{2d_{1}} $, on trouve finalement \[ |\mathfrak{S}_{d,\xx}|\ll (dk)^{2d_{1}+\varepsilon}. \]

\end{proof}

On d\'eduit des lemmes \ref{Sx} et \ref{Jx} le r\'esultat suivant : 

\begin{lemma}\label{lemmex}
Soit  $ \xx\in \mathcal{A}_{2}^{\lambda}(\ZZ) $. On suppose fix\'es $ \theta\in [0,1] $ et $ P_{2}\geqslant 1  $ tels que $ (dk)^{2d_{1}}P_{2}^{-d_{2}+3\theta(d_{2}-1)}<1 $. On suppose de plus que $ K_{2}>2(d_{2}-1) $. On a alors que \[ N_{d,\xx}(P_{2})=\mathfrak{S}_{d,\xx}J_{d,\xx}(dk)^{m-r}P_{2}^{n-r+1-d_{2}}+ O(E_{2}) +O(E_{3}), \] avec \[ E_{2}=(dk)^{4d_{1}+m-r}P^{n-r+1-d_{2}-\eta(\theta)}, \]\[ E_{3}=(dk)^{2d_{1}+m-r+\varepsilon}P_{2}^{n-r+1-d_{2}-\Delta_{2}(\theta,K_{2})+\varepsilon} \]
et $ \varepsilon>0 $ arbitrairement petit. 
\end{lemma}
\begin{proof}
Nous avons d\'ej\`a vu avec le lemme \ref{intermediairex}

\begin{multline*}
N_{d,\xx}(P_{2})=(dk)^{m-r}P_{2}^{n-r+1-d_{2}}\mathfrak{S}_{d,\xx}(\phi(d,k,\theta))J_{d,\xx}(\tilde{\phi}(\theta))+O(E_{1})+O(E_{2}), \end{multline*} 
o\`u \[ E_{1}=(dk)^{d_{1}+m-r+\varepsilon}P_{2}^{n-r+1-d_{2}-\Delta_{2}(\theta,K_{2})+\varepsilon} \ll E_{3}. \]
Par ailleurs, d'apr\`es les lemmes \ref{Sx} et \ref{Jx}, on a \begin{multline*}
\left|\mathfrak{S}_{d,\xx}(\phi(d,k,\theta))J_{d,\xx}(\tilde{\phi}(\theta))-\mathfrak{S}_{d,\xx}J_{d,\xx}\right| \\ \leqslant \left|\mathfrak{S}_{d,\xx}(\phi(d,k,\theta))-\mathfrak{S}_{d,\xx}\right|\left|J_{d,\xx}\right|+ \left|\mathfrak{S}_{d,\xx}(\phi(d,k,\theta)) \right| \left|  J_{d,\xx}(\tilde{\phi}(\theta))-J_{d,\xx}\right| \\ \ll (dk)^{2d_{1}+2\varepsilon}P_{2}^{\theta(2(d_{2}-1)-K_{2})}+ (dk)^{2d_{1}+\varepsilon}P_{2}^{\theta((d_{2}-1)-K_{2})}\max\left\{P_{2}^{\varepsilon},(dk)^{\varepsilon}\right\} \end{multline*} et en multlipliant cette in\'egalit\'e par $ (dk)^{m-r}P_{2}^{n-r+1-d_{2}} $ on obtient un terme d'erreur  \[ (dk)^{2d_{1}+m-r+\varepsilon}P_{2}^{n-r+1-d_{2}-\Delta_{2}(\theta,K_{2})+\varepsilon}=E_{3},
\]
d'o\`u le r\'esultat. 
\end{proof}

En fixant $ \theta>0 $ tel que $ \theta<\frac{1}{5(d_{2}-1)} $ (de sorte que $ \eta(\theta)>0 $), on obtient le corollaire suivant : 

\begin{cor}\label{corollairex}
Soit  $ \xx\in \mathcal{A}_{2}^{\lambda}(\ZZ) $. On suppose que $ K_{2}>2(d_{2}-1) $. Il existe alors un r\'eel $ \delta>0 $ arbitrairement petit tel que : \[ N_{d,\xx}(P_{2})=\mathfrak{S}_{d,\xx}J_{d,\xx}(dk)^{m-r}P_{2}^{n-r+1-d_{2}}+O\left( (dk)^{m-r+4d_{1}}P_{2}^{n-r+1-d_{2}-\delta}\right), \] uniform\'ement pour tout $ k<d^{-1}P_{2}^{\frac{d_{2}-1}{2d_{1}}} $. 
\end{cor}
\begin{rem}
La condition d'uniformit\'e  $ k<d^{-1}P_{2}^{\frac{d_{2}-1}{2d_{1}}} $ d\'ecoule de la condition $ (dk)^{2d_{1}}P_{2}^{-d_{2}+3\theta(d_{2}-1)}<1 $ du lemme\;\ref{lemmex}.
\end{rem}
Dans ce qui va suivre, pour $ P_{2}=P_{1}^{u} $, avec $ u\geqslant 1 $ on introduit la fonction \begin{equation}
g_{2}(u,\delta)=\left(1-\frac{5d_{1}}{u}-\delta\right)^{-1}5(d_{2}-1)\left(\frac{3d_{1}}{u}+2\delta\right),
\end{equation} ainsi que \begin{multline}
N_{d,2}(P_{1},P_{2})=\Card\left\{ (\xx,\yy,\zz)\in \ZZ^{n+2} \; | \; \xx \in \mathcal{A}_{2}^{\lambda}(\ZZ), \; |\xx|\leqslant P_{1}, \;\right. \\ \left. |\yy| \leqslant d|\xx|P_{2},\;   |\zz|\leqslant P_{2}, \; F(d\xx,\yy,\zz)=0 \right\}.
\end{multline}
On a alors le r\'esultat ci-dessous : 
\begin{prop}\label{propx}
On suppose $ K_{2}>2(d_{2}-1) $, $ d_{2}\geqslant 2 $, $ P_{2}=P_{1}^{u} $ avec $ u $ suppos\'e strictement sup\'erieur \`a $ 5d_{1} $, et de plus que \[ K_{2}-2(d_{2}-1)>g_{2}(u,\delta). \] Alors : \begin{multline*} N_{d,2}(P_{1},P_{2})=\left(\sum_{\xx\in P_{1}\BB_{1}\cap \mathcal{A}_{2}^{\lambda}(\ZZ) }\mathfrak{S}_{d,\xx}J_{d,\xx}d^{m-r}|\xx|^{m-r}\right)P_{2}^{n-r+1-d_{2}} \\ + O\left(d^{m-r+5d_{1}}P_{1}^{m+1-d_{1}}P_{2}^{n-r+1-d_{2}-\delta}\right), \end{multline*} pour $ \delta>0 $ arbitrairement petit. 
\end{prop}
\begin{proof}
Commen\c{c}ons par d\'emontrer que la propri\'et\'e est triviale lorsque $ d^{2d_{1}}>PP_{2}^{-\frac{3}{5}-\delta} $. Dans ce cas on observe que d'une part (en utilisant les lemmes\;\ref{Sx} et\;\ref{Jx}) : \begin{multline*} \left(\sum_{\xx\in P_{1}\BB_{1}\cap \mathcal{A}_{2}^{\lambda}(\ZZ) }\mathfrak{S}_{d,\xx}J_{d,\xx}d^{m-r}|\xx|^{m-r}\right)P_{2}^{n-r+1-d_{2}} \\ \ll d^{2d_{1}+2\varepsilon}P_{1}^{m+1}P_{2}^{n-r+1-d_{2}} \\  \ll d^{m-r+5d_{1}}P_{1}^{m+1}P_{2}^{n-r+1-d_{2}}\underbrace{d^{-2d_{1}}}_{<P^{-1}P_{2}^{\frac{3}{5}+\delta}} \\ \ll d^{m-r+5d_{1}}P_{1}^{m+1-d_{1}}P_{2}^{n-r+1-d_{2}-\delta} \end{multline*}
et d'autre part  : \begin{align*}
N_{d,2}(P_{1},P_{2}) & \ll d^{m-r}P_{1}^{m+1}P_{2}^{n-r+1} \\ & \ll d^{m-r+5d_{1}}P_{1}^{m+1}P_{2}^{n-r+1}P^{-\frac{5}{2}}P_{2}^{\frac{3}{2}+\frac{5\delta}{2}} \\ & \ll d^{m-r+5d_{1}}P_{1}^{m+1-d_{1}}P_{2}^{n-r+1-d_{2}-\delta}.
\end{align*}
L'\'egalit\'e du lemme est donc trivialement v\'erifi\'ee. \\

Supposons \`a pr\'esent que $ d^{2d_{1}}\leqslant PP_{2}^{-\frac{3}{5}-\delta} $. Si $ \theta $ est tel que \begin{equation}\label{condsupp} d^{2d_{1}}P_{1}^{2d_{1}} P_{2}^{-d_{2}+3\theta(d_{2}-1)}< 1, \end{equation} alors d'apr\`es le lemme \ref{lemmex} : \[ N_{d,2}(P_{1},P_{2})=\left(\sum_{\xx\in P_{1}\BB_{1}\cap \mathcal{A}_{2}^{\lambda}(\ZZ) }\mathfrak{S}_{d,\xx}J_{d,\xx}d^{m-r}|\xx|^{m-r}\right)P_{2}^{n-r+1-d_{2}}+ O(\mathcal{E}_{2})+O(\mathcal{E}_{3} ), \] avec \begin{align*}
\mathcal{E}_{2} & =\sum_{\substack{\xx\in P_{1}\BB_{1}\cap \mathcal{A}_{2}^{\lambda}(\ZZ)} }(d|\xx|)^{4d_{1}+m-r}P_{2}^{n-r+1-d_{2}-\eta(\theta)} \\ & \ll d^{4d_{1}+m-r}P_{1}^{4d_{1}+m+1}P_{2}^{n-r+1-d_{2}-\eta(\theta)} \\  & =  d^{4d_{1}+m-r}P_{1}^{m+1-d_{1}}P_{2}^{n-r+1+\frac{5d_{1}}{u}-d_{2}-\eta(\theta)}.
\end{align*}
et \begin{align*}
\mathcal{E}_{3} & =\sum_{\substack{\xx\in P_{1}\BB_{1}\cap \mathcal{A}_{2}^{\lambda}(\ZZ)\\ |\xx|=k} }(dk)^{2d_{1}+m-r+\varepsilon}P_{2}^{n-r+1-d_{2}-\Delta_{2}(\theta,K_{2})+\varepsilon}\\ & \ll d^{2d_{1}+m-r+\varepsilon}P_{1}^{m+1-d_{1}}P_{2}^{n-r+1-d_{2}+\frac{3d_{1}}{u}-\Delta_{2}(\theta,K_{2})+2\varepsilon}.
\end{align*}
Rappelons que $ \eta(\theta)=1-5\theta(d_{2}-1) $. On choisit alors \[ \theta=\frac{1}{5(d_{2}-1)}\left(1-\frac{5d_{1}}{u}-\delta\right), \] de sorte que d'une part, en utilisant l'in\'egalit\'e $ d^{2d_{1}}\leqslant PP_{2}^{-\frac{3}{5}+\delta} $ on a  :  \[ d^{2d_{2}}P_{1}^{2d_{1}} P_{2}^{-d_{2}+3\theta(d_{2}-1)}= d^{2d_{2}}P_{1}^{2d_{1}} P_{2}^{-d_{2}+\frac{3}{5}(1-\frac{5d_{1}}{u})}=d^{2d_{2}}P^{-1}P_{2}^{\frac{3}{5}}\leqslant P^{-\delta},  \] la condition\;\eqref{condsupp} est donc satisfaite, et on a de plus \[ \frac{5d_{1}}{u}-\eta(\theta)=-\delta. \]  On a par ailleurs, puisque $ K_{2}-2(d_{2}-1)>g_{2}(u,\delta) $, \[ K_{2}-2(d_{2}-1)>\theta^{-1}\left(\frac{3d_{1}}{u}+2\delta\right), \] et donc \[ \Delta_{2}(\theta,K_{2})-\frac{3d_{1}}{u}>2\delta, \] d'o\`u le r\'esultat. 
\end{proof}

\subsection{Le cas $ d_{2}=1 $}

Lorsque $ d_{2}=1 $, on peut obtenir des r\'esultats semblables \`a ceux des propositions \ref{corollairex} et \ref{propx} en utilisant des r\'esultats de g\'eom\'etrie des r\'eseaux. On introduit la d\'efinition suivante issue de \cite[Definition 2.1]{Wi} :
\begin{Def}
Soit $ S $ un sous-ensemble de $ \RR^{n} $, et soit $ c $ un entier tel que $ 0\leqslant c \leqslant d $. Pour $ M\in \NN $ et $ L>0 $, on dit que \emph{$ S $ appartient \`a $ \Lip(n,c,M,L) $} s'il existe $ M $ applications $ \phi : [0,1]^{n-c}\ra \RR^{d} $ v\'erifiant : \[ ||\phi(\xx)-\phi(\yy)||_{2}\leqslant L||\xx-\yy||_{2}, \] $ ||.||_{2} $ d\'esignant la norme euclidienne, telles que $ S $ soit recouvert par les images de ces applications. 
\end{Def}

On a le r\'esultat suivant (cf. \cite[Lemme 2]{MV}) : 
\begin{lemma}\label{geomnomb}
Soit $ S\subset\RR^{n} $ un ensemble bord\'e dont le bord $ \partial S $ appartient \`a $ \Lip(n,1,M,L) $. L'ensemble $ S $ est alors mesurable et si $ \Lambda $ est un r\'eseau de $ \RR^{n} $ de premier minimum successif $ \lambda_{1} $, on a \[ \left| \card(S\cap\Lambda)-\frac{\Vol(S)}{\det(\Lambda)}\right|\leqslant c(n)M\left(\frac{L}{\lambda_{1}}+1\right)^{n-1}, \] o\`u $ c(n) $ est une constante ne d\'ependant que de $ n $.
\end{lemma} 
Soit $ \xx\in \mathcal{A}_{2}^{\lambda}(\ZZ) $ fix\'e de norme $ |\xx|=k $. Puisque $ d_{2}=1 $ le polyn\^ome $ F(d\xx,\yy,\zz) $ est une forme lin\'eaire en $ (\yy,\zz) $ que l'on peut r\'e\'ecrire \[ F(d\xx,\yy,\zz)=\sum_{j=r+1}^{m}A_{j}(d\xx)y_{j}+\sum_{j=m+1}^{n+1}B_{j}(d\xx)z_{j}, \] avec $ A_{j}(d\xx) $ ou $ B_{j}(d\xx) $ non tous nuls (car $ \xx\in \mathcal{A}_{2}^{\lambda}(\ZZ) $). On note alors $ H_{d,\xx} $ l'hyperplan de $ \RR^{n-r+1} $ d\'efini par \[ F(d\xx,\yy,\zz)=0 .\] On note par ailleurs $ C_{d,\xx} $ le corps convexe $ \BB_{d,\xx}\cap H_{d,\xx} $ o\`u \[ \BB_{d,\xx}=\left\{(\yy,\zz) \; | \;|\yy|\leqslant dk, \; |\zz|\leqslant 1 \right\},\]    et $ \Lambda_{d,\xx} $ le r\'eseau $ \ZZ^{n-r+1}\cap H_{d,\xx}  $. Nous allons appliquer le lemme \ref{geomnomb} \`a $ S=P_{2}C_{d,\xx} $ et $ \Lambda=\Lambda_{d,\xx} $ vus respectivement comme un sous-ensemble et un r\'eseau de $ H_{d,\xx} $ que l'on identifiera \`a $ \RR^{n-r+1} $. Nous allons pour cela montrer que $  \partial C_{d,\xx}\in \Lip(n-r,1,2^{n-r+1}(dk)^{m-r},(n-r-1)\sqrt{n-r+1}) $. Une face du polytope $ C_{d,\xx} $ est obtenue en prenant l'intersection d'une face $ \mathcal{F} $ du polytope $ \BB_{d,\xx} $ avec $ H_{d,\xx} $. Consid\'erons par exemple l'intersection (suppos\'ee non vide) de la face $ \mathcal{F}=\{\zz\in \BB_{d,\xx}\; | \; z_{n+1}=1\} $ avec $ H_{d,\xx} $. Pour simplifier les notations, on pose \[ \left\{ \begin{array}{rcl}\alpha_{j}=
A_{j}(d\xx) & \mbox{pour} & j\in\{r+1,...,m\} \\ \beta_{j}=B_{j}(d\xx) & \mbox{pour} & j\in\{m+1,...,n+1\}
\end{array}   \right.,    \] de sorte que $ H_{d,\xx} $ a pour \'equation $ \alpha_{r+1}y_{r+1}+...+\alpha_{m}y_{m}+\beta_{m+1}z_{m+1}+...+\beta_{n+1}z_{n+1}=0 $ (les $ \alpha_{k} $ ou les $ \beta_{k} $ \'etant non tous nuls). Par ailleurs on remarque que l'on peut subdiviser $ C_{d,\xx} $ en une union de $ 2^{n-m+1}(dk)^{m-r} $ polytopes $ C_{d,\xx,\aa,\ee} $ plus petits en posant \[ C_{d,\xx}=\bigcup_{\aa=(a_{r+1},...,a_{m})\in \{-dk,...,dk-1\}^{m-r}}\bigcup_{\ee=(\varepsilon_{1},...,\varepsilon_{n})\in \{-1,0\}^{n-m+1}}C_{d,\xx,\aa,\ee}, \] \[ C_{d,\xx,\aa,\ee}=H_{\xx}\cap\left(\left(\prod_{j=r+1}^{m}[a_{j},a_{j}+1]\right)\times \left(\prod_{j=m+1}^{n+1}[\varepsilon_{j},\varepsilon_{j}+1]\right)\right). \]
On peut par cons\'equent subdiviser chaque face $ \mathcal{F}\cap H_{d,\xx} $ en consid\'erant $ \mathcal{F}\cap C_{d,\xx,\aa,\ee} $ pour tout $ (\aa,\ee) $. Pour un couple $ (\aa,\ee) $ fix\'e, et pour tout $ \zz\in \mathcal{F}\cap C_{d,\xx,\aa,\ee} $, on a alors \[ \alpha_{r+1}y_{r+1}+...+\alpha_{m}y_{m}+\beta_{m+1}z_{m+1}+...+\beta_{n}z_{n}+\beta_{n+1}=0 \] avec $ \max\{\max_{r+1\leqslant j \leqslant m}|\alpha_{j}|,\max_{m+1\leqslant j \leqslant n}|\beta_{j}|\}\neq 0 $ puisque l'intersection $ \mathcal{F}\cap H_{d,\xx} $ est non vide. Supposons, par exemple, que \[ \max\{\max_{r+1\leqslant j \leqslant m}|\alpha_{j}|,\max_{m+1\leqslant j \leqslant n}|\beta_{j}|\}=|\beta_{n}| ,\] on a alors $ z_{n}=-\frac{\beta_{n+1}}{\beta_{n}}-\sum_{j=r+1}^{m}\frac{\alpha_{j}}{\beta_{n}}y_{j}-\sum_{j=m+1}^{n-1}\frac{\beta_{j}}{\beta_{n}}z_{j} $, et on peut construire l'application $ \phi_{\mathcal{F},\aa,\ee} : [0,1]^{n-1}\ra C_{d,\xx,\aa,\ee}\subset \RR^{n-r+1} $ d\'efinie par \begin{multline*} \phi_{\mathcal{F},\aa,\ee}(t_{r+1},...,t_{n-1})= \left(a_{r+1}+t_{r+1},...,a_{m}+t_{m}, \varepsilon_{m+1}+t_{m+1},...,\varepsilon_{n-1}+t_{n-1},\right. \\ \left.-\frac{\beta_{n+1}}{\beta_{n}}-\sum_{j=r+1}^{m}\frac{\alpha_{j}}{\beta_{n}}(a_{j}+t_{j})-\sum_{j=m+1}^{n-1}\frac{\beta_{j}}{\beta_{n}}(\varepsilon_{j}+t_{j}),1\right). \end{multline*}

On remarque alors que $ \mathcal{F}\cap C_{d,\xx,\aa,\ee}\cap H_{d,\xx}\subset \phi_{\mathcal{F},\aa,\ee}([0,1]^{n-1}) $ et que \begin{multline*} ||\phi_{\mathcal{F},\aa,\ee}(\tt)-\phi_{\mathcal{F},\aa,\ee}(\tt')||_{2}  \leqslant \sqrt{n-r+1}||\phi_{\mathcal{F},\aa,\ee}(\tt)-\phi_{\mathcal{F},\aa,\ee}(\tt')||_{\infty} \\ \leqslant \sqrt{n-r+1}\max\left(1,\sum_{j=r+1}^{m}\frac{|\alpha_{j}|}{|\beta_{n}|}+\sum_{j=m+1}^{n-1}\frac{|\beta_{j}|}{|\beta_{n}|}\right)||\tt-\tt'||_{\infty} \\  \leqslant (n-r-1)\sqrt{n-r+1}||\tt-\tt'||_{2}. \end{multline*} On a donc $  \partial C_{d,\xx}\in \Lip(n-r,1,2^{n-r+1}(dk)^{m-r},(n-r-1)\sqrt{n-r+1}) $ et par cons\'equent \[  \partial P_{2}C_{d,\xx}\in \Lip(n,1,2^{n-r+1}(dk)^{m-r},(n-r-1)\sqrt{n-r+1}P_{2}). \] De plus puisque $ \Lambda_{d,\xx}\subset \ZZ^{n-r+1} $ le premier minimum successif de ce r\'eseau est sup\'erieur ou \'egal \`a $ 1 $. Ainsi, puisque
\begin{equation}  N_{d,\xx}(P_{2})=\{(\yy,\zz)\in P_{2}\BB_{d,\xx}\cap \ZZ^{n-r+1}\; | \; F(d\xx,\yy,\zz)=0 \}=\card(\Lambda_{d,\xx}\cap P_{2}C_{d,\xx})\end{equation} le lemme \ref{geomnomb} nous donne un analogue du corollaire \ref{corollairex}  \begin{lemma}\label{corollairex'} On a :  \begin{equation}\label{reseau} N_{d,\xx}(P_{2})=\frac{\Vol(C_{d,\xx})}{\det(\Lambda_{d,\xx})}P_{2}^{n-r}+O((dk)^{m-r}P_{2}^{n-r-1}), \end{equation}  uniform\'ement pour tout $ \xx $ tel que $ |\xx|=k $. \end{lemma}
On d\'eduit alors de ce lemme un r\'esultat analogue \`a la proposition \ref{propx} : 

\begin{prop}\label{propx'}
On suppose $ d_{2}=1 $, $ P_{2}=P_{1}^{u} $ et de plus que $ u>d_{1} $.  Alors : \begin{multline*} N_{d,2}(P_{1},P_{2})=\left(\sum_{\xx\in P_{1}\BB_{1}\cap \mathcal{A}_{2}^{\lambda}(\ZZ) }\frac{\Vol(C_{d,\xx})}{\det(\Lambda_{d,\xx})}\right)P_{2}^{n-r} \\ + O\left(d^{m-r}P_{1}^{m+1-d_{1}}P_{2}^{n-r-\delta}\right), \end{multline*}pour un certain $ \delta>0 $ arbitrairement petit.
\end{prop}
\begin{proof}
D'apr\`es le lemme pr\'ec\'edent : \[ N_{d,2}(P_{1},P_{2})=\left(\sum_{\xx\in P_{1}\BB_{1}\cap \mathcal{A}_{2}^{\lambda}(\ZZ) }\frac{\Vol(C_{d,\xx})}{\det(\Lambda_{d,\xx})}\right)P_{2}^{n-r}+ O(\mathcal{E}), \] avec \begin{align*}
\mathcal{E} & =\sum_{\substack{\xx\in P_{1}\BB_{1}\cap \mathcal{A}_{2}^{\lambda}(\ZZ)\\ |\xx|=k} }(dk)^{m-r}P_{2}^{n-r-1} \\ & \ll d^{m-r}P_{1}^{m+1}P_{2}^{n-r-1} \\  & =  d^{m-r}P_{1}^{m+1-d_{1}}P_{2}^{n-r+\frac{d_{1}}{u}-1} \\ & \ll d^{m-r}P_{1}^{m+1-d_{1}}P_{2}^{n-r-\delta} 
\end{align*} car on a suppos\'e $ u>d_{1} $.
\end{proof}

Les r\'esultats des propositions \ref{propx} et \ref{propx'}  se r\'ev\`eleront cruciaux pour donner plus tard des estimations de $ N_{d,2}(P_{1},P_{2}) $ ind\'ependamment de $ u $. Mais avant cela, nous allons, dans la section qui va suivre, chercher \`a \'etablir des r\'esultats analogues \`a ceux obtenus dans cette section pour $ \zz $ fix\'e.

\section{Troisi\`eme \'etape}

Nous allons \`a pr\'esent chercher \`a \'evaluer, pour $ l\in \NN^{\ast} $ fix\'e la somme \[ \sum_{k\leqslant P_{1}}h_{d}(k,l), \] o\`u $ h_{d} $ est la fonction d\'efinie par \eqref{fonctionh}. On fixe donc \[l=\max\left( \left\lfloor\frac{|\yy|}{d|\xx|}\right\rfloor,|\zz|\right). \] Il sera alors n\'ecessaire de distinguer les cas $ \left\lfloor\frac{|\yy|}{d|\xx|}\right\rfloor < |\zz| $ et $ \left\lfloor\frac{|\yy|}{d|\xx|}\right\rfloor \geqslant|\zz| $. 

\subsection{Premier cas }
On suppose ici $ \left\lfloor\frac{|\yy|}{d|\xx|}\right\rfloor< |\zz|=l $. On choisira donc de fixer $ \zz $ de norme $ |\zz|=l $. Plut\^ot que calculer directement $ \sum_{k\leqslant P_{1}}h_{d}(k,l) $, nous allons, dans un premier temps, chercher \`a \'evaluer \begin{equation}
N_{d,\zz}(P_{1})=\Card\left\{ (\xx,\yy)\in \ZZ^{m+1}\; |\; |\xx|\leqslant P_{1}, \; |\yy|< dl|\xx|, \; F(d\xx,\yy,\zz)=0 \right\}.
\end{equation} 

On introduit la s\'erie g\'en\'eratrice \begin{equation}
S_{d,\zz}(\alpha)=\sum_{|\xx|\leqslant P_{1}}\sum_{|\yy|\leqslant dl|\xx|}e\left( \alpha F(d\xx,\yy,\zz)\right).
\end{equation}
On a alors comme pr\'ec\'edemment $ N_{d,\zz}(P_{1})=\int_{0}^{1}S_{d,\zz}(\alpha)d\alpha $.

\subsubsection{Sommes d'exponentielles}

Comme nous l'avons fait dans les sections pr\'ec\'edentes, on commence par \'etablir une in\'egalit\'e de type Weyl. \`A cette fin, on remarque que \[ |\yy|< d|\xx|l \Leftrightarrow |\xx|> \frac{|\yy|}{dl} \Leftrightarrow |\xx|\geqslant \left\lfloor\frac{|\yy|}{dl}\right\rfloor+1. \] On pose alors $ N= \left\lfloor\frac{|\yy|}{dl}\right\rfloor $ (ce qui \'equivaut \`a dire que $ |\yy|\in [d(N-1)l, dNl[ $), et on remarque que : \[ S_{d,\zz}(\alpha)=\sum_{N=0}^{P_{1}-1}S_{d,\zz,N}(\alpha), \] o\`u \begin{equation}
S_{d,\zz,N}(\alpha)=\sum_{\substack{ N+1\leqslant |\xx|\leqslant P_{1}}}\sum_{\substack{d(N-1)l\leqslant|\yy|< dNl}}e(\alpha F(d\xx,\yy,\zz)).
\end{equation}
Comme dans la section $ 3.1 $, \'etant donn\'e que le polyn\^ome $ F(\xx,\yy,\zz) $ est homog\`ene de degr\'e $ d_{1} $ en $ (\xx,\yy) $, on obtient sans difficult\'e la majoration \begin{multline*} 
|S_{d,\zz,N}(\alpha)|^{2^{d_{1}-1}}\ll \left(P_{1}^{r+1}\right)^{2^{d_{1}-1}-d_{1}}\left((dlP_{1})^{m-r}\right)^{2^{d_{1}-1}-d_{1}} \\ \sum_{\substack{\xx^{(1)},\yy^{(1)} \\ |\xx^{(1)}|\leqslant P_{1} \\ |\yy^{(1)}|\leqslant dlP_{1}}}...\sum_{\substack{\xx^{(d_{1}-1)},\yy^{(d_{1}-1)} \\ |\xx^{(d_{2}-1)}|\leqslant P_{1} \\ |\yy^{(d_{1}-1)}|\leqslant dlP_{1}}} \prod_{j=0}^{m}\min\left\{ H_{j}, \left|\left| \alpha\gamma_{d,\zz,j}\left((\xx^{(i)},\yy^{(i)})_{i\in \{1,...,d_{1}-1\}}\right)\right|\right|^{-1}\right\}
\end{multline*}
avec \[ H_{j}=\left\{ \begin{array}{rcl} P_{1} & \mbox{si} & j\in \{0,...,r\} \\ dlP_{1} & \mbox{si} & j\in \{r+1,...,m\},\end{array}\right. \] \[ \gamma_{d,\zz,j}\left((\xx^{(i)},\yy^{(i)})_{i\in \{1,...,d_{1}-1\}}\right)=\sum_{\ii=(i_{1},...,i_{d_{1}-1})\in \{0,...,m\}^{d_{1}-1}}F_{d,\zz,\ii,j}u_{i_{1}}^{(1)}...u_{i_{d_{1}-1}}^{(d_{1}-1)},\] o\`u \[ u_{i}=\left\{ \begin{array}{rcl} x_{i} & \mbox{si} & i\in \{0,...,r\} \\ y_{i} & \mbox{si} & i\in \{r+1,...,m\}, \end{array}\right.\] et les coefficients $ F_{d,\zz,\ii,j} $ sont sym\'etriques en $ (i_{1},...,i_{d_{1}-1},j)\in \{0,...,m\}^{d_{2}}  $. Remarquons que l'on peut \'ecrire \[ F_{d,\zz,\ii,j}=d^{f_{\ii,j}}F_{\zz,\ii,j}, \] avec \[ f_{\ii,j}=\card\{k\in\{1,...,d_{1}\}\; |\; i_{k}\in \{0,...,r\} \}, \] (en posant $ i_{d_{1}}=j $). \`A partir de l\`a, on montre, comme dans la section \ref{Weyl} que 
\begin{multline*} |S_{d,\zz,N}(\alpha)|^{2^{d_{1}-1}}\ll \left(P_{1}^{r+1+\varepsilon}\right)^{2^{d_{1}-1}-d_{1}+1}\left((dlP_{1})^{m-r+\varepsilon}\right)^{2^{d_{1}-1}-d_{1}+1} \\  M_{d,\zz}\left(\alpha,P_{1},dlP_{1},P_{1}^{-1},(dlP_{1})^{-1}\right), \end{multline*}
o\`u pour tous r\'eels strictement positifs $ H_{1},H_{2},B_{1},B_{2} $ : \begin{multline}\label{Mz} M_{d,\zz}\left(\alpha,H_{1},H_{2},B_{1}^{-1},B_{2}^{-1}\right) \\ = \card\left\{ (\xx^{(1)},\yy^{(1)},...,\xx^{(d_{1}-1},\yy^{(d_{2}-1)}) \; | \; \forall i\in \{1,...,d_{1}-1\} \; |\xx^{(i)}|\leqslant H_{1}, \; \right.\\  |\yy^{(i)}|\leqslant H_{2}, \;\et \forall j \in\{0,...,r\}\;\left|\left| \alpha\gamma_{d,\zz,j}\left((\xx^{(i)},\yy^{(i)})_{i\in \{1,...,d_{2}-1\}}\right)\right|\right|\leqslant B_{1}^{-1}, \\ \left. \et \; \forall j \in \{r+1,...,m\}\;\left|\left| \alpha\gamma_{\zz,j}\left((\xx^{(i)},\yy^{(i)})_{i\in \{1,...,d_{1}-1\}}\right)\right|\right|\leqslant B_{2}^{-1}\right\}.  \end{multline} On en d\'eduit, en sommant sur $ N $, le lemme ci-dessous : \begin{lemma}
Pour tous $ P>1 $, $ \kappa>0 $ et pour tout $ \varepsilon>0 $ arbitrairement petit, l'une au moins des assertions suivantes est v\'erifi\'ee : \begin{enumerate}
\item $ |S_{d,\zz}(\alpha)|\ll d^{\frac{(d_{1}-1)(r+1)}{2^{d_{1}-1}}}P_{1}^{m+2+\varepsilon}(dl)^{m-r+\varepsilon}P^{-\kappa}, $ 
\item  \begin{multline*} M_{d,\zz}\left(\alpha,P_{1},dlP_{1},P_{1}^{-1},(dlP_{1})^{-1}\right) \\ \gg d^{(d_{1}-1)(r+1)}\left(P_{1}^{r+1}\right)^{d_{1}-1}\left((dlP_{1})^{m-r}\right)^{d_{1}-1}P^{-2^{d_{1}-1}\kappa}. \end{multline*}
\end{enumerate}
\end{lemma}

On fixe alors $ (\xx^{(i)},\yy^{(i)})_{i\in \{1,...,d_{2}-2\}} $ et on applique le lemme \ref{geomnomb2} avec les variables $ \xx^{(d_{1}-1)},\yy^{(d_{1}-1)} $ et les formes lin\'eaires $ \alpha\gamma_{d,\zz,j} $ pour $ j\in \{0,...,m\} $, et en choisissant : $  Z_{2}=1 $, $ Z_{1}=d^{-1}P_{1}^{-1}P^{\theta} $, $ a_{j}=P_{1} $ pour tout $ j\in \{0,...,r\} $, et $ a_{j}=dlP_{1} $ pour $ j\in \{r+1,...,m\} $ de sorte que \[ \begin{array}{lrclcrcl}\forall j \in \{0,...,r\}, & a_{j}Z_{2} & = & P_{1}, & \;  & a_{j}Z_{1}& = &P^{\theta}/d \\ \forall j\in \{r+1,...,m\}, &  a_{j}Z_{2} & = & dlP_{1}, & \;  & a_{j}Z_{1}& = & lP^{\theta} \\ \forall j \in \{0,...,m\} , &  a_{j}^{-1}Z_{2} & = & P_{1}^{-1}, & \;  & a_{j}^{-1}Z_{1} & = & d^{-1}P_{1}^{-2}P^{\theta} \\ \forall j\in \{r+1,...,m\}, & a_{j}^{-1}Z_{2} & = & (dlP_{1})^{-1}, & \;  & a_{j}^{-1}Z_{1} & = & d^{-2}P_{1}^{-2}l^{-1}P^{\theta} 
\end{array}
\]
avec $ P>0 $ fix\'e, et $  \theta\in [0,1] $ tels que $ P^{\theta}\leqslant P_{1} $. En appliquant ce proc\'ed\'e aux autres familles de variables $ \xx^{(i)},\yy^{(i)} $, on obtient finalement la majoration suivante :
 \begin{multline*} 
M_{d,\zz}\left(\alpha,P_{1},dlP_{1},P_{1}^{-1},(dlP_{1})^{-1}\right) \\ \ll \left(\frac{dP_{1}}{P^{\theta}}\right)^{(d_{1}-1)(m+1)}
  M_{d,\zz}\left(\alpha,P^{\theta}/d,lP^{\theta},d^{-(d_{1}-1)}P_{1}^{-d_{1}}P^{(d_{1}-1)\theta},d^{-d_{1}}l^{-1}P_{1}^{-d_{1}}P^{(d_{1}-1)\theta}\right)
\end{multline*}
En appliquant le lemme \ref{lemmedebile}, on a par ailleurs que 
 \begin{multline*} 
  M_{d,\zz}\left(\alpha,P^{\theta}/d,lP^{\theta},d^{-(d_{1}-1)}P_{1}^{-d_{1}}P^{(d_{1}-1)\theta},d^{-d_{1}}l^{-1}P_{1}^{-d_{1}}P^{(d_{1}-1)\theta}\right) \\ \ll l^{(d_{1}-1)(m-r)}M_{d,\zz}\left(\alpha,P^{\theta}/d,P^{\theta},d^{-(d_{1}-1)}P_{1}^{-d_{1}}P^{(d_{1}-1)\theta},d^{-d_{1}}l^{-1}P_{1}^{-d_{1}}P^{(d_{1}-1)\theta}\right)
\end{multline*}
On a donc le lemme suivant : 

\begin{lemma}
Pour tous $ P>1 $, $ \kappa>0 $ et tout $ \varepsilon>0 $ arbitrairement petit, l'une au moins des assertions suivantes est v\'erifi\'ee : \begin{enumerate}
\item $ |S_{d,\zz}(\alpha)|\ll d^{m-r+\varepsilon+\frac{(d_{1}-1)(r+1)}{2^{d_{1}-1}}}l^{m-r+\varepsilon}P_{1}^{m+2+\varepsilon}P^{-\kappa}, $ 
\item  \begin{multline*} M_{d,\zz}\left(\alpha,P^{\theta}/d,P^{\theta},d^{-(d_{1}-1)}P_{1}^{-d_{1}}P^{(d_{1}-1)\theta},d^{-d_{1}}P_{1}^{-d_{1}}P^{(d_{1}-1)\theta}\right) \\ \gg \left(P^{\theta}\right)^{(m+1)(d_{1}-1)}P^{-2^{d_{1}-1}\kappa}. \end{multline*}
\end{enumerate}
\end{lemma}

On introduit \`a pr\'esent les nouvelles familles d'arcs majeurs 

\begin{equation}
\mathfrak{M}_{a,q}^{(1),\zz}(\theta)=\left\{ \alpha\in [0,1[ \; | \;  2|\alpha q-a|\leqslant d^{-(d_{1}-1)}P_{1}^{-d_{1}}P^{(d_{1}-1)\theta}\right\},
\end{equation}
\begin{equation}
\mathfrak{M}^{(1),\zz}(\theta)=\bigcup_{\substack{q\leqslant dl^{d_{2}}P^{(d_{1}-1)\theta} \\ d|q}}\bigcup_{\substack{0\leqslant a <q  }}\mathfrak{M}_{a,q}^{(1),\zz}(\theta).
\end{equation}
\begin{equation}
\mathfrak{M}_{a,q}^{(2),\zz}(\theta)=\left\{ \alpha\in [0,1[ \; | \;  2|\alpha q-a|\leqslant d^{-d_{1}}P_{1}^{-d_{1}}P^{(d_{1}-1)\theta}\right\},
\end{equation}
\begin{equation}
\mathfrak{M}^{(2),\zz}(\theta)=\bigcup_{q\leqslant l^{d_{2}}P^{(d_{1}-1)\theta}}\bigcup_{\substack{0\leqslant a <q \\ \PGCD(a,q)=1 }}\mathfrak{M}_{a,q}^{(2),\zz}(\theta).
\end{equation}
\begin{equation}
\mathfrak{M}^{\zz}(\theta)=\mathfrak{M}^{(1),\zz}(\theta)\cup\mathfrak{M}^{(2),\zz}(\theta)
\end{equation}
On a alors comme dans les sections pr\'ec\'edente : \begin{lemma}\label{dilemme33}
Si $ P>1 $, $ \kappa>0 $ et $ \varepsilon>0 $ arbitrairement petit, l'une au moins des assertions suivantes est v\'erifi\'ee : \begin{enumerate}
\item $ |S_{d,\zz}(\alpha)|\ll d^{m-r+\varepsilon+\frac{(d_{1}-1)(r+1)}{2^{d_{1}-1}}}l^{m-r+\varepsilon}P_{1}^{m+2+\varepsilon}P^{-\kappa}, $ 
\item le r\'eel $ \alpha $ appartient \`a $ \mathfrak{M}^{\zz}(\theta) $,
\item  \begin{multline*} 
\Card\left\{(\xx^{(i)},\yy^{(i)})_{i\in \{1,...,d_{1}-1\}}, \; |\xx^{(i)}|\leqslant P^{\theta}/d, \; |\yy^{(i)}|\leqslant P^{\theta}, \; \right. \\ \left. \et \; \forall j \in \{r+1,...,n+1\},\;  \gamma_{d,\zz,j}\left((\xx^{(i)},\yy^{(i)})_{i\in \{1,...,d_{1}-1\}}\right)=0\right\} \\  \gg \left(P^{\theta}\right)^{(m+1)(d_{1}-1)}P^{-2^{d_{1}-1}\kappa}. \end{multline*}
\end{enumerate}
\end{lemma}
Pour un $ \zz $ fix\'e, on d\'efinit \`a pr\'esent :  

\begin{multline}
V_{1,\zz}^{\ast}=\left\{ (\xx,\yy) \in \CC^{m+1} \; | \;  \; \forall i \in \{0,...,r\}, \;  \frac{\partial F}{\partial x_{i}}(\xx,\yy,\zz) =0 \right.   \\ \left. \et \; \forall j \in \{r+1,...,m\}, \frac{\partial F}{\partial y_{j}}(\xx,\yy,\zz) =0\right\}.
\end{multline}
On note par ailleurs : \begin{equation}
\mathcal{A}_{1}^{\mu}=\left\{ \zz \in \CC^{n-m+1} \; | \; \dim V_{1,\zz}^{\ast} < \dim V_{1}^{\ast} -(n-m+1)+\mu \right\},
\end{equation}
o\`u $ \mu\in \NN $ est un param\`etre que nous pr\'eciserons ult\'erieurement. Par abus de langage on note \[\mathcal{A}_{1}^{\mu}(\ZZ)=\mathcal{A}_{1}^{\mu}\cap \ZZ^{n-m+1}. \] 
On a alors une propri\'et\'e analogue \`a la proposition \ref{propA2}
\begin{prop}\label{propA1}
L'ensemble $ \mathcal{A}_{1}^{\mu} $ est un ouvert de Zariski de $ \AA_{\CC}^{n-m+1} $, et de plus, on a que \[ \Card\left\{ \zz \in [-P_{2},P_{2}]^{n-m+1}\cap (\mathcal{A}_{1}^{\mu})^{c}\cap \ZZ^{n-m+1} \right\}\ll P_{2}^{n-m+1-\mu}. \] 
\end{prop}
On commence par remarquer que le cardinal de la condition $ 3 $ peut \^etre major\'e par \begin{multline*}\Card\left\{(\xx^{(i)},\yy^{(i)})_{i\in \{1,...,d_{1}-1\}}, \; |\xx^{(i)}|\leqslant P^{\theta}, \; |\yy^{(i)}|\leqslant P^{\theta}, \; \right. \\ \left. \et \; \forall j \in \{r+1,...,n+1\},\;  \gamma_{\zz,j}\left((\xx^{(i)},\yy^{(i)})_{i\in \{1,...,d_{1}-1\}}\right)=0\right\}, \end{multline*} o\`u \[ \gamma_{\zz,j}\left((\xx^{(i)},\yy^{(i)})_{i\in \{1,...,d_{1}-1\}}\right)=\sum_{\ii=(i_{1},...,i_{d_{1}-1})\in \{0,...,m\}^{d_{1}-1}}F_{\zz,\ii,j}u_{i_{1}}^{(1)}...u_{i_{d_{1}-1}}^{(d_{1}-1)}.\]
Puis, comme dans la section pr\'ec\'edente, en choisissant $ \kappa=K_{1}\theta $ avec \begin{equation}
K_{1}=(n+2-\dim V_{1}^{\ast}-\mu)/2^{d_{1}-1}
\end{equation}
on d\'eduit du lemme \ref{dilemme33} : 

\begin{lemma}\label{dilemme34}
Si $ \varepsilon>0 $ est un r\'eel arbitrairement petit, l'une au moins des assertions suivantes est v\'erifi\'ee : \begin{enumerate}
\item $ |S_{d,\zz}(\alpha)|\ll d^{m-r+\varepsilon+\frac{(d_{1}-1)(r+1)}{2^{d_{1}-1}}}l^{m-r+\varepsilon}P_{1}^{m+2+\varepsilon}P^{-K_{1}\theta}, $ 
\item le r\'eel $ \alpha $ appartient \`a $ \mathfrak{M}^{\zz}(\theta) $.
\end{enumerate}
\end{lemma}
Pour tout le reste de cette section, on fixera $ P=P_{1} $.

\subsubsection{M\'ethode du cercle}

On fixe un r\'eel $ \theta\in [0,1] $. On suppose de plus que \begin{equation}
K_{1}>2(d_{1}-1).
\end{equation}On notera \begin{equation}
\phi_{1}(d,l,\theta)=dl^{d_{2}}P_{2}^{(d_{1}-1)\theta},
\end{equation} \begin{equation}
\Delta_{1}(\theta,K_{1})=\theta(K_{1}-2(d_{1}-1))
\end{equation}
On supposera de plus $ \theta $ est tel que \[ \Delta_{1}(\theta,K_{1})>1. \]
Comme pr\'ec\'edemment nous allons v\'erifier que les arcs mineurs fournissent bien un terme d'erreur.

\begin{lemma}\label{arcsmin1}
Pour tout $ \zz\in \mathcal{A}_{1}^{\mu}(\ZZ) $, et si $ d_{1}\geqslant 2 $, on a la  majoration : \[ \int_{\alpha\notin \mathfrak{M}^{\zz}(\theta)}|S_{\zz}(\alpha)|d\alpha \ll d^{m-r+\varepsilon+\frac{(d_{1}-1)(r+1)}{2^{d_{1}-1}}}l^{d_{2}+m-r+\varepsilon}P_{1}^{m+2-d_{1}-\Delta_{1}(\theta,K_{1})+\varepsilon}. \] \end{lemma}

\begin{proof}
Consid\'erons une suite  \[ 0<\theta=\theta_{0}<\theta_{1}<...<\theta_{T-1}<\theta_{T}=1 \] telle que \begin{equation}
2(\theta_{i+1}-\theta_{i})(d_{1}-1)<\varepsilon 
\end{equation}
et $ T\ll P_{1}^{\varepsilon} $ pour $ \varepsilon>0 $ arbitrairement petit (et $ P_{1} $ assez grand). Puisque $ \zz\in \mathcal{A}_{1}^{\mu}(\ZZ) $, par le lemme \ref{dilemme34} on a  \begin{multline*}
\int_{\alpha\notin \mathfrak{M}^{\zz}(\theta_{T})}|S_{d,\xx}(\alpha)|d\alpha \\  \ll d^{m-r+\varepsilon+\frac{(d_{1}-1)(r+1)}{2^{d_{1}-1}}}l^{m-r+\varepsilon}P_{1}^{m+2-K_{1}\theta_{T}+\varepsilon} \\  \ll d^{m-r+\varepsilon+\frac{(d_{1}-1)(r+1)}{2^{d_{1}-1}}}l^{m-r+\varepsilon}P_{1}^{m+2-d_{2}-\Delta_{1}(\theta,K_{1})+\varepsilon}.
\end{multline*}
Par ailleurs, on remarque que \begin{multline*}
\Vol\left(\mathfrak{M}^{\zz}(\theta)\right)  \ll \Vol\left(\mathfrak{M}^{(1),\zz}(\theta)\right)+\Vol\left(\mathfrak{M}^{(2),\zz}(\theta)\right) \\  \ll \sum_{q\leqslant dl^{d_{2}}P_{2}^{(d_{1}-1)\theta}}\sum_{\substack{0\leqslant a <q }}q^{-1}d^{-(d_{1}-1)}P_{1}^{-d_{1}+(d_{1}-1)\theta} \\ +\sum_{q\leqslant l^{d_{2}}P_{2}^{(d_{1}-1)\theta}}\sum_{\substack{0\leqslant a <q \\ \PGCD(a,q)=1}}q^{-1}d^{-d_{1}}P_{1}^{-d_{1}+(d_{1}-1)\theta} \\  \ll  d^{(2-d_{1)}}l^{d_{2}}P_{1}^{-d_{1}+2(d_{1}-1)\theta}.
\end{multline*} 

On a alors pour tout $ i\in \{0,...,T-1\} $ : \begin{multline*}
\int_{\alpha\in \mathfrak{M}^{\zz}(\theta_{i+1})\setminus  \mathfrak{M}^{\zz}(\theta_{i})}|S_{d,\zz}(\alpha)|d\alpha \\  \ll d^{m-r+\varepsilon+\frac{(d_{1}-1)(r+1)}{2^{d_{1}-1}}}l^{m-r+\varepsilon}P_{1}^{m+2-K_{1}\theta_{i}+\varepsilon}\Vol\left(\mathfrak{M}^{\zz}(\theta_{i+1})\right) \\   \ll   d^{m-r+\varepsilon+\frac{(d_{1}-1)(r+1)}{2^{d_{1}-1}}+(2-d_{1})}l^{m-r+d_{2}+\varepsilon}P_{1}^{m+2-K_{1}\theta_{i}+\varepsilon-d_{1}+2(d_{1}-1)\theta_{i+1}} \\  \ll d^{m-r+\varepsilon+\frac{(d_{1}-1)(r+1)}{2^{d_{1}-1}}+(2-d_{1})}l^{m-r+d_{2}+\varepsilon}P_{1}^{m+2-d_{1}-\Delta_{1}(\theta,K_{1})+\varepsilon}
\end{multline*}
 et  on obtient le r\'esultat souhait\'e en sommant sur les $ i\in \{0,...,T-1\} $. 
\end{proof}

On introduit une nouvelle famille d'arcs majeurs :

\begin{equation}
\mathfrak{M}_{a,q}^{'d,\zz}(\theta)=\left\{ \alpha\in [0,1[ \; | \;  2|\alpha q-a|\leqslant qd^{-d_{1}}P_{1}^{-d_{1}}P^{(d_{1}-1)\theta}\right\},
\end{equation}
\begin{equation}
\mathfrak{M}^{'d,\zz}(\theta)=\bigcup_{q\leqslant \phi_{1}(d,l,\theta)}\bigcup_{\substack{0\leqslant a <q \\ \PGCD(a,q)=1 }}\mathfrak{M}_{a,q}^{d,\zz'}(\theta),
\end{equation}
et on v\'erifie que $ \mathfrak{M}^{d,\zz}(\theta)\subset\mathfrak{M}^{'d,\zz}(\theta) $. On a alors le r\'esultat analogue au lemme \ref{disjoint2} :

\begin{lemma}\label{disjoint1} Si $ d_{1}\geqslant 2 $ et si l'on suppose $ l^{2d_{2}}P_{1}^{-d_{1}+3\theta(d_{1}-1)}<1 $, alors les arcs majeurs $ \mathfrak{M}_{a,q}^{'d,\zz}(\theta) $ sont disjoints deux \`a deux. \end{lemma}\begin{proof}
Supposons qu'il existe $ \alpha\in\mathfrak{M}_{a,q}^{'d,\zz}(\theta)\cap\mathfrak{M}_{a',q'}^{'d,\zz}(\theta)  $ pour $ (a,q)\neq (a',q') $, $ q,q'\leqslant \phi_{1}(d,l,\theta) $, $ 0\leqslant a<q $, $ 0\leqslant a'<q' $ et $ \PGCD(a,q)=\PGCD(a',q')=1 $. On a alors \[ 1 \leqslant qq'd^{-d_{1}}P_{1}^{-d_{1}+\theta(d_{1}-1)} \leqslant d^{2-d_{1}} l^{2d_{2}}P_{1}^{-d_{1}+3\theta(d_{1}-1)} \leqslant l^{2d_{2}}P_{1}^{-d_{1}+3\theta(d_{1}-1)}, \]
d'o\`u le r\'esultat. 
\end{proof}

Comme pr\'ec\'edemment, on d\'eduit des lemmes \ref{disjoint1} et \ref{arcsmin1} que : 

\begin{multline} N_{d,\zz}(P_{1})=\sum_{q\leqslant \phi_{1}(d,l,\theta) }\sum_{\substack{0\leqslant a <q \\ \PGCD(a,q)=1}}\int_{\alpha\in \mathfrak{M}_{a,q}^{d,\zz'}(\theta)}S_{d,\zz}(\alpha)d\alpha \\ + O\left( d^{m-r+\varepsilon+\frac{(d_{1}-1)(r+1)}{2^{d_{1}-1}}}l^{d_{2}+m-r+\varepsilon}P_{1}^{m+2-d_{1}-\Delta_{1}(\theta,K_{1})+\varepsilon} \right). 
\end{multline}

On consid\`ere $ \alpha\in  \mathfrak{M}_{a,q}^{'d,\zz}(\theta) $. On pose $ \beta=\alpha-\frac{a}{q} $ et donc $ |\beta|\leqslant d^{-d_{1}}P_{1}^{-d_{1}+(d_{1}-1)\theta} $. De la m\^eme mani\`ere que nous avons \'etabli le lemme \ref{separation2}, on d\'emontre :

\begin{lemma}\label{separation1}
On a l'estimation \begin{multline*}
S_{d,\zz}(\alpha)=d^{m-r}l^{m-r}P_{1}^{m+1}q^{-(m+1)}S_{a,q,d}(\zz)I_{\zz}(d^{d_{1}}P_{1}^{d_{1}}\beta) \\ +O\left( d^{m-r+1}l^{2d_{2}+m-r}P_{1}^{m+2\theta(d_{1}-1)} \right), \end{multline*} avec \begin{equation}
S_{a,q,d}(\zz)=\sum_{(\bb_{1},\bb_{2})\in (\ZZ/q\ZZ)^{r+1}\times (\ZZ/q\ZZ)^{m-r}}e\left(\frac{a}{q}F(d\bb_{1},\bb_{2},\zz)\right),
\end{equation} \begin{equation}
I_{\zz}(\beta)=\int_{\substack{(\uu,\vv)\in [-1,1]^{r+1}\times [-1,1]^{m-r} \\ |\vv|< |\uu|}}e\left(\beta F(\uu,l\vv,\zz)\right)d\uu d\vv. 
\end{equation}
\end{lemma}

Par ailleurs, en posant \begin{equation}
\tilde{\phi}_{1}(\theta)=\frac{1}{2}P_{1}^{\theta(d_{1}-1)},
\end{equation}\begin{equation}
\eta_{1}(\theta)=1-5\theta(d_{1}-1),
\end{equation}
on d\'emontre un analogue du lemme \ref{intermediairex} :
 \begin{lemma}\label{intermediairez} Pour $ \zz\in \mathcal{A}_{1}^{\mu}(\ZZ) $, et $ \varepsilon>0 $ arbitrairement petit, on a l'estimation suivante : \begin{multline*}
N_{d,\zz}(P_{1})=d^{m-r-d_{1}}l^{m-r}P_{1}^{m+1-d_{1}}\mathfrak{S}_{d,\zz}(\phi_{1}(d,l,\theta))J_{\zz}(\tilde{\phi}_{1}(\theta)) \\ + O\left(d^{m-r+\varepsilon+\frac{(d_{1}-1)(r+1)}{2^{d_{1}-1}}}l^{d_{2}+m-r+\varepsilon}P_{1}^{m+2-d_{1}-\Delta_{1}(\theta,K_{1})+\varepsilon}\right) \\ + O\left(d^{m-r+3-d_{1}}l^{4d_{2}+m-r}P_{1}^{m+1-d_{1}-\eta(\theta)}\right), \end{multline*} o\`u \begin{equation}
\mathfrak{S}_{d,\zz}(\phi(d,l,\theta))=\sum_{q\leqslant \phi(d,l,\theta)}q^{-(m+1)}\sum_{\substack{0\leqslant a < q \\ \PGCD(a,q)=1}}S_{a,q,d}(\zz),
\end{equation}\begin{equation}
J_{\zz}(\tilde{\phi}(\theta))=\int_{|\beta|\leqslant \tilde{\phi}(\theta)}I_{\zz}(\beta)d\beta. 
\end{equation}
\end{lemma}

On pose \`a pr\'esent : 

\begin{equation}
\mathfrak{S}_{d,\zz}=\sum_{q=1}^{\infty}q^{-(m+1)}\sum_{\substack{0\leqslant a < q \\ \PGCD(a,q)=1}}S_{a,q,d}(\zz),
\end{equation}
\begin{equation}
J_{\zz}=\int_{\RR}I_{\zz}(\beta)d\beta.
\end{equation}
\begin{lemma}\label{Jz}
Soit $ \zz\in \mathcal{A}_{1}^{\mu}(\ZZ) $, et $ \varepsilon>0 $ arbitrairement petit. On suppose de plus que $ d_{1}\geqslant 2 $ et que $ \theta<\frac{1}{2(d_{1}-1)} $. L'int\'egrale $ J_{\zz} $ est absolument convergente, et on a : \[ |J_{\zz}(\tilde{\phi}(\theta))-J_{\zz}|\ll l^{d_{2}+\varepsilon}P_{1}^{-\frac{K_{1}\theta}{2}+2\theta(d_{1}-1)} . \] On a de plus $ |J_{\zz}|\ll l^{d_{2}+\varepsilon}. $
\end{lemma}
\begin{proof}
On consid\`ere $ \beta $ tel que $ |\beta|\geqslant \tilde{\phi}(\theta) $. On choisit alors des param\`etres $ P $ et $ \theta' $ tels que \begin{equation}
|\beta|=\frac{1}{2}P^{\theta'(d_{1}-1)}, 
\end{equation} \begin{equation}\label{supp}
 P^{1-K_{1}\theta'}=P^{-1+2\theta'(d_{1}-1)}l^{2d_{2}}.
\end{equation} Remarquons que ces deux \'egalit\'es impliquent \begin{equation}\label{valeurexacte}
\theta'=\frac{(d_{1}-1)^{-1}\log(2|\beta|)}{\left(2+\frac{K_{1}}{(d_{1}-1)}\right)\log(2|\beta|)+2d_{2}\log(l)}
\end{equation}
donc en particulier \begin{equation}
\theta'\gg \min\left\{1,\frac{\log(|\beta|)}{\log(l)}\right\}.
\end{equation} Par ailleurs, on a d'apr\`es\;\eqref{supp} : \[ P^{-2+3\theta'(d_{1}-1)}l^{2d_{2}}=P^{\theta'(d_{1}-1-K_{1})}<1, \]donc, pour $ d_{1}\geqslant 2 $, \[  P^{-d_{1}+3\theta'(d_{1}-1)}l^{2d_{2}}<1, \] et ainsi, d'apr\`es le lemme \ref{disjoint2}, les arcs majeurs $ \mathfrak{M}_{a,q}(\theta') $ correspondant \`a $ P $ et $ \theta' $ sont disjoints deux \`a deux. Le r\'eel $ P^{-d_{1}}\beta $ appartient au bord de $ \mathfrak{M}_{0,1}(\theta') $, et donc par le lemme \ref{dilemme34} appliqu\'e \`a $ d=1 $, on a l'estimation : 

\[ |S_{1,\zz}(P^{-d_{1}}\beta )|\ll l^{m-r+\varepsilon}P^{m+2-K_{1}\theta'+\varepsilon}. \] 
Par le lemme \ref{separation1} on a : \[ S_{1,\zz}(P^{-d_{1}}\beta )=l^{m-r}P^{m+1}I_{\zz}(\beta)+O\left(l^{m-r+2d_{2}}P^{m+2\theta'(d_{1}-1)} \right). \] On a ainsi : \begin{align*} |I_{\zz}(\beta)| & \ll l^{\varepsilon}P^{1-K_{1}\theta'+\varepsilon}+ l^{2d_{2}}P^{-1+2\theta'(d_{1}-1)} 
\\ & \ll l^{\varepsilon}P^{1-K_{1}\theta'+\varepsilon} \\ & \ll l^{\varepsilon}|\beta|^{\frac{1}{\theta'(d_{1}-1)}-\frac{K_{1}}{(d_{1}-1)}+\frac{\varepsilon}{\theta'(d_{1}-1)}}. \end{align*}
\'Etant donn\'e que $ \theta'\gg \min\left\{1,\frac{\log(|\beta|)}{\log(l)}\right\} $, \[ |\beta|^{\frac{\varepsilon}{\theta'(d_{1}-1)}}\ll \max\{|\beta|^{\varepsilon'}, l^{\varepsilon'}\}, \] pour $ \varepsilon'>0 $ arbitrairement petit.
D'autre part, d'apr\`es \eqref{valeurexacte} : \begin{align*}
|\beta|^{\frac{1}{\theta'(d_{1}-1)}-\frac{K_{1}}{(d_{1}-1)}} & \ll |\beta|^{\left(1-\frac{K_{1}}{2(d_{1}-1)}\right)}|\beta|^{\frac{\log(l^{d_{2}})}{\log(2|\beta|)}} \\ & \ll l^{d_{2}} |\beta|^{\left(1-\frac{K_{1}}{2(d_{1}-1)}\right)}. \end{align*}
On a ainsi
\begin{align*} |J_{\zz}(\tilde{\phi}(\theta))-J_{\zz}| & \ll l^{d_{2}+\varepsilon}\int_{|\beta|>\tilde{\phi}(\theta)}|\beta|^{\left(1-\frac{K_{1}}{2(d_{1}-1)}\right)+\varepsilon} d\beta \\ & \ll  l^{d_{2}+\varepsilon}\tilde{\phi}(\theta)^{2-\frac{K_{1}}{2(d_{1}-1)}+\varepsilon}  \\ & \ll l^{d_{2}+\varepsilon}P_{1}^{2\theta(d_{1}-1)-\frac{K_{1}\theta}{2}+\varepsilon} \end{align*} avec $ \varepsilon $ arbitrairement petit.

D'autre part, en choisissant $ P_{1} \ll 1   $, cette in\'egalit\'e donne \[ |J_{\zz}(\tilde{\phi}(\theta))-J_{\zz}| \ll l^{d_{2}+\varepsilon}, \] et puisque $ |J_{\zz}(\tilde{\phi}(\theta))|\ll 1 $ lorsque $ P_{1} \ll 1   $, on a imm\'ediatement \[ |J_{\zz}| \ll l^{d_{2}+\varepsilon} \]

\end{proof}

On introduit pour $ \zz $ fix\'e et $ P\geqslant 1 $ la nouvelle s\'erie g\'en\'eratrice : 
\[ S_{d,\zz}'(\alpha)=\sum_{|\xx|\leqslant P}\sum_{|\yy|\leqslant P}e\left(\alpha F(d\xx,\yy,\zz)\right). \] De la m\^eme mani\`ere que pour le lemme \ref{dilemme34}, on \'etablit : 

\begin{lemma}\label{dilemme34bis}
Si $ \varepsilon>0 $ est un r\'eel arbitrairement petit, l'une au moins des assertions suivantes est v\'erifi\'ee : \begin{enumerate}
\item $ |S_{d,\zz}'(\alpha)|\ll d^{\frac{(d_{1}-1)(r+1)}{2^{d_{1}-1}}}P^{m+1+\varepsilon-K_{1}\theta}, $ 
\item le r\'eel $ \alpha $ appartient \`a $ \bigcup_{q\leqslant dl^{d_{2}}P^{(d_{1}-1)\theta}}\bigcup_{\substack{0\leqslant a <q\\ \PGCD(a,q)=1}}\mathfrak{M}^{\zz}_{a,q}(\theta) $.
\end{enumerate}
\end{lemma}

De la m\^eme mani\`ere que pour le lemme \ref{Sx}, on en d\'eduit :

\begin{lemma}\label{Sz}

Soit $ \zz\in \mathcal{A}_{1}^{\mu}(\ZZ) $, et $ \varepsilon>0 $ arbitrairement petit. On suppose de plus que $ d_{1}\geqslant 2 $. La s\'erie $ \mathfrak{S}_{\zz} $ est absolument convergente, et on a : \[ |\mathfrak{S}_{d,\zz}(\phi_{1}(d,l,\theta))-\mathfrak{S}_{d,\zz}|\ll d^{\frac{(d_{1}-1)(r+1)}{2^{d_{1}-1}}+2+\varepsilon}l^{2d_{2}+\varepsilon}P_{1}^{\theta(2(d_{1}-1)-K_{1})}. \] On a de plus $ |\mathfrak{S}_{\zz}|\ll d^{\frac{(d_{1}-1)(r+1)}{2^{d_{1}-1}}+2+\varepsilon}l^{2d_{2}+\varepsilon}. $
\end{lemma}
\begin{proof}
On consid\`ere $ q>\phi(d,k,\theta) $, $ \alpha=\frac{a}{q} $ avec $ 0\leqslant a<q $ et $ \PGCD(a,q)=1 $. On a alors  $ S_{a,q,d}(\zz)=S_{d,\zz}'(\alpha) $ avec $ P=q $. On consid\`ere $ \theta' $ tel que $ q=dl^{d_{2}} q^{(d_{1}-1)\theta'} $. Si $ \theta''=\theta'-\nu $ pour $ \nu>0 $ arbitrairement petit, alors s'il existait $ a',q'\in \ZZ $ tels que $ 0\leqslant a'<q' $, $ \PGCD(a',q')=1 $, $ q'\leqslant dl^{d_{2}}q^{\theta''(d_{1}-1)} <q $ et $ \alpha \in \mathfrak{M}_{a',q'}^{\zz}(\theta'') $, on aurait alors \[ 1 \leqslant |aq'-a'q|\leqslant q^{1-d_{1}+\theta'(d_{1}-1)}, \] ce qui est absurde pour $ d_{1}\geqslant 2 $. On a donc, par le lemme pr\'ec\'edent : \[ |S_{a,q,d}(\zz)|\ll d^{\frac{(d_{1}-1)(r+1)}{2^{d_{1}-1}}}q^{m+1+\varepsilon-K_{1}\theta'}. \] Par cons\'equent \begin{align*}
\left| \mathfrak{S}_{d,\zz}(\phi_{1}(d,l,\theta))-\mathfrak{S}_{d,\zz}\right| & \ll d^{\frac{(d_{1}-1)(r+1)}{2^{d_{1}-1}}}\sum_{q>\phi_{1}(d,l,\theta)}q^{-(m+1)}\sum_{0\leqslant a <q}|S_{a,q,d}(\zz)| \\ & \ll d^{\frac{(d_{1}-1)(r+1)}{2^{d_{1}-1}}}\sum_{q>\phi_{1}(d,l,\theta)}q^{-(m+1)}\sum_{0\leqslant a <q}q^{m+1+\varepsilon-K_{1}\theta'} \\ & \ll 
d^{\frac{(d_{1}-1)(r+1)}{2^{d_{1}-1}}}\sum_{q>\phi(d,l,\theta)}q^{-\frac{K_{1}}{(d_{1}-1)}+1+\varepsilon}l^{\frac{d_{2}K_{1}}{(d_{1}-1)}}d^{\frac{K_{1}}{(d_{1}-1)}} \\ & \ll d^{\frac{(d_{1}-1)(r+1)}{2^{d_{1}-1}}+2+\varepsilon}l^{2d_{2}+\varepsilon}P_{1}^{\theta(2(d_{1}-1)-K_{1})+\varepsilon}.
\end{align*}
En prenant $ P_{1}\ll 1 $ cette majoration donne : \[ \left| \mathfrak{S}_{d,\zz}(\phi_{1}(d,l,\theta))-\mathfrak{S}_{d,\zz}\right|  \ll d^{\frac{(d_{1}-1)(r+1)}{2^{d_{1}-1}}+2+\varepsilon}l^{2d_{2}+\varepsilon}, \] et en consid\'erant la majoration triviale $ | \mathfrak{S}_{d,\zz}(\phi_{1}(d,l,\theta))|\ll d^{2}l^{2d_{2}} $, on trouve finalement \[ |\mathfrak{S}_{d,\zz}|\ll d^{\frac{(d_{1}-1)(r+1)}{2^{d_{1}-1}}+2+\varepsilon}l^{2d_{2}+\varepsilon}. \]
\end{proof}

On d\'eduit alors des lemmes \ref{Sz} et \ref{Jz} : 

\begin{lemma}\label{lemmez}
Soit  $ \zz\in \mathcal{A}_{1}^{\mu}(\ZZ) $. On suppose fix\'es $ \theta\in [0,1] $ et $ P_{1}\geqslant 1  $ tels que $ l^{2d_{2}}P_{1}^{-d_{1}+3\theta(d_{1}-1)}<1 $. On suppose de plus que $ K_{1}>4(d_{1}-1) $ et $ d_{1}\geqslant 2 $. On a alors que \[ N_{d,\zz}(P_{1})=\mathfrak{S}_{d,\zz}J_{\zz}d^{m-r-d_{1}}l^{m-r}P_{1}^{m+1-d_{1}}+ O(E_{2}) +O(E_{3}), \] avec \[ E_{2}=d^{m-r+3-d_{1}}l^{4d_{2}+m-r}P^{m+1-d_{1}-\eta(\theta)}, \]\[ E_{3}=d^{m-r+\frac{(d_{1}-1)(r+1)}{2^{d_{1}-1}}+\varepsilon}l^{3d_{2}+m-r+\varepsilon}P_{1}^{m+1-d_{1}+2\theta(d_{1}-1)-\frac{K_{1}\theta}{2}+\varepsilon} \]
et $ \varepsilon>0 $ arbitrairement petit. 
\end{lemma}
\begin{proof}
Par le lemme \ref{intermediairez} :

\begin{multline*}
N_{d,\zz}(P_{1})=d^{m-r-d_{1}}l^{m-r}P_{1}^{m+1-d_{1}}\mathfrak{S}_{d,\zz}(\phi_{1}(d,l,\theta))J_{\zz}(\tilde{\phi}_{1}(\theta))+O(E_{1})+O(E_{2}), \end{multline*} 
o\`u \[ E_{1}=d^{m-r+\varepsilon+\frac{(d_{1}-1)(r+1)}{2^{d_{1}-1}}}l^{d_{2}+m-r+\varepsilon}P_{1}^{m+1-d_{1}-\Delta_{1}(\theta,K_{1})+\varepsilon} \ll E_{3}. \]
Par ailleurs, d'apr\`es les lemmes \ref{Sz} et \ref{Jz}, on a \begin{multline*}
\left|\mathfrak{S}_{d,\zz}(\phi_{1}(d,l,\theta))J_{\zz}(\tilde{\phi}_{1}(\theta))-\mathfrak{S}_{d,\zz}J_{\zz}\right| \\ \leqslant \left|\mathfrak{S}_{d,\zz}(\phi_{1}(d,l,\theta))-\mathfrak{S}_{d,\zz}\right|\left|J_{\zz}\right|+ \left|\mathfrak{S}_{d,\zz}(\phi_{1}(d,l,\theta)) \right| \left|  J_{\zz}(\tilde{\phi}_{1}(\theta))-J_{\zz}\right| \\ \ll d^{\frac{(d_{1}-1)(r+1)}{2^{d_{1}-1}}+2+\varepsilon}l^{3d_{2}+2\varepsilon}P_{1}^{\theta(2(d_{1}-1)-K_{1})}+ d^{2}l^{3d_{2}+2\varepsilon}P_{1}^{\theta(2(d_{1}-1)-\frac{K_{1}}{2})+\varepsilon},
\end{multline*} et en multipliant par $ d^{m-r-d_{1}}l^{m-r}P_{1}^{m+1-d_{1}} $, on obtient un terme d'erreur \[ d^{m-r-d_{1}\frac{(d_{1}-1)(r+1)}{2^{d_{1}-1}}+2+\varepsilon}l^{3d_{2}+2\varepsilon}P_{1}^{m+1-d_{1}+2\theta(d_{1}-1)-\frac{K_{1}\theta}{2}+\varepsilon}, \]
d'o\`u le r\'esultat. 
\end{proof}

En fixant $ \theta>0 $ tel que $ \theta<\frac{1}{5(d_{1}-1)} $, on obtient le corollaire suivant : 

\begin{cor}\label{corollairez}
Soit  $ \zz\in \mathcal{A}_{1}^{\mu}(\ZZ) $. On suppose que $ K_{1}>4(d_{1}-1) $ et $ d_{1}\geqslant 2 $. Il existe alors un r\'eel $ \delta>0 $ arbitrairement petit tel que : \begin{multline*} N_{d,\zz}(P_{1})=\mathfrak{S}_{d,\zz}J_{\zz}d^{m-r-d_{1}}l^{m-r}P_{1}^{m+1-d_{1}} \\ +O\left( d^{m-r}\max\{d^{\frac{(d_{1}-1)(r+1)}{2^{d_{1}-1}}+\varepsilon},d^{3-d_{1}}\}l^{m-r+4d_{2}}P_{1}^{m+1-d_{1}-\delta}\right), \end{multline*} uniform\'ement pour tout $ l<P_{1}^{\frac{d_{1}-1}{2d_{2}}} $. 
\end{cor}
On pose \`a pr\'esent $ P_{1}=P_{2}^{b} $, avec $ b\geqslant 1 $ et on introduit la fonction \begin{equation}
g_{1}(b,\delta)=\left(1-\frac{5d_{2}}{b}-\delta\right)^{-1}5(d_{1}-1)\left(\frac{4d_{2}}{b}+2\delta\right),
\end{equation} ainsi que \begin{multline}
\tilde{N}_{d,1}^{(1)}(P_{1},P_{2})=\Card\left\{ (\xx,\yy,\zz)\in \ZZ^{n+2} \; | \; \zz \in \mathcal{A}_{1}^{\mu}(\ZZ), \; |\xx|\leqslant P_{1}, \;\right. \\ \left. |\yy| < d|\xx|P_{2},\;   |\zz|\leqslant P_{2}, \; |\yy|\leqslant d|\xx| |\zz|,  \;  F(d\xx,\yy,\zz)=0 \right\}  \\ =\Card\left\{ (\xx,\yy,\zz)\in \ZZ^{n+2} \; | \; \zz \in \mathcal{A}_{1}^{\mu}(\ZZ), \; |\xx|\leqslant P_{1}, \;\right. \\ \left. |\yy| < d|\xx|P_{2},\;   |\zz|\leqslant P_{2}, \; \left\lfloor \frac{|\yy|}{d|\xx|}\right\rfloor < |\zz|,  \;  F(d\xx,\yy,\zz)=0 \right\} .
\end{multline}
On a alors la proposition suivante qui est l'analogue de \ref{propx} : 
\begin{prop}\label{propz}
On suppose $ K_{1}>4(d_{1}-1) $, $ P_{1}=P_{2}^{b} $ et de plus que \[ \frac{K_{1}}{2}-2(d_{1}-1)>g_{1}(b,\delta). \] Alors : \begin{multline*} \tilde{N}_{d,1}^{(1)}(P_{1},P_{2})=\left(\sum_{\zz\in P_{2}\BB_{3}\cap \mathcal{A}_{1}^{\mu}(\ZZ) }\mathfrak{S}_{d,\zz}J_{\zz}|\zz|^{m-r}\right)d^{m-r-d_{1}}P_{1}^{m+1-d_{1}} \\ + O\left( d^{m-r}\max\{d^{\frac{(d_{1}-1)(r+1)}{2^{d_{1}-1}}+\varepsilon},d^{3-d_{1}}\}P_{1}^{m+1-d_{1}-\delta}P_{2}^{n-r+1-d_{2}}\right), \end{multline*} pour $ \delta>0 $ arbitrairement petit. 
\end{prop}
\begin{proof}
On sait, d'apr\`es le lemme \ref{lemmez} que :\[ \tilde{N}_{d,1}^{(1)}(P_{1},P_{2})=\left(\sum_{\zz\in P_{2}\BB_{3}\cap \mathcal{A}_{1}^{\mu}(\ZZ) }\mathfrak{S}_{d,\zz}J_{\zz}|\zz|^{m-r}\right)d^{m-r-d_{1}}P_{1}^{m+1-d_{1}}+ O(\mathcal{E}_{2})+O(\mathcal{E}_{3} ), \] avec \begin{align*}
\mathcal{E}_{2} & =d^{m-r+3-d_{1}}\sum_{l=1}^{P_{2}}\sum_{\substack{\zz\in P_{2}\BB_{3}\cap \mathcal{A}_{1}^{\mu}(\ZZ)\\ |\zz|=l} }l^{4d_{2}+m-r}P_{1}^{m+1-d_{1}-\eta(\theta)} \\ & \ll d^{m-r+3-d_{1}}P_{1}^{m+1-d_{1}-\eta(\theta)}P_{2}^{4d_{2}+n-r+1} \\  & =  d^{m-r+3-d_{1}}P_{1}^{m+1-d_{1}+\frac{5d_{2}}{b}-\eta(\theta)}P_{2}^{n-r+1-d_{2}}.
\end{align*}
et \begin{align*}
\mathcal{E}_{3} & =d^{m-r+\frac{(d_{1}-1)(r+1)}{2^{d_{1}-1}}+\varepsilon}\sum_{l=1}^{P_{2}}\sum_{\substack{\zz\in P_{1}\BB_{1}\cap \mathcal{A}_{1}^{\mu}(\ZZ)\\ |\zz|=l} }l^{3d_{2}+m-r+\varepsilon}P_{1}^{n-r+1-d_{1}+2\theta(d_{1}-1)-\frac{K_{1}\theta}{2}+\varepsilon}\\ & \ll d^{m-r+\frac{(d_{1}-1)(r+1)}{2^{d_{1}-1}}+\varepsilon}P_{1}^{m+1-d_{1}+2\theta(d_{1}-1)-\frac{K_{1}\theta}{2}+\frac{4d_{2}}{b}+2\varepsilon}P_{2}^{n-r+1-d_{2}}.
\end{align*}
Rappelons que $ \eta(\theta)=1-5\theta(d_{1}-1) $. On choisit alors \[ \theta=\frac{1}{5(d_{1}-1)}\left(1-\frac{5d_{2}}{b}-\delta\right), \] de sorte que \[ \frac{5d_{2}}{b}-\eta(\theta)=-\delta. \] On a par ailleurs, puisque $ 2\theta(d_{1}-1)-\frac{K_{1}\theta}{2}>g_{1}(b,\delta) $, \[ -2\theta(d_{1}-1)+\frac{K_{1}\theta}{2}>\left(\frac{4d_{2}}{b}+2\delta\right), \]  d'o\`u le r\'esultlat. 
\end{proof}

\subsection{Deuxi\`eme cas}

On suppose \`a pr\'esent $ l=\left\lfloor\frac{|\yy|}{d|\xx|}\right\rfloor\geqslant |\zz| $. 
Dans cette partie nous fixerons l'entier $ l $ et $ \zz $ de norme $ |\zz|\leqslant l $, et nous allons \'evaluer \begin{equation}
N_{d,l,\zz}(P_{1})=\Card\left\{ (\xx,\yy)\in \ZZ^{m+1}\; |\; |\xx|\leqslant P_{1}, \; dl|\xx|\leqslant |\yy| <d(l+1)|\xx| , \; F(d\xx,\yy,\zz)=0 \right\}.
\end{equation} 

Pour cela on introduit la s\'erie g\'en\'eratrice \begin{equation}
S_{d,l,\zz}(\alpha)=\sum_{|\xx|\leqslant P_{1}}\sum_{dl|\xx|\leqslant |\yy| <d(l+1)|\xx|}e\left( \alpha F(d\xx,\yy,\zz)\right),
\end{equation}
de sorte que $ N_{d,l,\zz}(P_{1})=\int_{0}^{1}S_{d,l,\zz}(\alpha)d\alpha $.\\

Les r\'esultats que nous obtiendrons dans cette section seront sensiblement identiques \`a ceux de la section pr\'ec\'edente, \`a quelques modifications pr\`es.

\subsubsection{Somme d'exponentielles}

Dans ce qui va suivre, pour un $ \yy $ donn\'e, on note $ N=\left\lfloor\frac{|\yy|}{dl}\right\rfloor $ et $ M=\left\lfloor\frac{|\yy|}{d(l+1)}\right\rfloor $. On a alors \[ dl|\xx|\leqslant |\yy| <d(l+1)|\xx| \Leftrightarrow M < |\xx| \leqslant N. \]
On note alors, pour $ N,M\in \{0,...,P_{1}\} $ : \begin{equation}
S_{d,N,M,l,\zz}(\alpha)=\sum_{M <|\xx|\leqslant N}\sum_{\substack{dNl\leqslant |\yy| <d(N+1)l \\ dM(l+1)\leqslant |\yy| < d(M+1)(l+1) }}e\left( \alpha F(d\xx,\yy,\zz)\right),
\end{equation}
et on a \begin{equation}
S_{d,l,\zz}(\alpha)=\sum_{N=1}^{P_{1}}\sum_{M=1}^{P_{1}}S_{d,N,M,l,\zz}(\alpha).
\end{equation}

En appliquant la m\'ethode de diff\'erencietion de Weyl des sections pr\'ec\'edentes on montre que : 

\begin{multline*}
|S_{d,N,M,l,\zz}(\alpha)|^{2^{d_{1}-1}}\ll \left(P_{1}^{r+1+\varepsilon}\right)^{2^{d_{1}-1}-d_{1}+1}\left((dlP_{1})^{m-r+\varepsilon}\right)^{2^{d_{1}-1}-d_{1}+1} \\  M_{d,\zz}\left(\alpha,P_{1},dlP_{1},P_{1}^{-1},(dlP_{1})^{-1}\right), \end{multline*}
o\`u $ M_{d,\zz}\left(\alpha,P_{1},dlP_{1},P_{1}^{-1},(dlP_{1})^{-1}\right) $ a \'et\'e d\'efini dans la section pr\'ec\'edente (cf. \eqref{Mz}). \\

Puis, en sommant sur $ M $ et $ N $, on en d\'eduit le lemme ci-dessous : \begin{lemma}
Pour tous $ P>1 $, $ \kappa>0 $ et tout $ \varepsilon>0 $ arbitrairement petit, l'une au moins des assertions suivantes est v\'erifi\'ee : \begin{enumerate}
\item $ |S_{d,l,\zz}(\alpha)|\ll d^{m-r+\frac{(d_{1}-1)(r+1)}{2^{d_{1}-1}}+\varepsilon}P_{1}^{m+3+\varepsilon}l^{m-r+\varepsilon}P^{-\kappa}, $ 
\item  $ M_{d,\zz}\left(\alpha,P_{1},dlP_{1},P_{1}^{-1},(dlP_{1})^{-1}\right)\gg \left(P_{1}^{r+1}\right)^{d_{1}-1}\left((dlP_{1})^{m-r}\right)^{d_{1}-1}P^{-2^{d_{1}-1}\kappa}. $
\end{enumerate}
\end{lemma}

Par les m\^emes arguments que ceux employ\'es dans la section pr\'ec\'edente, on en d\'eduit l'\'equivalent du lemme \ref{dilemme34},

\begin{lemma}\label{dilemme44}
Si $ \varepsilon>0 $ est un r\'eel arbitrairement petit, et si $ \zz\in \mathcal{A}_{1}^{\mu}(\ZZ) $, l'une au moins des assertions suivantes est v\'erifi\'ee : \begin{enumerate}
\item $ |S_{d,l,\zz}(\alpha)|\ll  d^{m-r+\frac{(d_{1}-1)(r+1)}{2^{d_{1}-1}}+\varepsilon}l^{m-r+\varepsilon}P_{1}^{m+3+\varepsilon}P^{-K_{1}\theta}, $ 
\item le r\'eel $ \alpha $ appartient \`a $ \mathfrak{M}^{l}(\theta) $,
\end{enumerate}
o\`u l'on a not\'e \begin{equation}
\mathfrak{M}^{l}(\theta)=\mathfrak{M}^{(1),l}(\theta)\cup\mathfrak{M}^{(2),l}(\theta)
\end{equation}\begin{equation}
\mathfrak{M}_{a,q}^{(1),l}(\theta)=\left\{ \alpha\in [0,1[ \; | \;  2|\alpha q-a|\leqslant d^{-(d_{1}-1)}P_{1}^{-d_{1}}P^{(d_{1}-1)\theta}\right\},
\end{equation}
\begin{equation}
\mathfrak{M}^{(1),l}(\theta)=\bigcup_{\substack{q\leqslant dl^{d_{2}}P^{(d_{1}-1)\theta} \\ d|q}}\bigcup_{\substack{0\leqslant a <q  }}\mathfrak{M}_{a,q}^{(1),l}(\theta).
\end{equation}
\begin{equation}
\mathfrak{M}_{a,q}^{(2),l}(\theta)=\left\{ \alpha\in [0,1[ \; | \;  2|\alpha q-a|\leqslant d^{-d_{1}}P_{1}^{-d_{1}}P^{(d_{1}-1)\theta}\right\},
\end{equation}
\begin{equation}
\mathfrak{M}^{(2),l}(\theta)=\bigcup_{q\leqslant l^{d_{2}}P^{(d_{1}-1)\theta}}\bigcup_{\substack{0\leqslant a <q \\ \PGCD(a,q)=1 }}\mathfrak{M}_{a,q}^{(2),l}(\theta).
\end{equation}

\end{lemma}
\`A partir d'ici, on fixe \`a nouveau $ P=P_{1} $.

\subsubsection{M\'ethode du cercle}

Pour les arcs mineurs, les calculs effectu\'es pour \'etablir le lemme \ref{arcsmin1}

\begin{lemma}\label{arcsmin1bis}
Pour tout $ \zz\in \mathcal{A}_{1}^{\mu}(\ZZ) $, on a la  majoration : \[ \int_{\alpha\notin \mathfrak{M}^{l}(\theta)}|S_{d,l,\zz}(\alpha)|d\alpha \ll d^{m-r+\varepsilon+\frac{(d_{1}-1)(r+1)}{2^{d_{1}-1}}}l^{d_{2}+m-r+\varepsilon}P_{1}^{m+3-d_{1}-\Delta_{1}(\theta,K_{1})+\varepsilon}. \] \end{lemma}

Pour les arcs majeurs, on a les \'equivalents des lemmes \ref{separation1} et \ref{intermediairez} :

\begin{lemma}\label{separation1bis}
On a l'estimation \begin{multline*}
S_{d,l,\zz}(\alpha)=d^{m-r}l^{m-r}P_{1}^{m+1}q^{-(m+1)}S_{a,q,d}(\zz)I_{l,\zz}(d^{d_{1}}P_{1}^{d_{1}}\beta) \\ +O\left( d^{m-r+1}l^{2d_{2}+m-r}P_{1}^{m+2\theta(d_{1}-1)} \right), \end{multline*} avec \begin{equation}
S_{d,a,q}(\zz)=\sum_{(\bb_{1},\bb_{2})\in (\ZZ/q\ZZ)^{r+1}\times (\ZZ/q\ZZ)^{m-r}}e\left(\frac{a}{q}F(d\bb_{1},\bb_{2},\zz)\right),
\end{equation} \begin{equation}
I_{l,\zz}(\beta)=\int_{\substack{(\uu,\vv)\in [-1,1]^{r+1}\times [-1,1]^{m-r} \\ |\uu|\leqslant |\vv|<\left(1+\frac{1}{l}\right) |\uu|}}e\left(\beta F(\uu,l\vv,\zz)\right)d\uu d\vv. 
\end{equation}
\end{lemma}
 \begin{lemma}\label{intermediairezbis} Pour $ \zz\in \mathcal{A}_{1}^{\mu}(\ZZ) $, et $ \varepsilon>0 $ arbitrairement petit, on a l'estimation suivante : \begin{multline*}
N_{d,l,\zz}(P_{1})=d^{m-r-d_{1}}l^{m-r}P_{1}^{m+1-d_{1}}\mathfrak{S}_{d,\zz}(\phi_{1}(d,l,\theta))J_{l,\zz}(\tilde{\phi}_{1}(\theta)) \\ + O\left(d^{m-r+\varepsilon+\frac{(d_{1}-1)(r+1)}{2^{d_{1}-1}}}l^{d_{2}+m-r+\varepsilon}P_{1}^{m+3-d_{1}-\Delta_{1}(\theta,K_{1})+\varepsilon}\right) \\ + O\left(d^{m-r+3-d_{1}}l^{4d_{2}+m-r}P_{1}^{m+1-d_{1}-\eta(\theta)}\right), \end{multline*} o\`u \begin{equation}
\mathfrak{S}_{d,\zz}(\phi(d,l,\theta))=\sum_{q\leqslant \phi(d,l,\theta)}q^{-(m+1)}\sum_{\substack{0\leqslant a < q \\ \PGCD(a,q)=1}}S_{a,q,d}(\zz),
\end{equation}\begin{equation}
J_{l,\zz}(\tilde{\phi}(\theta))=\int_{|\beta|\leqslant \tilde{\phi}(\theta)}I_{l,\zz}(\beta)d\beta. 
\end{equation}
\end{lemma}

Si l'on note : 

\begin{equation}
J_{l,\zz}=\int_{\RR}I_{l,\zz}(\beta)d\beta,
\end{equation}
on montre alors comme pour le lemme \ref{Jz} :
\begin{lemma}\label{Jzbis}
Soit $ \zz\in \mathcal{A}_{1}^{\mu}(\ZZ) $, et $ \varepsilon>0 $ arbitrairement petit. On suppose de plus que $ d_{1}\geqslant 2 $. L'int\'egrale $ J_{l,\zz} $ est absolument convergente, et on a : \[ |J_{l,\zz}(\tilde{\phi}(\theta))-J_{l,\zz}|\ll l^{\frac{4d_{2}}{3}+\varepsilon}P_{1}^{-\frac{K_{1}}{3}+\frac{7}{3}\theta(d_{1}-1)+\varepsilon} . \] On a de plus $ |J_{\zz}|\ll l^{\frac{4d_{2}}{3}+\varepsilon}. $
\end{lemma}
et on d\'eduit des lemmes\;\ref{intermediairezbis},\;\ref{Jzbis} et\;\ref{Sz} :
\begin{lemma}\label{lemmezbis}
Soit  $ \zz\in \mathcal{A}_{1}^{\mu}(\ZZ) $. On suppose fix\'es $ \theta\in [0,1] $ et $ P_{1}\geqslant 1  $ tels que $ l^{2d_{2}}P_{1}^{-d_{1}+3\theta(d_{1}-1)}<1 $. On suppose de plus que $ K_{1}>7(d_{1}-1) $ et $ d_{1}\geqslant 2 $. On a alors que \[ N_{d,l,\zz}(P_{1})=\mathfrak{S}_{d,\zz}J_{l,\zz}d^{m-r-d_{1}}l^{m-r}P_{1}^{m+1-d_{1}}+ O(E_{2}) +O(E_{3}), \] avec \[ E_{2}=d^{m-r+3-d_{1}}l^{4d_{2}+m-r}P^{m+1-d_{1}-\eta(\theta)}, \]\[ E_{3}=d^{m-r+\frac{(d_{1}-1)(r+1)}{2^{d_{1}-1}}+\varepsilon}l^{\frac{10d_{2}}{3}+m-r+\varepsilon}P_{1}^{m+1-d_{1}+\frac{7}{3}\theta(d_{1}-1)-\frac{K_{1}\theta}{3}+\varepsilon} \]
et $ \varepsilon>0 $ arbitrairement petit. 
\end{lemma}
\begin{cor}\label{corollairezbis}
Soit  $ \zz\in \mathcal{A}_{1}^{\mu}(\ZZ) $. On suppose que $ K_{1}>7(d_{1}-1) $ et $ d_{1}\geqslant 2 $. Il existe alors un r\'eel $ \delta>0 $ arbitrairement petit tel que : \begin{multline*} N_{d,l,\zz}(P_{1})=\mathfrak{S}_{d,\zz}J_{l,\zz}d^{m-r-d_{1}}l^{m-r}P_{1}^{m+1-d_{1}} \\ +O\left( d^{m-r}\max\{d^{\frac{(d_{1}-1)(r+1)}{2^{d_{1}-1}}+\varepsilon},d^{3-d_{1}}\}l^{m-r+4d_{2}}P_{1}^{m+1-d_{1}-\delta}\right), \end{multline*} uniform\'ement pour tout $ l<P_{1}^{\frac{d_{1}-1}{2d_{2}}} $. 
\end{cor}
On pose \`a pr\'esent $ P_{1}=P_{2}^{b} $, avec $ b\geqslant 1 $ et on introduit la fonction \begin{equation}
g_{1}'(b,\delta)=\left(1-\frac{5d_{2}}{b}-\delta\right)^{-1}5(d_{1}-1)\left(\frac{10d_{2}}{3b}+2\delta\right),
\end{equation} ainsi que \begin{multline}
\widetilde{N}_{d,1}^{(2)}(P_{1},P_{2})=\Card\left\{ (\xx,\yy,\zz)\in \ZZ^{n+2} \; | \; \zz \in \mathcal{A}_{1}^{\mu}(\ZZ), \; |\xx|\leqslant P_{1}, \;\right. \\ \left. |\yy| \leqslant d|\xx|P_{2},\;   |\zz|\leqslant P_{2}, \; \left\lfloor\frac{|\yy|}{d|\xx|}\right\rfloor\geqslant |\zz|,  \; F(d\xx,\yy,\zz)=0 \right\}.
\end{multline}
On a alors la proposition suivante qui est l'analogue de \ref{propz} : 
\begin{prop}\label{propzbis}
On suppose $ K_{1}>7(d_{1}-1) $, $ d_{1}\geqslant 2 $, $ P_{1}=P_{2}^{b} $ et de plus que \[ \frac{K_{1}\theta}{3}-\frac{7}{3}\theta(d_{1}-1)>g_{1}'(b,\delta). \] Alors : \begin{multline*} \widetilde{N}_{d,1}^{(2)}(P_{1},P_{2})=d^{m-r-d_{1}}\left(\sum_{l=1}^{P_{2}}\sum_{\substack{\zz\in P_{2}\BB_{3}\cap \mathcal{A}_{1}^{\mu}(\ZZ)\\ |\zz|\leqslant l }}\mathfrak{S}_{\zz}J_{l,\zz}l^{m-r}\right)P_{1}^{m+1-d_{1}} \\ + O\left(d^{m-r}\max\{d^{\frac{(d_{1}-1)(r+1)}{2^{d_{1}-1}}+\varepsilon},d^{3-d_{1}}\}P_{1}^{m+1-d_{1}-\delta}P_{2}^{n-r+1-d_{2}}\right), \end{multline*} pour $ \delta>0 $ arbitrairement petit. 
\end{prop}

\section{Quatri\`eme \'etape}

L'objectif est \`a pr\'esent de regrouper les r\'esultats obtenus pour en d\'eduire une formule asymptotique pour $ N_{d}(P_{1},P_{2}) $ avec $ n $ assez grand et $ P_{1},P_{2} $ quelconques. \\

On d\'efinit dans un premier temps $ b_{1} $ comme le r\'eel minimisant la fonction \begin{multline}
b\mt \max\{ 2^{\tilde{d}}(bd_{1}+d_{2}), 2^{\tilde{d}}(5b+2)(\tilde{d}+1), 2^{d_{1}-1}(4(d_{1}-1)+2g_{1}(b,\delta)+\lceil bd_{1}+d_{2}+\delta\rceil) \\  2^{d_{1}-1}(7(d_{1}-1)+3g_{1}'(b,\delta)+\lceil bd_{1}+d_{2}+\delta\rceil) \}
\end{multline}et on notera $ m_{1} $ le minimum correspondant. On d\'efinit de m\^eme  $ u_{1} $ le r\'eel minimisant \begin{equation}
u \mt \max\{ 2^{\tilde{d}}(d_{1}+ud_{2}),7.2^{\tilde{d}}(\tilde{d}+1), 2^{d_{2}-1}(2(d_{2}-1)+g_{2}(u,\delta)+\lceil d_{1}+ud_{2}+\delta\rceil)\}
\end{equation} et $ m_{2} $ le minimum correspondant. On note par ailleurs $ m=\max\{m_{1},m_{2}\} $.\\
Un calcul en $ b=10d_{2} $ et $ u=10d_{1} $ montre que \begin{equation}
 2^{d_{1}+d_{2}}\leqslant m\leqslant 13d_{2}(d_{1}+d_{2})2^{d_{1}+d_{2}}. 
\end{equation}
\`A partir d'ici on fixe \begin{equation}
\mu=\lceil b_{1}d_{1}+d_{2}+\delta\rceil, \; \; \; \lambda=\lceil d_{1}+u_{1}d_{2}+\delta\rceil
\end{equation}

On commence par \'etablir le lemme suivant : 

\begin{lemma}\label{lemmeunif}
On suppose $ n+2-\max\{\dim V_{1}^{\ast},\dim V_{2}^{\ast}\} >m $. 
On a alors pour tout $ P_{2}\geqslant 1 $ : \begin{multline*} \sum_{\zz\in P_{2}\BB_{3}\cap \mathcal{A}_{1}^{\mu}(\ZZ) }\mathfrak{S}_{d,\zz}J_{\zz}d^{m-r-d_{1}}|\zz|^{m-r}+\sum_{l=1}^{P_{2}}\sum_{\substack{\zz\in P_{2}\BB_{3}\cap \mathcal{A}_{1}^{\mu}(\ZZ)\\ |\zz|\leqslant l }}\mathfrak{S}_{d,\zz}J_{l,\zz}d^{m-r-d_{1}}l^{m-r} \\ =d^{m-r-d_{1}}\mathfrak{S}_{d}JP_{2}^{n-r+1-d_{2}}+ O(d^{m-r}\max\{d^{\frac{(r+1)(d_{1}-1)}{2^{d_{1}-1}}+\varepsilon}, d^{3-d_{1}}\}P_{2}^{n-r+1-d_{2}-\delta}).  \end{multline*}
\end{lemma}

\begin{proof}
On choisit dans un premier temps $ P_{1} $ tel que $ P_{1}=P_{2}^{b_{1}} $. D'apr\`es les propositions \ref{propz} et \ref{propzbis}, on a \begin{multline} 
\widetilde{N}_{d,1}^{(1)}(P_{1},P_{2})+\widetilde{N}_{d,1}^{(2)}(P_{1},P_{2}) \\ =\left(\sum_{\zz\in P_{2}\BB_{3}\cap \mathcal{A}_{1}^{\mu}(\ZZ) }\mathfrak{S}_{d,\zz}J_{\zz}|\zz|^{m-r}+\sum_{l=1}^{P_{2}}\sum_{\substack{\zz\in P_{2}\BB_{3}\cap \mathcal{A}_{1}^{\mu}(\ZZ)\\ |\zz|\leqslant l }}\mathfrak{S}_{d,\zz}J_{l,\zz}l^{m-r}\right)d^{m-r-d_{1}}P_{1}^{m+1-d_{1}} \\ + O\left(d^{m-r}\max\{d^{\frac{(r+1)(d_{1}-1)}{2^{d_{1}-1}}+\varepsilon}, d^{3-d_{1}}\}P_{1}^{m+1-d_{1}-\delta}P_{2}^{n-r+1-d_{2}}\right) \end{multline}
Notons \`a pr\'esent \begin{multline} \tilde{N}_{d,1}(P_{1},P_{2})=\Card\left\{ (\xx,\yy,\zz)\in \ZZ^{n+2} \; | \; \zz \in \mathcal{A}_{1}^{\mu}(\ZZ), \; |\xx|\leqslant P_{1}, \;\right. \\ \left.  \max\left(\left\lfloor\frac{|\yy|}{d|\xx|}\right\rfloor,|\zz|\right)\leqslant P_{2} ,  \; F(d\xx,\yy,\zz)=0 \right\}.    \end{multline}
et \begin{multline} N_{d,1}(P_{1},P_{2})=\Card\left\{ (\xx,\yy,\zz)\in \ZZ^{n+2} \; | \; \zz \in \mathcal{A}_{1}^{\mu}(\ZZ), \; |\xx|\leqslant P_{1}, \;\right. \\ \left.  \max\left(\frac{|\yy|}{d|\xx|},|\zz|\right)\leqslant P_{2} ,  \; F(d\xx,\yy,\zz)=0 \right\}.    \end{multline}
On remarque d'une part que \[ N_{d,1}(P_{1},P_{2})\leqslant \widetilde{N}_{d,1}^{(1)}(P_{1},P_{2})+\widetilde{N}_{d,1}^{(2)}(P_{1},P_{2})=\widetilde{N}_{d,1}(P_{1},P_{2}) \leqslant N_{d,1}(P_{1},P_{2}+1), \] et d'autre part, en utilisant la proposition \ref{propA1} : \begin{align*}
N_{d,1}(P_{1},P_{2})& = N_{d}(P_{1},P_{2})+O\left(\sum_{\zz\in P_{2}\BB_{3}\cap (\mathcal{A}_{1}^{\mu})^{c}(\ZZ)}d^{m-r}P_{1}^{m+1}P_{2}^{m-r}\right) \\ & =N_{d}(P_{1},P_{2})+O\left(d^{m-r}P_{1}^{m+1}P_{2}^{m-r}P_{2}^{n-r+1-\mu}\right) \\ &  =N_{d}(P_{1},P_{2})+O\left(d^{m-r}P_{1}^{m+1-d_{1}}P_{2}^{n-r+1-d_{2}-\delta}\right),
\end{align*} 
par d\'efinition de $ \mu $.\\

Par ailleurs, \'etant donn\'e que $ n+2-\max\{\dim V_{1}^{\ast},\dim V_{2}^{\ast}\}>2^{\tilde{d}}(5b_{1}+2)(\tilde{d}+1) $, la proposition \ref{propgeneral} donne : \begin{multline*} N_{d}(P_{1},P_{2})=\sigma_{d} d^{m-r-d_{1}}P_{1}^{m+1-d_{1}}P_{2}^{n-r+1-d_{2}} \\ +O\left(d^{m-r}\max\{d^{d_{1}(r+1)/2^{\tilde{d}}+\varepsilon},d^{3-d_{1}}\}P_{1}^{m+1-d_{1}-\delta}P_{2}^{n-r+1-d_{2}-\delta}\right). \end{multline*}
On a donc que \begin{multline*}
|N_{d,1}^{(1)}(P_{1},P_{2})+\widetilde{N}_{d,1}^{(2)}(P_{1},P_{2})-N_{d}(P_{1},P_{2})| \\ \ll N_{d}(P_{1},P_{2}+1)-N_{d}(P_{1},P_{2})+O\left(d^{m-r}P_{1}^{m+1-d_{1}}P_{2}^{n-r+1-d_{2}-\delta}\right) \\ \ll \sigma_{d} P_{1}^{m+1-d_{1}}((P_{2}+1)^{n-r+1-d_{2}}-P_{2}^{n-r+1-d_{2}}) \\ +O\left(d^{m-r}\max\{d^{d_{1}(r+1)/2^{\tilde{d}}+\varepsilon},d^{3-d_{1}}\}P_{1}^{m+1-d_{1}}P_{2}^{n-r+1-d_{2}-\delta}\right) \\ \ll d^{m-r}\max\{d^{d_{1}(r+1)/2^{\tilde{d}}+\varepsilon},d^{3-d_{1}}\}P_{1}^{m+1-d_{1}}P_{2}^{n-r+1-d_{2}-\delta},
\end{multline*} \'etant donn\'e que $ \sigma_{d}=d^{m-r-d_{1}}\mathfrak{S}_{d}J \ll d^{m-r-d_{1}}\max\{d^{\frac{d_{1}(r+1)}{2^{\tilde{d}}}},d^{2}\} $, d'apr\`es la remarque \ref{remseriesing}. \\

Par cons\'equent, \begin{multline*} \left(\sum_{\zz\in P_{2}\BB_{3}\cap \mathcal{A}_{1}^{\mu}(\ZZ) }\mathfrak{S}_{d,\zz}J_{\zz}|\zz|^{m-r}+\sum_{l=1}^{P_{2}}\sum_{\substack{\zz\in P_{2}\BB_{3}\cap \mathcal{A}_{1}^{\mu}(\ZZ)\\ |\zz|\leqslant l }}\mathfrak{S}_{d,\zz}J_{l,\zz}l^{m-r}\right)d^{m-r-d_{1}}P_{1}^{m+1-d_{1}} \\ = \sigma_{d} P_{1}^{m+1-d_{1}} P_{2}^{n-r+1-d_{2}} \\ + O\left(d^{m-r}\max\{d^{\frac{(r+1)(d_{1}-1)}{2^{d_{1}-1}}+\varepsilon}, d^{3-d_{1}}\}P_{1}^{m+1-d_{1}}P_{2}^{n-r+1-d_{2}-\delta}\right) \end{multline*} en simplifiant par $ P_{1}^{m+1-d_{1}} $ on obtient le r\'esultat. 
\end{proof}
On d\'emontre de m\^eme : 
\begin{lemma}
On suppose $ n+2-\max\{\dim V_{1}^{\ast},\dim V_{2}^{\ast}\} >m $. 
Alors pour tout $ P_{1}\geqslant 1 $ et $ d_{2}\geqslant 2 $ : \begin{multline*} \sum_{\xx\in P_{1}\BB_{1}\cap \mathcal{A}_{2}^{\lambda}(\ZZ) }\mathfrak{S}_{d,\xx}J_{d,\xx}d^{m-r}|\xx|^{m-r}=\sigma_{d} P_{1}^{m+1-d_{1}} \\ + O(d^{m-r}\max\{d^{d_{1}(r+1)/2^{\tilde{d}}+\varepsilon},d^{4d_{1}}\}P_{1}^{m+1-d_{1}-\delta}).  \end{multline*}
\end{lemma}
Pour le cas $ d_{2}=1 $, en notant $ u_{1}'=d_{1}+\delta $, $  m_{2}' =7d_{1}2^{d_{1}-1} $ et $ \lambda'=\lceil d_{1}+u_{1}'+\delta\rceil $ on trouve : 
\begin{lemma}
On suppose $ n+2-\max\{\dim V_{1}^{\ast},\dim V_{2}^{\ast}\}>m'=\max\{m_{1},m_{2}'\} $. Alors si $ d_{2}=1 $ et $ P_{1}\geqslant 1 $ : \begin{multline*} \sum_{\xx\in P_{1}\BB_{1}\cap \mathcal{A}_{2}^{\lambda}(\ZZ) }\frac{\Vol(C_{d,\xx})}{\det(\Lambda_{d,\xx})}=\sigma_{d} P_{1}^{m+1-d_{1}} \\ + O(d^{m-r}\max\{d^{d_{1}(r+1)/2^{d_{1}-1}+\varepsilon},d^{3-d_{1}}\}P_{1}^{m+1-d_{1}-\delta}).  \end{multline*}
\end{lemma}

Nous sommes en mesure de d\'emontrer la proposition suivante : 
\begin{prop}\label{propN1}
On suppose $ d_{1}\geqslant 2 $, $ P_{1}\geqslant P_{2} $ et $ n+2-\max\{\dim V_{1}^{\ast},\dim V_{2}^{\ast}\} >m $. On a alors \begin{multline*}\widetilde{N}_{d,1}(P_{1},P_{2})=\sigma_{d} P_{1}^{m+1-d_{1}} P_{2}^{n-r+1-d_{2}} \\ +O\left(d^{m-r}\max\{d^{\frac{(r+1)(d_{1}-1)}{2^{d_{1}-1}}+\varepsilon}, d^{3-d_{1}}\}P_{1}^{m+1-d_{1}}P_{2}^{n-r+1-d_{2}-\delta}\right). \end{multline*}
\end{prop}
\begin{proof}
On suppose dans un premier temps que $ b\geqslant b_{1} $. On a alors puisque $ n+2-\max\{\dim V_{1}^{\ast},\dim V_{2}^{\ast}\} >m  $ et puisque les fonctions $ g_{1} $ et $ g_{1}' $ sont d\'ecroissantes en $ b $ :  \[ \frac{K_{1}}{2}-2(d_{1}-1)>g_{1}(b_{1},\delta)>g_{1}(b,\delta), \] \[ \frac{K_{1}}{3}-\frac{7}{3}(d_{1}-1)>g_{1}'(b_{1},\delta)>g_{1}'(b,\delta). \] 
Par cons\'equent on peut appliquer les propositions \ref{propz} et \ref{propzbis} et on a alors  \begin{multline*} 
N_{d,1}^{(1)}(P_{1},P_{2})+\widetilde{N}_{d,1}^{(2)}(P_{1},P_{2}) \\ =\left(\sum_{\zz\in P_{2}\BB_{3}\cap \mathcal{A}_{1}^{\mu}(\ZZ) }\mathfrak{S}_{d,\zz}J_{\zz}|\zz|^{m-r}+\sum_{l=1}^{P_{2}}\sum_{\substack{\zz\in P_{2}\BB_{3}\cap \mathcal{A}_{1}^{\mu}(\ZZ)\\ |\zz|\leqslant l }}\mathfrak{S}_{d,\zz}J_{l,\zz}l^{m-r}\right)d^{m-r-d_{1}}P_{1}^{m+1-d_{1}} \\ + O\left(d^{m-r}\max\{d^{\frac{(r+1)(d_{1}-1)}{2^{d_{1}-1}}+\varepsilon}, d^{3-d_{1}}\}P_{1}^{m+1-d_{1}-\delta}P_{2}^{n-r+1-d_{2}}\right) \\ =\sigma_{d} P_{1}^{m+1-d_{1}} P_{2}^{n-r+1-d_{2}} \\ +O\left(d^{m-r}\max\{d^{\frac{(r+1)(d_{1}-1)}{2^{d_{1}-1}}+\varepsilon}, d^{3-d_{1}}\}P_{1}^{m+1-d_{1}}P_{2}^{n-r+1-d_{2}-\delta}\right)
\end{multline*}
en utilisant le lemme pr\'ec\'edent. 
On remarque d'autre part que \[ \widetilde{N}_{d,1}^{(1)}(P_{1},P_{2})+\widetilde{N}_{d,1}^{(2)}(P_{1},P_{2})= \widetilde{N}_{d,1}(P_{1},P_{2}) \] et ainsi que  \begin{multline*} \widetilde{N}_{d,1}(P_{1},P_{2})=\sigma_{d} P_{1}^{m+1-d_{1}} P_{2}^{n-r+1-d_{2}} \\ +O\left(d^{m-r}\max\{d^{\frac{(r+1)(d_{1}-1)}{2^{d_{1}-1}}+\varepsilon}, d^{3-d_{1}}\}P_{1}^{m+1-d_{1}}P_{2}^{n-r+1-d_{2}-\delta}\right). \end{multline*}
Si l'on suppose \`a pr\'esent $ b< b_{1} $, on a alors \[ K>\max\{ b_{1}d_{1}+d_{2}, (5b_{1}+2)(\tilde{d}+1)>\max\{ bd_{1}+d_{2}, (5b+2)(\tilde{d}+1)\}. \]
Par la proposition \ref{propgeneral}, on a donc \begin{multline*} N_{d}(P_{1},P_{2})=\sigma_{d} P_{1}^{m+1-d_{1}}P_{2}^{n-r+1-d_{2}} \\ +O\left(d^{m-r}\max\{d^{\frac{(r+1)(d_{1}-1)}{2^{d_{1}-1}}+\varepsilon}, d^{3-d_{1}}\}P_{1}^{m+1-d_{1}-\delta}P_{2}^{n-r+1-d_{2}-\delta}\right). \end{multline*}
Or comme dans la d\'emonstration du lemme \ref{lemmeunif}, on a que : \begin{multline*} \widetilde{N}_{d,1}(P_{1},P_{2})=N_{d}(P_{1},P_{2}) \\ +O\left(d^{m-r}\max\{d^{\frac{(r+1)(d_{1}-1)}{2^{d_{1}-1}}+\varepsilon}, d^{3-d_{1}}\}P_{1}^{m+1-d_{1}}P_{2}^{n-r+1-d_{2}-\delta}\right), \end{multline*} ce qui cl\^ot la d\'emonstration. 
\end{proof}
Si l'on note \begin{multline} \tilde{N}_{d,2}(P_{1},P_{2})=\Card\left\{ (\xx,\yy,\zz)\in \ZZ^{n+2} \; | \; \xx \in \mathcal{A}_{2}^{\lambda}(\ZZ), \; |\xx|\leqslant P_{1}, \;\right. \\ \left.  \max\left(\left\lfloor\frac{|\yy|}{d|\xx|}\right\rfloor,|\zz|\right)\leqslant P_{2} ,  \; F(d\xx,\yy,\zz)=0 \right\},    \end{multline} on a un r\'esultat analogue :  

\begin{prop}\label{propN2}
Si l'on suppose $ d_{1},d_{2}\geqslant 2 $, $ P_{1}\leqslant P_{2} $ et $ n+2-\linebreak\max\{\dim V_{1}^{\ast},\dim V_{2}^{\ast}\} >m $, On a alors \begin{multline*}\widetilde{N}_{d,2}(P_{1},P_{2})=\sigma_{d} P_{1}^{m+1-d_{1}} P_{2}^{n-r+1-d_{2}} \\ +O\left(d^{m-r}\max\{d^{d_{1}(r+1)/2^{\tilde{d}}+\varepsilon},d^{5d_{1}}\}P_{1}^{m+1-d_{1}-\delta}P_{2}^{n-r+1-d_{2}}\right). \end{multline*}
Si l'on a $ d_{1}\geqslant 2 $, $ d_{2}=1 $ , $ P_{1}\leqslant P_{2} $ et $ n+2-\max\{\dim V_{1}^{\ast},\dim V_{2}^{\ast}\} >m' $,  alors \begin{multline*}\widetilde{N}_{d,2}(P_{1},P_{2})=\sigma_{d} P_{1}^{m+1-d_{1}} P_{2}^{n-r} \\ +O\left(d^{m-r}\max\{d^{d_{1}(r+1)/2^{d_{1}-1}+\varepsilon},d^{3-d_{1}}\}P_{1}^{m+1-d_{1}-\delta}P_{2}^{n-r}\right). \end{multline*}
\end{prop}
Consid\'erons \`a pr\'esent l'ouvert de Zariski \begin{equation}
U=\mathcal{A}_{2}^{\lambda}\times \AA_{\CC}^{m-r}\times \mathcal{A}_{1}^{\mu}\subset \AA_{\CC}^{n+2}. 
\end{equation}
On note alors 
 \begin{multline} \widetilde{N}_{d,U}(P_{1},P_{2})=\Card\left\{ (\xx,\yy,\zz)\in \ZZ^{n+2}\cap U \; | \; |\xx|\leqslant P_{1},  \;\right. \\ \left. \;  \max\left(\left\lfloor\frac{|\yy|}{d|\xx|}\right\rfloor,|\zz|\right)\leqslant P_{2} ,  \; F(d\xx,\yy,\zz)=0 \right\},    \end{multline} 
On en d\'eduit : \begin{prop}\label{propNU}
Si l'on suppose $ d_{1},d_{2}\geqslant 2 $ et $ n+2-  \max\{\dim V_{1}^{\ast},\dim V_{2}^{\ast}\} >m $, on a alors \begin{multline*}\widetilde{N}_{d,U}(P_{1},P_{2})=\sigma_{d} P_{1}^{m+1-d_{1}} P_{2}^{n-r+1-d_{2}} \\ +O_{\delta}\left(d^{m-r}\max\{d^{\frac{(r+1)(d_{1}-1)}{2^{d_{1}-1}}+\varepsilon}, d^{5d_{1}}\}P_{1}^{m+1-d_{1}}P_{2}^{n-r+1-d_{2}}\min\{P_{1},P_{2}\}^{-\delta}\right), \end{multline*} pour $ \delta>0 $ arbitrairement petit. Pour $ d_{1}\geqslant 2 $, $ d_{2}=1 $ , $ P_{1}\leqslant P_{2} $ et $ n+2-\max\{\dim V_{1}^{\ast},\dim V_{2}^{\ast}\} >m' $, on a \begin{multline*} \widetilde{N}_{d,U}(P_{1},P_{2})=\sigma_{d} P_{1}^{m+1-d_{1}} P_{2}^{n-r} \\ +O_{\delta}\left(d^{m-r}\max\{d^{\frac{d_{1}(r+1)}{2^{d_{1}-1}}+\varepsilon}, d^{3-d_{1}}\}P_{1}^{m+1-d_{1}}P_{2}^{n-r}\min\{P_{1},P_{2}\}^{-\delta}\right). \end{multline*}
\end{prop}
\begin{proof}
On suppose $ P_{1}\geqslant P_{2} $. On \'evalue le terme d'erreur \begin{align*}
| \widetilde{N}_{d,U}(P_{1},P_{2})-\tilde{N}_{d,1}(P_{1},P_{2})| & \ll \sum_{\xx\in \mathcal{A}_{2}^{\lambda}(\ZZ)\cap P_{1}\BB_{1}}d^{m-r}P_{1}^{m-r}P_{2}^{n-r+1} \\ & \ll d^{m-r}P_{1}^{m+1-\lambda}P_{2}^{n-r+1} \\ & \ll d^{m-r}P_{1}^{m+1-d_{1}-u_{1}d_{2}-\delta}P_{2}^{n-r+1} \\ & \ll d^{m-r}P_{1}^{m+1-d_{1}-\delta}P_{2}^{n-r+1-d_{2}}
\end{align*}
car $ u_{1}\geqslant 1 $. Pour $ P_{1}\leqslant P_{2} $, on obtient le m\^eme r\'esultat avec $ | \widetilde{N}_{d,U}(P_{1},P_{2})-\tilde{N}_{d,2}(P_{1},P_{2})|  $. 
\end{proof}

 \section{Cinqui\`eme \'etape}

\subsection{Un r\'esultat interm\'ediaire}
Nous allons \`a pr\'esent utiliser la formule obtenue pour $  \widetilde{N}_{U}(P_{1},P_{2}) $ dans la proposition \ref{propNU} pour trouver une formule asymptotique pour $ N_{d,U}(B) $. Pour r\'esoudre ce probl\`eme, nous allons appliquer une version l\'eg\`erement modifi\'ee (tenant compte de la d\'ependance en $ d $ des fonctions de comptage) de la m\'ethode  d\'evelopp\'ee par Blomer et Br\"{u}dern dans \cite{BB} pour le cas les hypersurfaces diagonales des espaces multiprojectifs, et reprise dans la section $ 9 $ de \cite{S2}. \\
 
 Pour tout $ d\in \NN^{\ast} $, on consid\`ere une fonction $ f_{d} : \NN^{2}\ra [0,+\infty[ $. Conform\'ement aux notations de \cite{BB}, on dira que $ f_{d} $ est une $ (\beta_{1},\beta_{2},C_{d},D,\alpha, $ $\upsilon,\delta) $-fonction si elle v\'erifie les trois conditions suivantes : \begin{enumerate}
\item \label{I} On a \[ \sum_{\substack{k\leqslant K \\ l\leqslant L }}f_{d}(k,l)=C_{d}K^{\beta_{1}}L^{\beta_{2}}+O(d^{\upsilon}K^{\beta_{1}}L^{\beta_{2}}\min\{K,L\}^{-\delta}) \] pour tous $ K,L\geqslant 1 $. 
\item \label{II}Il existe des fonctions $ c_{1,d},c_{2,d} : \NN \ra [0,+\infty[ $ telles que : \[ \sum_{\substack{l \leqslant L }}f_{d}(k,l)=c_{d,1}(k)L^{\beta_{2}}+O\left(k^{D}d^{\upsilon}L^{\beta_{2}-\delta}\right), \] uniform\'ement pour tous $ L\geqslant 1 $ et $ k\leqslant d^{-1}L^{\alpha} $, \[ \sum_{\substack{k \leqslant K }}f_{d}(k,l)=c_{d,2}(l)K^{\beta_{1}}+O\left(l^{D}d^{\upsilon}K^{\beta_{1}-\delta}\right), \] uniform\'ement pour tous $ K\geqslant 1 $ et $ l\leqslant d^{-1}K^{\alpha} $.
\end{enumerate}

 Nous allons alors d\'emontrer, en nous inspirant des arguments de \cite[\S 9]{S2}, la proposition suivante qui est une adaptation de \cite[Th\'eor\`eme 2.1]{BB} pour le cas d'une famille de fonctions d\'ependant d'un param\`etre $ d $ : \begin{prop}\label{propfin}
Si $ (f_{d})_{d\in \NN^{\ast}} $ est une famille de $ (\beta_{1},\beta_{2},C_{d},D,\alpha,\upsilon,\delta) $-fonctions avec $ (C_{d})_{d\in \NN^{\ast}} $ telle que $ C_{d}\ll d^{\upsilon} $, alors on a, pour tout $ d $, la formule asymptotique : \[ \sum_{k^{\beta_{1}}l^{\beta_{2}}\leqslant P}f_{d}(k,l)=C_{d}P\log(P)+O(d^{\upsilon+\delta}\log(d)P). \]
\end{prop}

On consid\`ere $ (f_{d})_{d\in \NN^{\ast}} $ une famille de $ (\beta_{1},\beta_{2},C,D,\alpha,\upsilon,\delta) $-fonctions, avec $ C_{d}\ll d^{\upsilon} $ et on d\'efinit \[ F_{d}(K,L)=\sum_{k\leqslant K}\sum_{l\leqslant L}f_{d}(k,l). \] 

\begin{lemma}
Pour tout $ d\in \NN^{\ast} $, on a les estimations : \[ \sum_{k\leqslant K}c_{d,1}(k)=C_{d}K^{\beta_{1}}+ O\left(d^{\upsilon}K^{\beta_{1}-\delta}\right), \]\[ \sum_{l\leqslant L}c_{d,1}(l)=C_{d}L^{\beta_{1}}+ O\left(d^{\upsilon}L^{\beta_{2}-\delta}\right). \]
\end{lemma}
\begin{proof}
D'apr\`es la condition \ref{I}, on a \begin{equation}\label{FonctionFd}
F_{d}(K,L)=C_{d}K^{\beta_{1}}L^{\beta_{2}}+O(d^{\upsilon}K^{\beta_{1}}L^{\beta_{2}}\min\{K,L\}^{-\delta}).
\end{equation}  Pour $ L\geqslant 1 $ et $ K\leqslant L^{\alpha} $, la condition \ref{II} implique : \begin{align*}
 F_{d}(K,L) & =\sum_{k\leqslant K}\left(\sum_{l\leqslant L}f_{d}(k,l)\right) \\ & =\sum_{k\leqslant K}\left(c_{d,1}(k)L^{\beta_{2}}+O\left(k^{D}d^{\upsilon}L^{\beta_{2}-\delta}\right)\right) \\ &  =L^{\beta_{2}}\sum_{k\leqslant K}c_{d,1}(k)+O\left(d^{\upsilon}K^{D+1}L^{\beta_{2}-\delta}\right). 
\end{align*}
En choisissant $ L $ tel que $ K\leqslant L^{\alpha} $ et $ K^{D+1}L^{-\delta}=O(K^{\beta_{1}-\delta}) $, on obtient alors en utilisant la formule \eqref{FonctionFd} : \[ \sum_{k\leqslant K}c_{d,1}(k)=C_{d}K^{\beta_{1}}+ O\left(d^{\upsilon}K^{\beta_{1}-\delta}\right). \]\end{proof}

\begin{lemma}
On fixe un r\'eel $ \mu $ tel que \begin{equation}\label{mu1}
0<\beta_{1}\mu<\frac{1}{2},
\end{equation}\begin{equation}\label{mu2}
\mu\left(1+\alpha\frac{\beta_{1}}{\beta_{2}}\right)\leqslant \frac{\alpha}{\beta_{2}},
\end{equation}\begin{equation}\label{mu3}
\mu\left(D-\beta_{1}+1+\delta\frac{\beta_{1}}{\beta_{2}}\right)<\frac{\delta}{2\beta_{2}}.
\end{equation}
On pose \[ T_{d,1}=\sum_{k\leqslant d^{-1}P^{\mu}}\sum_{P^{\frac{1}{2}}<l^{\beta_{2}}\leqslant P/k^{\beta_{1}}}f_{d}(k,l). \]
On a alors \[ T_{d,1}=\beta_{1}\mu C_{d}P\log(P)+O(d^{\upsilon+\delta}\log(d)P). \]
\end{lemma}
\begin{proof}
On remarque dans un premier temps que : \[ T_{d,1}=\sum_{k\leqslant d^{-1}P^{\mu}}\sum_{k^{\beta_{1}}l^{\beta_{2}}\leqslant P}f_{d}(k,l)-F_{d}(d^{-1}P^{\mu},P^{\frac{1}{2\beta_{2}}}) \]
avec \[ F_{d}(d^{-1}P^{\mu},P^{\frac{1}{2\beta_{2}}})=O\left(d^{\upsilon}P^{\beta_{1}\mu+\frac{1}{2}}\right)=O(d^{\upsilon}P). \]
D'autre part, par l'hypoth\`ese \eqref{mu2}, on a pour tout $ k\leqslant d^{-1}P^{\mu} $, \[ k^{1+\alpha\frac{\beta_{1}}{\beta_{2}}}\leqslant d^{-(1+\alpha\frac{\beta_{1}}{\beta_{2}})}P^{(1+\alpha\frac{\beta_{1}}{\beta_{2}})\mu}\leqslant d^{-1}P^{\frac{\alpha}{\beta_{2}}}, \] et donc $ k\leqslant d^{-1}\left(P^{\frac{1}{\beta_{2}}}/k^{\frac{\beta_{1}}{\beta_{2}}}\right)^{\alpha} $. La condition \ref{II} donne alors : \begin{align*}
 T_{d,1} & =\sum_{k\leqslant d^{-1}P^{\mu}}\left(c_{d,1}(k)\left(P^{\frac{1}{\beta_{2}}}/k^{\frac{\beta_{1}}{\beta_{2}}}\right)^{\beta_{2}}+O\left(k^{D}d^{\upsilon}\left(P^{\frac{1}{\beta_{2}}}/k^{\frac{\beta_{1}}{\beta_{2}}}\right)^{\beta_{2}-\delta}\right)\right)+O(d^{\upsilon}P) \\ & =\left(\sum_{k\leqslant d^{-1}P^{\mu}}\frac{c_{d,1}(k)}{k^{\beta_{1}}}\right)P+O\left(\left(\sum_{k\leqslant d^{-1}P^{\mu}}k^{D-\beta_{1}+\delta\frac{\beta_{1}}{\beta_{2}}}\right)d^{\upsilon}P^{1-\frac{\delta}{\beta_{2}}}\right)+O(d^{\upsilon}P) \\ & =\left(\sum_{k\leqslant d^{-1}P^{\mu}}\frac{c_{d,1}(k)}{k^{\beta_{1}}}\right)P+O\left(P^{\mu(D-\beta_{1}+1+\delta\frac{\beta_{1}}{\beta_{2}})}\right)d^{\upsilon}P^{1-\frac{\delta}{\beta_{2}}}+O(d^{\upsilon}P) \\ & =\left(\sum_{k\leqslant d^{-1}P^{\mu}}\frac{c_{d,1}(k)}{k^{\beta_{1}}}\right)P+O(d^{\upsilon}P).
\end{align*}
Il nous faut \`a pr\'esent \'evaluer $ \sum_{k\leqslant d^{-1}P^{\mu}}\frac{c_{d,1}(k)}{k^{\beta_{1}}} $. Par sommation par parties, et en utilisant le lemme pr\'ec\'edent on a : \begin{multline*}
\sum_{k\leqslant d^{-1}P^{\mu}}\frac{c_{d,1}(k)}{k^{\beta_{1}}}  =d^{\beta_{1}}P^{-\mu\beta_{1}}\sum_{k\leqslant d^{-1}P^{\mu}}c_{d,1}(k)+\beta_{1}\int_{1}^{d^{-1}P^{\mu}}t^{-\beta_{1}-1}\sum_{k\leqslant t}c_{d,1}(k)dt \\  =d^{\beta_{1}}P^{-\mu\beta_{1}}\left(C_{d}d^{-\beta_{1}}P^{\mu\beta_{1}}+O\left(d^{\upsilon-\beta_{1}+\delta}P^{\mu\beta_{1}-\delta\mu}\right)\right) \\ +\beta_{1}\int_{1}^{d^{-1}P^{\mu}}t^{-\beta_{1}-1}\left(C_{d} t^{\beta_{1}}+O(d^{\upsilon} t^{\beta_{1}-\delta})\right)dt \\  =C_{d}+O\left(d^{\upsilon+\delta}P^{-\delta\mu}\right)+\beta_{1}C_{d}\log(P^{\mu})+O(d^{\upsilon}\log(d)) \\  =\beta_{1}C_{d}\log(P^{\mu})+O(d^{\upsilon+\delta})
\end{multline*}
\end{proof}

\begin{lemma}
On suppose $ 0<\mu<\min\{\frac{1}{2\beta_{1}},\frac{1}{2\beta_{2}}\} $, et on d\'efinit : \[ T_{d,2}=\sum_{d^{-1}P^{\mu}<k\leqslant P^{\frac{1}{2\beta_{1}}}}\sum_{P^{\frac{1}{2}}<l^{\beta_{2}}\leqslant P/k^{\beta_{1}}}f_{d}(k,l). \] On a alors \[ T_{d,2}=(\frac{1}{2}-\beta_{1}\mu)C_{d}P\log(P)+O(d^{\upsilon+\delta}\log(d)P). \]
\end{lemma}
\begin{proof}
On fixe $ d\in \NN^{\ast} $. On consid\`ere un entier $ J $ assez grand et on d\'efinit $ \theta>0 $ via : \[ (1+\theta)^{J}=dP^{\frac{1}{2\beta_{1}}-\mu}.  \] On consid\`ere alors des r\'eels $ d^{-1}P^{\mu}\leqslant K<K'\leqslant P^{\frac{1}{2\beta_{1}}} $ avec $ K'=K(1+\theta) $. On d\'efinit alors : \[ V(K)=\sum_{K<k\leqslant K'}\sum_{P^{\frac{1}{2}}<l^{\beta_{2}}\leqslant P/k^{\beta_{1}}}f_{d}(k,l), \] \[ V_{-}(K)=\sum_{K<k\leqslant K'}\sum_{P^{\frac{1}{2}}<l^{\beta_{2}}\leqslant P/(K')^{\beta_{1}}}f_{d}(k,l), \] \[ V_{+}(K)=\sum_{K<k\leqslant K'}\sum_{P^{\frac{1}{2}}<l^{\beta_{2}}\leqslant P/K^{\beta_{1}}}f_{d}(k,l),\] et on remarque que : \[  V_{-}(K)\leqslant V(K) \leqslant V_{+}(K). \]
Par ailleurs, \begin{multline*} V_{+}(K)=F_{d}\left(K',P^{\frac{1}{\beta_{2}}}/K^{\frac{\beta_{1}}{\beta_{2}}}\right)-F_{d}\left(K,P^{\frac{1}{\beta_{2}}}/K^{\frac{\beta_{1}}{\beta_{2}}}\right) \\ -F_{d}\left(K',P^{\frac{1}{2\beta_{2}}}\right)+F_{d}\left(K,P^{\frac{1}{2\beta_{2}}}\right). \end{multline*}
Or, on a : \begin{multline*}F_{d}\left(K',P^{\frac{1}{\beta_{2}}}/K^{\frac{\beta_{1}}{\beta_{2}}}\right)-F_{d}\left(K,P^{\frac{1}{\beta_{2}}}/K^{\frac{\beta_{1}}{\beta_{2}}}\right) \\ =C_{d}((K')^{\beta_{1}}-K^{\beta_{1}})PK^{-\beta_{1}}+O\left(d^{\upsilon}(K')^{\beta_{1}}PK^{-\beta_{1}}\min\{K',P^{\frac{1}{\beta_{2}}}K^{-\frac{\beta_{1}}{\beta_{2}}}\}^{-\delta}\right) \\ =C_{d}((1+\theta)^{\beta_{1}}-1)P+O\left(d^{\upsilon+\delta}(1+\theta)^{\beta_{1}}P^{1-\mu\delta}\right),
\end{multline*}
d'apr\`es \eqref{mu1}. 
En remarquant que $ (1+\theta)^{\beta_{1}}=1+\beta_{1}\theta +O(\theta^{2}) $, on obtient : \[ F_{d}\left(K',P^{\frac{1}{\beta_{2}}}/K^{\frac{\beta_{1}}{\beta_{2}}}\right)-F_{d}\left(K,P^{\frac{1}{\beta_{2}}}/K^{\frac{\beta_{1}}{\beta_{2}}}\right)=C_{d}\beta_{1}\theta P+O(d^{\upsilon+\delta}P^{1-\mu\delta})+O(d^{\upsilon}\theta^{2}P). \] De la m\^eme mani\`ere on trouve  \[ F_{d}\left(K',P^{\frac{1}{2\beta_{2}}}\right)-F_{d}\left(K,P^{\frac{1}{2\beta_{2}}}\right)=C_{d}\beta_{1}\theta K^{\beta_{1}} P^{\frac{1}{2}}+O(d^{\upsilon}P^{1-\mu\delta})+O(d^{\upsilon}\theta^{2}P). \] On en d\'eduit : \[ V_{+}(K)=C_{d}\beta_{1}\theta P+C_{d}\beta_{1}\theta K^{\beta_{1}} P^{\frac{1}{2}}+O(d^{\upsilon}P^{1-\mu\delta})+O(d^{\upsilon}\theta^{2}P). \]
Par des arguments analogues, on obtient la m\^eme estimation pour $ V_{-}(K) $, et donc \[ V(K)=C_{d}\beta_{1}\theta P+C_{d}\beta_{1}\theta K^{\beta_{1}} P^{\frac{1}{2}}+O(d^{\upsilon+\delta}P^{1-\mu\delta})+O(d^{\upsilon}\theta^{2}P). \]
On pose \`a pr\'esent, pour tout entier $ j $ tel que $ 0\leqslant j<J $, \[ K_{j}=d^{-1}P^{\mu}(1+\theta)^{j}. \] On a alors \begin{align*}
T_{d,2} & =\sum_{0\leqslant j<J}V(K_{j}) \\ & =C_{d}\beta_{1}(J\theta) P+C_{d}\beta_{1}\theta P^{\frac{1}{2}} \sum_{j=0}^{J-1}K_{j}^{\beta_{1}} +O(d^{\upsilon+\delta}JP^{1-\mu\delta})+O(d^{\upsilon}J\theta^{2}P).
\end{align*}
Or on a : \begin{align*}
\theta \sum_{j=0}^{J-1}K_{j}^{\beta_{1}} & = \theta d^{-\beta_{1}}P^{\beta_{1}\mu}\frac{(1+\theta)^{J\beta_{1}}-1}{(1+\theta)^{\beta_{1}}-1} \\ & =  d^{-\beta_{1}}P^{\beta_{1}\mu}\frac{dP^{\frac{1}{2}-\beta_{1}\mu}-1}{\beta_{1}+O(\theta)} \\ & =\frac{1}{\beta_{1}}P^{\frac{1}{2}}+O(d^{-\beta_{1}}P^{\beta_{1}\mu})+O(P^{\frac{1}{2}}\theta).  
\end{align*}
On obtient alors : \begin{align*} T_{d,2} & =C_{d}\beta_{1}(J\theta) P+C_{d}P+O(d^{\upsilon-\beta_{1}}P^{\frac{1}{2}+\beta_{1}\mu}) \\ & \; \; \; \;+O(d^{\upsilon}\theta^{2}P) +O(d^{\upsilon+\delta}JP^{1-\mu\delta})+O(d^{\upsilon}J\theta^{2}P)\\ & =C_{d}\beta_{1}(J\theta) P+O(d^{\upsilon}J\theta^{2}P)+O(d^{\upsilon}P)+O(d^{\upsilon+\delta}JP^{1-\mu\delta}). \end{align*}

On choisit \`a pr\'esent : \[ J=\left\lfloor P^{\frac{\mu\delta}{2}}\left(\left(\frac{1}{2\beta_{1}}-\mu\right)\log(P)+\log(d)\right)\right\rfloor \] Par d\'efinition de $ \theta $ on a : \[ J\log(\theta+1)=\left(\frac{1}{2\beta_{1}}-\mu\right)\log(P)+\log(d), \] et donc \begin{multline*} \theta=J^{-1}\left(\left(\frac{1}{2\beta_{1}}-\mu\right)\log(P)+\log(d)\right) \\ +O\left(J^{-2}\left(\left(\frac{1}{2\beta_{1}}-\mu\right)\log(P)+\log(d)\right)^{2}\right), \end{multline*} et on en d\'eduit \begin{multline*} J\theta=\left(\left(\frac{1}{2\beta_{1}}-\mu\right)\log(P)+\log(d)\right) \\ +O\left(P^{-\frac{\mu\delta}{2}}\left(\left(\frac{1}{2\beta_{1}}-\mu\right)\log(P)+\log(d)\right)\right). \end{multline*} On a par cons\'equent :  \begin{multline*} T_{d,2}=C_{d}\beta_{1}\left(\frac{1}{2\beta_{1}}-\mu\right) P\log(P)+C_{d}\beta_{1}\log(d) P \\ +O(d^{\upsilon+\delta}\log(d)P^{1-\frac{\mu\delta}{2}}\log(P))+O(d^{\upsilon}P), \end{multline*} et le lemme est d\'emontr\'e.

\end{proof}
\begin{proof}[D\'emonstration de la proposition \ref{propfin}]
On \'ecrit \begin{align*}
\sum_{k^{\beta_{1}}l^{\beta_{2}}\leqslant P}f_{d}(k,l) & =\sum_{\substack{k^{\beta_{1}}l^{\beta_{2}}\leqslant P \\ P^{\frac{1}{2}}<l^{\beta_{2}}}}f_{d}(k,l)+ \sum_{\substack{k^{\beta_{1}}l^{\beta_{2}}\leqslant P \\ P^{\frac{1}{2}}<k^{\beta_{1}}}}f_{d}(k,l)-F_{d}(P^{\frac{1}{2\beta_{1}}},P^{\frac{1}{2\beta_{2}}}) \\ & =\sum_{\substack{k^{\beta_{1}}l^{\beta_{2}}\leqslant P \\ P^{\frac{1}{2}}<l^{\beta_{2}}}}f_{d}(k,l)+ \sum_{\substack{k^{\beta_{1}}l^{\beta_{2}}\leqslant P \\ P^{\frac{1}{2}}<k^{\beta_{1}}}}f_{d}(k,l)+O(d^{\upsilon}P).
\end{align*}
On remarque alors que \[\sum_{\substack{k^{\beta_{1}}l^{\beta_{2}}\leqslant P \\ P^{\frac{1}{2}}<l^{\beta_{2}}}}f_{d}(k,l)  =T_{d,1}+T_{d,2}  =\frac{1}{2}C_{d}P\log(P)+O(d^{\upsilon+\delta}\log(d)P),\]
d'apr\`es les deux lemmes pr\'ec\'edents. Par sym\'etrie, on obtient exactement le m\^eme r\'esultat pour $ \sum_{\substack{k^{\beta_{1}}l^{\beta_{2}}\leqslant P \\ P^{\frac{1}{2}}<k^{\beta_{1}}}}f_{d}(k,l) $, et la proposition est d\'emontr\'ee.  \end{proof}
\begin{rem}\label{derniererem}
Par les m\^emes arguments, on d\'emontre, sous les m\^emes hypoth\`eses que :  
\[ \sum_{k^{\beta_{1}}(l+1)^{\beta_{2}}\leqslant P}f_{d}(k,l)=C_{d}P\log(P)+O(d^{\upsilon+\delta}\log(d)P). \]
Cette remarque nous sera utile dans ce qui va suivre.
\end{rem}

\subsection{Formule asymptotique pour $ N_{d,U}(B) $}

L'id\'ee est alors d'appliquer la proposition \ref{propfin} \`a la fonction $ h_{d}(k,l) $ d\'efinie en \eqref{fonctionh}. Pour cela nous allons montrer que cette fonction est bien une $ (\beta_{1},\beta_{2},C_{d},D,\alpha,\upsilon,\delta) $-fonction (pour des constantes $ C_{d},\delta,\beta_{1},\beta_{2},\alpha,\upsilon,D $ que nous pr\'eciserons).\\

 Remarquons avant tout que, d'apr\`es la proposition \ref{propNU}, la fonction $ h $ v\'erifie bien la condition \ref{I} avec $ \beta_{1}=m+1-d_{1} $, $ \beta_{2}=n-r+1-d_{2} $, $ C_{d}=\sigma_{d} $ et $ \upsilon=m-r+\max\{\frac{(r+1)(d_{1}-1)}{2^{d_{1}-1}}+\varepsilon, 5d_{1}\} $. D'autre part, par les corollaires \ref{corollairez} et \ref{corollairezbis},  pour tout $ \zz\in \mathcal{A}_{1}^{\mu}(\ZZ) $ et $ P_{2}\leqslant P_{1} $, on a
 \begin{multline*} N_{d,\zz}(P_{1})=\mathfrak{S}_{d,\zz}J_{\zz}d^{m-r-d_{1}}l^{m-r}P_{1}^{m+1-d_{1}} +O\left( d^{\upsilon}l^{m-r+4d_{2}}P_{1}^{m+1-d_{1}-\delta}\right) \end{multline*}
 \begin{multline*} N_{d,l,\zz}(P_{1})=\mathfrak{S}_{d,\zz}J_{l,\zz}d^{m-r-d_{1}}l^{m-r}P_{1}^{m+1-d_{1}}+O\left( d^{\upsilon}l^{m-r+4d_{2}}P_{1}^{m+1-d_{1}-\delta}\right) \end{multline*}
 uniform\'ement pour tout $ \zz, l <P_{1}^{\frac{d_{1}-1}{2d_{2}}} $ et $ \zz\in \mathcal{A}_{1}^{\mu}(\ZZ) $. En notant \begin{multline} \widetilde{N}_{d,U,l}(P_{1})=\card\{(\xx,\yy,\zz)\in  \ZZ^{n+2}\cap U \; | \;  F(\xx,\yy,\zz)=0 \; |\xx|\leqslant P_{1} \\ l=\max\left(\left\lfloor\frac{|\yy|}{d|\xx|}\right\rfloor,|\zz|\right)\} \end{multline}
On a alors que \begin{multline*}
\widetilde{N}_{d,U,l}(P_{1}) = \sum_{k\leqslant P_{1}}h_{d}(k,l) =\sum_{\substack{\zz\in \mathcal{A}_{1}^{\mu}(\ZZ)\\ |\zz|= l}}N_{d,\zz}(P_{1}) + \sum_{\substack{\zz\in \mathcal{A}_{1}^{\mu}(\ZZ)\\ |\zz|\leqslant l}}N_{d,l,\zz}(P_{1}) \\  =\left(\sum_{\substack{\zz\in \mathcal{A}_{1}^{\mu}(\ZZ)\\|\zz|= l}}\mathfrak{S}_{d,\zz}J_{\zz}+ \sum_{\substack{\zz\in \mathcal{A}_{1}^{\mu}(\ZZ)\\|\zz|\leqslant l}}\mathfrak{S}_{d,\zz}J_{l,\zz}\right)l^{m-r}d^{m-r-d_{1}}P_{1}^{m+1-d_{1}} \\ +O\left( d^{\upsilon}l^{n-r+1+4d_{2}}P_{1}^{m+1-d_{1}-\delta}\right)
\end{multline*} 
uniform\'ement pour tout $  l <P_{1}^{\frac{d_{1}-1}{2d_{2}}} $. On a de m\^eme, d'apr\`es le corollaire \ref{corollairex} :  
 \begin{multline*} N_{d,\xx}(P_{2})=\mathfrak{S}_{d,\xx}J_{d,\xx}d^{m-r}k^{m-r}P_{2}^{n-r+1-d_{2}}  +O\left( d^{\upsilon}k^{m-r+4d_{1}}P_{2}^{n-r+1-d_{2}-\delta}\right), \end{multline*} uniform\'ement pour tout $ k<P_{2}^{\frac{d_{2}-1}{2d_{1}}} $ et $ \xx\in \mathcal{A}_{2}^{\lambda}(\ZZ)  $. En notant 
 \begin{multline} \widetilde{N}_{d,U,k}(P_{2})=\card\{(\xx,\yy,\zz)\in  \ZZ^{n+2}\cap U \; | \;  F(\xx,\yy,\zz)=0, \; |\xx|=k\\ \max\left(\left\lfloor\frac{|\yy|}{d|\xx|}\right\rfloor,|\zz|\right)\leqslant P_{2}\} \end{multline}
On a alors que \begin{align*}
\widetilde{N}_{d,U,k}(P_{2})& = \sum_{l\leqslant P_{2}}h_{d}(k,l) =\sum_{\substack{\xx\in \mathcal{A}_{2}^{\lambda}(\ZZ)\\ |\xx|= k}}N_{d,\xx}(P_{2})  \\ & =\sum_{\substack{\xx\in \mathcal{A}_{2}^{\lambda}(\ZZ)\\|\xx|= k}}\mathfrak{S}_{d,\xx}J_{d,\xx}k^{m-r}d^{m-r}P_{2}^{n-r+1-d_{2}} \\ & +O\left( d^{\upsilon}k^{m+1+4d_{1}}P_{2}^{n-r+1-d_{2}-\delta}\right)
\end{align*} uniform\'ement pour tout $ k<d^{-1}P_{2}^{\frac{d_{2}-1}{2d_{1}}}  $. 
Par cons\'equent, $ h_{d} $ v\'erifie bien la condition \ref{II} avec\[ c_{d,1}(k)=\sum_{\substack{\xx\in \mathcal{A}_{2}^{\lambda}(\ZZ)\\|\xx|= k}}\mathfrak{S}_{d,\xx}J_{d,\xx}k^{m-r}d^{m-r}, \]
 \[ c_{d,2}(l)=\left(\sum_{\substack{\zz\in \mathcal{A}_{1}^{\mu}(\ZZ)\\|\zz|= l}}\mathfrak{S}_{d,\zz}J_{\zz}+ \sum_{\substack{\zz\in \mathcal{A}_{1}^{\mu}(\ZZ)\\|\zz|\leqslant l}}\mathfrak{S}_{d,\zz}J_{l,\zz}\right)l^{m-r}d^{m-r-d_{1}}, \] 
\[ D=\max\{m+1+4d_{1}, n-r+1+4d_{2}\}\] et \[ \alpha=\min\{\frac{d_{2}-1}{2d_{1}}, \frac{d_{1}-1}{2d_{2}}\} .\]

On a donc montr\'e que $ h_{d} $ est une $ (m+1-d_{1},n-r+1-d_{2},\sigma_{d},D,\alpha, \upsilon,\delta) $-fonction, et donc en notant \begin{multline*}
\tilde{N}_{d,U}^{(1)}(B)=\card\left\{ (\xx,\yy,\zz)\in \ZZ^{n+2}\cap U\; | \; \xx\neq\0, \; (\yy,\zz)\neq (\0,\0), \; \right. \\ \left. \; F(d\xx,\yy,\zz)=0, \;   |\xx|^{m+1-d_{1}}\max\left( \left\lfloor\frac{|\yy|}{d|\xx|}\right\rfloor,|\zz|\right)^{n-r+1-d_{2}}\leqslant B \right\} 
\end{multline*}et \begin{multline*}
\tilde{N}_{d,U}^{(2)}(B)=\card\left\{ (\xx,\yy,\zz)\in \ZZ^{n+2}\cap U\; | \; \xx\neq\0, \; (\yy,\zz)\neq (\0,\0), \; \right. \\ \left. \; F(d\xx,\yy,\zz)=0, \;   |\xx|^{m+1-d_{1}}\max\left( \left\lfloor\frac{|\yy|}{d|\xx|}\right\rfloor+1,|\zz|+1\right)^{n-r+1-d_{2}} \leqslant B \right\} 
\end{multline*}
la proposition \ref{propfin} et la remarque \ref{derniererem} donnent : \[ \tilde{N}_{d,U}^{(i)}(B)=\sigma_{d} B\log(B)+O\left(d^{\upsilon}\log(d)B\right) \] pour $ i\in\{1,2\} $.
Par ailleurs, on observe que \[ \tilde{N}_{d,U}^{(2)}(B)\leqslant N_{d,U}(B)\leqslant \tilde{N}_{d,U}^{(1)}(B), \] et on en d\'eduit finalement :
\begin{prop}\label{preconclusion}
Si $ d_{1},d_{2}\geqslant 2 $ et $ n+2-\max\{\dim V_{1}^{\ast},\dim V_{2}^{\ast}\} >m $, ou si $ d_{1}\geqslant 2 $, $ d_{2}=1 $ et $ n+2-\max\{\dim V_{1}^{\ast},\dim V_{2}^{\ast}\} >m' $, alors pour tout $ B\geqslant 1 $, on a la formule asymptotique : \[ N_{d,U}(B)=\sigma_{d} B\log(B)+O\left(d^{\upsilon+\delta}\log(d)B\right), \] pour un certain $ \delta>0 $ arbitrairement petit. 
\end{prop}

\section{Conclusion et interpr\'etation des constantes}

Nous somme \`a pr\'esent en mesure de donner une formule asymptotique pour  \begin{multline*} \mathcal{N}_{U}(B)=\frac{1}{4}\card\{(\xx,\yy,\zz)\in U\cap\ZZ^{n+2} \; | \;  \PGCD(\xx)=1, \; \PGCD(\yy,\zz)=1 \; \\ F(\xx,\yy,\zz) =0, \; H(\xx,\yy,\zz)\leqslant B\}. \end{multline*}

On remarque en effet que si $ N_{d,e}(B) $ d\'esigne
\begin{multline*}\card\{(d\xx,e\yy,e\zz)\in U\cap(d\ZZ^{r+1}\times e\ZZ^{n-r+1}) \; | \;  F(d\xx,e\yy,e\zz) =0,  \; H(d\xx,e\yy,e\zz)\leqslant B\} \\ =\card\left\{(\xx,\yy,\zz)\in U\cap(\ZZ^{r+1}\times \ZZ^{n-r+1}) \; | \;  F(d\xx,\yy,\zz) =0,  \; \right. \\ \left.  |\xx|^{m+1-d_{1}}\max\{\frac{|\yy|}{d|\xx|},|\zz|\}^{n-r+1-d_{2}}\leqslant B/(d^{m+1-d_{1}}e^{n-r+1-d_{2}})\right\} \\ =N_{d,U}(B/(d^{m+1-d_{1}}e^{n-r+1-d_{2}})) \end{multline*} et
 \begin{multline*} \tilde{N}_{k,l}(B)=\card\{(k\xx,l\yy,l\zz)\in U\cap(k\ZZ^{r+1}\times l\ZZ^{m-r}\times l\ZZ^{n-r+1}) \; | \; \PGCD(\xx)=1, \;\\ \PGCD(\yy,\zz)=1, \;   \;  F(\xx,\yy,\zz) =0, \; H(k\xx,l\yy,l\zz)\leqslant B\}  \end{multline*} (pour $ d,e,k,l \in \NN $), alors on a 
  \[ N_{d,e}(B)=\sum_{d|k}\sum_{e|l}\tilde{N}_{k,l}(B). \] Par inversions de M\"{o}bius successives, et en utilisant la proposition \ref{preconclusion}, on obtient : \begin{multline*} \mathcal{N}_{U}(B) =\frac{1}{4}\tilde{N}_{1,1}(B)  =\frac{1}{4}\sum_{d\in \NN^{\ast}}\mu(d)\sum_{e\in \NN^{\ast}}\mu(e)N_{d,e}(B)\\  =\frac{1}{4}\sum_{d,e \in \NN^{\ast}}\mu(d)\mu(e)N_{d,U}(B/(d^{m+1-d_{1}}e^{n-r+1-d_{2}})) \\  = \frac{1}{4}\sum_{d,e \in \NN^{\ast}}\frac{\mu(d)\mu(e)}{d^{m+1-d_{1}}e^{n-r+1-d_{2}}}\sigma_{d}B\log(B) \\  +O\left(\sum_{d,e \in \NN^{\ast}}\frac{\mu(d)\mu(e)}{d^{m+1-d_{1}}e^{n-r+1-d_{2}}}d^{\upsilon+\delta}\log(d)B\right) \\ = \frac{1}{4}\left(\sum_{e\in \NN^{\ast}}\frac{\mu(e)}{e^{n-r+1-d_{2}}}\right)\left(\sum_{d \in \NN^{\ast}}\frac{\mu(d)}{d^{m+1-d_{1}}}\sigma_{d}\right)B\log(B) +O\left(B\right),
 \end{multline*} 
car $ \upsilon=m-r+\max\{\frac{(r+1)(d_{1}-1)}{2^{d_{1}-1}}+\varepsilon, 5d_{1}\}<(m+1-d_{1})+2 $, pour $ r $ choisi assez grand, i.e. pour $ r\geqslant 6d_{1}-3 $. Par ailleurs on peut r\'e\'ecrire  \[\sum_{e\in \NN^{\ast}}\frac{\mu(e)}{e^{n-r+1-d_{2}}}=\prod_{p \in \mathcal{P}}\left(1-\frac{1}{p^{n-r+1-d_{2}}}\right) \] et \[ \sum_{d \in \NN^{\ast}}\frac{\mu(d)}{d^{m+1-d_{1}}}\sigma_{d}=J\sum_{d \in \NN^{\ast}}\frac{\mu(d)}{d^{m+1-d_{1}}}\mathfrak{S}_{d}d^{m-r-d_{1}}=J\mathfrak{S} \]
pour \begin{equation}
\mathfrak{S}=\sum_{d\in \NN^{\ast}}\frac{\mu(d)}{d^{r+1}}\mathfrak{S}_{d}.
\end{equation}
On obtient donc finalement \begin{prop}\label{conclusion}
Pour $ d_{1},d_{2}\geqslant 2 $, $ n+2-\max\{\dim V_{1}^{\ast},\dim V_{2}^{\ast}\} >m $ et $ r\geqslant 6d_{1}-3 $, on a : \[\mathcal{N}_{U}(B)=  \sigma B\log(B)+O(B), \]
lorsque $ B \ra \infty $, o\`u l'on a not\'e $ \sigma=\frac{1}{4}J\mathfrak{S}\prod_{p \in \mathcal{P}}\left(1-\frac{1}{p^{n-r+1-d_{2}}}\right) $. On a de plus la m\^eme formule pour $ d_{1}\geqslant 2 $, $ d_{2}=1 $, $ n+2-\max\{\dim V_{1}^{\ast},\dim V_{2}^{\ast}\} >m' $ et $ r\geqslant 6d_{1}-3 $.
\end{prop}

Nous allons \`a pr\'esent donner une interpr\'etation des constantes introduites, et d\'emontrer que l'expression obtenue est bien en accord avec les formules conjectur\'ees par Peyre dans \cite{Pe}. \\

Rappelons que l'on a not\'e $ \pi : X_{0}\ra X  $ la projection du torseur universel $ X_{0}=(\AA^{r+1}\setminus\{\0\})\times(\AA^{n-r+1}\setminus\{\0\})  $ sur la vari\'et\'e torique ambiante $ X $. On consid\`ere un point $ (\xx,\yy,\zz)\in Y_{0} $ tel que $ \frac{\partial F}{\partial t_{j}}(\xx,\yy,\zz)\neq \0 $, o\`u \[t_{j}=\left\{ \begin{array}{lll} x_{j} & \mbox{si} & j\in\{0,...,r\} \\ y_{j} & \mbox{si} & j\in\{r+1,...,m\}   \\ z_{j} & \mbox{si} & j\in\{m+1,...,n+1\}\end{array}\right. \] et on note $ P=\pi(\xx,\yy,\zz) $. La forme de Leray $ \omega_{L} $ sur un voisinage de $ (\xx,\yy,\zz) $ sur lequel $ \frac{\partial F}{\partial t_{j}}\neq \0 $ est alors donn\'ee par \begin{multline*} \omega_{L}(\xx,\yy,\xx)=\frac{(-1)^{n+2-j}}{\frac{\partial F}{\partial t_{j}}(\xx,\yy,\zz)}dt_{0}\wedge...\wedge \widehat{dt_{j}}\wedge ...\wedge dt_{n+1}(\xx,\yy,\zz).  \end{multline*}Pour toute place $ \nu\in \Val(\QQ) $ la forme de Leray induit une mesure locale $ \omega_{L,\nu} $. \\

On suppose \`a pr\'esent que le point $ (\xx,\yy,\zz) $ est tel que, par exemple, $ x_{0}\neq 0 $, $ z_{m+1}\neq 0 $ et $ \frac{\partial F}{\partial z_{n+1}}(\xx,\yy,\zz)\neq 0 $. Pour toute place $ \nu $ de $ \QQ $, on consid\`ere le morphisme \[ \begin{array}{rcl} \rho : X_{\QQ_{\nu}} & \ra & \AA_{\QQ_{\nu}}^{n-1} \\  \pi(\xx,\yy,\zz) & \mt & \left(\frac{x_{1}}{x_{0}},...,\frac{x_{r}}{x_{0}},\frac{y_{r+1}}{x_{0}z_{m+1}},...,\frac{y_{m}}{x_{0}z_{m+1}},\frac{z_{m+2}}{z_{m+1}},...,\frac{z_{n}}{z_{m+1}}\right).
\end{array} \] 
Par le th\'eor\`eme d'inversion locale, il existe un voisinage ouvert de $ P $ not\'e $ V $ sur lequel $ \rho $ est bien d\'efini et induit un isomorphisme analytique sur $ \rho(V) $. On pose $ W=\pi^{-1}(V) $. Si l'on note \[ \begin{array}{l}
\uu=(1,u_{1},...,u_{r}) \\ \vv=(v_{r+1},...,v_{m}) \\ \ww=(1,w_{m+2},...,w_{n+1}) 
\end{array}, \]la mesure de Tamagawa $ \omega_{\nu} $ est d\'efinie par \[ \rho_{\ast}\omega_{\nu}=\frac{du_{1,\nu}...du_{r,\nu}dv_{r+1,\nu}...dv_{m,\nu}dw_{m+2,\nu}...dw_{n,\nu}}{h_{\nu}(\uu,\vv,\ww)\left|\frac{\partial F}{\partial z_{n+1}}(\uu,\vv,\ww)\right|_{\nu}}, \] o\`u $ w_{n+1} $ est implicitement d\'efini par $ F(\uu,\vv,\ww)=0 $, et \[ h_{\nu}(\uu,\vv)=h^{(1)}_{\nu}(\uu)h^{(2)}_{\nu}(\uu,\vv,\ww) \] pour \[ \begin{array}{l}
h^{(1)}_{\nu}(\uu)=|\uu|_{\nu}^{m+1-d_{1}} \\ h^{(2)}_{\nu}(\uu,\vv,\ww)=\max\left(\frac{|\vv|_{\nu}}{|\uu|_{\nu}},|\ww|_{\nu}\right)^{n-r+1-d_{2}}, 
\end{array}\]
o\`u pour tout vecteur $ \xx=(x_{1},...,x_{N}) $, \[ |\xx|_{\nu}=\max_{1\leqslant i\leqslant N}|x_{i}|_{\nu}. \]

\subsection{\'Etude de l'int\'egrale singuli\`ere $ J $ }

 Rappelons que l'int\'egrale $ J $ est d\'efinie par \[ J =\int_{\RR}\int_{\substack{|\yy|\leqslant |\xx|\leqslant 1 \\ |\zz|\leqslant 1}}e(\beta F(\xx,\yy,\zz))d\xx d\yy d\zz d\beta. \]
et cette int\'egrale est absolument convergente. On pose par ailleurs : \[ \sigma_{\infty}(Y)=\int_{\substack{\pi^{-1}(Y)\cap\{|\yy|\leqslant |\xx|\leqslant 1 \\ |\zz|\leqslant 1\}}}\omega_{L,\infty}. \]
Nous allons montrer que l'int\'egrale $ J $ co\"{i}ncide avec $ \sigma_{\infty}(Y) $. Il nous suffit de le v\'erifier localement i.e. montrons que pour tout ouvert $ V' $ de $ \{(\xx,\yy,\zz) \; | \; |\yy|\leqslant |\xx|\leqslant 1, \; |\zz|\leqslant 1\} $ sur lequel, par exemple, $ \frac{\partial F}{\partial z_{n+1}}(\xx,\yy,\zz)\neq 0 $, \[ \int_{V'\cap\pi^{-1}( Y)} \omega_{L,\infty}=\int_{V'\cap\pi^{-1}( Y)}\frac{1}{\left|\frac{\partial F}{\partial z_{n+1}}(\xx,\yy,\zz)\right|}d\xx d\yy d\hat{\zz}, \] (avec $ d\hat{\zz}=dz_{m+1}...dz_{n} $) co\"{i}ncide avec \[ 
J_{V'}=\int_{\RR}\int_{V'}e(\beta F(\xx,\yy,\zz))d\xx d\yy d\zz d\beta. \]Consid\'erons donc un tel ouvert $ V' $. 
On note alors $ t=F(\xx,\yy,\zz) $, et $ z_{n+1}  $ est alors d\'efini implicitement par $ \xx,\yy,\hat{\zz},t $ sur $ V' $. On note $ z_{n+1}=g(\xx,\yy,\hat{\zz},t) $. Par changement de variables, on a alors : \[ J_{V'}= \int_{\RR}\int_{\RR}\int_{[-1,1]^{n+1}}\frac{\chi(t,\xx,\yy,\hat{\zz})e(\beta t)}{\left|\frac{\partial F}{\partial z_{n+1}}(\xx,\yy,g(\xx,\yy,\hat{\zz},t))\right|}d\xx d\yy d\hat{\zz}dt d\beta    \] o\`u \[ \chi(t,\xx,\yy,\hat{\zz})=\left\{ \begin{array}{l}
1 \; \; \mbox{si} \; \; (\xx,\yy,\hat{\zz},g(\xx,\yy,\hat{\zz},t))\in V' \\ 0 \; \; \mbox{sinon}
\end{array}\right.  \]

La fonction $ t\mt \frac{\chi(t,\xx,\yy,\hat{\zz})e(\beta t)}{\left|\frac{\partial F}{\partial z_{n+1}}(\xx,\yy,g(\xx,\yy,\hat{\zz},t))\right|} $ est \`a variations born\'ees, par cons\'equent, par application des r\'esultats d'analyse de Fourier (voir \cite[9.43]{W-W}) on a que \begin{align*} J_{V'} & =\int_{[-1,1]^{n+1}}\frac{\chi(0,\xx,\yy,\hat{\zz})}{\left|\frac{\partial F}{\partial z_{n+1}}(\xx,\yy,g(\xx,\yy,\hat{\zz},0))\right|}d\xx d\yy d\hat{\zz} \\ & =\int_{V'\cap \pi^{-1}(Y)} \omega_{L,\infty}. \end{align*} Remarquons que ces calculs constituent un \'equivalent du travail effectu\'e par Igusa dans \cite[\S IV.6]{Ig} pour le cas des int\'egrales de fonctions indicatrices. \\

Nous allons \`a pr\'esent interpr\'eter cette constante $ J $ en termes de mesures de Tamagawa. Plus pr\'ecis\'ement, en notant $ \tau_{\infty}=\omega_{\infty} $, nous allons d\'emontrer le r\'esultat suivant : 
\begin{lemma}\label{Tamagawa infini} On a 
\[ \tau_{\infty}=\frac{(m+1-d_{1})(n-r+1-d_{2})}{4}\sigma_{\infty}. \]
\end{lemma}
\begin{proof}
Il nous suffit de montrer que par exemple pour l'ouvert $ V $ d\'efini pr\'ec\'edemment, on a $ \tau_{\infty}(V)=\frac{(m+1-d_{1})(n-r+1-d_{2})}{4}\sigma_{\infty}(V) $. Par d\'efinition de la mesure de Leray, \[ \sigma_{\infty}(V)=\int_{\substack{\pi^{-1}(V)\cap \{|\xx|\leqslant 1\\ |\yy|\leqslant |\xx|, \;  |\zz|\leqslant  1\}}}\frac{d\xx d\yy d\hat{\zz}}{\left|\frac{\partial F}{\partial z_{n+1}}(\xx,\yy,\zz)\right|}. \] On remarque que \[ \max_{i}|x_{i}|\leqslant 1\; \Leftrightarrow \; |x_{0}|\leqslant \left(\max_{i}\frac{|x_{i}|}{|x_{0}|}\right)^{-1}. \] On applique alors les changements de variables $ x_{i}=x_{0}u_{i} $, $ y_{j}=z_{m+1}x_{0}v_{j} $ et $ z_{k}=z_{m+1}w_{k} $ dans l'int\'egrale ci-dessus. On a alors que \begin{align*} \left\{\begin{array}{l}
|\yy|\leqslant |\xx|\leqslant 1 \\ |\zz|\leqslant 1
\end{array}\right.  \; & \Leftrightarrow \; \left\{\begin{array}{l}
|x_{0}|\leqslant (|\uu|)^{-1} \\ |z_{m+1}||\vv|\leqslant |\uu| \\ |z_{m+1}|\leqslant |\ww|^{-1}
\end{array}\right. \\ & \Leftrightarrow \; \left\{\begin{array}{l}
|x_{0}|^{m+1-d_{1}}\leqslant h^{(1)}_{\infty}(\uu)^{-1} \\ |z_{m+1}|^{n-r+1-d_{2}}\leqslant h^{(2)}_{\infty}(\uu,\vv,\ww)^{-1}
\end{array} \right. \;\end{align*} On obtient donc  \begin{multline*} \sigma_{\infty}(V)= \int_{V}\frac{1}{\left|\frac{\partial F}{\partial z_{n+1}}(\xx,\yy,\zz)\right|} \int_{\substack{|x_{0}|^{m+1-d_{1}}\leqslant h_{\infty}^{(1)}(\uu)^{-1} \\  |z_{m+1}|^{n-r+1}\leqslant h_{\infty}^{(2)}(\uu,\vv,\ww)^{-1} }} \\  |x_{0}|^{m-d_{1}}|z_{m+1}|^{n-r-d_{2}}dx_{0}dz_{m+1}d\uu d\vv d\hat{\ww} \\ =\frac{4}{(m+1-d_{1})(n-r+1-d_{2})}\int_{\rho(V)}\frac{d\uu d\vv d\hat{\ww}}{h_{\infty}(\uu,\vv,\ww)\left|\frac{\partial F}{\partial z_{n+1}}(\xx,\yy,\zz)\right|} \\ =\frac{4}{(m+1-d_{1})(n-r+1-d_{2})}\int_{V}\omega_{\infty}. \end{multline*}

\end{proof}

\subsection{\'Etude de la s\'erie singuli\`ere $ \mathfrak{S} $}

Rappelons que $ \mathfrak{S} $ est d\'efinie par : 
\[ \mathfrak{S}=\sum_{d\in \NN^{\ast}}\frac{\mu(d)}{d^{r+1}}\mathfrak{S}_{d}, \] avec \[ \mathfrak{S}_{d}=\sum_{q=1}^{\infty}A_{d}(q) \] o\`u \[ A_{d}(q)=q^{-(n+2)}\sum_{a\in (\ZZ/q\ZZ)^{\ast}}\sum_{(\bb_{1},\bb_{2},\bb_{3})\in (\ZZ/q\ZZ)^{n+2}}e\left(\frac{a}{q}F(d\bb_{1},\bb_{2},\bb_{3})\right). \]
\begin{lemma}
Pour tout $ d\in \NN^{\ast} $, la fonction $ A_{d} $ est multiplicative. 
\end{lemma}
\begin{proof}
On consid\`ere deux entiers $ q_{1},q_{2} $ tels que $ \PGCD(q_{1},q_{2})=1 $, et posons $ q=q_{1}q_{2} $. Montrons qu'alors $ A_{d}(q)=A_{d}(q_{1})A_{d}(q_{2}) $. On remarque que \begin{multline*}
A_{d}(q_{1})A_{d}(q_{2})=q^{-(n+2)}\sum_{\substack{a_{1}\in (\ZZ/q_{1}\ZZ)^{\ast}\\a_{2}\in (\ZZ/q_{2}\ZZ)^{\ast}}}\sum_{\substack{(\bb_{1}^{(1)},\bb_{2}^{(1)},\bb_{3}^{(1)})\in (\ZZ/q_{1}\ZZ)^{n+2} \\ (\bb_{1}^{(2)},\bb_{2}^{(2)},\bb_{3}^{(2)})\in (\ZZ/q_{2}\ZZ)^{n+2}}}\\ e\left(\frac{a_{1}q_{2}F(d\bb_{1}^{(1)},\bb_{2}^{(1)},\bb_{3}^{(1)})+a_{2}q_{1}F(d\bb_{1}^{(2)},\bb_{2}^{(2)},\bb_{3}^{(2)})}{q}\right). 
\end{multline*} 
Or si l'on consid\`ere l'unique \'el\'ement $ (\bb_{1},\bb_{2},\bb_{3})\in (\ZZ/q\ZZ)^{n+2} $ tel que \[ \forall i\in \{1,2,3\}, \; \;  \left\{\begin{array}{l} \bb_{i}\equiv \bb_{i}^{(1)} (q_{1}) \\ \bb_{i}\equiv \bb_{i}^{(2)} (q_{2})
\end{array}\right. \]
on a \[  \left\{\begin{array}{l} q_{2}F(d\bb_{1}^{(1)},\bb_{2}^{(1)},\bb_{3}^{(1)})\equiv q_{2}F(d\bb_{1},\bb_{2},\bb_{3})\;(q) \\q_{1}F(d\bb_{1}^{(2)},\bb_{2}^{(2)},\bb_{3}^{(2)}) \equiv q_{1}F(d\bb_{1},\bb_{2},\bb_{3})\;(q)
\end{array}\right. \]
et ainsi : \begin{multline*}
A_{d}(q_{1})A_{d}(q_{2})=q^{-(n+2)}\sum_{\substack{a_{1}\in (\ZZ/q_{1}\ZZ)^{\ast}\\a_{2}\in (\ZZ/q_{2}\ZZ)^{\ast}}}\sum_{(\bb_{1},\bb_{2},\bb_{3})\in (\ZZ/q\ZZ)^{n+2}} \\ e\left(\frac{a_{1}q_{2}+a_{2}q_{1}}{q}F(d\bb_{1},\bb_{2},\bb_{3})\right). 
\end{multline*} 
Or l'application : \[ \begin{array}{rcl}
(\ZZ/q_{1}\ZZ)^{\ast}\times (\ZZ/q_{2}\ZZ)^{\ast} & \ra & (\ZZ/q\ZZ)^{\ast} \\ (a_{1},a_{2}) &\mt & a_{1}q_{2}+a_{2}q_{1}
\end{array} \]
est bijective. On obtient donc finalement : \[
A_{d}(q_{1})A_{d}(q_{2})=A_{d}(q). \]
\end{proof}
Puisque $ \mathfrak{S}_{d} $ est de plus absolument convergente (cf. lemme \ref{seriesing}), on a : \[ \mathfrak{S}_{d}=\prod_{p\in \mathcal{P}}\sigma_{d,p} \] o\`u \[ \sigma_{d,p}=\sum_{k=0}^{\infty}A_{d}(p^{k}). \] On remarque par ailleurs que pour tous $ d,k\in \NN^{\ast} $ :  \begin{multline*} \sum_{(\bb_{1},\bb_{2},\bb_{3})\in (\ZZ/p^{k}\ZZ)^{n+2}}e\left(\frac{a}{p^{k}}F(d\bb_{1},\bb_{2},\bb_{3})\right) \\ = \sum_{(\bb_{1},\bb_{2},\bb_{3})\in (\ZZ/p^{k}\ZZ)^{n+2}}e\left(\frac{a}{p^{k}}F(p^{v_{p}(d)}\bb_{1},\bb_{2},\bb_{3})\right) \end{multline*}
et donc : \[ A_{d}(p^{k})=A_{p^{v_{p}(d)}}(p^{k}). \] Par cons\'equent, pour tout $ d\in \NN^{\ast} $, on a : \[  \frac{\mu(d)}{d^{r+1}}\mathfrak{S}_{d}=\prod_{p\in \mathcal{P}} B_{p^{v_{p}(d)}} \] o\`u pour tout $ \nu\in \NN^{\ast} $ : \[  B_{p^{\nu}}=\frac{\mu(p^{\nu})}{p^{\nu(r+1)}}\sum_{k=0}^{\infty}A_{p^{\nu}}(p^{k}). \] Remarquons que $ B_{p^{\nu}}=0 $ pour tout $ \nu\geqslant 2 $. La s\'erie $ \sum_{d\in \NN^{\ast}}\frac{\mu(d)}{d^{r+1}}\mathfrak{S}_{d} $ \'etant absolument convergente, on a alors : \begin{align*}
\sum_{d\in \NN^{\ast}}\frac{\mu(d)}{d^{r+1}}\mathfrak{S}_{d} & = \prod_{p\in \mathcal{P}}\left(\sum_{\nu=0}^{\infty}B_{p^{\nu}}\right) \\ & = \prod_{p\in \mathcal{P}}\left(\sum_{\nu=0}^{1}B_{p^{\nu}}\right) \\ & = \prod_{p\in \mathcal{P}}\underbrace{\left(\sum_{k=0}^{\infty}\left(A_{1}(p^{k})-\frac{A_{p}(p^{k})}{p^{r+1}}\right)\right)}_{\sigma_{p}'}.
\end{align*}
Notons \`a pr\'esent \begin{equation}
M_{p}(k)=\Card\left\{ (\xx,\yy,\zz)\in (\ZZ/p^{k}\ZZ)^{n+2} \; | \; \xx\not\equiv \0 \;(p), \; F(\xx,\yy,\zz)\equiv 0 \;(p^{k})\right\}
\end{equation}
\begin{lemma}
Pour tout entier $ N>0 $, on a \[ \sum_{k=0}^{N}\left(A_{1}(p^{k})-\frac{A_{p}(p^{k})}{p^{r+1}}\right)=\frac{M_{p}(N)}{p^{N(n+1)}}, \] et donc \[ \sigma_{p}'=\lim_{N\ra\infty}\frac{M_{p}(N)}{p^{N(n+1)}}. \]
\end{lemma}
\begin{proof} On pose $ q=p^{N} $. Il est imm\'ediat que \begin{multline*}q^{-1}\sum_{t=0}^{q-1}\sum_{(\bb_{1},\bb_{2},\bb_{3})\in (\ZZ/q\ZZ)^{n+2}}e\left(\frac{t}{q}F(\bb_{1},\bb_{2},\bb_{3})\right) \\ = \Card\left\{ (\xx,\yy,\zz)\in (\ZZ/q\ZZ)^{n+2} \; | \;  \; F(\xx,\yy,\zz)\equiv 0 \;(q)\right\},\end{multline*} et de m\^eme \begin{multline*}q^{-1}\sum_{t=0}^{q-1}\sum_{(\bb_{1},\bb_{2},\bb_{3})\in (\ZZ/q\ZZ)^{n+2}}e\left(\frac{a}{q}F(p\bb_{1},\bb_{2},\bb_{3})\right) \\ = p^{r+1}\Card\left\{ (\xx,\yy,\zz)\in (\ZZ/q\ZZ)^{n+2} \; | \; \xx\equiv \0 \; (p) \; F(\xx,\yy,\zz)\equiv 0 \;(q)\right\}\end{multline*}
On a donc \begin{multline*}
M_{p}(N)  =q^{-1}\sum_{t=0}^{q-1}\sum_{(\bb_{1},\bb_{2},\bb_{3})\in (\ZZ/q\ZZ)^{n+2}} \\ \left(e\left(\frac{t}{q}F(\bb_{1},\bb_{2},\bb_{3})\right)-\frac{1}{p^{r+1}}e\left(\frac{t}{q}F(p\bb_{1},\bb_{2},\bb_{3})\right)\right) \\  =q^{-1}\sum_{q_{1}|q}\sum_{\substack{0\leqslant a<q_{1}\\ \PGCD(a,q_{1})=1}}\sum_{(\bb_{1},\bb_{2},\bb_{3})\in (\ZZ/q\ZZ)^{n+2}} \\ \left(e\left(\frac{a}{q_{1}}F(\bb_{1},\bb_{2},\bb_{3})\right)  -\frac{1}{p^{r+1}}e\left(\frac{a}{q_{1}}F(p\bb_{1},\bb_{2},\bb_{3})\right)\right) \\  =p^{-N}\sum_{k=1}^{N}\frac{p^{N(n+2)}}{p^{k(n+2)}}\sum_{\substack{a\in (\ZZ/p^{k}\ZZ)^{\ast}}}\sum_{(\bb_{1},\bb_{2},\bb_{3})\in (\ZZ/p^{k}\ZZ)^{n+2}} \\ \left(e\left(\frac{a}{p^{k}}F(\bb_{1},\bb_{2},\bb_{3})\right)-\frac{1}{p^{r+1}}e\left(\frac{a}{p^{k}}F(p\bb_{1},\bb_{2},\bb_{3})\right)\right) \\  = p^{N(n+1)}\sum_{k=0}^{N}\left(A_{1}(p^{k})-\frac{A_{p}(p^{k})}{p^{r+1}}\right)
\end{multline*}
\end{proof}
Nous allons \`a pr\'esent interpr\'eter les constantes $ \sigma_{p}' $ en terme de mesures de Tamagawa $ \tau_{p} $ d\'efinies par \[ \tau_{p}=\left(1-\frac{1}{p}\right)^{2}\omega_{p}. \] Pour cela nous commen\c{c}ons par \'etablir deux lemmes interm\'ediaires : 

\begin{lemma}\label{lemmeinter1}
Pour tout $ N\in \NN^{\ast} $, on note \begin{multline*} W_{p}^{\ast}(N)=\{ (\xx,\yy,\zz)\in (\ZZ_{p}/p^{N})^{n+2} \; | \; \xx\not\equiv \0\; (p), \; \\ (\yy,\zz)\not\equiv \0\; (p), F(\xx,\yy,\zz)\equiv 0 \; (p^{r}) \} \end{multline*} ainsi que $ M_{p}^{\ast}(N)=\Card W^{\ast}(N) $. Il existe alors un entier $ N_{0} $ tel que pour tout $ N\geqslant N_{0} $ : \[ \int_{\substack{(\xx,\yy,\zz)\in \ZZ_{p}^{n+2} \\ \xx\not\equiv \0\; (p), \; (\yy,\zz)\not\equiv \0\; (p) \\ F(\xx,\yy,\zz)=0 }}\omega_{L,p}=\frac{M_{p}^{\ast}(N)}{p^{N(n+1)}}. \]
\end{lemma}
\begin{proof}
Soit $ (\xx,\yy,\zz)\in \ZZ_{p}^{n+2} $. Dans tout ce qui suit, on note \[ [\xx,\yy,\zz]_{N}=(\xx,\yy,\zz)\mod p^{N} .\] On \'ecrit alors : \begin{align}
\int_{\substack{\{(\xx,\yy,\zz)\in \ZZ_{p}^{n+2}, \; \xx \not\equiv \0\;(p) \\ (\yy,\zz)\not\equiv \0\;(p),\; F(\xx,\yy,\zz)=0\}}}\omega_{L,p} & =\sum_{\substack{(\xx,\yy,\zz)\mod p^{N} \\ \xx\not\equiv \0\; (p), \; (\yy,\zz)\not\equiv \0\; (p) \\ F(\xx,\yy,\zz)\equiv 0 \; (p^{N})}}\int_{\substack{\{(\uu,\vv,\ww)\in \ZZ_{p}^{n+2}, \\ [\uu,\vv,\ww]_{N}=(\xx,\yy,\zz)   \\  F(\uu,\vv,\ww)=0\}}}\omega_{L,p}(\uu,\vv,\ww) \\ & =\sum_{(\xx,\yy,\zz)\in W_{p}^{\ast}(N)}\int_{\substack{\{(\uu,\vv,\ww)\in \ZZ_{p}^{n+2}, \\ [\uu,\vv,\ww]_{N}=(\xx,\yy,\zz)   \\  F(\uu,\vv,\ww)=0\}}}\omega_{L,p}(\uu,\vv,\ww).
\end{align}
Puisque $ Y $ est lisse, il existe un $ N>0 $ assez grand tel que, pour tout $ (\xx,\yy,\zz)\in (\ZZ_{p}/p^{N})^{n+2} $ tel que $ \xx\not\equiv \0\;  (p) $, $(\yy,\zz)\not\equiv \0\; (p)$, et $ F(\xx,\yy,\zz)=0 $  : \[ c=\inf_{i,j,k}\lbrace v_{p}\left(\frac{\partial F}{\partial x_{i}}(\xx,\yy,\zz)\right), \; v_{p}\left(\frac{\partial F}{\partial y_{j}}(\xx,\yy,\zz)\right), \; v_{p}\left(\frac{\partial F}{\partial z_{k}}(\xx,\yy,\zz)\right)\rbrace \] soit non nul et constant sur la classe d\'efinie par $ (\xx,\yy,\zz) $. On peut supposer que $ N>c $ et que $ c=v_{p}\left(\frac{\partial F}{\partial z_{n+1}}(\xx,\yy,\zz)\right) $. On consid\`ere $ (\uu,\vv,\ww)\in \ZZ_{p}^{n+2} $ tel que $ [\uu,\vv,\ww]_{N}=(\xx,\yy,\zz) $, et $ (\uu',\vv',\ww')\in \ZZ_{p}^{n+2} $ quelconque. On a alors \begin{multline*} 
F(\uu+\uu',\vv+\vv',\ww+\ww')=F(\uu,\vv,\ww)+\sum_{i=0}^{r}\frac{\partial F}{\partial x_{i}}(\uu,\vv,\ww)u_{i}'\\+\sum_{j=r+1}^{m}\frac{\partial F}{\partial y_{j}}(\uu,\vv,\ww)v_{j}' +\sum_{k=m+1}^{n+1}\frac{\partial F}{\partial z_{k}}(\uu,\vv,\ww)w_{k}'+G(\uu,\vv,\ww,\uu',\vv',\ww'), 
\end{multline*}

o\`u $ G(\uu,\vv,\ww,\uu',\vv',\ww') $ est une somme de termes contenant au moins deux facteurs $ u_{i}' $, $ v_{j}' $ ou $ w_{k}' $. Ainsi, on a donc, si $ (\uu',\vv',\ww')\in (p^{N}\ZZ_{p})^{n+2} $ : \[ F(\uu+\uu',\vv+\vv',\ww+\ww')\equiv F(\uu,\vv,\ww) \; (p^{N+c}). \] Par cons\'equent, l'image de $ F(\uu,\vv,\ww) $ dans $ \ZZ_{p}/p^{N+c} $ d\'epend uniquement de $ (\uu,\vv,\ww)\mod p^{N}=(\xx,\yy,\zz) $, on note alors $ F^{\ast}(\xx,\yy,\zz) $ cette image. \\

 Si $ F^{\ast}(\xx,\yy,\zz)\neq 0 $, alors l'int\'egrale \[ \int_{\substack{\{(\uu,\vv,\ww)\in \ZZ_{p}^{n+2}, \\ [\uu,\vv,\ww]_{N}=(\xx,\yy,\zz)   \\  F(\uu,\vv,\ww)=0\}}}\omega_{L,p}(\uu,\vv,\ww) \] est nulle, et l'ensemble \[ \{(\uu,\vv,\ww) \mod p^{N+c} \; | \; [\uu,\vv,\ww]_{N}=(\xx,\yy,\zz), \; F(\uu,\vv,\ww)\equiv 0\; (p^{N+c}) \} \] est vide. \\

Si $ F^{\ast}(\xx,\yy,\zz)= 0 $ alors, par le lemme de Hensel, les applications coordonn\'ees $ X_{0},...,X_{r},Y_{r+1},...,Y_{m},Z_{m+1},...,Z_{n} $ d\'efinissent un isomorphisme de \[ \{(\uu,\vv,\ww)\in \ZZ_{p}^{n+2} \; | \; [\uu,\vv,\ww]_{N}=(\xx,\yy,\zz), \; F(\uu,\vv,\ww)=0 \} \] sur \[(\xx,\yy,\hat{\zz})+(p^{N}\ZZ_{p})^{n+1}, \] o\`u $ \hat{\zz}=(z_{m+1},...,z_{n}) $. Par cons\'equent, on a : \begin{multline*}
\int_{\substack{\{(\uu,\vv,\ww)\in \ZZ_{p}^{n+2}, \\ [\uu,\vv,\ww]_{N}=(\xx,\yy,\zz)   \\  F(\uu,\vv,\ww)=0\}}}\omega_{L,p}(\uu,\vv,\ww)  \\ =\int_{(\xx,\yy,\hat{\zz})+(p^{N}\ZZ_{p})^{n+1}}p^{c}du_{0,p}...du_{r,p}dv_{r+1,p}...dv_{m,p}dw_{m+1,p}...dw_{n,p} =p^{c-N(n+1)}.
\end{multline*} 
On a d'autre part, puisque $ F(\uu,\vv,\ww) \mod p^{N+c} $ ne d\'epend que de $ (\xx,\yy,\zz) $ : \begin{multline*}
p^{-(N+c)(n+1)}\card\{(\uu,\vv,\ww)\mod p^{N+c}\; | \; [\uu,\vv,\ww]_{N}=(\xx,\yy,\zz),\; \\F(\uu,\vv,\ww)\equiv 0(p^{N+c}) \} =p^{-(N+c)(n+1)}p^{(n+1)c}  = p^{c-N(n+1)}.
\end{multline*}
On a donc finalement : \begin{multline*}
\int_{\substack{\{(\xx,\yy,\zz)\in \ZZ_{p}^{n+2}, \; \xx \not\equiv \0 \; (p) \\ (\yy,\zz)\not\equiv \0\; (p), \;  F(\xx,\yy,\zz)=0\}}}\omega_{L,p}  =\sum_{\substack{(\xx,\yy,\zz)\in W_{p}^{\ast}(N) \\ F^{\ast}(\xx,\yy,\zz)=0}}p^{c-N(n+1)} \\ = \sum_{\substack{(\xx,\yy,\zz)\in W_{p}^{\ast}(N) }}p^{-(N+c)(n+1)}\card\{(\uu,\vv,\ww)\mod p^{N+c}\; | \; [\uu,\vv,\ww]_{N}=(\xx,\yy,\zz),\;\\ F(\uu,\vv,\ww)\equiv 0(p^{N+c}) \}  = \frac{M_{p}^{\ast}(N+c)}{p^{(N+c)(n+1)}}.
\end{multline*}

D'o\`u le r\'esultat.

\end{proof}

\begin{lemma}\label{lemmeinter2}
On a \[ \int_{\substack{(\xx,\yy,\zz)\in \ZZ_{p}^{n+2} \\ \xx\not\equiv \0\; (p), \; (\yy,\zz)\not\equiv \0\; (p) \\ F(\xx,\yy,\zz)=0 }}\omega_{L,p}=\left(1-\frac{1}{p^{n-r+1-d_{2}}}\right)\int_{\substack{(\xx,\yy,\zz)\in \ZZ_{p}^{n+2} \\ \xx\not\equiv \0\; (p), \; F(\xx,\yy,\zz)=0 }}\omega_{L,p}, \] et d'autre part \[ \lim_{N\ra \infty }\frac{M_{p}^{\ast}(N)}{p^{N(n+1)}}=\left(1-\frac{1}{p^{n-r+1-d_{2}}}\right)\sigma_{p}'. \]
\end{lemma}
\begin{proof}
La premi\`ere partie du lemme r\'esulte du fait que : \[  \omega_{L,p}(\xx,p\yy,p\zz)=p^{-(n-r+1-d_{2})}\omega_{L,p}(\xx,\yy,\zz). \] Pour la deuxi\`eme partie, on consid\`ere un entier $ j $ tel que $ N\geqslant jd_{2}+1 $ et on consid\`ere l'ensemble \begin{multline*} \tilde{N}(j)=\Card\{\xx\in (\ZZ_{p}/p^{N}\ZZ_{p})^{r+1}, \;(\yy,\zz) \in (p^{j}\ZZ_{p}/p^{N}\ZZ_{p})^{n-r+1} \; | \; \xx\not\equiv \0 \; (p), \\  \; (\yy,\zz)\not\equiv \0\; (p^{j+1}), F(\xx,\yy,\zz)\equiv 0 \; (p^{N}) \} \end{multline*}
On remarque que, pour tout $ N>jd_{2} $ \begin{multline*} \tilde{N}(j)=\Card\{\xx\in (\ZZ/p^{N}\ZZ)^{r+1}, \;(\yy,\zz) \in (\ZZ/p^{N-j}\ZZ)^{n-r+1} \; | \; \xx\not\equiv \0 \; (p), \\  \; (\yy,\zz)\not\equiv \0\; (p), F(\xx,\yy,\zz)\equiv 0 \; (p^{N-jd_{2}}) \} \\ = p^{(r+1)jd_{2}+(n-r+1)(jd_{2}-j)}M^{\ast}(N-jd_{2}) \end{multline*}
Soit $ N_{0} $ comme dans le lemme pr\'ec\'edent, et soit $ j_{0}=\lceil (N-N_{0})/d_{2}\rceil $. On a alors \begin{multline*}
M_{p}(N)=\sum_{0\leqslant jd_{2} \leqslant N-N_{0}}\tilde{N}(j) \\ + O\left(\Card\{(\xx,\yy,\zz)\in (\ZZ/p^{N}\ZZ)^{n+2} \; | \; (\yy,\zz)\equiv \0 \; (p^{j_{0}}) \}\right) \\ =\sum_{0\leqslant jd_{2}\leqslant N-N_{0} }p^{(r+1)jd_{2}+(n-r+1)(jd_{2}-j)}M_{p}^{\ast}(N-jd_{2})+O\left(p^{N(n+2)-j_{0}(n-r+1)}\right)
\end{multline*}
Or, d'apr\`es le lemme pr\'ec\'edent : \[ \frac{M_{p}^{\ast}(N-jd_{2})}{p^{(N-jd_{2})(n+1)}}=\frac{M_{p}^{\ast}(N)}{p^{N(n+1)}}, \] on obtient donc : \begin{align*} M_{p}(N) & =\sum_{0\leqslant j\leqslant N-N_{0} }p^{-j(n-r+1)+jd_{2}}M_{p}^{\ast}(N)+O\left(p^{N(n+2)-j_{0}(n-r+1)}\right) \\ & =M_{p}^{\ast}(N)\frac{1-p^{-(N-N_{0}+1)(n-r+1-d_{2})}}{1-p^{-(n-r+1-d_{2})}}+O\left(p^{N(n+2)-j_{0}(n-r+1)}\right), \end{align*} et puisque  $ \sigma_{p}'=\lim_{N\ra\infty}\frac{M_{p}(N)}{p^{N(n+1)}} $, on obtient le r\'esultat.

\end{proof}
On d\'eduit des lemmes \ref{lemmeinter1} et \ref{lemmeinter2} que \begin{equation}\label{sigma'} \sigma_{p}'=\int_{\substack{(\xx,\yy,\zz)\in \ZZ_{p}^{n+2} \\ \xx\not\equiv \0\; (p), \; F(\xx,\yy,\zz)=0 }}\omega_{L,p}.  \end{equation}

On conclut alors en utilisant le lemme ci-dessous : 

\begin{lemma}
On pose : \[ a(p)=\left(1-\frac{1}{p}\right)^{2}\left(1-\frac{1}{p^{n-r+1-d_{2}}}\right)^{-1}. \]
On a alors \[ \int_{\substack{(\xx,\yy,\zz)\in \ZZ_{p}^{n+2} \\ \xx\not\equiv \0\; (p), \; F(\xx,\yy,\zz)=0 }}\omega_{L,p}=\int_{\substack{Y_{0}(\QQ_{p})\cap \{|\xx|_{p}= 1 \\h_{p}^{2}(\xx,\yy,\zz)\leqslant 1 \}}}\omega_{L,p}(\xx,\yy,\zz)=a(p)\omega_{p}(Y(\QQ_{p})). \]
\end{lemma}
\begin{proof}
Il suffit de montrer que pour tout ouvert $ V $ de $ \AA_{\QQ_{p}}^{n-1}\subset X(\QQ_{p}) $ tel que pour tout $ P=\pi(\xx,\yy,\zz)\in V $ on a (par exemple) $ x_{0}z_{m+1}\neq 0 $ et $ \frac{\partial F}{\partial z_{n+1}}(\xx,\yy,\zz)\neq 0 $ (les autres cas se traitant de fa\c{c}on analogue) l'\'egalit\'e \[ \int_{\substack{\pi^{-1}(V)\cap\pi^{-1}(Y) \\ \cap \{|\xx|_{p}= 1, \;  h_{p}^{2}(\xx,\yy,\zz)\leqslant 1 \}}}\omega_{L,p}(\xx,\yy,\zz)=a(p)\omega_{p}(V\cap Y) \] est v\'erifi\'ee. Remarquons dans un premier temps que, pour un tel ouvert $ V $, on a\begin{align*} \left(1-\frac{1}{p}\right)\omega_{p}(V\cap Y)= \left(1-\frac{1}{p}\right)\int_{V\cap Y}\frac{du_{1,p}...du_{r,p}dv_{r+1,p}...dv_{m,p}dw_{m+2,p}...dw_{n,p}}{\left|\frac{\partial F}{\partial z_{n+1}}(\uu,\vv,\ww)\right|_{p}h_{p}^{1}(\uu) h_{p}^{2}(\uu,\vv,\ww) } \\ . \end{align*}
En appliquant deux fois le lemme $ 5.4.5 $ de \cite{Pe}, on obtient alors : \begin{multline*} \left(1-\frac{1}{p}\right)^{2}\left(1-\frac{1}{p^{m+1-d_{1}}}\right)^{-1}\left(1-\frac{1}{p^{n-r+1-d_{2}}}\right)^{-1}\omega_{p}(V) \\ =\int_{\substack{\pi^{-1}(V)\cap \pi^{-1}(Y) \\ \cap \{h_{p}^{(1)}(\xx)\leqslant 1, \;h_{p}^{2}(\xx,\yy,\zz)\leqslant 1 \}}}\frac{d\xx d\yy d\hat{\zz}}{\left|\frac{\partial F}{\partial z_{n+1}}(\xx,\yy,\zz)\right|_{p}} \\ =\int_{\substack{\pi^{-1}(V)\cap \pi^{-1}(Y)\\ \cap \{h_{p}^{(1)}(\xx)\leqslant  1, \; h_{p}^{2}(\xx,\yy,\zz)\leqslant 1 \}}}\omega_{L,p}(\xx,\yy,\zz). \end{multline*}
Or, \'etant donn\'e que  \[ \omega_{L,p}(p\xx,p\yy,\zz)=p^{-(m+1-d_{1})}\omega_{L,p}(\xx,\yy,\zz), \] on a \begin{multline*} \int_{\substack{(\xx,\yy,\zz)\in \pi^{-1}(V) \cap\pi^{-1}(Y) \\ \cap\{h_{p}^{(1)}(\xx)\leqslant  1, \; h_{p}^{2}(\xx,\yy,\zz)\leqslant 1  }}\omega_{L,p} \\ =\left(1-\frac{1}{p^{m+1-d_{1}}}\right)^{-1}\int_{\substack{X_{0}(\QQ_{p})\cap \pi^{-1}(Y) \\ \cap \{|\xx|_{p}= 1, \;h_{p}^{2}(\xx,\yy,\zz)\leqslant 1 \}}}\omega_{L,p}(\xx,\yy,\zz) \end{multline*}
et on obtient le r\'esultat souhait\'e.

\end{proof}

On d\'eduit de ce lemme et de la formule \eqref{sigma'} que \[ \left(1-\frac{1}{p^{n-r+1-d_{2}}}\right)\sigma_{p}'=\left(1-\frac{1}{p}\right)^{2}\omega_{p}(Y(\QQ_{p}))=\tau_{p}(Y(\QQ_{p})). \]

\subsection{Conclusion}

Rappelons que la formule asymptotique conjectur\'ee par Peyre dans \cite{Pe}, dans sa version corrig\'ee par Batyrev et Tschinkel, pour le nombre $ \mathcal{N}_{U}(B) $ de points de hauteur born\'ee  par $ B $ sur l'ouvert $ U $ de Zariski de la vari\'et\'e $ Y $ (pour la hauteur associ\'ee au fibr\'e anticanonique $ \omega_{Y}^{-1} $) est : \begin{equation}\alpha(Y)\beta(Y)\tau_{H}(Y)B\log(B)^{\rg(\Pic(Y))-1} \end{equation} 
o\`u \[ \alpha(Y)= \frac{1}{(\rg(\Pic(Y))-1)!}\int_{\Lambda_{\eff}^{1}(Y)^{\vee}}e^{-\langle \omega_{Y}^{-1},y\rangle}dy, \] \[ \Lambda_{\eff}^{1}(Y)^{\vee}=\{y\in \Pic(Y)\otimes \RR^{\vee}\; |\; \forall x\in \Lambda_{\eff}^{1}(Y), \langle x,y\rangle\geqslant 0 \} \] et \[ \beta(Y)=\card(H^{1}(\QQ,\Pic(\overline{Y}))), \] \[ \tau_{H}(Y)=\prod_{\nu\in \Val(\QQ)}\tau_{\nu}(Y(\QQ_{\nu})). \]
Dans le cas pr\'esent on a \[ \Pic(Y)=\ZZ[\tilde{D}_{0}]\oplus\ZZ[\tilde{D}_{n+1}]\simeq \ZZ^{2}, \;\; \rg(\Pic(Y))=2, \] \[ -[K_{Y}]=(m+1-d_{1})[\tilde{D}_{0}]+ (n-r+1-d_{2})[\tilde{D}_{n+1}], \] \[\Lambda_{\eff}^{1}(Y)=\RR^{+}[\tilde{D}_{0}]+\RR^{+}[\tilde{D}_{n+1}]\simeq (\RR^{+})^{2}. \] On a par cons\'equent : \begin{align*}
\alpha(Y)=\int_{[0,+\infty[^{2}}e^{-(m+1-d_{1})t_{1}-(n-r+1-d_{2})t_{2}}dt_{1}dt_{2}=\frac{1}{(m+1-d_{1})(n-r+1-d_{2})}. 
\end{align*} 
D'autre part $ \Pic(\overline{Y})\simeq \ZZ^{3} $, et le groupe de Galois $ \Gal(\overline{\QQ}/\QQ) $ agit trivialement sur $ \Pic(\overline{Y}) $, on a donc \[ \beta(Y)=1. \] Par ailleurs, d'apr\`es ce qui a \'et\'e vu dans les sections pr\'ec\'edentes, on a \[ \prod_{p\in \mathcal{P}}\tau_{p}(Y(\QQ_{p}))=\mathfrak{S}\prod_{p\in \mathcal{P}}\left(1-\frac{1}{p^{n-r+1-d_{2}}}\right) \] et \[ \tau_{\infty}(Y(\RR))=\frac{(m+1-d_{1})(n-r+1-d_{2})}{4}J. \]
Ainsi on a : \begin{equation*} \alpha(Y)\beta(Y)\tau_{H}(Y)B\log(B)^{\rg(\Pic(Y))-1} =\frac{1}{4}\mathfrak{S}J\prod_{p\in \mathcal{P}}\left(1-\frac{1}{p^{n-r+1-d_{2}}}\right)B\log(B), \end{equation*} et on retrouve bien la formule de la proposition \ref{conclusion}. 
Nous avons donc d\'emontr\'e le r\'esultat ci-dessous : 
\begin{thm}\label{thmconcl}
Pour $ d_{1},d_{2}\geqslant 2 $, $ n+2-\max\{\dim V_{1}^{\ast},\dim V_{2}^{\ast}\} >m\geqslant 2^{d_{1}+d_{2}} $ (avec $ m\leqslant 13d_{2}(d_{1}+d_{2})2^{d_{1}+d_{2}} $) et $ r\geqslant 6d_{1}-3 $, on a : \[\mathcal{N}_{U}(B)=  C_{H}(X) B\log(B)+O(B), \] o\`u $ C_{H}(X) $ est la constante conjectur\'ee par Peyre, lorsque $ B \ra \infty $. On a de plus la m\^eme formule pour $ d_{1}\geqslant 2 $, $ d_{2}=1 $, $ n+2-\max\{\dim V_{1}^{\ast},\dim V_{2}^{\ast}\} >m'\geqslant 2^{d_{1}+d_{2}} $ et $ r\geqslant 6d_{1}-3 $.
\end{thm}

\end{document}